\documentclass[10pt,a4paper,twoside,headinclude=true]{article} 

\usepackage{lipsum}
\usepackage{url}
\usepackage[T1]{fontenc}
\usepackage[utf8]{inputenc}
\usepackage[english]{babel}
\usepackage{mathrsfs}
\usepackage{rsfso}
\usepackage{amsmath}
\usepackage{amsthm,amssymb}
\usepackage{thmtools}
\usepackage{geometry}
\usepackage{enumitem}
\usepackage{hyperref}
\usepackage[nochapters, dottedtoc]{classicthesis} 
\usepackage{titling} 
\usepackage{comment}
\usepackage{imakeidx, multind}
\makeindex[name=symbol,title=Index of notations and symbols]
\makeindex[name=notion,title=Glossary]

\swapnumbers

\declaretheoremstyle[
spaceabove=\topsep, spacebelow=\topsep,
headfont=\scshape,
notefont=\mdseries, notebraces={(}{)},
bodyfont=\itshape,
headpunct=.~---,
postheadspace=0.5em,
qed=$\blacktriangle$
]{mythmstyle}
\declaretheorem[style=mythmstyle, name=Theorem, numberwithin=subsection]{thm}
\declaretheorem[style=mythmstyle, name=Lemma, sibling=thm]{lem}
\declaretheorem[style=mythmstyle, name=Proposition, sibling=thm]{prop}
\declaretheorem[style=mythmstyle, name=Proposition-Definition, sibling=thm]{propdef}
\declaretheorem[style=mythmstyle, name=Corollary, sibling=thm]{cor}

\declaretheoremstyle[
spaceabove=\topsep, spacebelow=\topsep,
headfont=\scshape,
notefont=\mdseries, notebraces={(}{)},
bodyfont=\normalfont,
headpunct=.~---,
postheadspace=0.5em,
qed=$\blacktriangle$
]{mydefstyle}
\declaretheorem[style=mydefstyle, name=Definition, sibling=thm]{defin}

\declaretheorem[style=mydefstyle, name=Example, sibling=thm]{ex}
\declaretheorem[style=mydefstyle, name=Examples, sibling=thm]{exs}
\declaretheorem[style=mydefstyle, name=Remark, sibling=thm]{rk}
\declaretheorem[style=mydefstyle, name=Remarks, sibling=thm]{rks}
\declaretheorem[style=mydefstyle, name=Notation, sibling=thm]{nb}
\declaretheorem[style=mydefstyle, name=Notations, sibling=thm]{nbs}
\declaretheorem[style=mydefstyle, name=Convention, sibling=thm]{conv}

\declaretheoremstyle[
spaceabove=\topsep, spacebelow=\topsep,
headfont=\scshape,
notefont=\mdseries, notebraces={(}{)},
bodyfont=\normalfont,
headpunct=.,
postheadspace=0.5em,
qed=$\blacktriangle$
]{mynoheadstyle}
\declaretheorem[style=mynoheadstyle, name={}, sibling=thm]{noh}

\numberwithin{equation}{section}

\newcommand{\M}{\mathcal{M}}
\newcommand{\B}{\mathcal{B}}
\newcommand{\K}{\mathcal{K}}
\newcommand{\Lin}{\mathcal{L}}
\newcommand{\s}[1]{\mathscr{{#1}}}
\newcommand{\f}[1]{\mathfrak{{#1}}}
\newcommand{\GR}{{\mathbb R}}
\newcommand{\GC}{{\mathbb C}}

\newcommand{\QG}{\mathbb{G}}
\newcommand{\QH}{\mathbb{H}}

\newcommand{\C}{\mathcal{C}}

\newcommand{\tens}{\otimes}

\newcommand{\cst}[5]{{#1}\underset{{#3}}{_{#2}\tens_{#4}}{#5}}
\newcommand{\reltens}[4]{{#1}{\,_{#2}\tens_{#3}\,}{#4}}

\newcommand{\fprod}[4]{{#1}{\,_{#2}\star_{#3}\,}{#4}}

\newcommand{\nsf}{n.s.f.~}
\newcommand{\GNS}{G.N.S.~}

\newcommand{\ext}{\mathrm{ext}}
\newcommand{\id}{\mathrm{id}}

\newcommand{\restr}[2]{{#1}\!\!\restriction_{#2}}
\newcommand{\ind}{{\rm Ind}_{\QG_1}^{\QG_2}}
\newcommand{\iind}{{\rm Ind}_{\QG_2}^{\QG_1}}
\newcommand{\eo}{{\cal E}_0}
\newcommand{\er}{{\cal E}_{A,R}}
\newcommand{\delr}{{\delta_{\er}}}
\newcommand{\delo}{{\delta_{\eo}}}
\newcommand{\betr}{{\beta_{\er}}}

\newcommand{\kk}{{\sf KK}}

\makeatletter
\newcommand{\raisemath}[1]{\mathpalette{\raisem@th{#1}}}
\newcommand{\raisem@th}[3]{\raisebox{#1}{$#2#3$}}
\makeatother

\pdfmapfile{+rsfso.map}
\DeclareSymbolFont{rsfso}{U}{rsfso}{m}{n}
\DeclareSymbolFontAlphabet{\mathscr}{rsfso}

\begin{document}

\pagestyle{scrheadings}

\clearscrheadfoot
\cohead{\footnotesize{{\sc ACTIONS OF MEASURED QUANTUM GROUPOIDS ON A FINITE BASIS}}}
\rohead{\thepage}
\cehead{\footnotesize{{\sc J.\ CRESPO}}}
\lehead{\thepage}


\setcounter{tocdepth}{2} 

    \title{\rmfamily\normalfont\spacedallcaps{Actions of measured quantum groupoids on a finite basis}}
    \author{\small{by}\\[1em] \spacedlowsmallcaps{jonathan crespo}}
    \date{}
    \maketitle
	\thispagestyle{empty}

\vspace{-3em}

    \begin{abstract}
        \noindent In this article, we generalize to the case of measured quantum groupoids on a finite basis some important results concerning actions of locally compact quantum groups on C*-algebras \cite{BSV}. Let $\cal G$ be a measured quantum groupoid on a finite basis. We prove that if $\cal G$ is regular, then any weakly continuous action of $\cal G$ on a C*-algebra is necessarily strongly continuous. Following \cite{BS1}, we introduce and investigate a notion of $\cal G$-equivariant Hilbert C$^*$-modules. By applying the previous results and a version of the Takesaki-Takai duality theorem obtained in \cite{BC} for actions of $\cal G$, we obtain a canonical equivariant Morita equivalence between a given $\cal G$-C$^*$-algebra $A$ and the double crossed product $(A\rtimes{\cal G})\rtimes\widehat{\cal G}$.

\medbreak        
        
    \end{abstract}
     
    \tableofcontents

    
\section*{Introduction}\addcontentsline{toc}{section}{Introduction}

The notion of monoidal equivalence of compact quantum groups has been introduced by Bichon, De Rijdt and Vaes in \cite{BRV}. Two compact quantum groups $\QG_1$ and $\QG_2$ are said to be monoidally equivalent if their categories of representations are equivalent as monoidal C*-categories. They have proved that $\QG_1$ and $\QG_2$ are monoidally equivalent if and only if there exists a unital C*-algebra equipped with commuting continuous ergodic actions of full multiplicity of $\QG_1$ on the left and of $\QG_2$ on the right.\hfill\break
Many crucial results of the geometric theory of free discrete quantum groups rely on the monoidal equivalence of their dual compact quantum groups. Among the applications of monoidal equivalence to this theory, we mention the contributions to randow walks and their associated boundaries \cite{VV,RV}, CCAP property and Haagerup property \cite{DFY}, the Baum-Connes conjecture and K-amenability \cite{V1,VeVo}.

\medskip

In his Ph.D.~thesis \cite{DC}, De Commer has extended the notion of monoidal equivalence to the locally compact case. Two locally compact quantum groups $\QG_1$ and $\QG_2$ (in the sense of Kustermans and Vaes \cite{KV2}) are said to be monoidally equivalent if there exists a von Neumann algebra equipped with a left Galois action of $\QG_1$ and a right Galois action of $\QG_2$ that commute. He proved that this notion is completely encoded by a measured quantum groupoid (in the sense of Enock and Lesieur \cite{E08}) on the basis $\GC^2$. Such a groupoid is called a colinking measured quantum groupoid.\hfill\break
The measured quantum groupoids  have been introduced and studied by Lesieur and Enock (see \cite{E08,Le}). Roughly speaking, a measured quantum groupoid (in the sense of Enock-Lesieur) is an octuple ${\cal G}=(N,M,\alpha,\beta,\Gamma,T,T',\nu)$, where $N$ and $M$ are von Neumann algebras (the basis $N$ and $M$ are the algebras of the groupoid corresponding respectively to the space of units and the total space for a classical groupoid), $\alpha$ and $\beta$ are faithful normal *-homomorphisms from $N$ and $N^{\rm o}$ (the opposite algebra) to $M$ (corresponding to the source and target maps for a classical groupoid) with commuting ranges, $\Gamma$ is a coproduct taking its values in a certain fiber product, $\nu$ is a normal semi-finite weight on $N$ and $T$ and $T'$ are operator-valued weights satisfying some axioms. \hfill\break
In the case of a finite-dimensional basis $N$, the definition has been greatly simplified by De Commer \cite{DC2,DC} and we will use this point of view in this article. Broadly speaking, we can take for $\nu$ the non-normalized Markov trace on the C*-algebra $N=\bigoplus_{1\leqslant l\leqslant k}{\rm M}_{n_l}(\GC)$\index[symbol]{ma@${\rm M}_{n_l}(\GC)$, square matrices of order $n_l$ with entries in $\GC$}. The relative tensor product of Hilbert spaces (resp.\ the fiber product of von Neumann algebras) is replaced by the ordinary tensor product of Hilbert spaces (resp.\  von Neumann algebras). The coproduct takes its values in $M\tens M$ but is no longer unital. In the following, these objects will be referred to as 'measured quantum groupoids on a finite basis'.

\medskip

In \cite{BC}, the authors introduce a notion of (strongly) continuous actions on C*-algebras of measured quantum groupoids on a finite basis. They extend the construction of the crossed product, the dual action and give a version of the Takesaki-Takai duality generalizing the Baaj-Skandalis duality theorem \cite{BS1} in this setting.\hfill\break
If a colinking measured quantum groupoid ${\cal G}$, associated with a monoidal equivalence of two locally compact quantum groups $\QG_1$ and $\QG_2$, acts (strongly) continuously on a C*-algebra $A$, then $A$ splits up as a direct sum $A=A_1\oplus A_2$ of C*-algebras and the action of ${\cal G}$ on $A$ restricts to an action of $\QG_1$ (resp.\ $\QG_2$) on $A_1$ (resp.\ $A_2$).\hfill\break
They also extend the induction procedure to the case of monoidally equivalent regular locally compact quantum groups. To any continuous action of $\QG_1$ on a C*-algebra $A_1$, they associate canonically a C*-algebra $A_2$ endowed with a continuous action of $\QG_2$. As important consequences of this construction, we mention the following:
\begin{itemize}
	\item a one-to-one functorial correspondence between the continuous actions of the quantum groups $\QG_1$ and $\QG_2$, which generalizes the compact case \cite{RV} and the case of deformations by a 2-cocycle \cite{NT14};
	\item a complete description of the continuous actions of colinking measured quantum groupoids;
	\item the equivalence of the categories $\kk_{\QG_1}$ and $\kk_{\QG_2}$, which generalizes to the regular locally compact case a result of Voigt \cite{V1}.
\end{itemize}
The proofs of the above results rely crucially on the regularity of the quantum groups $\QG_1$ and $\QG_2$. They prove that the regularity of $\QG_1$ and $\QG_2$ is equivalent to the regularity in the sense of \cite{E05} (see also \cite{Tim1,Tim2}) of the associated colinking measured quantum groupoid. By pas\-sing, this result solves some questions raised in \cite{NT14} in the case of deformations by a 2-cocycle.

\medskip

In this article, we generalize to the case of (semi-)regular measured quantum groupoid on a finite basis some important properties of (semi-)regular locally compact quantum groups \cite{BS2,Ba95}. This work will give us enough formulas to generalize some crucial results of \cite{BSV} concerning actions of (semi-)regular locally compact quantum groups.\newline
More precisely, if $\cal G$ is a semi-regular measured quantum groupoid on a finite basis, then the space consisting of the continuous elements of any action of $\cal G$ is a C*-algebra. Moreover, if $\cal G$ is regular we prove that any weakly continuous action of $\cal G$ is necessarily continuous in the strong sense.\newline
We introduce a notion of action of $\cal G$ on Hilbert C*-modules in line with the corresponding notion for quantum groups \cite{BS1}. A $\cal G$-equivariant Hilbert C*-module is a Hilbert C*-module endowed with a continuous action (in a sense that will be specified). By using the previous result, if $\cal G$ is regular we prove that any action of $\cal G$ on a Hilbert C*-module is necessarily continuous. We are able to define the notion of $\cal G$-equivariant Morita equivalence of $\cal G$-C*-algebras. By applying a version of the Takesaki-Takai duality theorem obtained in \cite{BC}, we prove that any $\cal G$-C*-algebra $A$ is $\cal G$-equivariantly Morita equivalent to its double crossed product $(A\rtimes{\cal G})\rtimes\widehat{\cal G}$ in a canonical way.\newline

This article is organized as follows.\newline
$\bullet$ {\it Chapter 1.} We recall the general conventions and notations used throughout this paper.\newline
$\bullet$ {\it Chapter 2.} We make an overview of the theory of locally compact quantum groups (cf.\ \cite{KV2} and \cite{BS2}). We recall the construction of the Hopf C*-algebra associated with a locally compact quantum group and the notion of action of locally compact quantum groups in the C*-algebraic setting. We also recall the notion of equivariant Hilbert C*-modules (cf.\ \cite{BS1}).\newline
$\bullet$ {\it Chapter 3.} We make a very brief survey of the theory of measured quantum groupoid (cf.\ \cite{Le,E08}) and we recall the simplified definition in the case where the basis is finite-dimensional and the associated C*-algebraic structure provided by De Commer in \cite{DC2,DC}. In the last section, we make an outline of the reflection technique across a Galois object provided by De Commer (cf.\ \cite{DC,DC3}), the construction and the structure of the colinking measured quantum groupoid associated with monoidally equivalent locally compact quantum groups. We also recall the precise description of the C*-algebraic structure of colinking measured quantum groupoids (cf.\ \cite{BC}).\newline
$\bullet$ {\it Chapter 4.} In this chapter, we make a review of the notions of regularity and semi-regularity for measured quantum groupoids on a finite basis (cf.\ \cite{E05,Tim1,Tim2,BC}) and we obtain new relations equivalent to the (semi-)regularity generalizing some results of Baaj and Skandalis \cite{BS2,Ba95}. Given a (semi-)regular measured quantum groupoid, we derive new relations that will play a crucial role in the subsequent chapters.\newline
$\bullet$ {\it Chapter 5.} In the first section of this chapter, we recall the definitions and the main results of \cite{BC} concerning the notion of (strongly) continuous action of measured quantum groupoids on a finite basis on C*-algebras. We also recall the version of the Takesaki-Takai duality theorem obtained in \cite{BC} in this framework. The second section is dedicated to a brief overview of  C*-algebras acted upon by a colinking measured quantum groupoid (cf.\ \cite{BC}). In the last section, we generalize to the setting of measured quantum groupoids on a finite basis the results of Baaj, Skandalis and Vaes \cite{BSV} concerning the notion of weakly/strongly continuous action of (semi-)regular locally compact quantum groups.\newline
$\bullet$ {\it Chapter 6.} We introduce the notion of action of measured quantum groupoid on a finite basis on Hilbert C*-module and we investigate in detail the case of a colinking measured quantum groupoid. In the last paragraph, we provide a direct approach of the induction procedure for equivariant Hilbert C*-modules equivalent to that obtained in \cite{BC}. In particular, if ${\cal G}_{\QG_1,\QG_2}$ is a colinking measured quantum groupoid associated with two monoidally equivalent regular locally compact groups $\QG_1$ and $\QG_2$ we obtain one-to-one correspondences between the actions of $\QG_1$, $\QG_2$ and ${\cal G}_{\QG_1,\QG_2}$ on Hilbert C*-modules.\newline
$\bullet$ {\it Chapter 7.} In this chapter, we introduce and discuss the notion of equivariant Morita equivalence. Given a $\cal G$-C*-algebra $A$, we prove that $A$ and its double crossed product $(A\rtimes{\cal G})\rtimes\widehat{\cal G}$ are $\cal G$-equivariantly Morita equivalent in a canonical way.\newline
$\bullet$ {\it Chapter 8.} In the appendix of this article, we have assembled a very brief review of the main notions and notations of the non-commutative measure and integration theory. We can also find some notations and important results used throughout this paper.

\medskip

In a forthcoming article \cite{C2}, we use the results of this paper to generalize those of Baaj and Skandalis concerning the equivariant Kasparov theory (cf.\ \S 6 \cite{BS1} and 7.7 b) \cite{BS2}).

\begin{center}
{\bf Acknowledgements}
\end{center}

Some of the results of this article were part of the author's Ph.D.\ thesis and he wishes
to thank his advisor Prof.\ S.\ Baaj for his supervision. The author is also very grateful to Prof.\ K.\ De Commer for fruitful discussions on measured quantum groupoids and for the financial support of the F.W.O.

\section{Preliminary notations}\label{sectionNotations}

We specify here some elementary notations and conventions used in this article. For more notations, we refer the reader to the appendix and the index of this article.

\medskip

$\bullet$ For all subset $X$ of a normed vector space $E$, we denote $\langle X\rangle$ (resp.\ $[X]$) the linear span (resp.\ closed linear span) of $X$ in $E$. If $X,\,Y\subset E$, we denote $XY:=\{xy\,;\, x\in X,\, y\in Y\}$, where $xy$ denotes the product/composition of $x$ and $y$ or the evaluation of $x$ at $y$ (when these operations make sense). If $X$ is a subset of a *-algebra $A$, we denote by $X^*$ the subset $\{x^*\,;\, x\in X\}$ of $A$.\hfill\break
$\bullet$ We denote by $\tens$ the tensor product of Hilbert spaces, the tensor product of von Neumann algebras, the minimal tensor product of C*-algebras or the external tensor product of Hilbert C*-modules. We also denote by $\odot$ (resp.\ $\odot_A$) the algebraic tensor product over the field of complex numbers $\GC$ (resp.\ an algebra $A$).\newline
$\bullet$ Let $A$ and $B$ be C*-algebras. We denote by $\M(A)$ (resp.\ $\widetilde{A}$) the C*-algebra consisting of the multipliers of $A$ (resp.\ the C*-algebra obtained from $A$ by adjunction of a unit element). We denote by $\widetilde{\M}(A\tens B)$ (or $\widetilde{\M}_B(A\tens B)$ in case of ambiguity, \S 1 \cite{BS1}) the $B$-relative multiplier algebra, {\it i.e.\ }the C*-algebra consisting of the elements $m$ of $\M(\widetilde{A}\tens B)$ such that the relations $(\widetilde{A}\tens B)m\subset A\tens B$ and $m(\widetilde{A}\tens B)\subset A\tens B$ hold.\newline
Let $\pi:A\rightarrow\M(B)$ be a (possibly degenerate) *-homomorphism. For all C*-algebra $D$, there exists a unique strictly continuous *-homomorphism $\pi\tens\id_D:\widetilde{\M}(A\tens D)\rightarrow\M(B\tens D)$ satisfying the relation
$
(\pi\tens\id_D)(x)(1_B\tens d)=(\pi\tens\id_D)(x(1_A\tens d))
$ 
for all $x\in\widetilde{\M}(A\tens D)$ and $d\in D$.
Indeed, denote by $\widetilde{\pi}$ the unital extension of $\pi$ to $\widetilde{A}$. The non-degenerate *-homomorphism $\widetilde{\pi}\tens\id_D:\widetilde{A}\tens D\rightarrow\M(B\tens D)$ uniquely extends to $\M(\widetilde{A}\tens D)$. By restricting to $\widetilde{\M}(A\tens D)$, we obtain the desired extension of $\pi\tens\id_D$ (\S 1 \cite{BS1}).\newline
$\bullet$ If $x$ and $y$ are two elements of an algebra $A$, we denote by $[x,\,y]$ the commutator of $x$ and $y$, {\it i.e.} the element of $A$ defined by $[x,\,y]:=xy-yx$.

\medskip

Let $\s H$ and $\s K$ be Hilbert spaces (all inner products are assumed to be anti-linear in the first variable and linear in the second variable).\newline
$\bullet$ We denote by $\B(\s H,\s K)$ (resp.\ $\K(\s H,\s K)$) the Banach space of bounded (resp.\ compact) linear operators from $\s H$ to $\s K$. For all $\xi\in\s K$ and $\eta\in\s H$, we denote by $\theta_{\xi,\eta}\in\B(\s H,\s K)$ the rank-one operator defined by $\theta_{\xi,\eta}(\zeta):=\langle\eta,\,\zeta\rangle\xi$ for all $\zeta\in\s H$. We have the relation $\K(\s H,\s K)=[\theta_{\xi,\eta}\,;\,\xi\in\s K,\,\eta\in\s K]$. Denote by $\B(\s H):=\B(\s H,\s H)$ (resp.\ $\K(\s H):=\K(\s H,\s H)$) the C*-algebra of bounded (resp.\ compact) linear operators on $\s H$. Recall that $\K(\s H)$ is a closed two-sided ideal of $\B(\s H)$ and $\B(\s H)=\M(\K(\s H))$.\newline 
$\bullet$ We denote by $\Sigma_{\s K\tens\s H}$ (or simply $\Sigma$) the flip map, that is to say the unitary operator $\s K\tens\s H\rightarrow\s H\tens\s K\,;\, \xi\tens\eta\mapsto\eta\tens\xi$.\hfill\break
$\bullet$ For $u\in\B(\s H)$, we denote by ${\rm Ad}_u$ the bounded operator on $\B(\s H)$ defined for all $x\in\B(\s H)$ by ${\rm Ad}_u(x)=uxu^*$.

\medskip

In this article, we will use the notion of (right) Hilbert C*-module over a C*-algebra and their tensor products (internal and external). All the definitions and conventions are those of \cite{Kas1}. In particular, let $\s E$ and $\s F$ be two Hilbert C*-modules over a C*-algebra $A$.\hfill\break
\noindent$\bullet$ We denote by $\Lin(\s E,\s F)$ the Banach space consisting of all adjointable operators from $\s E$ to $\s F$ and $\Lin(\s E)$ the C*-algebra $\Lin(\s E,\s E)$.\hfill\break
$\bullet$ For $\xi\in\s F$ and $\eta\in\s E$, we denote by $\theta_{\xi,\eta}$ the elementary operator of $\Lin(\s E,\s F)$ defined by $\theta_{\xi,\eta}(\zeta):=\xi\langle\eta,\,\zeta\rangle_A$ for all $\zeta\in\s E$. Let $\K(\s E,\s F):=[\theta_{\xi,\eta}\,;\,\xi\in\s F,\,\eta\in\s E]$ be the Banach space of ``compact'' adjointable operators. Denote by $\K(\s E)$ the C*-algebra $\K(\s E,\s E)$ consisting of the compact adjointable operators of $\Lin(\s E)$. Recall that $\K(\s E)$ is a closed two-sided ideal of $\Lin(\s E)$ and $\Lin(\s E)=\M(\K(\s E))$.\hfill\break
$\bullet$ Let $\s E^*:=\K(\s E,A)$. We have $\s E^*=\{T\in\Lin(\s E,A)\,;\,\exists\,\xi\in\s E,\,\forall\,\eta\in\s E,\, T(\eta)=\langle\xi,\,\eta\rangle_A\}$. We will identify $\s E=\K(A,\s E)\subset\Lin(A,\s E)$. Up to this identification, for $\xi\in\s E$ the operator $\xi^*\in\s E^*$ satisfies $\xi^*(\eta)=\langle\xi,\,\eta\rangle_A$ for all $\eta\in\s E$. We recall that $\s E^*$ is a Hilbert $\K(\s E)$-module for the inner product defined by $\langle T,\, T'\rangle_{\K(\s E)}:=T^*\circ T'$ for $T,T'\in\s E^*$ and the right action is defined by the composition of maps.

\medskip

In this article, we will also use the leg numbering notation. Let $\s H$ be a Hilbert space and $T\in\B(\s H\tens\s H)$. We define the operators $T_{12},T_{13},T_{23}\in\B(\s H\tens\s H)$ by setting $T_{12}:=T\tens 1$, $T_{23}:=1\tens T$ and $T_{13}:=(\Sigma\tens 1)(1\tens T)(\Sigma\tens 1)$. We can generalize the leg numbering notation for operators acting on any tensor product of Hilbert spaces and for adjointable operators acting on any external tensor product of Hilbert C*-modules over possibly different C*-algebras.

\section{Locally compact quantum groups}

\setcounter{thm}{0}

\numberwithin{thm}{section}
\numberwithin{prop}{section}
\numberwithin{lem}{section}
\numberwithin{cor}{section}
\numberwithin{propdef}{section}
\numberwithin{nb}{section}
\numberwithin{nbs}{section}
\numberwithin{rk}{section}
\numberwithin{rks}{section}
\numberwithin{defin}{section}
\numberwithin{ex}{section}
\numberwithin{exs}{section}
\numberwithin{noh}{section}
\numberwithin{conv}{section}

For the notations and conventions used in this article concerning the non-commutative integration theory and the canonical objects of the Tomita-Takesaki theory, we refer the reader to the appendix of this article (cf.\ \S\ref{integration}).
\begin{defin}\cite{KV2}
A locally compact quantum group\index[notion]{locally compact quantum group} is a pair $\QG=({\rm L}^{\infty}(\QG),\Delta)$, where ${\rm L}^{\infty}(\QG)$ is a von Neumann algebra and $\Delta:{\rm L}^{\infty}(\QG)\rightarrow{\rm L}^{\infty}(\QG)\tens{\rm L}^{\infty}(\QG)$ is a unital normal *-homomorphism satisfying the following conditions:
\begin{enumerate}
\item $(\Delta\tens\id)\Delta=(\id\tens\Delta)\Delta$;
\item there exist \nsf weights $\varphi$ and $\psi$ on ${\rm L}^{\infty}(\QG)$ such that:
\begin{enumerate}
\item $\varphi$ is left invariant, {\it i.e.}\ $\varphi((\omega\tens\id)\Delta(x))=\varphi(x)\omega(1)$, for all $\omega\in{\rm L}^{\infty}(\QG)_*^+$ and $x\in{\f M}_{\varphi}^+$,
\item $\psi$ is right inveriant, {\it i.e.}\ $\psi((\id\tens\omega)\Delta(x))=\psi(x)\omega(1)$, for all $\omega\in{\rm L}^{\infty}(\QG)_*^+$ and $x\in{\f M}_{\psi}^+$.
\end{enumerate}
\end{enumerate}
A left (resp.\ right) invariant \nsf weight on ${\rm L}^{\infty}(\QG)$ is called a left (resp.\ right) Haar weight on $\QG$.
\end{defin}

\begin{noh}
For a locally compact quantum group $\QG$, there exists a unique left (resp.\ right) Haar weight on $\QG$ up to a positive scalar \cite{KV2}. Let us fix a locally compact quantum group $\QG:=({\rm L}^{\infty}(\QG),\Delta)$. Let us fix a left Haar weight $\varphi$ on $\QG$. Let $({\rm L}^2(\QG),\pi,\Lambda)$ be the \GNS construction for $({\rm L}^{\infty}(\QG),\varphi)$. The left regular representation of $\QG$ is the multiplicative unitary \cite{KV2,BS2} $W\in\B({\rm L}^2(\QG)\tens {\rm L}^2(\QG))$ defined by
\[
W^*(\Lambda(x)\tens\Lambda(y))=(\Lambda\tens\Lambda)(\Delta(y)(x\tens 1)), \quad \text{for all } x,\, y\in{\f N}_{\varphi}.
\]
By identifying ${\rm L}^{\infty}(\QG)$ with its image by the \GNS representation $\pi$, we obtain:
\begin{itemize}
\item ${\rm L}^{\infty}(\QG)$ is the strong closure of the algebra $\{(\id\tens\omega)(W)\,;\, \omega\in\B({\rm L}^2(\QG))_*\}$;
\item $\Delta(x)=W^*(1\tens x)W$, for all $x\in{\rm L}^{\infty}(\QG)$.\qedhere
\end{itemize}
\end{noh}

\begin{noh}
The Hopf-von Neumann algebra $({\rm L}^{\infty}(\QG),\Delta)$ admits \cite{KV2} a unitary antipode map $R_{\QG}:{\rm L}^{\infty}(\QG)\rightarrow{\rm L}^{\infty}(\QG)$ and we can choose for right Haar weight on $\QG$ the weight $\psi$ defined by $\psi(x):=\varphi(R_{\QG}(x))$, for all $x\in {\rm L}^{\infty}(\QG)_+$. The Connes cocycle derivative \cite{Co2,Vaes1} of $\psi$ with respect to $\varphi$ is given by 
\[
(D\psi\,:\,D\varphi)_t:=\nu^{\,{\rm i}t^2/2}d^{\,{\rm i}t},\quad \text{for all } t\in\GR,
\] 
where $\nu>0$ is the scaling constant of $\QG$ and the operator $d\eta M$ is the modular element of $\QG$ \cite{KV2}. Let 
$
{\f N}_{\varphi}^d:=\{x\in M\,;\, xd^{1/2} \text{ is bounded and its closure } \overline{x d^{1/2}} \text{ belongs to } {\f N}_{\varphi}\}.
$
The \GNS construction \cite{Vaes1} for $({\rm L}^{\infty}(\QG),\psi)$ is given by $({\rm L}^2(\QG),\id,\Lambda_{\psi})$, where $\Lambda_{\psi}$ is the closure of the map ${\f N}_{\varphi}^d\rightarrow{\rm L}^2(\QG)\,;\,x\mapsto\Lambda(\overline{xd^{1/2}})$. We recall that $J_{\psi}=\nu^{\,{\rm i}/4}J$, where $J$ denotes the modular conjugation for $\varphi$.
\end{noh}

\begin{noh}
The right regular representation of the quantum group $\QG$ is the multiplicative unitary $V\in\B({\rm L}^2(\QG)\tens{\rm L}^2(\QG))$ defined by
\[
V(\Lambda_{\psi}(x)\tens\Lambda_{\psi}(y))=(\Lambda_{\psi}\tens\Lambda_{\psi})(\Delta(x)(1\tens y)),\quad \text{for all } x,\, y\in{\f N}_{\psi}.\qedhere
\]
\end{noh}

\begin{defin}
The quantum group $\widehat{\QG}$ dual of $\QG$ is defined by the Hopf-von Neumann algebra $({\rm L}^{\infty}(\widehat{\QG}),\widehat{\Delta})$, where:
\begin{itemize}
\item ${\rm L}^{\infty}(\widehat{\QG})$ is the strong closure of the algebra $\{(\id\tens\omega)(V)\,;\, \omega\in\B({\rm L}^2(\QG)\}$;
\item the coproduct $\widehat{\Delta}:{\rm L}^{\infty}(\widehat{\QG})\rightarrow{\rm L}^{\infty}(\widehat{\QG})\tens{\rm L}^{\infty}(\widehat{\QG})$ is defined by $\widehat{\Delta}(x):=V^*(1\tens x)V$ for all $x\in{\rm L}^{\infty}(\widehat{\QG})$.
\end{itemize} 
The quantum group $\widehat{\QG}$ admits left and right Haar weights \cite{KV2} and we can take the Hilbert space ${\rm L}^2(\QG)$ for \GNS space. We denote by $\widehat{J}$ the modular conjugation of the left Haar weight on $\widehat{\QG}$.
\end{defin}

		\subsection{Hopf C*-algebras associated with a quantum group}
	
\setcounter{thm}{0}

\numberwithin{thm}{subsection}
\numberwithin{prop}{subsection}
\numberwithin{lem}{subsection}
\numberwithin{cor}{subsection}
\numberwithin{propdef}{subsection}
\numberwithin{nb}{subsection}
\numberwithin{nbs}{subsection}
\numberwithin{rk}{subsection}
\numberwithin{rks}{subsection}
\numberwithin{defin}{subsection}
\numberwithin{ex}{subsection}
\numberwithin{exs}{subsection}
\numberwithin{noh}{subsection}
\numberwithin{conv}{subsection}	

We associate \cite{BS2,KV2} with the quantum group $\QG$ two Hopf C*-algebras $(S,\delta)$ and $(\widehat{S},\widehat{\delta})$ defined by:
\begin{itemize}
\item $S$ (resp.\ $\widehat{S}$) is the norm closure of the algebra $\{(\omega\tens\id)(V)\,;\,\omega\in\B({\rm L}^2(\QG))_*\}$ (resp.\ $\{(\id\tens\omega)(V)\,;\,\omega\in\B({\rm L}^2(\QG))_*\}$);
\item the coproduct $\delta:S\rightarrow\M(S\tens S)$ (resp.\ $\widehat\delta:\widehat{S}\rightarrow\M(\widehat{S}\tens \widehat{S})$) is given by: 
\[
\delta(x):=V(x\tens 1)V^*,\quad \text{for all } x\in S \quad \text{{\rm(}resp.\ }\widehat{\delta}(x):=V^*(1\tens x)V,\quad \text{for all } x\in\widehat{S}\text{{\rm)}}.
\]
\end{itemize}
We call $(S,\delta)$ (resp.\ $(\widehat{S},\widehat{\delta})$) the Hopf C*-algebra (resp.\ dual Hopf C*-algebra) associated with $\QG$. We can also denote by ${\rm C}_0(\QG):=S$ the Hopf C*-algebra of $\QG$. Note that ${\rm C}_0(\widehat{\QG})=\widehat{S}$. 

\begin{nbs}
\begin{itemize}
\item Consider the unitary operator $U:=\widehat{J}J\in\B({\rm L}^2(\QG))$. Since $U=\nu^{\,{\rm i}/4}J\widehat{J}$, we have $U^*=\nu^{-{\rm i}/4}U$. In particular, ${\rm Ad}_U={\rm Ad}_{U^*}$ on $\B({\rm L}^2(\QG))$.
\item We have the following non-degenerate faithful representation of the C*-algebra $S$ (resp.\ $\widehat{S}$):
\begin{align*}
L & :S\rightarrow\B({\rm L}^2(\QG))\; ;\; y\mapsto y;\quad R:S\rightarrow\B({\rm L}^2(\QG))\;;\;y\mapsto UL(y)U^*\\
\text{{\rm(}resp.\ }\rho & :\widehat{S}\rightarrow\B({\rm L}^2(\QG))\; ; \; x\mapsto x;\quad \lambda:\widehat{S}\rightarrow\B({\rm L}^2(\QG))\;;\;x\mapsto U\rho(x)U^*\text{{\rm)}}.\qedhere
\end{align*}
\end{itemize}
\end{nbs}
		
It follows from 2.15 \cite{KV2} that 
$
W=\Sigma(U\tens 1)V(U^*\tens 1)\Sigma
$
and $[W_{12},\,V_{23}]=0$. The right regular representation of $\widehat{\QG}$ is the multiplicative unitary $\widetilde{V}:=\Sigma(1\tens U)V(1\tens U^*)\Sigma$.

\begin{nb}\label{notC(V)}
Let $\s H$ be a Hilbert space and $X\in\B(\s H\tens\s H)$. We denote by ${\cal C}(X)$\index[symbol]{ca@${\cal C}(-)$} the norm closure of the subspace $\{(\id\tens\omega)(\Sigma X)\,;\, \omega\in\B(\s H)_*\}$ of $\B(\s H)$. If $X$ is a multiplicative unitary, then $\{(\id\tens\omega)(\Sigma X)\,;\, \omega\in\B(\s H)_*\}$ is a subalgebra of $\B(\s H)$ \cite{BS2}.
\end{nb}

\begin{defin}
\cite{BS2,Ba95} The quantum group $\QG$ is said to be regular (resp.\ semi-regular), if $\K({\rm L}^2(\QG))={\cal C}(V)$ (resp.\ $\K({\rm L}^2(\QG))\subset{\cal C}(V)$).
\end{defin}

Note that $\QG$ is regular (resp.\ semi-regular) if, and only if, $\K({\rm L}^2(\QG))={\cal C}(W)$ (resp.\ $\K({\rm L}^2(\QG))\subset{\cal C}(W)$).

		\subsection{Continuous actions of locally compact quantum groups}
		
We use the notations introduced in the previous paragraph. Let $A$ be a C*-algebra.

\begin{defin}
\begin{enumerate}
\item An action of the quantum group $\QG$ on $A$ is a non-degenerate *-homomorphism $\delta_A:A\rightarrow\M(A\tens S)$ satisfying $(\delta_A\tens\id_S)\delta_A=(\id_A\tens\delta)\delta_A$.
\item An action $\delta_A$ of $\QG$ on $A$ is said to be strongly (resp.\ weakly) continuous if 
\[
[\delta_A(A)(1_A\tens S)]=A\tens S \quad (\text{resp.\ } A=[(\id_A\tens\omega)\delta_A(A)\,;\,\omega\in\B({\rm L}^2(\QG))_*]).
\]
\item A $\QG$-C*-algebra is a pair $(A,\delta_A)$, where $A$ is a C*-algebra and $\delta_A:A\rightarrow\M(A\tens S)$ is a strongly continuous action of $\QG$ on $A$.\qedhere
\end{enumerate}
\end{defin}

If $\QG$ is regular, any weakly continuous action of $\QG$ is necessarily continuous in the strong sense, cf.\ 5.8 \cite{BSV}.

\begin{nbs}
Let $\delta_A:A\rightarrow\M(A\tens S)$ (resp.\ $\delta_A:A\rightarrow\M(A\tens\widehat{S})$) be a strongly continuous action of $\QG$ (resp.\ $\widehat{\QG}$) on the C*-algebra $A$. Denote by $\pi_L$ (resp.\ $\widehat{\pi}_{\lambda}$) the representation of $A$ on the Hilbert $A$-module $A\tens{\rm L}^2(\QG)$ defined by $\pi_L:=(\id_A\tens L)\delta_A$ (resp.\ $\widehat{\pi}_{\lambda}:=(\id_A\tens\lambda)\delta_A$).
\end{nbs}

\begin{defin}(cf.\ 7.1 \cite{BS2})
Let $(A,\delta_A)$ be a $\QG$-C*-algebra (resp.\ $\widehat{\QG}$-C*-algebra). We call (reduced) crossed product of $A$ by $\QG$ (resp.\ $\widehat{\QG}$), the C*-subalgebra $A\rtimes\QG$ (resp.\ $A\rtimes\widehat{\QG}$) of $\Lin(A\tens{\rm L}^2(\QG))$ generated by the products $\pi_L(a)(1_A\tens\rho(x))$ (resp.\ $\widehat{\pi}_{\lambda}(a)(1_A\tens L(x))$) for $a\in A$ and $x\in\widehat S$ (resp.\ $x\in S$).
\end{defin}

The crossed product $A\rtimes\QG$ (resp.\ $A\rtimes\widehat{\QG}$) is endowed with a strongly continuous action of $\widehat{\QG}$ (resp.\ $\QG$), cf.\ 7.3 \cite{BS2}. If $\QG$ is regular, then the Takesaki-Takai duality extends to this setting, cf.\ 7.5 \cite{BS2}.

\begin{defin}
Let $A$ and $B$ be two C*-algebras. Let $\delta_A:A\rightarrow\M(A\tens S)$ and $\delta_B:B\rightarrow\M(B\tens S)$ be two actions of $\QG$ on $A$ and $B$ respectively. A non-degenerate *-homomorphism $f:A\rightarrow\M(B)$ is said to be $\QG$-equivariant if $(f\tens\id_S)\delta_A=\delta_B\circ f$. We denote by ${\sf Alg}_{\QG}$ the category whose objects are the $\QG$-C*-algebras and the morphisms are the $\QG$-equivariant non-degenerate *-homomorphisms.
\end{defin}

		\subsection{Equivariant Hilbert C*-modules and bimodules}

\paragraph{Preliminaries.} In this paragraph, we briefly recall some classical notations and elementary facts concerning Hilbert C*-modules. Let $A$ be a C*-algebra and $\s E$ a Hilbert $A$-module.

\begin{nbs}\index[symbol]{ia@$\iota_A,\,\iota_{\s E},\,\iota_{\s E^*},\,\iota_{\K(\s E)}$, canonical morphisms}
Let us consider the following maps:
\begin{itemize}
\item $\iota_A:A\rightarrow\K(\s E\oplus A)$, the *-homomorphism given by $\iota_A(a)(\xi\oplus b)=0\oplus ab$ for all $a,b\in A$ and $\xi\in\s E$;
\item $\iota_{\s E}:\s E\rightarrow\K(\s E\oplus A)$, the bounded linear map given by $\iota_{\s E}(\xi)(\eta\oplus a)=\xi a\oplus 0$ for all $a\in A$ and $\xi,\eta\in\s E$;
\item $\iota_{\s E^*}:\s E^*\rightarrow\K(\s E\oplus A)$, the bounded linear map given by $\iota_{\s E^*}(\xi^*)(\eta\oplus a)=0\oplus\xi^*\eta$ for all $\xi,\eta\in\s E$ and $a\in A$;
\item $\iota_{\K(\s E)}:\K(\s E)\rightarrow\K(\s E\oplus A)$, the *-homomorphism given by $\iota_{\K(\s E)}(k)(\eta\oplus a)=k\eta\oplus 0$ for all $k\in\K(\s E)$, $\eta\in\s E$ and $a\in A$.\qedhere
\end{itemize}
\end{nbs}

The result below follows from straightforward computations.

\begin{prop}\label{prop32}
We have the following statements:
\begin{enumerate}
\item $\iota_{\s E}(\xi a)=\iota_{\s E}(\xi)\iota_A(a)$ and $\iota_A(a)\iota_{\s E^*}(\xi^*)=\iota_{\s E^*}(a\xi^*)$ for all $\xi\in\s E$ and $a\in A$;
\item $\iota_{\s E^*}(\xi^*)=\iota_{\s E}(\xi)^*$ and $\iota_{\K(\s E)}(\theta_{\xi,\eta})=\iota_{\s E}(\xi)\iota_{\s E}(\eta)^*$ for all $\xi,\eta\in\s E$;
\item $\iota_{\s E}(\xi)^*\iota_{\s E}(\eta)=\iota_A(\langle\xi,\,\eta\rangle)$ for all $\xi,\eta\in\s E$;
\item $\K(\s E\oplus A)$ is the C*-algebra generated by the set $\iota_A(A)\cup\iota_{\s E}(\s E)$.\qedhere
\end{enumerate}
\end{prop}

\begin{rks}\label{rk3}
\begin{enumerate}
\item For $a\in A$, $\xi\in\s E$ and $k\in\K(\s E)$, the operators $\iota_A(a)$, $\iota_{\s E}(\xi)$, $\iota_{\s E^*}(\xi^*)$ and $\iota_{\K(\s E)}(k)$ can be represented by 2-by-2 matrices acting on the Hilbert $A$-module $\s E\oplus A$ as follows:
\[
\iota_A(a)=\begin{pmatrix}0 & 0\\0 & a\end{pmatrix}\!;\quad \iota_{\s E}(\xi)=\begin{pmatrix}0 & \xi\\0 & 0\end{pmatrix}\!;\quad \iota_{\s E^*}(\xi^*)=\begin{pmatrix}0 & 0\\\xi^* & 0\end{pmatrix}\!;\quad \iota_{\K(\s E)}(k)=\begin{pmatrix}k & 0\\0 & 0\end{pmatrix}\!.
\]
Moreover, any operator $x\in\K(\s E\oplus A)$ can be written in a unique way as follows:
\[
x=\begin{pmatrix}k & \xi\\ \eta^* & a\end{pmatrix}\!, \quad \text{with } k\in\K(\s E),\ \xi,\eta\in\s E \; \text{and}\; a\in A.
\]
\item Note that $\iota_A$ and $\iota_{\K(\s E)}$ extend uniquely to strictly/*-strongly continuous unital *-homomor\-phisms $\iota_{A}:\M(A)\rightarrow\Lin(\s E\oplus A)$ and $\iota_{\K(\s E)}:\Lin(\s E)\rightarrow\Lin(\s E\oplus A)$. Besides, we have 
$
\iota_{A}(m)(\xi\oplus a)=0\oplus ma
$
and
$
\iota_{\K(\s E)}(T)(\xi\oplus a)=T\xi\oplus 0
$
for all $m\in\M(A)$, $T\in\Lin(\s E)$, $\xi\in\s E$ and $a\in A$. 
\item $\iota_{\s E^*}$ admits an extension to a bounded linear map $\iota_{\s E^*}:\Lin(\s E,A)\rightarrow\Lin(\s E\oplus A)$ in a straightforward way. Similarly, up to the identification $\s E=\K(A,\s E)$, we can also extend $\iota_{\s E}$ to a bounded linear map $\iota_{\s E}:\Lin(A,\s E)\rightarrow\Lin(\s E\oplus A)$.
\item As in 1, we can represent the operators $\iota_{A}(m)$, $\iota_{\K(\s E)}(T)$, $\iota_{\s E^*}(S)$ and $\iota_{\s E}(S^*)$, for $m\in\M(A)$, $T\in\Lin(\s E)$ and $S\in\Lin(A,\s E)$, by 2-by-2 matrices. Moreover, any operator $x\in\Lin(\s E\oplus A)$ can be written in a unique way as follows:
\[
x=\begin{pmatrix}T & S'\\ S^* & m\end{pmatrix}\!, \quad \text{with } T\in\Lin(\s E),\ S,S'\in\Lin(A,\s E) \; \text{and} \; m\in\M(A).\qedhere
\]
\end{enumerate}
\end{rks}

By using the matrix notations described above, we derive easily the following useful technical lemma:

\begin{lem}\label{ehmlem1}
Let $x\in\mathcal{L}(\s{E}\oplus A)$ {\rm(}resp.\ $x\in\K(\s E\oplus A))$. We have:
\begin{enumerate}
 \item $x\in\iota_{\s E}(\mathcal{L}(A,\s{E}))$ {\rm(}resp.\ $\iota_{\s E}(\s E))$ if, and only if, $x\iota_\s{E}(\xi)=0$ for all $\xi\in \s{E}$ and $\iota_A(a)x=0$ for all $a\in A$; in that case, we have $\iota_{A}(m)x=0$ for all $m\in\M(A)$;
 \item $x\in\iota_{\K(\s{E})}(\mathcal{L}(\s{E}))$ {\rm(}resp.\ $\iota_{\K(\s E)}(\K(\s E)))$ if, and only if, $x\iota_A(a)=0$ and $\iota_A(a)x=0$ for all $a\in A$; in that case, we have $x\iota_{A}(m)=\iota_{A}(m)x=0$ for all $m\in\M(A)$.\qedhere
\end{enumerate}
\end{lem}

Let us recall the notion of relative multiplier module, cf.\ 2.1 \cite{BS1}.

\begin{defin}
Let $A$ and $B$ be two C*-algebras and let $\s E$ be a Hilbert C*-module over $A$. Up to the identification $\s E\tens B=\K(A\tens B,\s E\tens B)$, we define $\widetilde{\M}(\s E\tens B)$\index[symbol]{mb@$\widetilde{\M}(\s E\tens B)$, relative multiplier module} (or $\widetilde{\M}_B(\s E\tens B)$ in case of ambiguity) to be the following subspace of $\Lin(A\tens B,\s E\tens B)$:
\[\{T\in\Lin(A\tens B,\s E\tens B) \; ; \; \forall x\in B, \; (1_{\s E}\tens x)T\in\s E\tens B\;\;\mathrm{and} \;\; T(1_A\tens x)\in\s E\tens B\}.\]
Note that $\widetilde{\M}(\s E\tens B)$ is a Hilbert C*-module over $\widetilde{\M}(A\tens B)$, whose $\widetilde{\M}(A\tens B)$-valued inner product is given by: 
\[
\langle\xi,\eta\rangle:=\xi^*\circ\eta,\quad \text{for all }\xi,\eta\in\widetilde{\M}(\s E\tens B)\subset\Lin(A\tens B,\s E\tens B).
\]
Note also that we have $\K(\widetilde{\M}(\s E\tens B))\subset\widetilde{\M}(\K(\s E)\tens B)$.
\end{defin}

\begin{propdef}\label{not4}
Let $B\subset\B(\s K)$ be a C*-algebra of operators on a Hilbert space $\s K$. For all $T\in\Lin(A\tens B,\s E\tens B)$ and $\omega\in\B(\s K)_*$, there exists a unique $(\id_{\s E}\tens\omega)(T)\in\Lin(A,\s E)$ such that
\[
\iota_{\s E}(\id_{\s E}\tens\omega)(T)=(\id_{\K(\s E\oplus A)}\tens\omega)(\iota_{\s E\tens B}(T))\in\Lin(\s E\oplus A),
\]
where $\iota_{\s E\tens B}:\Lin(A\tens B,\s E\tens B)\rightarrow\Lin((\s E\tens B)\oplus(A\tens B))=\M(\K(\s E\oplus A)\tens B)$. If $B\subset\B(\s K)$ is non-degenerate and $T\in\widetilde{\M}(\s E\tens B)$, then we have $(\id_{\s E}\tens\omega)(T)\in\s E$.
\end{propdef}

\begin{proof}
This is a direct consequence of \ref{ehmlem1} 1 and the fact that $\iota_{\s E\tens B}(T)\in\widetilde{\M}(\K(\s E\oplus A)\tens B)$ if $T\in\widetilde{\M}(\s E\tens B)$.
\end{proof}

\paragraph{Notion of equivariant Hilbert C*-module.}

In this paragraph, we recall the notion of equivariant Hilbert C*-module through the three equivalent pictures developed in \S 2 \cite{BS1}. Let us fix a $\QG$-C*-algebra $(A,\delta_A)$ and a Hilbert $A$-module $\s E$.

\begin{defin}
An action of the locally compact quantum group $\QG$ on $\s E$ is a linear map $\delta_{\s E}:\s E\rightarrow\widetilde{\M}(\s E\tens S)$ such that:
\begin{enumerate}
\item $\delta_{\s E}(\xi)\delta_A(a)=\delta_{\s E}(\xi a)$ and $\delta_A(\langle\xi,\,\eta\rangle)=\langle\delta_{\s E}(\xi),\,\delta_{\s E}(\eta)\rangle$, for all $a\in A$ and $\xi,\,\eta\in\s E$;
\item $[\delta_{\s E}(\s E)(A\tens S)]=\s E\tens S$;
\item the linear maps $\delta_{\s E}\tens\id_S$ and $\id_{\s E}\tens\delta$ extend to linear maps from $\Lin(A\tens S,\s E\tens S)$ to $\Lin(A\tens S\tens S,\s E\tens S\tens S)$ and we have $(\delta_{\s E}\tens\id_S)\delta_{\s E}=(\id_{\s E}\tens\delta)\delta_{\s E}$.
\end{enumerate}
An action $\delta_{\s E}$ of $\QG$ on $\s E$ is said to be continuous if we have 
$
[(1_{\s E}\tens S)\delta_{\s E}(\s E)]=\s E\tens S.
$
A $\QG$-equivariant Hilbert $A$-module is a Hilbert $A$-module $\s E$ endowed with a continuous action $\delta_{\s E}:\s E\rightarrow\widetilde{\M}(\s E\tens S)$ of $\QG$ on $\s E$.
\end{defin}

\begin{noh}
The datum of a continuous action of $\QG$ on $\s E$ is equivalent to that of a continuous action $\delta_{\K(\s E\oplus A)}:\K(\s E\oplus A)\rightarrow\M(\K(\s E\oplus A)\tens S)$ of $\QG$ on the linking C*-algebra $\K(\s E\oplus A)$ such that the *-homomorphism $\iota_A:A\rightarrow\K(\s E\oplus A)$ is $\QG$-equivariant, cf.\ 2.7 \cite{BS1}.
\end{noh}

\begin{noh}\label{unitary}
If $\delta_{\s E}$ is an action of $\QG$ on $\s E$, we have the unitary operator $\s V\in\Lin(\s E\tens_{\delta_A}(A\tens S),\s E\tens S)$ defined by $\s V(\xi\tens_{\delta_A}x):=\delta_{\s E}(\xi)x$ for all $\xi\in\s E$ and $x\in A\tens S$. This unitary satisfies the relation 
\[
(\s V\tens_{\GC}\id_S)(\s V\tens_{\delta_A\tens\id_S}1)=\s V\tens_{\id_A\tens\,\delta}1 \quad \text{in} \quad \Lin(\s E\tens_{\delta_A^2}(A\tens S\tens S),\s E\tens S\tens S), 
\]
where $\delta_A^2:=(\delta_A\tens\id_S)\delta_A=(\id_A\tens\delta)\delta_A$, cf.\ 2.3 and 2.4 (a) \cite{BS1} for the details. Conversely, if there exists a unitary operator $\s V\in\Lin(\s E\tens_{\delta_A}(A\tens S),\s E\tens S)$ satisfying the above relation and the fact that $\s VT_{\xi}\in\widetilde{\M}(\s E\tens S)$ for all $\xi\in\s E$, where $T_{\xi}\in\Lin(A\tens S,\s E\tens_{\delta_A}(A\tens S))$ is defined by $T_{\xi}(x):=\xi\tens_{\delta_A}x$ for all $x\in A\tens S$, then the map $\delta_{\s E}:\s E\rightarrow\widetilde{\M}(\s E\tens S)\,;\,\xi\mapsto\s VT_{\xi}$ is an action of $\QG$ on $\s E$, cf.\ 2.4 (b) \cite{BS1}.
\end{noh}

\begin{noh}
An action of $\QG$ on $\s E$ determines an action $\delta_{\K(\s E)}:\K(\s E)\rightarrow\widetilde{\M}(\K(\s E)\tens S)$ of $\QG$ on $\K(\s E)$ defined by $\delta_{\K(\s E)}(k)=\s V(k\tens_{\delta_A}1)\s V^*$ for all $k\in\K(\s E)$, where $\s V$ is the unitary operator associated to the action, cf.\ 2.8 \cite{BS1}. Moreover, if $\s E$ is a $\QG$-equivariant Hilbert module, then $\K(\s E)$ turns into a $\QG$-C*-algebra.
\end{noh}


\section{Measured quantum groupoids}

\setcounter{thm}{0}

\numberwithin{thm}{section}
\numberwithin{prop}{section}
\numberwithin{lem}{section}
\numberwithin{cor}{section}
\numberwithin{propdef}{section}
\numberwithin{nb}{section}
\numberwithin{nbs}{section}
\numberwithin{rk}{section}
\numberwithin{rks}{section}
\numberwithin{defin}{section}
\numberwithin{ex}{section}
\numberwithin{exs}{section}
\numberwithin{noh}{section}
\numberwithin{conv}{section}

For reminders concerning the relative tensor product of Hilbert spaces and the fiber product of von Neumann algebras, we refer the reader to the appendix of this article (cf.\ \ref{tensorproduct}).

\begin{defin}(cf.\ 3.7 \cite{E08}, 4.1 \cite{Le})
We call a measured quantum groupoid an octuple ${\cal G}=(N, M, \alpha, \beta, \Gamma, T, T', \nu)$, where:
\begin{itemize}
	\item $M$ and $N$ are von Neumann algebras;
	\item $\Gamma:M\rightarrow\fprod{M}{\beta}{\alpha}{M}$ is a faithful normal unital *-homomorphism, called the coproduct;
	\item $\alpha:N\rightarrow M$ and $\beta:N^{\rm o}\rightarrow M$ are faithful normal unital *-homormorphisms, called the range and source maps of $\cal G$;
	\item $T:M_+\rightarrow\alpha(N)_+^{\rm ext}$ and $T':M_+\rightarrow\beta(N^{\rm o})_+^{\rm ext}$ are \nsf operator-valued weights;
	\item $\nu$ is a \nsf weight on $N$;
\end{itemize}
such that the following conditions are satisfied:
\begin{enumerate}
	\item $[\alpha(n'),\,\beta(n^{\rm o})]=0$, for all $n,n'\in N$;
	\item $\Gamma(\alpha(n))=\reltens{\alpha(n)}{\beta}{\alpha}{1}$ and $\Gamma(\beta(n^{\rm o}))=\reltens{1}{\beta}{\alpha}{\beta(n^{\rm o})}$, for all $n\in N$;
	\item $\Gamma$ is coassociative, {\it i.e.\ }$(\fprod{\Gamma}{\beta}{\alpha}{\id})\Gamma=(\fprod{\id}{\beta}{\alpha}{\Gamma})\Gamma$;
	\item the \nsf weights $\varphi$ and $\psi$ on $M$ given by $\varphi=\nu\circ\alpha^{-1}\circ T$ and $\psi=\nu\circ\beta^{-1}\circ T'$ satisfy:
		\begin{itemize}
			\item $\forall x\in\f M_T^+, \, T(x)=(\fprod{\id}{\beta}{\alpha}{\varphi})\Gamma(x), \quad \forall x\in\f M_{T'}^+,\,T'(x)=(\fprod{\psi}{\beta}{\alpha}{\id})\Gamma(x)$,
			\item $\sigma^{\varphi}_t$ and $\sigma^{\psi}_s$ commute for all $s,t\in\GR$.\qedhere
		\end{itemize}
\end{enumerate}
\end{defin}

Let ${\cal G}=(N, M, \alpha, \beta, \Gamma, T, T', \nu)$ be a measured quantum groupoid. We denote by $(\s H,\pi,\Lambda)$ the \GNS construction for $(M,\varphi)$ where $\varphi:=\nu\circ\alpha^{-1}\circ T$. Let $(\sigma_t)_{t\in\GR}$, $\nabla$ and $J$ be respectively the modular automorphism group, the modular operator and the modular conjugation for $\varphi$. In the following, we identify $M$ with its image by $\pi$ in $\B(\s H)$. 
\begin{itemize}
	\item We have a coinvolutive *-antiautomomorphism $R_{\cal G}:M\rightarrow M$\index[symbol]{r@$R_{\cal G}$, unitary coinverse} such that $R_{\cal G}^2=\id_M$ (cf.\ 3.8 \cite{E08}).
\end{itemize}
From now on, we will assume that $T'=R_{\cal G}\circ T\circ R_{\cal G}$ and then also $\psi=\varphi\circ R_{\cal G}$.
\begin{itemize}
	\item There exist self-adjoint positive non-singular operators $\lambda$ and $d$ respectively affiliated to ${\cal Z}(M)$\index[symbol]{z@${\cal Z}(-)$, center} and $M$ such that
	 $
	 (D\psi:D\varphi)_t=\lambda^{{\rm i}t^2/2}d^{\,{\rm i}t}
	 $
for all $t\in\GR$. The operators $\lambda$ and $d$ are respectively called the scaling operator and the modular operator of $\cal G$. 
	\item The \GNS construction for $(M,\psi)$ is given by $(\s H,\pi_{\psi},\Lambda_{\psi})$, where: $\Lambda_{\psi}$ is the closure of the operator which sends any element $x\in M$ such that $x d^{1/2}$ is closable and its closure $\overline{x d^{1/2}}\in\f N_{\varphi}$ to $\Lambda_{\varphi}(\overline{x d^{1/2}})$; $\pi_{\psi}:M\rightarrow\B(\s H)$ is given by the formula $\pi_{\psi}(a)\Lambda_{\psi}(x)=\Lambda_{\psi}(ax)$.
	\item The modular conjugation $J_{\psi}$ for $\psi$ is given by $J_{\psi}=\lambda^{{\rm i}/4}J$.
	\item We will denote by
	$
	\s W_{\cal G}:\reltens{\s H}{\beta}{\alpha}{\s H}\rightarrow\reltens{\s H}{\alpha}{\widehat{\beta}}{\s H}
	$
	the pseudo-multiplicative unitary of $\cal G$\index[symbol]{wa@$\s W_{\cal G}$, pseudo-multiplicative unitary} (cf.\ 3.3.4 \cite{Val}, 3.6 \cite{E08}). 
\end{itemize}

\begin{propdef}(cf.\ 3.10 \cite{E08}) We define the (Pontryagin) dual of $\cal G$ to be the measured quantum groupoid $\widehat{\cal G}:=(N,\widehat{M},\alpha,\widehat{\beta},\widehat{\Gamma},\widehat{T},\widehat{R}\circ\widehat{T}\circ\widehat{R},\nu)$, where:
\begin{itemize}
	\item $\widehat{M}$ is the von Neumann algebra generated by 
	$
	\{(\omega\star\id)(\s W_{\cal G})\,;\,\omega\in\B(\s H)_*\}\subset\B(\s H);
	$
	\item $\widehat\beta:N^{\rm o}\rightarrow\widehat{M}$ is given by $\widehat\beta(n^{\rm o}):=J\alpha(n^*)J$ for all $n\in N$;
	\item $\widehat{\Gamma}:\widehat{M}\rightarrow\fprod{\widehat{M}}{\widehat\beta}{\alpha}{\widehat{M}}$ is given for all $x\in\widehat{M}$ by
	$
	\widehat{\Gamma}(x):=\sigma_{\alpha\widehat\beta}\s W_{\cal G}(\reltens{x}{\beta}{\alpha}{1})\s W_{\cal G}^*\sigma_{\widehat\beta\alpha};
	$
	\item there exists a unique \nsf weight $\widehat\varphi$ on $\widehat{M}$ whose \GNS construction is $(\s H,\id,\Lambda_{\widehat\varphi})$, where $\Lambda_{\widehat\varphi}$ is the closure of the operator 
	$
	(\omega\star\id)(\s W_{\cal G})\mapsto a_{\varphi}(\omega)
	$
defined for normal linear forms $\omega$ in a dense subspace of 
$
{\s I}_{\varphi}=\{\omega\in\B(\s H)_*\,;\,\exists k\in\GR_+,\,\forall x\in\f N_{\varphi},\, |\omega(x^*)|^2\leqslant k\varphi(x^*x)\}
$
and $a_{\varphi}(\omega)\in\s H$ satisfies
	$
	\omega(x^*)=\langle\Lambda_{\varphi}(x),\,a_{\varphi}(\omega)\rangle
	$
	for all $x\in\f N_{\varphi}$;
	\item $\widehat{T}$ is the unique \nsf operator-valued weight from $\widehat{M}$ to $\alpha(N)$ such that $\widehat\varphi=\nu\circ\alpha^{-1}\circ\widehat{T}$ and $\widehat{T}'=R_{\widehat{\cal G}}\circ\widehat{T}\circ R_{\widehat{\cal G}}$, where $R_{\widehat{\cal G}}:\widehat{M}\rightarrow\widehat{M}$ is given by $R_{\widehat{\cal G}}(x):=Jx^*J$ for all $x\in\widehat{M}$.
\end{itemize}
The pseudo-multiplicative unitary $\s W_{\widehat{\cal G}}$ of $\widehat{\cal G}$ is given by 
$
\s W_{\widehat{\cal G}}=\sigma_{\beta\alpha}\s W_{\cal G}^*\sigma_{\widehat\beta\alpha}.
$
\end{propdef}

We will denote by $\widehat{J}$ the modular conjugation for $\widehat\varphi$. Note that the scaling operator of $\widehat{\cal G}$ is $\lambda^{-1}$. In particular, we have $\lambda^{{\rm i}t}\in{\cal Z}(M)\cap{\cal Z}(\widehat{M})$ for all $t\in\GR$. 
\begin{itemize}
\item Let $\widehat\alpha(n):=J\beta(n^{\rm o})^*J=\widehat{J}\,\widehat\beta(n^{\rm o})^*\widehat{J}$ for $n\in N$. We recall the following relations (cf.\ 3.11 (v) \cite{E08}):
	$
	M\cap\widehat{M}=\alpha(N)$, $M\cap\widehat{M}'=\beta(N^{\rm o})$,
	$M'\cap\widehat{M}=\widehat\beta(N^{\rm o})$ and $M'\cap\widehat{M}'=\widehat\alpha(N)$.
\item Let $U:=\widehat{J}J\in\B(\s H)$\index[symbol]{u@$U$}. Then, $U^*=\lambda^{-i/4}U$ and $U^2=\lambda^{i/4}$ (cf.\ 3.11 (iv) \cite{E08}). In particular, $U$ is unitary. We have
$
\widehat\alpha(n)=U\alpha(n)U^*
$
and
$
\widehat\beta(n^{\rm o})=U\beta(n^{\rm o})U^*
$
for all $n\in N$. Since $\lambda^{-{\rm i}/4}\in{\cal Z}(M)$, we also have 
$
\widehat\alpha(n)=U^*\alpha(n)U
$
and
$
\widehat\beta(n^{\rm o})=U^*\beta(n^{\rm o})U
$
for all $n\in N$.
\end{itemize}

\begin{propdef}(cf.\ 3.12 \cite{E08})
\begin{itemize}
\item The octuple $(N^{\rm o},M,\beta,\alpha,\varsigma_{\beta\alpha}\circ\Gamma,R_{\cal G}\circ T\circ R_{\cal G},T,\nu^{\rm o})$ is a measured quantum groupoid denoted by $\cal G^{\rm o}$ and called the opposite of $\cal G$. The pseudo-multiplicative unitary of $\cal G^{\rm o}$ is given by $\s W_{\cal G^{\rm o}}=(_{\beta}\reltens{\widehat J}{\alpha}{\widehat\alpha}{\widehat J}_{\widehat\beta})\s W_{\cal G}(_{\beta}\reltens{\widehat J}{\alpha}{\alpha}{\widehat J}_{\beta})$.
\item  Let $C_M:M\rightarrow M'$ be the canonical *-antihomomorphism given by $C_M(x):=Jx^*J$ for all $x\in M$. Let us define:
	\[
	\Gamma^{\rm c}:=(\fprod{C_M}{\beta}{\alpha}{C_M})\circ\Gamma\circ C_M^{-1};\quad R_{\cal G}^{\rm c}:=C_M\circ R_{\cal G}\circ C_M^{-1};\quad T^{\rm c}=C_M\circ T\circ C_M^{-1}.
	\]
	Then, the octuple $(N^{\rm o},M',\widehat\beta,\widehat\alpha,\Gamma^{\rm c},T^{\rm c},R_{\cal G}^{\rm c}T^{\rm c}R_{\cal G}^{\rm c},\nu^{\rm o})$ is a measured quantum groupoid denoted by ${\cal G}^{\rm c}$ and called the commutant of $\cal G$. The pseudo-multiplicative unitary $\s W_{{\cal G}^{\rm c}}$ of ${\cal G}^{\rm c}$ is given by
	$
	\s W_{{\cal G}^{\rm c}}=(_{\widehat{\beta}}\reltens{J}{\alpha}{\alpha}{J}_{\widehat\beta})\s W_{\cal G}({}_{\beta}\reltens{J}{\widehat\alpha}{\alpha}{J}_{\widehat\beta}).
	$\qedhere
\end{itemize}
\end{propdef}

\begin{nbs}\label{pseudomult}
For a given measured quantum groupoid $\cal G$, we will need the following pseudo-multiplicative unitaries:\index[symbol]{va@$\widehat{\s V}$, $\s V$, $\widetilde{\s V}$}
 \[
 \widehat{\s V}  := {\s W}_{\cal G}; \quad   {\s V}  := {\s W}_{\widehat{({\cal G}^{\rm o})}} =   {\s W}_{(\widehat{\cal G})^{\rm c}}; \quad 
  \widetilde{\s V}  := {\s W}_{({\cal G}^{\rm o})^{\rm c}}.\qedhere
 \]
\end{nbs}

\begin{conv}\label{rkconv} 
Henceforth, we will refer to $(\widehat{\cal G})^{\rm c}$ instead of $\widehat{\cal G}$ as the dual of ${\cal G}$ since this groupoid is better suited for right actions of $\cal G$. We have
\[
(\widehat{\cal  G})^{\rm c} = (N^{\rm o}, \widehat M',  \beta, \widehat\alpha, \widehat\Gamma^{\rm c}, \widehat T^{\rm c}, \widehat T^{\rm c}{'}, \nu^{\rm o}),
\]
where the coproduct and the operator-valued weights are given by:
\begin{itemize}
	\item $\widehat\Gamma^{\rm c}(x) = ({\s W}_{(\widehat{\cal  G} )^{\rm c}})^* (\reltens{1}{\beta}{\alpha}{x}) {\s W}_{(\widehat{\cal  G} )^{\rm c}}$, for all $x\in \widehat M'$;
	\item $\widehat T^{\rm c}  = C_{\widehat M} \circ \widehat T \circ C_{\widehat M}^{-1}$, where 
$C_{\widehat M}  : \widehat M  \rightarrow \widehat M'$ ; $x \mapsto \widehat J x^* \widehat J$;
	\item $\widehat T^{\rm c}{'}  = R_{(\widehat{\cal  G} )^{\rm c}} \circ  \widehat T^{\rm c} \circ R_{(\widehat{\cal  G} )^{\rm c}}$.
\end{itemize}
Note also that the commutant weight  $\widehat{\varphi}^{\rm c} := \nu^{\rm o} \circ \beta^{-1} \circ \widehat T^{\rm c}$ derived from the weight $\widehat\varphi$ is left invariant for the coproduct $\widehat\Gamma^{\rm c}$.
In the following, we will simply denote by $\widehat{\cal G}$ the dual groupoid of $\cal G$ (since no ambiguity will arise with the Pontryagin dual). Note that the bidual groupoid is $({\cal G}^{\rm o})^{\rm c}=({\cal G}^{\rm c})^{\rm o}$. 
\end{conv} 

		\subsection{Case where the basis is finite-dimensional}\label{MQGfinitebasis}
	
\setcounter{thm}{0}

\numberwithin{thm}{subsection}
\numberwithin{prop}{subsection}
\numberwithin{lem}{subsection}
\numberwithin{cor}{subsection}
\numberwithin{propdef}{subsection}
\numberwithin{nb}{subsection}
\numberwithin{nbs}{subsection}
\numberwithin{rk}{subsection}
\numberwithin{rks}{subsection}
\numberwithin{defin}{subsection}
\numberwithin{ex}{subsection}
\numberwithin{exs}{subsection}
\numberwithin{noh}{subsection}
\numberwithin{conv}{subsection}
	
In \cite{DC}, De Commer provides an equivalent definition of a measured quantum groupoid on a finite basis. This definition is far more tractable since it allows us to circumvent the use of relative tensor products and fiber products.

\medskip
 
In the following, we fix a finite-dimensional C*-algebra
$
N:=\bigoplus_{1\leqslant l\leqslant k}{\rm M}_{n_l}(\GC)
$
endowed with the non-normalized Markov trace $\epsilon:=\bigoplus_{1\leqslant l\leqslant k}n_l\cdot{\rm Tr}_l$, where ${\rm Tr}_l$ denotes the non-normalized trace on ${\rm M}_{n_l}(\GC)$.\index[symbol]{ta@${\rm Tr}_l$, non-normalized Markov trace on ${\rm M}_{n_l}(\GC)$}

\medskip

We refer to \S \ref{tensorproduct} of the appendix for the definitions of $v_{\beta\alpha}$ and $q_{\beta\alpha}$.
Let us a fix a measured quantum groupoid ${\cal G}=(N,M,\alpha,\beta,\Gamma,T,T',\epsilon)$. We have a unital normal *-isomorphism
$
\fprod{M}{\beta}{\alpha}{M} \rightarrow q_{\beta\alpha}(M\tens M)q_{\beta\alpha}\,;\,
x \mapsto v_{\beta\alpha}^*xv_{\beta\alpha}
$
(cf.\ \ref{propcoiso}). Let $\Delta:M\rightarrow M\tens M$ be the (non necessarily unital) faithful normal *-homomor\-phism given by
$
\Delta(x)=v_{\beta\alpha}^*\Gamma(x)v_{\beta\alpha}
$
for all $x\in M$. We have $\Delta(1)=q_{\beta\alpha}$. This has led De Commer to the following equivalent definition of a measured quantum groupoid on a finite basis.

\begin{defin}\label{defMQG}(cf.\ 11.1.2 \cite{DC})
A measured quantum groupoid on the finite-dimen\-sional basis 
$
N
$
is an octuple ${\cal G}=(N,M,\alpha,\beta,\Delta,T,T',\epsilon)$, where:
\begin{itemize}
	\item $M$ is a von Neumann algebra, $\alpha:N\rightarrow M$ and $\beta:N^{\rm o}\rightarrow M$ are unital faithful normal *-homomorphisms;
	\item $\Delta:M\rightarrow M\tens M$ is a faithful normal *-homomorphism;
	\item $T:M_+\rightarrow\alpha(N)_+^{\ext}$ and $T':M_+\rightarrow\beta(N^{\rm o})_+^{\ext}$ are \nsf operator-valued weights;
\end{itemize}
such that the following conditions are satisfied:
\begin{enumerate}
	\item $[\alpha(n'),\,\beta(n^{\rm o})]=0$, for all $n,n'\in N$;
	\item $\Delta(1)=q_{\beta\alpha}$;
	\item $(\Delta\tens\id)\Delta=(\id\tens\Delta)\Delta$;
	\item $\Delta(\alpha(n))=\Delta(1)(\alpha(n)\tens 1)$ and $\Delta(\beta(n^{\rm o}))=\Delta(1)(1\tens\beta(n^{\rm o}))$, for all $n\in N$;
	\item the \nsf weights $\varphi$ and $\psi$ on $M$ given by $\varphi:=\epsilon\circ\alpha^{-1}\circ T$ and $\psi:=\epsilon\circ\beta^{-1}\circ T'$ satisfy:
		\[
		T(x)=(\id\tens\varphi)\Delta(x)\quad\text{for all } x\in\f M_T^+,\quad T'(x)=(\psi\tens\id)\Delta(x),\quad\text{for all } x\in\f M_{T'}^+;
		\]
	\item $\sigma_t^T\circ\beta=\beta$ and $\sigma_t^{T'}\circ\alpha=\alpha$, for all $t\in\GR$.\qedhere
\end{enumerate}
\end{defin}		

Let us fix a measured quantum groupoid ${\cal G}=(N,M,\alpha,\beta,\Delta,T,T',\epsilon)$.

\begin{nbs}
Let us consider the injective bounded linear map 
\[
\iota_{\,\widehat\alpha\alpha}^{\beta}:\B(\reltens{\s H}{\widehat\alpha}{\beta}{\s H},\reltens{\s H}{\beta}{\alpha}{\s H})  \rightarrow\B(\s H\tens\s H)\quad ;\quad
X  \mapsto v_{\beta\alpha}^*Xv_{\widehat\alpha\beta}.
\]
Similarly, we also define $\iota_{\!\beta\widehat\beta}^{\alpha}$ and $\iota_{\widehat\beta\beta}^{\widehat\alpha}$. Let\index[symbol]{vb@$V$, $W$, $\widetilde{V}$, multiplicative partial isometries}
\[
V:=\iota_{\,\widehat\alpha\alpha}^{\beta}(\s V),\quad W:=\iota_{\!\beta\widehat\beta}^{\alpha}(\widehat{\s V})\quad \text{and} \quad
\widetilde{V}:=\iota_{\widehat\beta\beta}^{\widehat\alpha}(\widetilde{\s V}),
\]
where $\s V=\s W_{\widehat{\cal G}}$, $\widehat{\s V}=\s  W_{\cal G}$ and $\widetilde{\s V}=\s W_{({\cal G}^{\rm o})^{\rm c}}$ (cf.\ \ref{pseudomult}).
\end{nbs}

In what follows, we recall the main properties satisfied by $V$, $W$ and $\widetilde{V}$. The proof of the results below are derived from the properties satisfied by the pseudo-multiplicative unitaries $\s V$, $\widehat{\s V}$ and $\widetilde{\s V}$ (cf.\ \cite{E08}, \S 11 \cite{DC} and \S 2 \cite{BC}).

\begin{prop}\label{inifinproj}(cf.\ 3.11 (iii), 3.12 (v), (vi) \cite{E08}, 2.2 \cite{BC}) 
The operators $V,W$ and $\widetilde{V}$ are multiplicative partial isometries acting on $\s H\tens\s H$ such that:
\begin{enumerate}
\item $W = \Sigma (U \otimes 1) V (U^* \otimes 1)\Sigma$,\quad $\widetilde V = \Sigma (1 \otimes U) V (1 \otimes U^*)\Sigma=(U\tens U)W(U^*\tens U^*)$;
\item $V^* = (J \otimes \widehat J) V (J \otimes \widehat J)$,\quad  $W^* = (\widehat J \otimes J) W (\widehat J \otimes J)$;
\item the initial and final projections are given by
\[
V^* V =q_{\widehat\alpha\beta}=\widetilde V \widetilde V^*,\quad W^*W  =q_{\beta\alpha}= V V^*,\quad W W^* = q_{\alpha\widehat\beta} \quad \text{and} \quad \widetilde V^*\widetilde V  = q_{\widehat\beta \widehat\alpha}.\qedhere
\]
\end{enumerate}
\end{prop}

\begin{prop}(cf.\ 3.8, 3.12 \cite{E08})
\begin{enumerate}
\item The von Neumann algebra $M$ {\rm(}resp.\ $\widehat{M}${\rm)} is the weak closure of
$\{(\id\tens\omega)(W)\,;\,\omega\in\B(\s H)_*\}$ {\rm(}resp.\ $\{(\omega\tens\id)(W)\,;\,\omega\in\B(\s H)_*\}${\rm)}.
\item We have
$
W\in M\tens\widehat{M}
$, 
$
V\in\widehat{M}'\tens M,
$
and
$
\widetilde{V}\in M'\tens\widehat{M}'.
$
In particular, we have the commutation relations $[W_{12},\, V_{23}] =  0$ and $[V_{12},\, \widetilde V_{23}] =  0$.
\item The coproduct $\Delta:M\rightarrow M\tens M$ of $\cal G$ (resp.\ $\widehat{\Delta}:\widehat{M}'\rightarrow\widehat{M}'\tens \widehat{M}'$ of $\widehat{\cal G}$) satisfies 
\begin{align*}
\Delta(x) &= V (x \otimes 1) V^* = W^* (1 \otimes x) W,\quad \text{for all } x\in M \\
\text{{\rm(}resp.\ } \widehat{\Delta}(x) &=V^* (1\otimes x) V = \widetilde V (x \otimes 1) \widetilde V^*,\quad \text{for all } x\in\widehat{M}'\text{{\rm)}}.\qedhere
\end{align*}
\end{enumerate}
\end{prop}

\begin{prop}\label{prop34}(cf.\ 3.2.\ (i), 3.6.\ (ii) \cite{E08} and 11.1.2 \cite{DC})
For all $n\in N$, we have:
\begin{enumerate}
	\item $[V, \alpha(n) \otimes 1] =0$,\quad $[V, \widehat \beta(n^{\rm o}) \otimes 1] = 0$,\quad   $[V, 1 \otimes \widehat\alpha(n)] = 0$,\quad $[V,  1 \otimes \widehat \beta(n^{\rm o})] = 0$;
	\item $V(1 \otimes \alpha(n)) =  (\widehat\alpha(n) \otimes 1) V$,\quad $V(\beta(n^{\rm o}) \otimes 1) =  (1 \otimes \beta(n^{\rm o}))V$;
	\item $[W, \widehat \beta(n^{\rm o}) \otimes 1] = 0$,\quad $[W, \widehat\alpha(n) \otimes 1] = 0$,\quad $[W, 1 \otimes \beta(n^{\rm o})]= 0$,\quad $[W, 1 \otimes \widehat\alpha(n)] = 0$;
	\item $W(1 \otimes \widehat \beta(n^{\rm o})) =  (\beta(n^{\rm o}) \otimes 1) W$,\quad $W(\alpha(n) \otimes 1) =  (1 \otimes \alpha(n))W$;
	\item $[\widetilde V, \alpha(n) \otimes 1] = 0$,\quad $[\widetilde V, \beta(n^{\rm o}) \otimes 1] = 0$,\quad $[\widetilde V, 1 \otimes \alpha(n)] = 0$,\quad $[\widetilde V, 1 \otimes \widehat \beta(n^{\rm o})] =  0$;
	\item $\widetilde V(1 \otimes \beta(n^{\rm o})) =  (\widehat \beta(n^{\rm o}) \otimes 1) \widetilde V$,\quad $\widetilde V(\widehat\alpha(n) \otimes 1) =  (1 \otimes \widehat\alpha(n))\widetilde V$.\qedhere
\end{enumerate}
\end{prop}

\begin{prop}\label{NonComRel}(cf.\ 11.1.4 \cite{DC})
For all $n\in N$, we have:
\begin{enumerate}
	\item $W(\beta(n^{\rm o})\tens 1)=W(1\tens\alpha(n))$,\quad $(1\tens\widehat{\beta}(n^{\rm o}))W=(\alpha(n)\tens 1)W$;
	\item $V(1\tens\beta(n^{\rm o}))=V(\widehat{\alpha}(n)\tens 1)$,\quad $(1\tens\alpha(n))V=(\beta(n^{\rm o})\tens 1)V$;
	\item $\widetilde{V}(\widehat{\beta}(n^{\rm o})\tens 1)=\widetilde{V}(1\tens\widehat{\alpha}(n))$,\quad $(1\tens\beta(n^{\rm o}))\widetilde{V}=(\widehat{\alpha}(n)\tens 1)\widetilde{V}$.\qedhere
\end{enumerate}
\end{prop}

		\subsection{Weak Hopf C*-algebras associated with a measured quantum groupoid on a finite basis}\label{WHC*A}
		
We recall the definitions and the main results concerning the weak Hopf C*-algebras associated with a measured quantum groupoid on a finite basis, cf.\ \S 11.2 \cite{DC} (with different notations and conventions, cf.\ \S 2.3 \cite{BC}). Let us fix a measured quantum groupoid ${\cal G}=(N,M,\alpha,\beta,\Delta,T,T',\epsilon)$ on the finite-dimensional basis $N=\bigoplus_{1\leqslant l\leqslant k}{\rm M}_{n_l}(\GC)$.

\begin{nbs}
With the notations of \S\ref{MQGfinitebasis}, we denote by $S$ (resp.\ $\widehat{S}$) the norm closure of the subalgebra\index[symbol]{sb@$S$, $\widehat{S}$, weak Hopf C*-algebras}
\[
\{(\omega\tens\id)(V)\,;\,\omega \in \B(\s H)_\ast\} \quad (\text{resp.}\,\{(\id\tens\omega)(V)\, ;\, \omega \in \B(\s H)_\ast\}).
\]
According to \S 11.2 \cite{DC}, we have the following statements:
\begin{itemize}
\item the Banach space $S$ (resp.\ $\widehat{S}$) is a non-degenerate C*-subalgebra of $\B(\s H)$, weakly dense in $M$ (resp.\ $\widehat{M}'$);
\item the C*-algebra $S$ (resp.\ $\widehat{S}$) is endowed with the faithful non-degenerate *-representa\-tions:\index[symbol]{l@$L$, $R$, $\rho$, $\lambda$, canonical representations of $S$ and $\widehat{S}$} 
\begin{align*}
L&:S\rightarrow\B(\s H)\,;\,x\mapsto x;\quad R:S\rightarrow\B(\s H)\,;\,x\mapsto UL(x)U^* \\
\text{{\rm(}resp.\ }\rho&:\widehat{S}\rightarrow\B(\s H)\,;\,x\mapsto x;\quad \lambda:\widehat{S}\rightarrow\B(\s H)\,;\,x\mapsto U\rho(x)U^*\text{{\rm)}};
\end{align*}
\item $\alpha(N)\subset\M(S)$, $\beta(N^{\rm o})\subset\M(S)$, $\beta(N^{\rm o})\subset\M(\widehat{S})$ and $\widehat{\alpha}(N)\subset\M(\widehat{S})$; 
\item $V\in\M( \widehat S \otimes S)$, $W \in \M( S \otimes \lambda(\widehat S))$ and $\widetilde V \in \M(R(S) \otimes \widehat S)$;
\item $\Delta$ {\rm (}resp.\ $\widehat{\Delta}${\rm )} restricts to a strictly continuous *-homomorphism $\delta:S\rightarrow\M(S\tens S)$ (resp.\ $\widehat{\delta}:\widehat{S}\rightarrow\M(\widehat{S}\tens\widehat{S})$), which uniquely extends to a strictly continuous *-homomorphism $\delta:\M(S)\rightarrow\M(S\tens S)$ (resp.\ $\widehat{\delta}:\M(\widehat{S})\rightarrow\M(\widehat{S}\tens\widehat{S})$) satisfying $\delta(1_S)=q_{\beta\alpha}$ (resp.\ $\widehat{\delta}(1_{\widehat{S}})=q_{\widehat\alpha\beta}$);
\item $\delta$ (resp.\ $\widehat{\delta}$) is coassociative and satisfies  
$
[\delta(S)(1_S \otimes S)] = \delta(1_S) (S \otimes S)=[\delta(S)(S \otimes 1_S)]
$
(resp.\ $[\widehat\delta(\widehat S)(1_{\widehat S} \otimes \widehat S)]  = \widehat\delta(1_{\widehat S}) (\widehat S \otimes \widehat S) = [\widehat\delta(\widehat S)(\widehat S \otimes 1_{\widehat S})]$);
\item the unital faithful *-homomorphisms $\alpha:N\rightarrow\M(S)$ and $\beta:N^{\rm o}\rightarrow\M(S)$ satisfy
	\[
	\delta(\alpha(n))=\delta(1_S)(\alpha(n)\tens 1_S) \; \text{ and } \; \delta(\beta(n^{\rm o}))=\delta(1_S)(1_S\tens\beta(n^{\rm o})),\; \text{ for all } n\in N;
	\]
\item the unital faithful *-homomorphisms $\beta:N^{\rm o}\rightarrow\M(\widehat{S})$ and $\widehat{\alpha}:N\rightarrow\M(\widehat{S})$ satisfy
	\[
	\widehat{\delta}(\beta(n^{\rm o}))=\widehat{\delta}(1_{\widehat{S}})(\beta(n^{\rm o})\tens 1_{\widehat{S}})\; \text{ and } \; \widehat{\delta}(\widehat{\alpha}(n))=\widehat{\delta}(1_{\widehat{S}})(1_{\widehat{S}}\tens\widehat{\alpha}(n)),\; \text{ for all } n\in N.\qedhere
	\]
\end{itemize}
\end{nbs}

\begin{defin}
With the above notations, we call the pair $(S,\delta)$ {\rm(}resp.\ $(\widehat{S},\widehat{\delta})${\rm)} the weak Hopf C*-algebra {\rm(}resp.\ dual weak Hopf C*-algebra{\rm)} associated with the measured quantum groupoid ${\cal G}$.
\end{defin}

\begin{rk}
With the notations of the above definition, the pair $(\widehat{S},\widehat{\delta})$ is the weak Hopf C*-algebra of $\widehat{\cal G}$ while its dual weak Hopf C*-algebra is the pair $(R(S),\delta_R)$, where $R(S)=USU^*$ and the coproduct $\delta_R$ is given by $\delta_R(y):=\widetilde{V}^*(1\tens y)\widetilde{V}$ for all $y\in R(S)$.
\end{rk}

		\subsection{Measured quantum groupoid associated with a monoidal equivalence}\label{sectionColinking}

We will recall the construction of the measured quantum groupoid associated with a monoidal equivalence between two locally compact quantum groups provided by De Commer \cite{DC,DC3}. First of all, we will need to recall the definitions and the crucial results of De Commer \cite{DC,DC3}.

\begin{defin}
Let $\QG$ be a locally compact quantum group. A right (resp.\ left) Galois action of $\QG$ on a von Neumann algebra $N$ is an ergodic integrable right (resp.\ left) action $\alpha_N:N\rightarrow N\tens {\rm L}^{\infty}(\QG)$ $($resp.\ $\gamma_N:N\rightarrow {\rm L}^{\infty}(\QG)\tens N)$ such that the crossed product $N\rtimes_{\alpha_N}\QG$ $($resp.\ $\QG\;_{\gamma_N}\!\!\ltimes N)$ is a type I factor. Then, the pair $(N,\alpha_N)$ (resp. $(N,\gamma_N)$) is called a right $($resp.\ left$)$ Galois object for $\QG$.
\end{defin}

Let $\QG$ be a locally compact quantum group and let us fix a right Galois object $(N,\alpha_N)$ for $\QG$. In his thesis, De Commer was able to build a locally compact quantum group $\QH$ equipped with a left Galois action $\gamma_N$ on $N$ commuting with $\alpha_N$, {\it i.e.\ }$(\id\tens\alpha_N)\gamma_N=(\gamma_N\tens\id)\alpha_N$. This construction is called the \textit{reflection technique} and $\QH$ is called the \textit{reflected locally compact quantum group across} $(N,\alpha_N)$.\hfill\break
In a canonical way, he was also able to associate a right Galois object $(O,\alpha_O)$ for $\QH$ and a left Galois action $\gamma_O:O\rightarrow {\rm L}^{\infty}(\QG)\tens O$ of $\QG$ on $O$ commuting with $\alpha_O$. Finally, De Commer has built a measured quantum groupoid 
\[
{\cal G}_{\QH,\QG}=(\GC^2,M,\alpha,\beta,\Delta,T,T',\epsilon)
\]
where: $M={\rm L}^{\infty}(\QH)\oplus N\oplus O\oplus {\rm L}^{\infty}(\QG)$; $\Delta:M\rightarrow M\tens M$ is made up of the coactions and coproducts of the constituents of $M$; the operator-valued weights $T$ and $T'$ are given by the invariants weights; the non-normalized Markov trace $\epsilon$ on $\GC^2$ is simply given by $\epsilon(a,b)=a+b$ for all $(a,b)\in\GC^2$. Moreover, the source and target maps $\alpha$ and $\beta$ have range in ${\cal Z}(M)$ and generate a copy of $\GC^4$.\hfill\break
Conversely, if ${\cal G}=(\GC^2,M,\alpha,\beta,\Delta,T,T',\epsilon)$ is a measured quantum groupoid whose source and target maps have range in ${\cal Z}(M)$ and generate a copy of $\GC^4$, then ${\cal G}$ is of the form ${\cal G}_{\QH,\QG}$ in a unique way, where $\QH$ and $\QG$ are locally compact quantum groups canonically associated with $\cal G$.

\medskip

In what follows, we fix a measured quantum groupoid ${\cal G}=(\GC^2,M,\alpha,\beta,\Delta,T,T',\epsilon)$ whose source and target maps have range in ${\cal Z}(M)$ and generate a copy of $\GC^4$. It is worth noticing that for such a groupoid we have:

\begin{lem}(cf.\ 2.21 \cite{BC})
$\widehat\alpha=\beta$ and $\widehat{\beta}=\alpha$.
\end{lem}

Following the notations introduced in \cite{DC}, we recall the precise description of the left and right regular representations $W$ and $V$ of $\cal G$ introduced in the previous section. We identify $M$ with its image by $\pi$ in $\B(\s H)$, where $(\s H,\pi,\Lambda)$ is the \GNS construction for $M$ endowed with the \nsf weight $\varphi=\epsilon\circ\alpha^{-1}\circ T$. We also consider the \nsf weight $\psi=\epsilon\circ\beta^{-1}\circ T'$. Denote by $(\varepsilon_1,\varepsilon_2)$ the standard basis of the vector space $\GC^2$.\index[symbol]{eaa@$(\varepsilon_1,\varepsilon_2)$, standard basis of $\GC^2$}  

\begin{nbs}\label{not10} Let us introduce some useful notations and make some remarks concerning them.
\begin{itemize}
	\item For $i,j=1,2$, we define the following nonzero central self-adjoint projection of $M$:\index[symbol]{pa@$p_{ij}$}
	\[
	p_{ij}:=\alpha(\varepsilon_i)\beta(\varepsilon_j).
	\]
	It follows from $\beta(\varepsilon_1)+\beta(\varepsilon_2)=1_M$ and $\alpha(\varepsilon_1)+\alpha(\varepsilon_2)=1_M$ that
	\[
	\alpha(\varepsilon_i)=p_{i1}+p_{i2} \quad \text{and} \quad \beta(\varepsilon_j)=p_{1j}+p_{2j},\quad \text{for all }i,j=1,2.
	\]
	\item We have 
	\begin{center}	
	$\Delta(1)=\alpha(\varepsilon_1)\tens \beta(\varepsilon_1)+\alpha(\varepsilon_2)\tens \beta(\varepsilon_2)$ \; and \; $\widehat{\Delta}(1)=\beta(\varepsilon_1)\tens \beta(\varepsilon_1)+\beta(\varepsilon_2)\tens \beta(\varepsilon_2)$
	\end{center} 
	since $\widehat{\alpha}=\beta$.
	\item Let $M_{ij}:=p_{ij}M$, for $i,j=1,2$. Then, $M_{ij}$ is a nonzero von Neumann subalgebra of $M$.\index[symbol]{mc@$M_{ij}$}
	\item Let $\s H_{ij}:=p_{ij}\s H$, for $i,j=1,2$. Then, $\s H_{ij}$ is a nonzero Hilbert subspace of $\s H$ for all $i,j=1,2$.\index[symbol]{h@$\s H_{ij}$}
	\item Let $\varphi_{ij}:=\varphi$\,$\restriction_{(M_{ij})_+}$ and $\psi_{ij}:=\psi$\,$\restriction_{(M_{ij})_+}$, for $i,j=1,2$. Then, $\varphi_{ij}$ and $\psi_{ij}$ are \nsf weights on $M_{ij}$.\index[symbol]{pb@$\varphi_{ij}$, $\psi_{ij}$}
	\item For all $i,j,k=1,2$, we denote by $\Delta_{ij}^k:M_{ij}\rightarrow M_{ik}\tens M_{kj}$ the unital normal *-homomorphism given by\index[symbol]{da@$\Delta_{ij}^k$} 
	\[
	\Delta_{ij}^k(x):=(p_{ik}\tens p_{kj})\Delta(x),\quad \text{for all } x\in M_{ij}.
	\]
	\item We have $Jp_{kl}=p_{kl}J$, $\widehat{J}p_{kl}=p_{lk}\widehat{J}$ and $Up_{kl}=p_{lk}U$ for all $k,l=1,2$. We define the anti-unitaries $J_{kl}:\s H_{kl}\rightarrow\s H_{kl}$, $\widehat{J}_{kl}:\s H_{kl}\rightarrow\s H_{lk}$ and the unitary $U_{kl}:\s H_{kl}\rightarrow\s H_{lk}$ by setting $J_{kl}=p_{kl}Jp_{kl}$, $\widehat{J}_{kl}=p_{lk}\widehat{J}p_{kl}$ and $U_{kl}=p_{lk}Up_{kl}=\widehat{J}_{kl}J_{kl}$.
	\item For $i,j,k,l=1,2$, let $\Sigma_{ij\tens kl}:=\Sigma_{\s H_{ij}\tens\s H_{kl}}:\s H_{ij}\tens\s H_{kl}\rightarrow\s H_{kl}\tens\s H_{ij}$.\qedhere
\end{itemize}
\end{nbs}

We readily obtain:
\[
M=\bigoplus_{i,j=1,2}M_{ij};\quad \s H=\bigoplus_{i,j=1,2}\s H_{ij};\quad \Delta(p_{ij})=p_{i1}\tens p_{1j}+p_{i2}\tens p_{2j},\text{ for all }i,j=1,2.
\]
Note that in terms of the parts $\Delta_{ij}^k$ of $\Delta$, the coassociativity condition reads as follows:
\[
(\Delta_{ik}^l\tens\id_{M_{kj}})\Delta_{ij}^k=(\id_{M_{il}}\tens\Delta_{lj}^k)\Delta_{ij}^l,\quad \text{for all } i,j,k,l=1,2.
\]
The \GNS representation for $(M_{ij},\varphi_{ij})$ is obtained by restriction of the \GNS representation of $(M,\varphi)$ to $M_{ij}$. In particular, the \GNS space $\s H_{\varphi_{ij}}$ is identified with $\s H_{ij}$.

\begin{prop}\label{prop35} For all $i,j,k,l=1,2$, we have:
\begin{align*}
(p_{ij} \otimes 1) V (p_{kl} \otimes 1) & = \delta_k^i (p_{ij} \otimes p_{jl}) V (p_{il} \otimes p_{jl});\\
(1 \otimes p_{ij}) W (1 \otimes p_{kl}) & = \delta_j^l (p_{ik} \otimes p_{ij}) W (p_{ik} \otimes p_{kj});\\
(1\tens p_{ji})\widetilde{V}(1\tens p_{lk}) & =\delta_j^l (p_{ki}\tens p_{ji})\widetilde{V}(p_{ki}\tens p_{jk}).\qedhere
\end{align*}
\end{prop}

\begin{nbs} The operators $V$, $W$ and $\widetilde{V}$ each splits up into eight unitaries\index[symbol]{vb@$V_{jl}^i$, $W_{ik}^j$, $\widetilde{V}_{ki}^j$} 
\[
V_{jl}^i : \s H_{il} \otimes \s H _{jl} \rightarrow \s H_{ij} \otimes \s H_{jl},\; 
W_{ik}^j : \s H_{ik} \otimes \s H_{kj} \rightarrow \s H_{ik} \otimes \s H_{ij} \text{ and }
\widetilde{V}_{ki}^j : \s H_{ki}\tens\s H_{jk}\rightarrow\s H_{ki}\tens\s H_{ji}
\] 
for $i,j,k,l=1,2$, given by $V_{jl}^i=(p_{ij} \otimes p_{jl}) V (p_{il} \otimes p_{jl})$, $ W_{ik}^j=(p_{ik} \otimes p_{ij}) W (p_{ik} \otimes p_{kj})$ and $\widetilde{V}_{ki}^j=(p_{ki}\tens p_{ji})\widetilde{V}(p_{ki}\tens p_{jk})$.
\end{nbs}

Let $i,j,k,l,l'=1,2$. These unitaries are related to each other by the following relations (cf.\ \ref{inifinproj}):
\[
W_{ik}^j=\Sigma_{ij\tens ik}(U_{ji}\tens 1)V_{ik}^j(U_{jk}^*\tens 1)\Sigma_{ik\tens kj};\;
\widetilde{V}_{ki}^j=\Sigma_{ji\tens ki}(1\tens U_{ik})V_{ik}^j(\tens U_{ik}^*)\Sigma_{ki\tens jk};
\]
\[
\widetilde{V}_{ki}^j=(U_{ik}\tens U_{ij})W_{ik}^j(U_{ik}^*\tens U_{kj}^*).
\]
Furthermore, we also have:
\[
(V_{jl}^i)^*=(J_{il}\tens\widehat{J}_{lj})V_{lj}^i(J_{ij}\tens\widehat{J}_{jl}) \quad \text{and} \quad 
(W_{ik}^l)^*=(\widehat{J}_{ki}\tens J_{kj})W_{ki}^j(\widehat{J}_{ik}\tens J_{ij}).
\]
Moreover, these unitaries satisfy the following pentagonal equations:
\begin{equation*}\label{penteq}
\begin{split}
(V_{jk}^i)_{12} (V_{kl}^i)_{13} (V_{kl}^j)_{23} = (V_{kl}^j)_{23} (V_{jl}^i)_{12};&\quad
(W_{ij}^k)_{12} (W_{ij}^l)_{13} (W_{jk}^l)_{23} = (W_{ik}^l)_{23}(W_{ij}^k)_{12};\\
(\widetilde{V}_{ji}^k)_{12}(\widetilde{V}_{ji}^l)_{13}&(\widetilde{V}_{kj}^l)_{23}=(\widetilde{V}_{ki}^l)_{23}(\widetilde{V}_{ji}^k)_{12}.
\end{split}
\end{equation*}
We also have the following commutation relations:
\begin{equation*}\label{comrel}
(V_{kj}^l)_{23}(W_{ll'}^j)_{12} = (W_{ll'}^k)_{12}(V_{kj}^{l'})_{23};\quad 
(V_{ki}^l)_{12}(\widetilde{V}_{ki}^{j})_{23}=(\widetilde{V}_{ki}^{j})_{23}(V_{ki}^l)_{12}.
\end{equation*}
Furthermore, we have
\begin{equation*}
\Delta_{ij}^k(x)=(W_{ik}^j)^*(1\tens x)W_{ik}^j=V_{kj}^i(x\tens 1)(V_{kj}^i)^*,\quad\text{for all } x\in M_{ij}.
\end{equation*}
Note that for all $\omega\in\B(\s H)_*$ we have:
\begin{equation*}\label{eqdefpi}
\begin{split}
({\rm id}  \otimes p_{jl}\omega p_{jl}) (V_{jl}^i) = p_{ij} ({\rm id}  \otimes \omega)(V) p_{il};& \quad (p_{ik}\omega p_{ik} \otimes {\rm id} )(W_{ik}^j) = 
p_{ij} (\omega \otimes {\rm id} )(W) p_{kj};\\
(p_{ki}\omega p_{ki} \tens \id)(\widetilde{V}_{ki}^j)&=p_{ji}(\id\tens\omega)(\widetilde{V})p_{jk}.
\end{split}
\end{equation*}

\begin{prop}\label{prop36}
Let $i , j=1,2$ such that $i\neq j$. With the notations of \ref{not10}, we have:
\begin{enumerate}
	\item $\QG_i:=(M_{ii} , \Delta_{ii}^i , \varphi_{ii}, \psi_{ii})$ is a locally compact quantum group whose left {\rm(}resp.\ right\,{\rm)} regular representation is $W_{ii}^i$ {\rm(}resp.\ $V_{ii}^i)$;
	\item $(M_{ij} , \Delta_{ij}^j)$ is a right Galois object for $\QG_j$ whose canonical implementation is $V_{jj}^i$;
	\item $(M_{ij} , \Delta_{ij}^i)$ is a left Galois object for $\QG_i$ whose canonical implementation is $W_{ii}^j$;
	\item the actions $\Delta_{ij}^j$ and $\Delta_{ij}^i$ on $M_{ij}$ commute;
	\item the Galois isometry associated with the right Galois object $(M_{ij},\Delta_{ij}^j)$ for $\QG_j$ (cf.\ 6.4.1, 6.4.2 \cite{DC}) is the unitary $\Sigma_{ij\tens jj}(W_{ij}^j)^*\Sigma_{ij\tens ij}$.\qedhere
\end{enumerate}
\end{prop}

\begin{defin}
A measured quantum groupoid $(\GC^2,M,\alpha,\beta,\Delta,T,T',\epsilon)$ such that the sour\-ce and 
target maps have range in ${\cal Z}(M)$ and generate a copy of $\GC^4$ will be denoted by ${\cal G}
_{\QG_1,\QG_2}$, where $\QG_i=(M_{ii},\Delta_{ii}^i,\varphi_{ii},\psi_{ii})$ (cf.\ \ref{prop36}) and will be called a colinking measured quantum groupoid.
\end{defin}

\begin{defin}
Let $\QG$ and $\QH$ be two locally compact quantum groups. We say that $\QG$ and $\QH$ are monoidally equivalent if there exists a colinking measured quantum groupoid ${\cal G}_{\QG_1,\QG_2}$ between two locally compact quantum groups $\QG_1$ and $\QG_2$ such that $\QH$ $($resp.\ $\QG)$ is isomorphic to $\QG_1$ $($resp.\ $\QG_2)$.\index[symbol]{g@${\cal G}_{\QG_1,\QG_2}$, colinking measured quantum groupoid}
\end{defin}

Let $(S,\delta)$ be the weak Hopf C*-algebra associated with $\cal G$. Note that 
\[
p_{ij}=\alpha(\varepsilon_i)\beta(\varepsilon_j)\in{\cal Z}(\M(S)),\quad\text{for all } i,j=1,2.
\]

\begin{nbs}
Let us recall the notations below (cf.\ 2.26 \cite{BC}).
\begin{enumerate}
	\item Let $S_{ij}:=p_{ij}S$, for $i,j=1,2$. Then, $S_{ij}$ is a C*-algebra (actually a closed two-sided ideal) of $S$ weakly dense in $M_{ij}$.\index[symbol]{sc@$S_{ij}$}
	\item For $i,j,k=1,2$, let 
	$
	\iota_{ij}^k:\M(S_{ik}\tens S_{kj})\rightarrow\M(S\tens S)
	$
	\index[symbol]{ib@$\iota_{ij}^k$}be the unique strictly continuous extension of the inclusion map $S_{ik}\tens S_{kj}\subset S\tens S$ satisfying $\iota_{ij}^k(1_{S_{ik}\tens S_{kj}})=p_{ik}\tens p_{kj}$. 
	\item Let $\delta_{ij}^k:S_{ij}\rightarrow\M(S_{ik}\tens S_{kj})$ be the unique *-homomorphism such that\index[symbol]{db@$\delta_{ij}^k$} 
\[
\iota_{ij}^k\circ\delta_{ij}^k(x)=(p_{ik}\tens p_{kj})\delta(x),\quad\text{for all }x\in S_{ij}.\qedhere
\]
\end{enumerate}
\end{nbs}

With these notations, we have:

\begin{prop}\label{prop5}(cf.\ 7.4.13, 7.4.14 \cite{DC}, 2.27 \cite{BC})
Let $i,j,k,l=1,2$.
\begin{enumerate}
	\item $(\delta_{ik}^l\tens\id_{S_{kj}})\delta_{ij}^k=(\id_{S_{il}}\tens\delta_{lj}^k)				\delta_{ij}^l$.
	\item $\delta_{ij}^k(x)=(W_{ik}^j)^*(1_{\s H_{ik}}\tens x)W_{ik}^j=V_{kj}^i(x\tens 1_{\s H_{kj}})(V_{kj}^i)^*$, for all $x\in S_{ij}$.
	\item $[\delta_{ij}^k(S_{ij})(1_{S_{ik}}\tens S_{kj})]=S_{ik}\tens S_{kj}=[\delta_{ij}				^k(S_{ij})(S_{ik}\tens 1_{S_{kj}})]$. In particular, we have 
	\[
	S_{kj}=[(\id_{S_{ik}}\tens\omega)\delta_{ij}^k(x)\,;\,x\in S_{ij},\,\omega\in\B(\s H_{kj})_*].
	\]
	\item The pair $(S_{jj},\delta_{jj}^j)$ is the Hopf C*-algebra associated with $\QG_j$.\qedhere
\end{enumerate}
\end{prop}

\section{Contributions to the notions of semi-regularity and regularity}\label{RegMQG}

\setcounter{thm}{0}

\numberwithin{thm}{section}
\numberwithin{prop}{section}
\numberwithin{lem}{section}
\numberwithin{cor}{section}
\numberwithin{propdef}{section}
\numberwithin{nb}{section}
\numberwithin{nbs}{section}
\numberwithin{rk}{section}
\numberwithin{rks}{section}
\numberwithin{defin}{section}
\numberwithin{ex}{section}
\numberwithin{exs}{section}
\numberwithin{noh}{section}
\numberwithin{conv}{section}

The notion of regular measured quantum groupoid has been introduced in \cite{E05} and stu\-died in the compact case. Note that this notion has been generalized in the setting of pseudo-multiplicative unitaries, cf.\ \cite{Tim1,Tim2}. The notion of semi-regular measured quantum groupoid has been introduced in \cite{BC,C}, where the notions of regularity and semi-regularity have been studied in the case of a finite-dimensional basis.

\medbreak

In this chapter, we fix a measured quantum groupoid ${\cal G}=(N,M,\alpha,\beta,\Delta,T,T',\epsilon)$ on the finite-dimensional basis $N=\bigoplus_{1\leqslant l\leqslant k}\,{\rm M}_{n_l}(\GC)$ and we use all the notations introduced in \S\S \ref{MQGfinitebasis}, \ref{WHC*A}. In the appendix (cf.\ \ref{defLeftBoundedVector}), for any $\xi\in\s H$ we have given the definition of the operator 
\begin{center}
$R^{\alpha}_{\xi}\in\B(\s H_{\epsilon},\s H)$ \quad (resp.\ $L^{\beta}_{\xi}\in\B(\s H_{\epsilon},\s H)$) 
\end{center} 
and the definition of the weakly dense ideal of $\alpha(N)'$ (resp.\ $\beta(N^{\rm o})'$) 
\begin{center}
$\K_{\alpha}:=[R^{\alpha}_{\xi}(R^{\alpha}_{\eta})^*\,;\,\xi,\,\eta\in\s H]$ \quad (resp.\ $\K_{\beta}:=[L^{\beta}_{\xi}(L^{\beta}_{\eta})^*\,;\,\xi,\,\eta\in\s H]$). 
\end{center}
Note that $\K_{\alpha}$ and $\K_{\beta}$ are C*-subalgebras of $\K:=\K(\s H)$\index[symbol]{ke@$\K$, C*-algebra of compact operators on the \GNS space ${\rm L}^2({\cal G})$}.

\medbreak

We first recall the following important consequence of the irreducibility (cf.\ 2.13 \cite{BC}) of $\cal G$.

\begin{prop}(cf.\ 2.15 \cite{BC})
The Banach spaces $[S\widehat{S}]$ and $\C(V)$ (cf.\ \ref{notC(V)}) are C*-algebras and we have $[S\widehat{S}]=U\C(V)U^*$.
\end{prop}

\begin{defin}(cf.\ 4.7 \cite{E05}, 2.37 \cite{BC})
The groupoid $\cal G$ is said to be semi-regular (resp.\ regular) if we have $\K_{\beta}\subset{\cal C}(V)$ (resp.\ $\K_{\beta}={\cal C}(V)$).
\end{defin}

\begin{prop}(cf.\ 2.8 \cite{BC}, 3.2.8 \cite{C})
The following statements are equivalent:
\begin{enumerate}[label=(\roman*)]
	\item $\cal G$ is semi-regular {\rm(}resp.\ regular{\rm)}, {\it i.e.} $\K_{\alpha}\subset \C(W)$ {\rm(}resp.\ $\K_{\alpha}=\C(W)${\rm)}; $\phantom{\K_{\beta}}$
	\item $\widehat{\cal G}$ is semi-regular {\rm(}resp.\ regular{\rm)}, {\it i.e.} $\K_{\beta}\subset \C(V)$ {\rm(}resp.\ $\K_{\beta}=\C(V)${\rm)};
	\item $({\cal G}^{\rm o})^{\rm c}$ is semi-regular {\rm(}resp.\ regular{\rm)}, {\it i.e.} $\K_{\widehat{\alpha}}\subset \C(\widetilde{V})$ {\rm(}resp.\ $\K_{\widehat{\alpha}}=\C(\widetilde{V})${\rm)}.\qedhere
\end{enumerate}
\end{prop}

\begin{prop}\label{propReg4}(cf.\ 2.8 \cite{BC}, 3.2.9 \cite{C})
The following statements are equivalent:
\begin{enumerate}[label=(\roman*)]
	\item $\cal G$ is semi-regular {\rm(}resp.\ regular{\rm)}; $\phantom{\K_{\widehat{\beta}}}$
	\item $\K_{\widehat{\beta}}\subset[\widehat{S}S]$ {\rm(}resp.\ $\K_{\widehat{\beta}}=[\widehat{S}S]${\rm)};
	\item $\K_{\alpha}\subset [R(S)\widehat{S}]$ {\rm(}resp.\ $\K_{\alpha}=[R(S)\widehat{S}]${\rm)}; $\phantom{\K_{\widehat{\beta}}}$
	\item $\K_{\widehat{\alpha}}\subset [S\lambda(\widehat{S})]$ {\rm(}resp.\ $\K_{\widehat{\alpha}}=[S\lambda(\widehat{S})]${\rm)}.
\end{enumerate}
In particular, if $\cal G$ is regular we have $[\widehat{S}S]\subset\K$, $[R(S)\widehat{S}]\subset\K$ and $[S\lambda(\widehat{S})]\subset\K$ {\rm(}and also $\C(V)\subset\K$, $\C(W)\subset\K$ and $\C(\widetilde{V})\subset\K${\rm)}.
\end{prop}

The semi-regularity and the regularity of colinking measured quantum groupoids have been treated in detail in \S 2.5 \cite{BC}.

\begin{thm}\label{theo7}(cf.\ 2.45 \cite{BC}) Let ${\cal G}_{\QG_1,\QG_2}$ be a colinking measured quantum groupoid associated with two monoidally equivalent locally compact quantum groups $\QG_1$ and $\QG_2$. The groupoid ${\cal G}_{\QG_1,\QG_2}$ is semi-regular (resp.\ regular) if, and only if, $\QG_1$ and $\QG_2$ are semi-regular (resp.\ regular).
\end{thm}

In the following, we use the multi-index notation introduced in the appendix \S \ref{appendix} of this article (cf.\ \ref{not13}, \ref{not14} and \ref{rk14}) with $\gamma:=\alpha$ and $\pi:=\beta$.

\begin{lem}\label{lemReg} For all $\xi,\eta\in\s H$, we have 
\[
R_{\xi}^{\alpha}(R_{\eta}^{\alpha})^*=\sum_{I\in\s I}n_I^{-1}\cdot e_{\vphantom{\overline{I}}I}\theta_{\xi,\eta}e_{\overline{I}}
\quad \text{and} \quad
L_{\xi}^{\beta}(L_{\eta}^{\beta})^*=\sum_{I\in\s I}n_I^{-1}\cdot f_{\vphantom{\overline{I}}I}\theta_{\xi,\eta}f_{\overline{I}}.
\qedhere
\]
\end{lem}

\begin{proof}
For $\eta,\zeta\in\s H$, let $X_{\eta,\zeta}\in N$ be defined by $X_{\eta,\zeta}:=\sum_{I\in\s I} n_I^{-1}\langle e_I\eta,\,\zeta\rangle \cdot\varepsilon_I$. For all $x\in N$ and $\eta\in\s H$, we have $R^{\alpha}_{\eta}\Lambda_{\epsilon}(x)=\alpha(x)\eta=\sum_{I\in\s I} x_I \cdot e_I\eta$. $\vphantom{\delta_{\overline{I}}^J}$Let $\eta,\zeta\in\s H$ and $I\in\s I$. 
$
\langle (R_{\eta}^{\alpha})^*\zeta,\, \Lambda_{\epsilon}(\varepsilon_I)\rangle=\langle \zeta,\, e_I\eta\rangle.
$
By disjunction elimination, we prove that $\epsilon(\varepsilon_{JI})=\delta_{\overline{I}}^J n_I$ for all $I,J\in\s I$. On the other hand, we have$\vphantom{\delta_{\overline{I}}^J}$
\begin{align*}
\langle\Lambda_{\epsilon}(X_{\eta,\zeta}),\,\Lambda_{\epsilon}(\varepsilon_I)\rangle &= \epsilon(X_{\eta,\zeta}^*\varepsilon_I) \\
&=\epsilon\left(\sum_{J\in\s I} n_J^{-1}\overline{\langle e_{\overline{J}}\eta,\,\zeta\rangle} \cdot\varepsilon_{JI}\right)\\
&=\sum_{J\in\s I} n_J^{-1}\langle \zeta,\, e_{\overline{J}}\eta\rangle\epsilon(\varepsilon_{JI})\\
&=\langle \zeta,\, e_{I}\eta\rangle.
\end{align*}
Hence, $\langle (R_{\eta}^{\alpha})^*\zeta,\, \Lambda_{\epsilon}(\varepsilon_I)\rangle=\langle\Lambda_{\epsilon}(X_{\eta,\zeta}),\,\Lambda_{\epsilon}(\varepsilon_I)\rangle$. Hence, $(R_{\eta}^{\alpha})^*\zeta=\Lambda_{\epsilon}(X_{\eta,\zeta})$ for all $\eta,\zeta\in\s H$. Let $\xi,\eta\in\s H$. For all $\zeta\in\s H$, we have
\[
R_{\xi}^{\alpha}(R_{\eta}^{\alpha})^*\zeta=R_{\xi}^{\alpha}\Lambda_{\epsilon}(X_{\eta,\zeta})=\sum_{I\in\s I}(X_{\eta,\zeta})_I\cdot e_I\xi=\sum_{I\in\s I}n_I^{-1}\langle \eta,\,e_{\overline{I}}\zeta\rangle\cdot e_I\xi=\sum_{I\in\s I}n_I^{-1}\theta_{e_I\xi,\,e_I\eta}(\zeta).
\]
Hence, $R_{\xi}^{\alpha}(R_{\eta}^{\alpha})^*=\sum_{I\in\s I}n_I^{-1}\cdot e_{\vphantom{\overline{I}}I}\theta_{\xi,\eta}e_{\overline{I}}$. The second formula is proved in a similar way.
\end{proof}

We refer to \ref{propdef5} and \ref{ProjectionBis} for the definition of the operators $q_{\alpha}$, $q_{\beta}$ and $q_{\widehat{\alpha}}$. The propositions \ref{propReg1}, \ref{propReg2} and \ref{propReg3} below have to be compared with their corresponding statements in the quantum group case, cf.\ 3.2 b), 3.6 b) and 3.6 d) \cite{BS2}.

\begin{prop}\label{propReg1} The following statements are equivalent:
\begin{enumerate}[label=(\roman*)]
\item ${\cal G}$ is regular {\rm(}resp.\ semi-regular{\rm)};
\item $[(\K\tens 1)W(1\tens\K)]=[(\K\tens 1)q_{\alpha}(1\tens\K)]$ {\rm(}resp.\ $\supset[(\K\tens 1)q_{\alpha}(1\tens\K)]${\rm)}; $\phantom{\widetilde{V}q_{\beta}}$
\item $[(\K\tens 1)V(1\tens\K)]=[(\K\tens 1)q_{\beta}(1\tens\K)]$ {\rm(}resp.\ $\supset[(\K\tens 1)q_{\beta}(1\tens\K)]${\rm)}; $\phantom{\widetilde{V}}$
\item $[(\K\tens 1)\widetilde{V}(1\tens\K)]=[(\K\tens 1)q_{\widehat{\alpha}}(1\tens\K)]$ {\rm(}resp.\ $\supset[(\K\tens 1)q_{\widehat{\alpha}}(1\tens\K)]${\rm)}.\qedhere
\end{enumerate}
\end{prop}

\begin{proof} It is known that ${\cal G}$ is regular (resp.\ semi-regular) if, and only if, $\widehat{\cal G}$ is regular (resp.\ semi-regular). Therefore, it suffices to prove that (i) is equivalent to (ii). We have 
\[
[(\K\tens 1)W(1\tens\K)]=\Sigma({\cal C}(W)\tens\K) \quad \text{(cf.\ 3.1 \cite{BS2}}\text{{\rm)}}
\]
(${\cal C}(W)$ is a C$^*$-algebra regardless of the regularity of $\cal G$). Note that $\cal G$ is regular (resp.\ semi-regular) if, and only if, $\Sigma({\cal C}(W)\tens\K)=\Sigma(\K_{\alpha}\tens\K)$ (resp.\ $\Sigma({\cal C}(W)\tens\K)\supset\Sigma(\K_{\alpha}\tens\K)$). Let $\xi,\eta,\zeta,\chi\in\s H$. We have $e_{\vphantom{\overline{I}}I}\theta_{\xi,\,\eta}e_{\overline{I}}\tens\theta_{\zeta,\,\chi}=\theta_{e_I\xi,\, e_I\eta}\tens\theta_{\zeta,\, \chi}=\theta_{e_I\xi\tens\zeta,\, e_I\eta\tens\chi}$, for all $I\in\s I$. Hence, $\Sigma(e_{\vphantom{\overline{I}}I}\theta_{\xi,\, \eta}e_{\overline{I}}\tens\theta_{\zeta,\, \chi})=\theta_{\zeta\tens e_I\xi,\, e_I\eta\tens\chi}=\theta_{\zeta,\, e_I\eta}\tens\theta_{e_I\xi,\, \chi}=\theta_{\zeta,\, \eta}e_{\overline{I}}\tens e_{\vphantom{\overline{I}}I}\theta_{\xi,\, \chi}$ for all $I\in\s I$.
By Lemma \ref{lemReg}, we obtain
\begin{equation*}
\Sigma(R_{\xi}^{\alpha}(R_{\eta}^{\alpha})^*\tens\theta_{\zeta,\chi})=(\theta_{\zeta,\eta}\tens 1)q_{\alpha}(1\tens\theta_{\xi,\chi}).
\end{equation*}
Hence, $\Sigma(\K_{\alpha}\tens\K)=[(\K\tens 1)q_{\alpha}(1\tens\K)]$ and the equivalence $({\rm (i)}\Leftrightarrow{\rm (ii)})$ is proved.
\end{proof}

\begin{prop}\label{propReg2} If $\cal G$ is regular (resp.\ semi-regular), we have:
\begin{enumerate}
\item $[(S\tens 1)W(1\tens\K)]=[(S\tens 1)q_{\alpha}(1\tens\K)]$ {\rm(}resp.\ $\supset[(S\tens 1)q_{\alpha}(1\tens\K)]${\rm)};
$\phantom{q_{\widehat{\alpha}}q_{\beta}}$
\item $[(\K\tens 1)V(1\tens S)]=[(\K\tens 1)q_{\beta}(1\tens S)]$ {\rm(}resp.\ $\supset[(\K\tens 1)q_{\beta}(1\tens S)]${\rm)};
$\phantom{\widetilde{V}}$
\item $[(R(S)\tens 1)\widetilde{V}(1\tens\K)]=[(R(S)\tens 1)q_{\widehat{\alpha}}(1\tens\K)]$ {\rm(}resp.\ $\supset[(R(S)\tens 1)q_{\widehat{\alpha}}(1\tens\K)]${\rm)}.\qedhere
\end{enumerate} 
\end{prop}

\begin{proof} Assume that $\cal G$ is regular (resp.\ semi-regular). Let us prove the first statement. The others will be obtained by using similar arguments. Let $a,b\in\K$, $\omega\in\B(\s H)_*$ and $y=(\id\tens a\omega)(W)$. We have 
\begin{align*}
(y\tens 1)W(1\tens b)&=(\id\tens\omega\tens\id)(W_{12}W_{13}(1\tens a\tens b))\\
&=(\id\tens\omega\tens\id)(W_{23}W_{12}(1\tens W^*(a\tens b))).
\end{align*}
However, $W^*(\K\tens\K)=q_{\beta\alpha}(\K\tens\K)$. Moreover, since $[W,\,1\tens\beta(n^{\rm o})]=0$ for all $n\in N$, we have $[W_{12},q_{\beta\alpha,23}]=0$. Hence, $W_{23}W_{12}q_{\beta\alpha,23}=W_{23}W_{12}$. We obtain (cf.\ \ref{propReg1})
\begin{align*}
[(S\tens 1)W(1\tens\K)] &=[(\id\tens a\omega\tens\id)(W_{23}W_{12}(1\tens 1\tens b))\,;\,\omega\in\B(\s H)_*,\, a,\,,b\in\K]\\
&=[(\id\tens\omega a\tens\id)(W_{23}W_{12}(1\tens 1\tens b))\,;\,\omega\in\B(\s H)_*,\, a,b\in\K]\\
&=[(\id\tens\omega\tens\id)(((a\tens 1)W(1\tens b))_{23}W_{12})\,;\,\omega\in\B(\s H)_*,\, a,b\in\K]\\
&=[(\id\tens\omega\tens\id)(((a\tens 1)q_{\alpha}(1\tens b))_{23}W_{12})\,;\,\omega\in\B(\s H)_*,\, a,b\in\K] \\
(\text{resp.\ }&\supset [(\id\tens\omega\tens\id)(((a\tens 1)q_{\alpha}(1\tens b))_{23}W_{12})\,;\,\omega\in\B(\s H)_*,\, a,b\in\K]).
\end{align*}
However, for all $\omega\in\B(\s H)_*$ and $a,b\in\K$ we have
\[
(\id\tens\omega\tens\id)(((a\tens 1)q_{\alpha}(1\tens b))_{23}W_{12})=(\id\tens\omega a\tens\id)(q_{\alpha,23}W_{12}(1\tens 1\tens b)).
\]
Since $(1\tens\alpha(n))W=W(\alpha(n)\tens 1)$ for all $n\in N$, we have $q_{\alpha,23}W_{12}=W_{12}q_{\alpha,13}$. Hence, 
\[
(\id\tens\omega\tens\id)(((a\tens 1)q_{\alpha}(1\tens b))_{23}W_{12})=
((\id\tens\omega a)(W)\tens 1)q_{\alpha}(1\tens b)
\]
and the result is proved.
\end{proof}	

\begin{prop}\label{propReg3} 
If $\cal G$ is regular {\rm(}resp.\ semi-regular{\rm)}, then we have:
\begin{enumerate}
\item $[(S\tens 1)W(1\tens\lambda(\widehat{S}))]=[(S\tens 1)q_{\alpha}(1\tens\lambda(\widehat{S}))]$ {\rm(}resp.\ $\supset[(S\tens 1)q_{\alpha}(1\tens\lambda(\widehat{S}))]${\rm)}; $\phantom{q_{\beta}}$
\item $[(\widehat{S}\tens 1)V(1\tens S)]=[(\widehat{S}\tens 1)q_{\beta}(1\tens S)]$ {\rm(}resp.\ $\supset[(\widehat{S}\tens 1)q_{\beta}(1\tens S)]${\rm)};
\item $[(R(S)\tens 1)\widetilde{V}(1\tens\widehat{S})]=[(R(S)\tens 1)q_{\widehat{\alpha}}(1\tens\widehat{S})]$ {\rm(}resp.\ $\supset[(R(S)\tens 1)q_{\widehat{\alpha}}(1\tens\widehat{S})]${\rm)}.
\end{enumerate}
In particular, we have $[(S\tens 1)W(1\tens\lambda(\widehat{S}))]\subset S\tens\lambda(\widehat{S})$, $[(\widehat{S}\tens 1)V(1\tens S)]\subset \widehat{S}\tens S$ and $[(R(S)\tens 1)\widetilde{V}(1\tens\widehat{S})]\subset R(S)\tens\widehat{S}$.
\end{prop}

\begin{proof} We have the pentagonal equation $V_{12}V_{13}=V_{23}V_{12}V_{23}^*$. Since $V\in\M(\widehat{S}\tens S)$ is a partial isometry, we have $V^*(\widehat{S}\tens S)=q_{\widehat{\alpha}\beta}(\widehat{S}\tens S)$. Since $[V,\,1\tens\widehat{\alpha}(n)]=0$ for all $n\in N$, we have $[V_{12},\,q_{\widehat{\alpha}\beta,23}]=0$. Hence, $V_{23}V_{12}q_{\widehat{\alpha}\beta,23}=V_{23}V_{12}$. Hence,
\begin{align*}
[(\widehat{S}\tens 1)V(1\tens S)]&=[((\id\tens\omega)(V)\tens 1)V(1\tens y)\,;\,\omega\in\B(\s H)_*,\, y\in S] \\
&=[(\id\tens\omega\tens \id)(V_{12}V_{13}(1\tens 1\tens y))\,;\,\omega\in\B(\s H)_*,\, y\in S]   \phantom{\widehat{S}}\\
&=[(\id\tens\omega\tens \id)(V_{23}V_{12}V_{23}^*(1\tens x\tens y))\,;\,\omega\in\B(\s H)_*,\, y\in S,\,x\in\widehat{S}]   \phantom{\widehat{S}}\\
&=[(\id\tens\omega\tens \id)(V_{23}V_{12}(1\tens x\tens y))\,;\, \omega\in\B(\s H)_*,\, x\in\widehat{S},\, y\in S]   \phantom{\widehat{S}}\\
&=[(\id\tens\omega\tens \id)(V_{23}(1\tens 1\tens y)V_{12})\,;\, \omega\in\B(\s H)_*,\, y\in S]  \phantom{\widehat{S}}\\
&=[(\id\tens\omega\tens \id)(((a\tens 1)V(1\tens y))_{23}V_{12})\,;\,\omega\in\B(\s H)_*,\, a\in\K,\,y\in S].   \phantom{\widehat{S}}\\[-.5em]
\intertext{Let $X:=[(\id\tens\omega\tens \id)(((a\tens 1)q_{\beta}(1\tens y))_{23}V_{12})\,;\,\omega\in\B(\s H)_*,\, a\in\K,\,y\in S]$. Since $\cal G$ is regular (resp.\ semi-regular), it follows from \ref{propReg2} that 
\begin{center}
$[(\widehat{S}\tens 1)V(1\tens S)]=X$ \quad (resp.\ $[(\widehat{S}\tens 1)V(1\tens S)]\supset X$).
\end{center} 
However, since $(1\tens\beta(n^{\rm o}))V=V(\beta(n^{\rm o})\tens 1)$ for all $n\in N$, we have}~\\[-2em]
X&=[(\id\tens\omega a\tens \id)(q_{\beta,23}V_{12}(1\tens 1\tens y))\,;\,\omega\in\B(\s H)_*,\, a\in\K,\,y\in S]   \phantom{\widehat{S}}\\
&=[(\id\tens\omega\tens \id)(V_{12}q_{\beta,13}(1\tens 1\tens y))\,;\,\omega\in\B(\s H)_*,\,y\in S]   \phantom{\widehat{S}}\\
&=[((\id\tens\omega)(V)\tens 1)q_{\beta}(1\tens y)\,;\,\omega\in\B(\s H)_*,\,y\in S]   \phantom{\widehat{S}}\\
&=[(\widehat{S}\tens 1)q_{\beta}(1\tens S)].
\end{align*}
The second statement is proved and the third one follows by applying it to $\widehat{\cal G}$. We obtain the first statement by combining the third one with the formulas $W=(U^*\tens U^*)\widetilde{V}(U\tens U)$ and $\widehat{\alpha}={\rm Ad}_U\circ\alpha$. Finally, the last statement follows from the inclusions $\beta(N^{\rm o})\subset\M(S)$, $\widehat{\beta}(N^{\rm o})\subset\M(\widehat{S})$, $\alpha(N)\subset\M(S)$ and $\widehat{\alpha}(N)\subset\M(\widehat{S})$.
\end{proof}

In the result below, we refer again to \ref{propdef5} and \ref{ProjectionBis} for the definition of the operators $q_{\beta\widehat{\beta}}$, $q_{\widehat{\alpha}\alpha}$ and $q_{\widehat{\beta}\beta}$.

\begin{cor}\label{corReg}
If $\cal G$ is regular (resp.\ semi-regular), then we have:
\begin{enumerate}
\item $[(1\tens\lambda(\widehat{S}))W(S\tens 1)]=[(1\tens\lambda(\widehat{S}))q_{\beta\widehat{\beta}}(S\tens 1)]$ {\rm(}resp.\ $\supset[(1\tens\lambda(\widehat{S}))q_{\beta\widehat{\beta}}(S\tens 1)]${\rm)};
\item $[(1\tens S)V(\widehat{S}\tens 1)]=[(1\tens S)q_{\widehat{\alpha}\alpha}(\widehat{S}\tens 1)]$ {\rm(}resp.\ $\supset[(1\tens S)q_{\widehat{\alpha}\alpha}(\widehat{S}\tens 1)]${\rm)};  
$\phantom{q_{\beta\widehat{\beta}}}$
\item $[(1\tens\widehat{S})\widetilde{V}(R(S)\tens 1)]=[(1\tens\widehat{S})q_{\widehat{\beta}\beta}(R(S)\tens 1)]$ {\rm(}resp.\ $\supset[(1\tens\widehat{S})q_{\widehat{\beta}\beta}(R(S)\tens 1)]${\rm)}.
\end{enumerate}
If $\cal G$ is regular, then we have $[(1\tens\lambda(\widehat{S}))W(S\tens 1)]\subset S\tens\lambda(\widehat{S})$, $[(1\tens S)V(\widehat{S}\tens 1)]\subset \widehat{S}\tens S$ and $[(1\tens\widehat{S})\widetilde{V}(R(S)\tens 1)]\subset R(S)\tens\widehat{S}$.
\end{cor}

\begin{proof}
This is a direct consequence of Proposition \ref{propReg3} and the formulas $\widehat\beta={\rm Ad}_U\circ\beta$, $\widehat\alpha={\rm Ad}_U\circ\alpha$, $W=\Sigma(U\tens 1)V(U^*\tens 1)\Sigma$  and $\widetilde{V}=\Sigma(1\tens U)V(1\tens U^*)\Sigma$. The second statement follows from the inclusions $\beta(N^{\rm o})\subset\M(S)$, $\widehat{\beta}(N^{\rm o})\subset\M(\widehat{S})$, $\alpha(N)\subset\M(S)$ and $\widehat{\alpha}(N)\subset\M(\widehat{S})$.
\end{proof}

\section{Measured quantum groupoids on a finite basis in action}

\setcounter{thm}{0}

\numberwithin{thm}{subsection}
\numberwithin{prop}{subsection}
\numberwithin{lem}{subsection}
\numberwithin{cor}{subsection}
\numberwithin{propdef}{subsection}
\numberwithin{nb}{subsection}
\numberwithin{nbs}{subsection}
\numberwithin{rk}{subsection}
\numberwithin{rks}{subsection}
\numberwithin{defin}{subsection}
\numberwithin{ex}{subsection}
\numberwithin{exs}{subsection}
\numberwithin{noh}{subsection}
\numberwithin{conv}{subsection}

	\subsection{Continuous actions, crossed product and biduality}\label{subsectionBiduality}

In this section, we fix a measured quantum groupoid ${\cal G}=(N,M,\alpha,\beta,\Delta,T,T',\epsilon)$ on the finite-dimensional basis $N=\bigoplus_{1\leqslant l\leqslant k}\,{\rm M}_{n_l}(\GC)$ and we use all the notations introduced in \S\S\ \ref{MQGfinitebasis}, \ref{WHC*A}. In the following, we recall the definitions, notations and results of \S\S\ 3.1, 3.2.1, 3.2.2 and 3.3.1 \cite{BC} (see also \cite{C} chapter 4).

		\subsubsection{Notion of actions of measured quantum groupoids on a finite basis}\label{CoAction}

\begin{lem}\label{lem19}
Let $A$ and $B$ be two C*-algebras, $f:A\rightarrow\M(B)$ a *-homomorphism and $e\in\M(B)$. The following statements are equivalent:
\begin{enumerate}[label=(\roman*)]
	\item there exists an approximate unit $(u_{\lambda})_{\lambda}$ of $A$ such that $f(u_{\lambda})\rightarrow e$ with respect to the strict topology;
	\item $f$ extends to a strictly continuous *-homomorphism $f:\M(A)\rightarrow\M(B)$, necessarily unique, such that $f(1_A)=e$;
	\item $[f(A)B]=eB$.
\end{enumerate}
In that case, $e$ is a self-adjoint projection, for all approximate unit $(v_{\mu})_{\mu}$ of $A$ we have $f(v_{\mu})\rightarrow e$ with respect to the strict topology and $[Bf(A)]=Be$.
\end{lem}

\begin{defin}\label{defactmqg}
An action of $\cal G$ on a C*-algebra $A$ is a pair $(\beta_A,\delta_A)$ consisting of a non-degenerate *-homomorphism $\beta_A:N^{\rm o}\rightarrow\M(A)$ and a faithful *-homomorphism $\delta_A:A\rightarrow\M(A\tens S)$ such that:
\begin{enumerate}
\item $\delta_A$ extends to a strictly continuous *-homomorphism $\delta_A:\M(A)\rightarrow\M(A\tens S)$ such that $\delta_A(1_A)=q_{\beta_A\alpha}$ (cf.\ \ref{ProjectionCAlg});
\item $(\delta_A\tens\id_S)\delta_A=(\id_A\tens\delta)\delta_A$;
$\phantom{q_{\beta_A\alpha}\ref{rk12}}$ 
\item $\delta_A(\beta_A(n^{\rm o}))=q_{\beta_A\alpha}(1_A\tens\beta(n^{\rm o}))$, for all $n\in N$.
\end{enumerate}
We say that the action $(\beta_A,\delta_A)$ is strongly continuous if we have
\[
[\delta_A(A)(1_A\tens S)]=q_{\beta_A\alpha}(A\tens S).
\]
If that case, we say that the triple $(A,\beta_A,\delta_A)$ is a $\cal{G}$-C*-algebra. 
\end{defin}

\begin{rks}\label{rk12}
\begin{itemize}
\item By \ref{lem19}, the condition 1 is equivalent to requiring that for some (and then any) approximate unit $(u_{\lambda})$ of $A$, we have $\delta_A(u_{\lambda})\rightarrow q_{\beta_A\alpha}$ with respect to the strict topology of $\M(A\tens S)$. It is also equivalent to $[\delta_A(A)(A\tens S)]=q_{\beta_A\alpha}(A\tens S)$.
	\item Condition 1 implies that the *-homomorphisms $\delta_A\tens\id_S$ and $\id_A\tens\delta$ extend uniquely to strictly continuous *-homomorphisms from $\M(A\tens S)$ to $\M(A\tens S\tens S)$ such that 
$
(\delta_A\tens\id_S)(1_{A\tens S})=q_{\beta_A\alpha,12}
$
and 
$
(\id_A\tens\delta)(1_{A\tens S})=q_{\beta\alpha,23}.
$ In particular, condition 2 does make sense and we denote by $\delta_A^2:=(\delta_A\tens\id_S)\delta_A:A\rightarrow\M(A\tens S\tens S)$\index[symbol]{dca@$\delta_A^2$, iterated coaction map} the iterated coaction map.\qedhere
\end{itemize}
\end{rks}

\begin{exs}\label{ex2} Let us give two basic examples.
\begin{itemize}
\item $(S,\beta,\delta)$ is a $\cal{G}$-C*-algebra.
\item Let $\beta_{N^{\rm o}}:=\id_{N^{\rm o}}$. Let $\delta_{N^{\rm o}}:N^{\rm o}\rightarrow\M(N^{\rm o}\tens S)$ be the faithful unital *-homomorphism given by $\delta_{N^{\rm o}}(n^{\rm o}):=q_{\beta_{N^{\rm o}}\alpha}(1_{N^{\rm o}}\tens\beta(n^{\rm o}))$ for all $n\in N$. Then, the pair $(\beta_{N^{\rm o}},\delta_{N^{\rm o}})$ is an action of $\cal{G}$ on $N^{\rm o}$ called the trivial action.\qedhere
\end{itemize}
\end{exs}

\begin{prop}\label{prop39}
Let $(\delta_A,\beta_A)$ be an action of $\cal G$ on $A$. We have the following statements:
\begin{enumerate}
\item the iterated coaction map $\delta_A^2$ extends uniquely to a strictly continuous *-homomorphism $\delta_A^2:\M(A)\rightarrow\M(A\tens S\tens S)$ such that $\delta_A^2(1_A)=q_{\beta_A\alpha,12}q_{\beta\alpha,23}$; moreover, we have $\delta_A^2(m)=(\delta_A\tens\id_S)\delta_A(m)=(\id_A\tens\delta)\delta_A(m)$ for all $m\in\M(A)$;
\item for all $n\in N$, we have $\delta_A(\beta_A(n^{\rm o}))=(1_A\tens\beta(n^{\rm o}))q_{\beta_A\alpha}$;
\item if $(\beta_A,\delta_A)$ is strongly continuous, then we have $[(1_A\tens S)\delta_A(A)]=(A\tens S)q_{\beta_A\alpha}$.\qedhere
\end{enumerate}
\end{prop}

Let us provide a more explicit definition of what an action of the dual measured quantum groupoid $\widehat{\cal G}$ on a C*-algebra $B$ is.

\begin{defin}
An action of $\widehat{\cal G}$ on a C*-algebra $B$ is a pair $(\alpha_B,\delta_B)$ consisting of a non-degenerate *-homomorphism $\alpha_B:N\rightarrow\M(B)$ and a faithful *-homomorphism $\delta_B:B\rightarrow\M(B\tens\widehat{S})$ such that:
\begin{enumerate}
\item $\delta_B$ extends to a strictly continuous *-homomorphism $\delta_B:\M(B)\rightarrow\M(B\tens\widehat{S})$ such that $\delta_B(1_B)=q_{\alpha_B\beta}$ (cf.\ \ref{ProjectionCAlg});
\item $(\delta_B\tens\id_{\widehat{S}})\delta_B=(\id_B\tens\widehat\delta)\delta_B$;
\item $\delta_B(\alpha_B(n))=q_{\alpha_B\beta}(1_B\tens\widehat\alpha(n))$, for all $n\in N$.
\end{enumerate}
We say that the action $(\alpha_B,\delta_B)$ is strongly continuous if we have  
\[
[\delta_B(B)(1_B\tens\widehat S)]=q_{\alpha_B\beta}(B\tens\widehat S).
\]
If $(\delta_B,\alpha_B)$ is a strongly continuous action of $\widehat{\cal G}$ on $B$, we say that the triple $(B,\alpha_B,\delta_B)$ is a  $\widehat{\cal G}$-C*-algebra. 
\end{defin}

\begin{rks} As for actions of ${\cal G}$, we have:
\begin{itemize}
\item the condition 1 is equivalent to requiring that for some $\vphantom{q_{\alpha_B\beta}}$(and then any) approximate unit $(u_{\lambda})_{\lambda}$ $\vphantom{\widehat{S}}$of $B$ we have $\delta_B(u_{\lambda})\rightarrow q_{\alpha_B\beta}$ with respect to the strict topology, which is also equivalent to the relation $[\delta_B(B)(B\tens\widehat{S})]=q_{\alpha_B\beta}(B\tens\widehat{S})$;
\item the *-homomorphisms $\id_B\tens\widehat\delta$ and $\delta_B\tens\id_{\widehat{S}}$ extend uniquely to strictly continuous *-homomorphisms from $\vphantom{\delta_B^2}$ $\M(B\tens \widehat{S})$ to $\M(B\tens\widehat{S}\tens\widehat{S})$ such that $(\id_B\tens\widehat{\delta})(1_{B\tens\widehat{S}})=q_{\widehat{\alpha}\beta,23}$ and $(\delta_B\tens\id_{\widehat{S}})(1_{B\tens\widehat{S}})=q_{\alpha_B\beta,12}\vphantom{\delta_B^2\widehat{S}}$. In particular, condition 2 does make sense and we denote by $\delta_B^2:=(\delta_B\tens\id_{\widehat{S}})\delta_B:B\rightarrow\M(B\tens\widehat{S}\tens\widehat{S})$ the iterated coaction map.\qedhere
\end{itemize}
\end{rks}

\begin{exs} Let us give two basic examples:
\begin{itemize}
\item $(\widehat{S},\widehat{\alpha},\widehat{\delta})$ is a $\widehat{\cal{G}}$-C*-algebra;
\item Let $\alpha_N:=\id_N$ and $\delta_{N}:N\rightarrow\M(N\tens \widehat{S})\,;\, n\mapsto q_{\alpha_N\beta}(1_{N}\tens\widehat{\alpha}(n))$; then, the pair $(\alpha_N,\delta_{N})$ is an action of $\widehat{\cal{G}}$ on $N$ called the trivial action.\qedhere
\end{itemize}
\end{exs}

\begin{prop}
Let $(\alpha_B,\delta_B)$ be an action of $\widehat{\cal G}$ on $B$. We have the following statements:
\begin{enumerate}
\item the iterated coaction map $\delta_B^2$ extends uniquely to a strictly continuous *-homomorphism $\delta_B^2:\M(B)\rightarrow\M(B\tens\widehat{S}\tens\widehat{S})$ such that $\delta_B^2(1_B)=q_{\alpha_B\beta,12}q_{\widehat{\alpha}\beta,23}$; moreover, we have $\delta_B^2(m)=(\delta_B\tens\id_{\widehat{S}})\delta_B(m)=(\id_B\tens\widehat{\delta})\delta_B(m)$ for all $m\in\M(B)$;
\item for all $n\in N$, we have $\delta_B(\alpha_B(n))=(1_B\tens\widehat{\alpha}(n))q_{\alpha_B\beta}$;$\vphantom{\widehat{S}}$ 
\item if $(\alpha_B,\delta_B)$ is strongly continuous, then we have $[(1_B\tens\widehat{S})\delta_B(B)]=(B\tens \widehat{S})q_{\alpha_B\beta}$.\qedhere
\end{enumerate}
\end{prop}

\begin{defin}\label{defEquiHom}
For $i=1,2$, let $A_i$ {\rm(}resp.\ $B_i${\rm)} be a C*-algebra. For $i=1,2$, let $(\beta_{A_i},\delta_{A_i})$ {\rm(}resp.\ $(\alpha_{B_i},\delta_{B_i})${\rm)} be an action of $\cal{G}$ {\rm(}resp.\ $\widehat{\cal G}${\rm)} on $A_i$ {\rm(}resp.\,$B_i${\rm)}. A non-degenerate *-homomorphism 
$
f:A_1\rightarrow\M(A_2)
$ {\rm(}resp.\ 
$
f:B_1\rightarrow\M(B_2)
${\rm)} 
is said to be $\cal{G}$-equivariant {\rm(}resp.\ $\widehat{\cal G}$-equivariant{\rm)} if $(f\tens\id_S)\delta_{A_1}=\delta_{A_2}\circ f$ and $ f\circ\beta_{A_1}=\beta_{A_2}$ (resp.\ $(f\tens\id_{\widehat{S}})\delta_{B_1}=\delta_{B_2}\circ f$ and $f\circ\alpha_{B_1}=\alpha_{B_2}$).
\end{defin}

\begin{rk}\label{rkEqMorph} With the notations and hypotheses of \ref{defEquiHom}, if $f$ satisfies the relation $(f\tens\id_S)\delta_{A_1}=\delta_{A_2}\circ f$ (resp.\ $(f\tens\id_{\widehat{S}})\delta_{B_1}=\delta_{B_2}\circ f$), then $f$ satisfies necessarily the relation $ f\circ\beta_{A_1}=\beta_{A_2}$ (resp.\ $f\circ\alpha_{B_1}=\alpha_{B_2}$), {\it i.e.}\ $f$ is $\cal G$-equivariant (resp.\ $\widehat{\cal G}$-equivariant). Indeed, let $n\in N$. For all $a\in A_1$ and $x\in A_2$, we have
\begin{align*}
\delta_{A_2}(f(\beta_{A_1}(n^{\rm o}))f(a)x)&=(f\tens\id_S)\delta_{A_1}(\beta_{A_1}(n^{\rm o})a)\delta_{A_2}(x)\\
&=(1_{A_2}\tens\beta(n^{\rm o}))(f\tens\id_S)\delta_{A_1}(a)\delta_{A_2}(x)\\
&=(1_{A_2}\tens\beta(n^{\rm o}))\delta_{A_2}(f(a)x)\\
&=\delta_{A_2}(\beta_{A_2}(n^{\rm o})f(a)x).
\end{align*}
Hence, $f(\beta_{A_1}(n^{\rm o}))f(a)x=\beta_{A_2}(n^{\rm o})f(a)x$ for all $a\in {A_1}$ and $x\in {A_2}$ since $\delta_{A_2}$ is faithful. Hence, we have $f(\beta_{A_1}(n^{\rm o}))=\beta_{A_2}(n^{\rm o})$ since $f$ is non-degenerate.
\end{rk}

\begin{nb}
We denote by ${\sf Alg}_{\mathcal{G}}$\index[symbol]{ac@${\sf Alg}_{\mathcal{G}}$, category of $\cal G$-C*-algebras} the category whose objects are the $\cal{G}$-C*-algebras and whose set of arrows between $\cal{G}$-C*-algebras is the set of $\cal G$-equivariant non-degenerate *-homomorphisms.
\end{nb}

		\subsubsection{Crossed product and dual action}

Let us fix a strongly continuous action $(\beta_A,\delta_A)$ of $\cal G$ on a C*-algebra $A$. 

\begin{nbs}\label{notCrossedProduct}
The *-representation\index[symbol]{pc@$\pi_L$} 
\begin{center}
$
\pi_L:=(\id_A\tens L)\circ\delta_A:A\rightarrow\Lin(A\tens\s H)
$
\end{center}
of $A$ on the Hilbert $A$-module $A\tens\s H$ extends uniquely to a strictly/*-strongly continuous faithful *-representation $\pi_L:\M(A)\rightarrow\Lin(A\tens\s H)$ such that $\pi_L(1_A)=q_{\beta_A\alpha}$. Moreover, we have
$
\pi_L(m)=\pi_L(m)q_{\beta_A\alpha}=q_{\beta_A\alpha}\pi_L(m)
$ 
for all $m\in\M(A)$. Consider the Hilbert $A$-module
\[
{\cal E}_{A,L}:=q_{\beta_A\alpha}(A\tens\s H).
\]
\index[symbol]{eb@${\cal E}_{A,L}$}By restricting $\pi_L$, we obtain a strictly/*-strongly continuous faithful unital *-representation
\[
\pi:\M(A)\rightarrow\Lin({\cal E}_{A,L})\; ;\; m\mapsto \restr{\pi_L(m)}{{\cal E}_{A,L}}\!\!.
\]
We have $[1_A\tens T,\,q_{\beta_A\alpha}]=0$ for all $T\in\M(\widehat{S})$. We then obtain a strictly/*-strongly continuous unital *-representation
\[
\widehat{\theta}:\M(\widehat{S})\rightarrow\Lin({\cal E}_{A,L})\; ;\; T\mapsto \restr{(1_A\tens T)}{{\cal E}_{A,L}}\!\!.
\]
\index[symbol]{pd@$\pi$, $\widehat{\theta}$}Note that if $\beta_A$ is faithful, then so is $\widehat{\theta}$.
\end{nbs}

\begin{propdef}
The norm closed subspace of $\Lin({\cal E}_{A,L})$ spanned by the products of the form $\pi(a)\widehat{\theta}(x)$ for $a\in A$ and $x\in\widehat{S}$ is a C*-subalgebra called the (reduced) crossed product of $A$ by the strongly continuous action $(\beta_A,\delta_A)$ of $\cal G$ and denoted by $A\rtimes{\cal G}$.\index[symbol]{ad@$A\rtimes{\cal G}$, crossed product}
\end{propdef}

In particular, the morphism $\pi$ (resp.\ $\widehat{\theta}$) defines a faithful unital *-homomorphism (resp.\ unital *-homomorphism) $\pi:\M(A)\rightarrow\M(A\rtimes{\cal G})$ (resp.\  $\widehat{\theta}:\widehat{S}\rightarrow\M(A\rtimes{\cal G})$).

\medbreak

Since $[\widetilde{V},\,\alpha(n)\tens 1]=0$, we have $[\widetilde{V}_{23},q_{\beta_A\alpha,12}]=0$. The operator $\widetilde{V}_{23}\in\Lin(A\tens\s H\tens\s H)$ restricts to a partial isometry 
\[
X:=\restr{\widetilde{V}_{23}}{{\cal E}_{A,L}\tens\s H}\,\in\Lin({\cal E}_{A,L}\tens\s H),
\]
whose initial and final projections are $X^*X=\restr{q_{\widehat{\beta}\alpha,23}}{{\cal E}_{A,L}\tens\s H}$ and $XX^*=\restr{q_{\widehat{\alpha}\beta,23}}{{\cal E}_{A,L}\tens\s H}$.

\begin{propdef}\label{defDualAct}
Let
\[
\delta_{A\rtimes{\cal G}}:A\rtimes{\cal G}\rightarrow\Lin({\cal E}_{A,L}\tens\s H) \quad \text{and} \quad \alpha_{A\rtimes{\cal G}}:N\rightarrow\M(A\rtimes{\cal G})
\] 
be the linear maps defined by:
\begin{itemize}
\item $\delta_{A\rtimes{\cal G}}(b):=X(b\tens 1)X^*$, for all
$b\in A\rtimes{\cal G}$;
\item $\alpha_{A\rtimes{\cal G}}(n):=\widehat{\theta}(\widehat{\alpha}(n))=(1_A\tens\widehat{\alpha}(n))\!\!\restriction_{{\cal E}_{A,L}}$, for all $n\in N$.
\end{itemize}
Then, $\delta_{A\rtimes{\cal G}}$ is a faithful *-homomorphism and $\alpha_{A\rtimes{\cal G}}$ is a non-degenerate *-homomorphism. Moreover, we have the following statements:
\begin{enumerate}
\item $\delta_{A\rtimes{\cal G}}(\pi(a)\widehat{\theta}(x))=(\pi(a)\tens 1_{\widehat{S}})
(\widehat{\theta}\tens\id_{\widehat{S}})\widehat{\delta}(x)$, for all $a\in A$ and $x\in\widehat{S}$; in particular, $\delta_{A\rtimes{\cal G}}$ takes its values in $\M((A\rtimes{\cal G})\tens\widehat{S})$;
\item $\alpha_{A\rtimes{\cal G}}(n)\pi(a)\widehat{\theta}(x)=\pi(a)\widehat{\theta}(\widehat{\alpha}
(n)x)$ and $\pi(a)\widehat{\theta}(x)\alpha_{A\rtimes{\cal G}}(n)=\pi(a)\widehat{\theta}(x
\widehat{\alpha}(n))$ for all $n\in N$, $a\in A$ and $x\in\widehat{S}$.\qedhere
\end{enumerate}
\end{propdef}

\begin{propdef}
With the notations of \ref{defDualAct}, the pair $(\alpha_{A\rtimes{\cal G}},\delta_{A\rtimes{\cal G}})$ is a strongly continuous action of $\widehat{\cal G}$ on ${A\rtimes{\cal G}}$ called the dual action of $(\beta_A,\delta_A)$.
\end{propdef}

In a similar way, we define the crossed product of a C*-algebra $B$ by a strongly continuous action $(\alpha_B,\delta_B)$ of the dual measured quantum groupoid $\widehat{\cal G}$. 

\begin{nbs}
The *-representation\index[symbol]{pe@$\widehat{\pi}_{\lambda}$} 
\begin{center}
$
\widehat\pi_\lambda:=(\id_B\tens\lambda)\circ\delta_B:B\rightarrow\Lin(B\tens\s H)
$
\end{center}
of $B$ on the Hilbert $B$-module $B\tens\s H$ extends uniquely to a strictly/*-strongly continuous faithful *-representation $\widehat\pi_\lambda:\M(B)\rightarrow\Lin(B\tens\s H)$ such that $\widehat\pi_\lambda(1_B)=q_{\alpha_B\widehat\beta}$. Moreover, we have
$
\widehat\pi_\lambda(m)=\widehat\pi_\lambda(m)q_{\alpha_B\widehat\beta}=q_{\alpha_B\widehat\beta}\widehat\pi_\lambda(m)
$, 
for all $m\in\M(B)$. Consider the Hilbert $B$-module
\[
{\cal E}_{B,\lambda}:=q_{\alpha_B\widehat\beta}(B\tens\s H).
\]
\index[symbol]{ec@${\cal E}_{B,\lambda}$}By restricting $\widehat\pi_\lambda$, we obtain a strictly/*-strongly continuous faithful unital *-representation
\[
\widehat\pi:\M(B)\rightarrow\Lin({\cal E}_{B,\lambda})\; ;\; m\mapsto \restr{\widehat\pi_\lambda(m)}{{\cal E}_{B,\lambda}}\!.
\]
We have $[1_B\tens T,\,q_{\alpha_B\widehat\beta}]=0$ for all $T\in\M(S)$. We then obtain a strictly/*-strongly continuous unital *-representation
\[
\theta:\M(S)\rightarrow\Lin({\cal E}_{B,\lambda})\; ;\; T\mapsto \restr{(1_B\tens T)}{{\cal E}_{B,\lambda}}\!.
\]
\index[symbol]{pf@$\widehat{\pi}$, $\theta$}Note that if $\alpha_B$ is faithful, then so is $\theta$.
\end{nbs}

\begin{propdef}
The norm closed subspace of $\Lin({\cal E}_{B,\lambda})$ spanned by the products of the form $\widehat{\pi}(b)\theta(x)$ for $b\in B$ and $x\in S$ is a C*-subalgebra called the (reduced) crossed product of $B$ by the strongly continuous action $(\alpha_B,\delta_B)$ of $\widehat{\cal G}$ and denoted by $B\rtimes\widehat{\cal G}$.\index[symbol]{bc@$B\rtimes\widehat{\cal G}$, crossed product}
\end{propdef}

In particular, the morphism $\widehat\pi$ (resp.\ ${\theta}$) defines a faithful unital *-homomorphism (resp.\ unital *-homomorphism) $\widehat\pi:\M(B)\rightarrow\M(B\rtimes\widehat{\cal G})$ (resp.\  ${\theta}:S\rightarrow\M(B\rtimes\widehat{\cal G})$).

\medbreak

Since $[V,\,\beta(n^{\rm o})\tens 1]=0$, we have $[V_{23},q_{\alpha_B\beta,12}]=0$. The operator $V_{23}\in\Lin(B\tens\s H\tens\s H)$ restricts to a partial isometry 
\[
Y:=\restr{V_{23}}{{\cal E}_{B,\lambda}\tens\s H}\,\in\Lin({\cal E}_{B,\lambda}\tens\s H),
\]
whose initial and final projections are $Y^*Y=\restr{q_{\widehat{\alpha}\beta,23}}{{\cal E}_{B,\lambda}\tens\s H}$ and $YY^*=\restr{q_{\beta\alpha,23}}{{\cal E}_{B,\lambda}\tens\s H}$.

\begin{propdef}\label{defDualActBis}
Let 
\[
\delta_{B\rtimes\widehat{\cal G}}:B\rtimes\widehat{\cal G}\rightarrow\Lin({\cal E}_{B,\lambda}\tens\s H) \quad \text{and} \quad \beta_{B\rtimes\widehat{\cal G}}:N^{\rm o}\rightarrow\Lin({\cal E}_{B,\lambda})
\]
be the linear maps defined by:
\begin{itemize}
\item $\delta_{B\rtimes\widehat{\cal G}}(c):=Y(c\tens 1_{\s H})Y^*$, for all $c\in B\rtimes\widehat{\cal G}$;
\item $\beta_{ B\rtimes\widehat{\cal G}}(n^{\rm o}):=\theta(\beta(n^{\rm o}))=\restr{(1_B\tens\beta(n^{\rm o}))}{{\cal E}_{B,\lambda}\tens\s H}$, for all $n\in N$.
\end{itemize}
Then, $\delta_{B\rtimes\widehat{\cal G}}$ is a faithful *-homomorphism and $\beta_{ B\rtimes\widehat{\cal G}}$ is a non-degenerate *-homomorphism. Moreover, we have the following statements:
\begin{enumerate}
\item $\delta_{B\rtimes\widehat{\cal G}}(\widehat\pi(b)\theta(x))=(\widehat\pi(b)\tens 1_{S})
(\theta\tens\id_S)\delta(x)$, for all $b\in B$ and $x\in S$; in particular, $\delta_{B\rtimes\widehat{\cal G}}$ takes its values in $\M((B\rtimes\widehat{\cal G})\tens S)$;
\item $\beta_{B\rtimes\widehat{\cal G}}(n^{\rm o})\widehat\pi(b)\theta(x)=\widehat\pi(b)\theta(\beta(n^{\rm o})x)$ and $\widehat\pi(b)\theta(x)\beta_{B\rtimes\widehat{\cal G}}(n^{\rm o})=\widehat\pi(b)\theta(x\beta(n^{\rm o}))$ for all $n\in N$, $b\in B$ and $x\in S$.\qedhere
\end{enumerate}
\end{propdef}

\begin{propdef}
With the notations of \ref{defDualActBis}, the pair $(\beta_{B\rtimes\widehat{\cal G}},\delta_{B\rtimes\widehat{\cal G}})$ is a strongly continuous action of ${\cal G}$ on ${B\rtimes\widehat{\cal G}}$ called the dual action of $(\alpha_B,\delta_B)$.
\end{propdef}

		\subsubsection{Takesaki-Takai duality}

Let $(\beta_A,\delta_A)$ be a strongly continuous action of the groupoid ${\cal G}$ on a C*-algebra $A$.

\begin{nbs}
The *-representation\index[symbol]{pg@$\pi_R$}
\[
\pi_R:=(\id_A\tens R)\circ\delta_A:A\rightarrow\Lin(A\tens\s H)
\]
of $A$ on the Hilbert $A$-module $A\tens\s H$ extends uniquely to a strictly/*-strongly continuous faithful *-representation $\pi_R:\M(A)\rightarrow\Lin(A\tens\s H)$ satisfying $\pi_R(m)=(\id_A\tens R)\delta_A(m)$ for all $m\in\M(A)$ and $\pi_R(1_A)=q_{\beta_A\widehat{\alpha}}$. Consider the Hilbert $A$-module
\[
\er:=q_{\beta_A\widehat{\alpha}}(A\tens\s H).
\]
\index[symbol]{ed@$\er$}We recall that the Banach space
\[
D:=[\pi_R(a)(1_A\tens\lambda(x)L(y))\,;\, a\in A,\, x\in\widehat{S},\, y\in S]
\]
is a C*-subalgebra of $\Lin(A\tens\s H)$ such that $dq_{\beta_A\widehat{\alpha}}=d=dq_{\beta_A\widehat{\alpha}}$ for all $d\in D$. Moreover, we have $D(A\tens\s H)=\er$. We also recall that there exists a unique strictly/*-strongly continuous faithful *-representation 
$
j_D:\M(D)\rightarrow\Lin(A\tens\s H)
$ 
\index[symbol]{ja@$j_D$}extending the inclusion map $D\subset\Lin(A\tens\s H)$ such that $j_D(1_D)=q_{\beta_A\widehat{\alpha}}$.
\end{nbs}

\begin{prop}\label{IsoTT}
There exists a unique *-isomorphism $\phi:(A\rtimes{\cal G})\rtimes\widehat{\cal G}\rightarrow D$ such that $\phi(\widehat{\pi}(\pi(a)\widehat{\theta}(x))\theta(y))=\pi_R(a)(1_A\tens\lambda(x)L(y))$ for all $a\in A$, $x\in\widehat{S}$ and $y\in S$.
\end{prop}

\begin{nbs}
We denote $\K:=\K(\s H)$ for short. Let $\delta_0:A\tens\K\rightarrow\M(A\tens\K\tens S)$ be the *-homomorphism defined by 
$
\delta_0(a\tens k)=\delta_A(a)_{13}(1_A\tens k\tens 1_S)
$
for all $a\in A$ and $k\in\K$. The morphism $\delta_0$ extends uniquely to a strictly continuous *-homomorphism $\delta_0:\M(A\tens\K)\rightarrow\M(A\tens\K\tens S)$ such that $\delta_0(1_{A\tens\K})=q_{\beta_A\alpha,13}$. Let ${\cal V}\in\Lin(\s H\tens S)$\index[symbol]{vb@${\cal V}$} be the unique partial isometry such that $(\id_{\K}\tens L)({\cal V})=V$.
\end{nbs}

\begin{thm}\label{BidulityTheo}
There exists a unique strongly continuous action $(\beta_D,\delta_D)$ of $\cal G$ on the C*-algebra $D=[\pi_R(a)(1_A\tens\lambda(x)L(y))\,;\,a\in A,\,x\in\widehat{S},\, y\in S]$ defined by the relations: 
\[
(j_D\tens\id_S)\delta_D(d)={\cal V}_{23}\delta_0(d){\cal V}_{23}^*,\quad d\in D ; \quad
j_D(\beta_D(n^{\rm o}))=q_{\beta_A\widehat{\alpha}}(1_A\tens\beta(n^{\rm o})),\quad n\in N.
\]
Moreover, the canonical *-isomorphism $\phi:(A\rtimes{\cal G})\rtimes\widehat{\cal G}\rightarrow D$ (cf.\ \ref{IsoTT}) is $\cal G$-equivariant. If the groupoid $\cal G$ is regular, then we have $D=q_{\beta_A\widehat{\alpha}}(A\tens\K)q_{\beta_A\widehat{\alpha}}$.
\end{thm}

The $\cal G$-C*-algebra $D$ defined above will be referred to as the bidual $\cal G$-C*-algebra of $A$.\index[symbol]{dcb@$D$, bidual $\cal G$-C*-algebra}

	\subsection{Case of a colinking measured quantum groupoid}\label{colinkingMQG}

In this section, we fix a colinking measured quantum groupoid ${\cal G}:={\cal G}_{\QG_1,\QG_2}$ associated with two monoidally equivalent locally compact quantum groups $\QG_1$ and $\QG_2$. We follow all the notations recalled in \S \ref{sectionColinking} concerning the objects associated with $\cal G$.\newline

In the following, we recall the notations and the main results of \S 3.2.3 \cite{BC} concerning the equivalent description of the $\cal G$-C*-algebras in terms of $\QG_1$-C*-algebras and $\QG_2$-C*-algebras. Let us fix a $\cal G$-C*-algebra $(A,\beta_A,\delta_A)$.

\begin{nbs}\label{not12}
\begin{itemize}
\item The morphism $\beta_A:\GC^2\rightarrow\M(A)$ is central. Let $q_j:=\beta_A(\varepsilon_j)$\index[symbol]{qa@$q_j$, $q_{A,j}$} for $j=1,2$. Then, $q_j$ is a central self-adjoint projection of $\M(A)$ and $q_1+q_2=1_A$. Let $A_j:=q_jA$ for $j=1,2$. For $j=1,2$, $A_j$ is a C*-subalgebra (actually a closed two-sided ideal) of $A$ and we have $A=A_1\oplus A_2$.
\item For $j,k=1,2$, let 
$
\pi_j^k:\M(A_k\tens S_{kj})\rightarrow\M(A\tens S)
$ 
\index[symbol]{ph@$\pi_j^k$, $\pi_{A,j}^k$}be the unique strictly continuous extension of the inclusion $A_k\tens S_{kj}\subset A\tens S$ such that $\pi_{j}^k(1_{A_k\tens S_{kj}})=q_k\tens p_{kj}$.
\end{itemize}
In case of ambiguity, we will denote $\pi_{A,j}^k$ and $q_{A,j}$ instead of $\pi_j^k$ and $q_j$.
\end{nbs}

\begin{prop}\label{actprop}
For all $j,k=1,2$, there exists a unique faithful non-degenerate *-homo\-morphism 
\[\delta_{A_j}^k:A_j\rightarrow\M(A_k\tens S_{kj})\]
\index[symbol]{dd@$\delta_{A_j}^k$}such that for all $x\in A_j$, we have
\[
\pi_j^k\circ\delta_{A_j}^k(x)=(q_k\tens p_{kj})\delta_A(x)=(q_k\tens 1_S)\delta_A(x)=(1_A\tens \alpha(\varepsilon_k))\delta_A(x)=(1_A\tens p_{kj})\delta_A(x).
\] 
Moreover, we have:
\begin{enumerate}
\item $\displaystyle{\delta_A(a)=\sum_{k,j=1,2}\pi_j^k\circ\delta_{A_j}^k(q_ja)}$, for all $a\in A$;
\item $(\delta_{A_k}^l\tens\id_{S_{kj}})\delta_{A_j}^k=(\id_{A_l}\tens\delta_{lj}^k)\delta_{A_j}^l$, for all $j,k,l=1,2$;
\item $[\delta_{A_j}^k(A_j)(1_{A_k}\tens S_{kj})]=A_k\tens S_{kj}$, for all $j,k=1,2$; in particular, we have

{
\centering
$
A_k=[(\id_{A_k}\tens\omega)\delta_{A_j}^k(a)\,;\,a\in A_j,\,\omega\in\B(\s H_{kj})_*];
$
\par
}
\item $\delta_{A_j}^j:A_j\rightarrow\M(A_j\tens S_{jj})$ is a strongly continuous action of $\QG_j$ on $A_j$.\qedhere
\end{enumerate}
\end{prop}

From this concrete description of $\cal G$-C*-algebras we can also give a convenient description of the $\cal G$-equivariant *-homomorphisms. With the above notations, we have the result below.

\begin{prop}\label{lem7}
Let $A$ and $B$ be two $\cal G$-C*-algebras. For $k=1,2$, let $\iota_k:\M(B_k)\rightarrow\M(B)$ be the unique strictly continuous extension of the inclusion map $B_k\subset B$ such that $\iota_k(1_{B_k})=q_{B,k}$.
\begin{enumerate}
\item Let $f:A\rightarrow\M(B)$ be a non-degenerate $\cal G$-equivariant *-homomorphism. Then, for all $j=1,2$, there exists a unique non-degenerate *-homomorphism $f_j:A_j\rightarrow\M(B_j)$ such that for $k=1,2$ we have
\begin{equation}\label{eqmorpheq}
(f_k\tens\id_{S_{kj}})\circ\delta_{A_j}^k=\delta_{B_j}^k\circ f_j.
\end{equation}
Moreover, we have $f(a)=\iota_1\circ f_1(aq_{A,1})+\iota_2\circ f_2(aq_{A,2})$ for all $a\in A$.
\item Conversely, for $j=1,2$ let $f_j:A_j\rightarrow\M(B_j)$ be a non-degenerate *-homomorphism such that {\rm(}\ref{eqmorpheq}\,{\rm)} holds for all $j,k=1,2$. Then, the map $f:A\rightarrow\M(B)$, defined for all $a\in A$ by 
\[
f(a):=\iota_1\circ f_1(aq_{A,1})+\iota_2\circ f_2(aq_{A,2}),
\]
is a non-degenerate $\cal G$-equivariant *-homomorphism.\qedhere
\end{enumerate} 
\end{prop}

The above results show that for $j=1,2$ we have a functor 
\begin{center}
${\sf Alg}_{\cal G}\rightarrow{\sf Alg}_{\QG_j}\,;\,(A,\beta_A,\delta_A)\mapsto(A_j,\delta_{A_j}^j).$
\end{center}
In \S 4 \cite{BC}, it has been proved that if $\cal G$ is regular (cf.\ \ref{theo7}), then $(A,\delta_A,\beta_A)\rightarrow(A_1,\delta_{A_1}^1)$ is an equivalence of categories. Moreover, the authors build explicitly the inverse functor $(A_1,\delta_{A_1})\rightarrow(A,\beta_A,\delta_A)$. More precisely, to any $\QG_1$-C*-algebra $(A_1,\delta_{A_1})$ they associate a $\QG_2$-C*-algebra $(A_2,\delta_{A_2})$ in a canonical way. Then, the C*-algebra $A:=A_1\oplus A_2$ can be equipped with a strongly continuous action $(\beta_A,\delta_A)$ of the groupoid $\cal G$. This allowed them to build the inverse functor $(A_1,\delta_{A_1})\rightarrow(A,\beta_A,\delta_A)$. The equivalence of categories $(A_1,\delta_{A_1})\rightarrow(A_2,\delta_{A_2})$ generalizes the correspondence of actions for monoidally equivalent \textit{compact} quantum groups of De Rijdt and Vander Vennet \cite{RV}. We bring to the reader's attention that an induction procedure has been developed by De Commer in the von Neumann algebraic setting (cf.\ \S 8 \cite{DC}).

\medskip

In the following, we recall the notations and the main results of \S 4 \cite{BC}. We assume that the quantum groups $\QG_1$ and $\QG_2$ are regular.

\begin{nbs}
Let $\delta_{A_1}:A_1\rightarrow\M(A_1\tens S_{11})$ be a continuous action of $\QG_1$ on a C*-algebra $A_1$. Let us denote:\index[symbol]{de@$\delta_{A_1}^{(2)}$}
\[
\delta_{A_1}^1:=\delta_{A_1},\quad \delta_{A_1}^{(2)}:=(\id_{A_1}\tens\delta_{11}^2)\delta_{A_1}:A_1\rightarrow\M(A_1\tens S_{12}\tens S_{21}).
\]
Then, $\delta_{A_1}^{(2)}$ is a faithful non-degenerate *-homomorphism. In the following, we will identify $S_{21}$ with a C*-subalgebra of $\B(\s H_{21})$. Let\index[symbol]{ic@$\ind(A_1)$, induced C*-algebra}
\[
{\rm Ind}_{\QG_1}^{\QG_2}(A_1):=[(\id_{A_1\tens S_{12}}\tens\omega)\delta_{A_1}^{(2)}(a)\,;\,a\in A_1,\,\omega\in\B(\s H_{21})_*]\subset\M(A_1\tens S_{12}).\qedhere
\]
\end{nbs}

\begin{prop}\label{propind4}
The Banach space $A_2:={\rm Ind}_{\QG_1}^{\QG_2}(A_1)\subset\M(A_1\tens S_{12})$ is a C*-algebra. Moreover, we have:
\begin{enumerate}
	\item $[A_2(1_{A_1}\tens S_{12})]=A_1\tens S_{12}=[(1_{A_1}\tens S_{12})A_2]$; in particular, $A_2\subset\M(A_1\tens S_{12})$ defines a faithful non-degenerate *-homomorphism and $\M(A_2)\subset\M(A_1\tens S_{12})$;
	\item let $\delta_{A_2}:=\restr{(\id_{A_1}\tens\delta_{12}^2)}{A_2}$, we have $\delta_{A_2}(A_2)\subset\M(A_2\tens S_{22})$ and $\delta_{A_2}$ is a continuous action of $\QG_2$ on $A_2$;
	\item the correspondence ${\rm Ind}_{\QG_1}^{\QG_2}:{\sf Alg}_{\QG_1}\rightarrow {\sf Alg}_{\QG_2}$ is functorial.\qedhere
\end{enumerate}
\end{prop}

By exchanging the roles of the quantum groups $\QG_1$ and $\QG_2$, we obtain {\it mutatis mutandis} a functor ${\rm Ind}_{\QG_2}^{\QG_1}:{\sf Alg}_{\QG_2}\rightarrow {\sf Alg}_{\QG_1}$.

\begin{prop}\label{propind1}
Let $j,k=1,2$ with $j\neq k$. Let $(A_j,\delta_{A_j})$ be a $\QG_j$-C*-algebra. Let 
\[
A_k:={\rm Ind}_{\QG_j}^{\QG_k}(A_j)\subset\M(A_j\tens S_{jk}) \quad \text{and} \quad C:={\rm Ind}_{\QG_k}^{\QG_j}(A_k)\subset\M(A_k\tens S_{kj})
\]
endowed with the continuous actions $\delta_{A_k}:=\restr{(\id_{A_j}\tens\delta_{jk}^k)}{A_k}$ and $\delta_C:=\restr{(\id_{A_k}\tens\delta_{kj}^j)}{C}$ respectively. Then, we have:
\begin{enumerate}
	\item $C\subset\M(A_k\tens S_{kj})\subset\M(A_j\tens S_{jk}\tens S_{kj})$ and $C=\delta_{A_j}^{(k)}(A_j)$;
	\item $\pi_j:A_j\rightarrow C\,;\,a\mapsto \delta_{A_j}^{(k)}(a):=(\id_{A_j}\tens\delta_{jj}^k)\delta_{A_j}(a)$
	is a $\QG_j$-equivariant *-isomorphism;\index[symbol]{pi@$\pi_j$}
	\item $\delta_{A_j}^k:A_j\rightarrow\M(A_k\tens S_{kj})\,;\,a\mapsto\delta_{A_j}^{(k)}(a):=(\id_{A_j}\tens\delta_{jj}^k)\delta_{A_j}(a)$
	is a faithful non-degenerate *-homomorphism.\qedhere
\end{enumerate}
\end{prop}

The above result shows that the functors $\ind$ and $\iind$ are inverse of each other.

\begin{nbs}\label{not11}
Let $(B_1,\delta_{B_1})$ be a $\QG_1$-C*-algebra. Let $(B_2,\delta_{B_2})$ be the induced $\QG_2$-C*-algebra, that is to say $B_2={\rm Ind}_{\QG_1}^{\QG_2}(B_1)$ and $\delta_{B_2}=\restr{(\id_{B_1}\tens\delta_{12}^2)}{B_2}$. In virtue of \ref{propind1}, we have four *-homomorphisms:
\[
\delta_{B_j}^k:B_j\rightarrow\M(B_k\tens S_{kj}),\quad j,k=1,2.
\]
Let us give a precise description of them. We denote $\delta_{B_1}^1:=\delta_{B_1}$ and $\delta_{B_2}^2:=\delta_{B_2}$. The *-homomorphism $\delta_{B_1}^2:B_1\rightarrow\M(B_2\tens S_{21})$ is given by
\[
b\in B_1\mapsto \delta_{B_1}^2(b):=\delta_{B_1}^{(2)}(b)\in\M(B_2\tens S_{21}) \;\;\text{{\rm(}with }\delta_{B_1}^{(2)}(b):=(\id_{B_1}\tens\delta_{11}^2)\delta_{B_1}^1(b),\text{ for } b\in B_1\text{{\rm)}}
\]
whereas the *-homomorphism $\delta_{B_2}^1:B_2\rightarrow\M(B_1\tens S_{12})$ is defined by the relation
\[
(\pi_1\tens\id_{S_{12}})\delta_{B_2}^1(b)=\delta_{B_2}^{(1)}(b) \quad \text{for } b\in B_2,
\]
where
$
\delta_{B_2}^{(1)}:=(\id_{B_2}\tens \delta_{22}^1)\delta_{B_2}^2
$ 
and
$
\pi_1:B_1\rightarrow{\rm Ind}_{\QG_2}^{\QG_1}(B_2)\,;\,b \mapsto \delta_{B_1}^{(2)}(b)
$
(cf.\ \ref{propind1} 2).
\end{nbs}

\begin{prop}\label{prop4}
Let $(A,\beta_A,\delta_A)$ be a ${\cal G}$-C*-algebra. Let $j,k=1,\,2$ with $j\neq k$. With the notations of \ref{actprop}, let
\[
(\widetilde{A}_j,\delta_{\widetilde{A}_j}):={\rm Ind}_{\QG_k}^{\QG_j}(A_k,\delta_{A_k}^k).
\]
If $x\in A_j$, then we have $\delta_{A_j}^k(x)\in\widetilde{A}_j\subset\M(A_k\tens S_{kj})$ and the map 
$
\widetilde{\pi}_j:A_j\rightarrow\widetilde{A}_j\,;\,x\mapsto\delta_{A_j}^k(x)
$
is a $\QG_j$-equivariant *-isomorphism.\index[symbol]{pj@$\widetilde{\pi}_j$}
\end{prop}

\begin{prop}\label{prop37}
Let $(B_1,\delta_{B_1})$ be a $\QG_1$-C*-algebra. Let $B_2={\rm Ind}_{\QG_1}^{\QG_2}(B_1)$ be the induced $\QG_2$-C*-algebra. Let $B:=B_1\oplus B_2$. For $j,\,k=1,2$ with $j\neq k$, let $\pi_j^k:\M(B_k\tens S_{kj})\rightarrow\M(B\tens S)$ be the strictly continuous *-homomorphism extending the canonical injection $B_k\tens S_{kj}\rightarrow B\tens S$ and $\delta_{B_j}^k:B_j\rightarrow\M(B_k\tens S_{kj})$ the *-homomorphisms defined in \ref{not11}.
Let $\beta_B:\GC^2\rightarrow\M(B)$ and $\delta_B:B\rightarrow\M(B\tens S)$ be the *-homomorphisms defined by:
\[
\beta_B(\lambda,\mu):=\begin{pmatrix}\lambda & 0\\ 0 & \mu\end{pmatrix}\!,\quad (\lambda,\mu)\in\GC^2;\quad
\delta_B(b):=\sum_{k,j=1,2}\pi_j^k\circ\delta_{B_j}^k(b_j),\quad b=(b_1,b_2)\in B.\]
Therefore, we have:
\begin{enumerate}
	\item $(\beta_B,\delta_B)$ is a strongly continuous action of ${\cal G}$ on $B$;
	\item the correspondence ${\sf Alg}_{\QG_1}\rightarrow{\sf Alg}_{\cal G}\,;\,(B_1,\delta_{B_1})\mapsto(B,\beta_B,\delta_B)$ is functorial;
	\item the functors ${\sf Alg}_{\QG_1}\rightarrow{\sf Alg}_{\cal G}$ and ${\sf Alg}_{\cal G}\rightarrow{\sf Alg}_{\QG_1}$ are inverse of each other.\qedhere
\end{enumerate}
\end{prop}

	\subsection{Actions of (semi-)regular measured quantum groupoids}\label{ActRegMQG}	

In this section, we fix a measured quantum groupoid ${\cal G}=(N,M,\alpha,\beta,\Delta,T,T',\epsilon)$ on a finite-dimensional basis $N:=\bigoplus_{1\leqslant l\leqslant k}\,{\rm M}_{n_l}(\GC)$ endowed with the non-normalized Markov trace $\epsilon=\bigoplus_{1\leqslant l\leqslant k}n_l\cdot{\rm Tr}_l$ and we use all the notations introduced in \S\ref{MQGfinitebasis}.

\medskip

We begin this section by a characterization of the regularity (resp.\ semi-regularity) of $\cal G$ in terms of the action of $\cal G$ on itself (cf.\ \ref{ex2}), which generalizes 2.6 \cite{BSV} to the setting of measured quantum groupoids on a finite basis. 

\begin{prop}
Let $S\rtimes{\cal G}$ be the crossed product of $S$ by the strongly continuous action $(\beta,\delta)$ of $\cal G$. Then, we have a canonical *-isomorphism $S\rtimes{\cal G}\simeq[S\widehat{S}]$. In particular, $\cal G$ is regular {\rm(}resp.\ semi-regular{\rm)} if, and only if, we have $\K_{\widehat{\beta}}=S\rtimes{\cal G}$ {\rm(}resp.\ $\K_{\widehat{\beta}}\subset S\rtimes{\cal G}${\rm)}.
\end{prop}

\begin{proof}
Let us identify $\Lin({\cal E}_{S,L})=\{T\in\Lin(S\tens\s H)\,;\,Tq_{\beta\alpha}=T=q_{\beta\alpha}T\}$. Let us denote by $j_{S\rtimes{\cal G}}:S\rtimes{\cal G}\rightarrow\B(\s H\tens\s H)$, the faithful *-representation defined by $j_{S\rtimes{\cal G}}(u)=(L\tens\id_{\K})(u)$ for all $u\in S\rtimes{\cal G}\subset\Lin({\cal E}_{S,L})$. Let $\pi:S\rightarrow\M(S\rtimes{\cal G})$ and $\widehat{\theta}:\widehat{S}\rightarrow\M(S\rtimes{\cal G})$ be the canonical morphisms (cf.\ \ref{notCrossedProduct}). We claim that there exists a unique *-isomorphism $\phi:S\rtimes{\cal G}\rightarrow [S\widehat{S}]$ such that
\[
\phi(\pi(s)\widehat{\theta}(x))=L(s)\rho(x),\quad \text{for all } s\in S \text{ and } x\in\widehat{S}.
\]
Let $s\in S$ and $x\in\widehat{S}$. Since $W\in M\tens\widehat{M}$ and $\rho(\widehat{S})\subset\widehat{M}'$, we have 
\begin{align*}
j_{S\rtimes{\cal G}}(\pi(s)\widehat{\theta}(x))&=(L\tens L)\delta(s)(1\tens\rho(x))\\
&=W^*(1\tens L(s))W(1\tens\rho(x))\\
&=W^*(1\tens L(s)\rho(x))W.
\end{align*}
Let $C:={\rm im}(j_{S\rtimes{\cal G}})=\{W^*(1\tens z)W\,;\,z\in[S\widehat{S}]\}$. The representation $j_{S\rtimes{\cal G}}$ induces a *-isomorphism $\psi:S\rtimes{\cal G}\rightarrow C$. Since $WW^*=q_{\alpha\widehat{\beta}}$ and $[1\tens z,\,q_{\alpha\widehat{\beta}}]=0$ for all $z\in[S\widehat{S}]$, the map 
\[
\chi:[S\widehat{S}]\rightarrow C \; ; \; z\rightarrow W^*(1\tens z)W
\] 
is a *-homomorphism satisfying $W\chi(z)W^*=q_{\alpha\widehat{\beta}}(1\tens z)$ for all $z\in[S\widehat{S}]$. Let $\omega\in\B(\s H)_*$ such that $\omega\circ\alpha=\epsilon$. We have 
\[
(\omega\tens\id)(W\chi(z)W^*)=z,\quad \text{for all } z\in[S\widehat{S}].
\] 
Hence, $\chi$ is a *-isomorphism. Hence, $\phi:=\chi^{-1}\circ\psi:S\rtimes{\cal G}\rightarrow[S\widehat{S}]\,;\,\pi(s)\widehat{\theta}(x)\mapsto L(s)\widehat{\theta}(x)$ is a *-isomorphism. The second statement of the proposition follows from \ref{propReg4}. 
\end{proof}

Proposition \ref{prop38} and Theorem \ref{theo6} are the generalizations of 5.7 and 5.8 of \cite{BSV} to measured quantum groupoids on a finite basis.

\begin{nbs}\label{not15}
Let $(\beta_A,\delta_A)$ be an action of $\cal G$ on a C*-algebra $A$. With the notations of \ref{not13} and \ref{not14}, let $e_I:=\alpha(\varepsilon_I)$ and $q_I:=\beta_A(\varepsilon_I)$ for all $I\in\s I$.\index[symbol]{qb@$q_I$}
\end{nbs}

\begin{lem}\label{lem23}
Let $(\beta_A,\delta_A)$ be an action of $\cal G$ on a C*-algebra $A$. With the notations of \ref{not15}, we have: 
\begin{enumerate}
\item $\displaystyle q_{\beta_A\alpha}=\sum_{I\in\s I}n_I^{-1} q_I\tens e_{\overline{I}}$;
\item $(q_I\tens 1_S)\delta_A(a)=(1_A\tens e_I)\delta_A(a)$, for all $a\in A$ and $I\in\s I$;$\phantom{\displaystyle \sum_{I\in\s I}}$
\item $\delta_A(a)(q_I\tens 1_S)=\delta_A(a)(1_A\tens e_I)$, for all $a\in A$ and $I\in\s I$.\qedhere
\end{enumerate}
\end{lem}

\begin{proof}
Statement 1 is just restatement of \ref{ProjectionBis}. By a straightforward computation, we verify that $(q_I\tens 1_S)q_{\beta_A\alpha}=(1_A\tens e_I)q_{\beta_A\alpha}$ for all $I\in\s I$. Statement 2 then follows from the fact that $\delta_A(1_A)=q_{\beta_A\alpha}$. The last statement follows from the second one by taking the adjoint.
\end{proof}

\begin{prop}\label{prop38} Let $(\beta_A,\delta_A)$ be an action of $\cal G$ on a C$^*$-algebra $A$. If $\cal G$ is semi-regular, the Banach space $[(\id_A\tens\omega)\delta_A(a)\,;\, a\in A,\, \omega\in\B(\s H)_*]\subset{\cal M}(A)$ is a C$^*$-algebra.
\end{prop}

\begin{proof}
Let us denote $T:=[(\id_A\tens\omega)\delta_A(a)\,;\, a\in A,\, \omega\in\B(\s H)_*]$. For all $a\in A$ and $\omega\in\B(\s H)_*$, we have $(\id_A\tens\omega)(\delta_A(a))^*=(\id_A\tens\overline{\omega})\delta_A(a^*)$. Hence, $T^*\subset T$. Let us prove that $TT\subset T$. Let $\omega,\,\phi\in\B(\s H)_*$, $a,\, b\in A$ and $x,\, y\in\K$. We have
\[
(\id_A\tens y\omega )\delta_A(a)(\id_A\tens \phi x)\delta_A(b)=(\id_A\tens\phi\tens\omega)(\delta_A(a)_{13}(1_A\tens x\tens y)\delta_A(b)_{12}).
\] 
By \ref{lem23} 1,2, we have
\begin{align*}
\delta_A(a)_{13}(1_A\tens x\tens y)\delta_A(b)_{12} &=
\sum_{I\in\s I}n_I^{-1} \delta_A(a)_{13}(1_A\tens x \tens e_{\overline{I}}y)(q_I\tens 1\tens 1)\delta_A(b)_{12}\\
&=\sum_{I\in\s I}n_I^{-1} \delta_A(a)_{13}(1_A\tens xe_I \tens e_{\overline{I}}y)\delta_A(b)_{12}\\
&=\delta_A(a)_{13}((x\tens 1)q_{\alpha}(1\tens y))_{23}\delta_A(b)_{12}.
\end{align*}
It follows from \ref{propReg1} that $(\id_A\tens y\omega )\delta_A(a)(\id_A\tens \phi x)\delta_A(b)$ is the norm limit of finite sums of elements of the form 
\begin{align*}
c&:=(\id_A\tens\phi\tens\omega)(\delta_A(a)_{13}((x'\tens 1)W(1\tens y'))_{23}\delta_A(b)_{12})\\
&=(\id_A\tens\phi x'\tens y'\omega)(\delta_A(a)_{13}W_{23}\delta_A(b)_{12}),
\end{align*}
where $x',y'\in\K$. By combining the following formulas; 
\[
\delta_A^2(a)=W_{23}^*\delta_A(a)_{13}W_{23};\;\; WW^*=q_{\alpha\widehat{\beta}}; \;\, [\delta_A(a)_{13},\,q_{\alpha\widehat{\beta},23}]=0 \text{ (since } \widehat{\beta}(N^{\rm o})\subset M');
\] 
we obtain $\delta_A(a)_{13}W_{23}=W_{23}\delta_{A}^{2}(a)$. Hence, we have $c=(\id_A\tens\psi)(\delta_A^2(a)_{13}\delta_A(b)_{12})$, where $\psi:=(\phi x'\tens y\omega)W\in\B(\s H\tens\s H)_*$. Therefore, $c$ is the norm limit of finite sums of elements of the form
$
(\id_A\tens\phi'\tens\omega')(\delta_A^2(a)_{13}\delta_A(b)_{12})=
(\id\tens\phi')\delta_A((\id_A\tens\omega')\delta_A(a)b),
$
where $\phi',\omega'\in\B(\s H)_*$. Hence, $(\id_A\tens y\omega )\delta_A(a)(\id_A\tens \phi x)\delta_A(b)\in T$.
\end{proof}

\begin{defin} Let $(\beta_A,\delta_A)$ be an action of $\cal G$ on a C$^*$-algebra $A$. We say that $(\beta_A,\delta_A)$ is weakly continuous if we have 
$
A=[(\id_A\tens\omega)\delta_A(a)\,;\, a\in A,\, \omega\in\B(\s H)_*].
$
\end{defin}

Note that any strongly continuous action $(\beta_A,\delta_A)$ of $\cal G$ on a C$^*$-algebra $A$ is necessarily continuous in the weak sense. Indeed, if $(\beta_A,\delta_A)$ is strongly continuous we have the inclusion $\delta_A(A)(1_A\tens S)\subset A\tens S$. Hence, $[(\id_A\tens\omega)\delta_A(a)\,;\, a\in A,\, \omega\in\B(\s H)_*]\subset A$ since $S\subset\B(\s H)$ is non-degenerate. Conversely, let $\omega\in\B(\s H)_*$ such that $\omega\circ\alpha=\epsilon$. We have $(\id_A\tens\omega)(q_{\beta_A\alpha})=1_A$. By writing $\omega=y\omega'$ for some $\omega'\in\B(\s H)_*$ and $y\in S$, we obtain $(\id_A\tens\omega')(q_{\beta_A\alpha}(a\tens y))=a$ for all $a\in A$.

\begin{thm}\label{theo6} If the groupoid ${\cal G}$ is regular, then any weakly continuous action of $\cal G$ is strongly continuous.
\end{thm}

\begin{proof}
Let us fix an action $(\beta_A,\delta_A)$ of $\cal G$ on a C$^*$-algebra $A$. Let us assume that $(\beta_A,\delta_A)$ is weakly continuous. Since $W\in\M(S\tens\K)$ is a partial isometry such that $W^*W=q_{\beta\alpha}$, we have $(S\tens\K)W=(S\tens\K)q_{\beta\alpha}$. We recall that $\delta_A(a)_{13}W_{23}=W_{23}\delta_{A}^{2}(a)$ for all $a\in A$ (cf.~proof of \ref{prop38}). By \ref{prop39} 1, we have $q_{\beta\alpha,23}\delta_A^2(a)=\delta_A^2(a)$ for all $a\in A$.
By combining the assertions of the above discussion with \ref{lem23} 3 and \ref{propReg2} 1, we have 
\begin{align*}
(A & \tens S) q_{\beta_A\alpha} = [((\id_A\tens\omega)\delta_A(a)\tens y)q_{\beta_A\alpha} \, ; \, a\in A,\, y\in S,\, \omega\in\B(\s H)_*]\\
&=[((\id_A\tens\id_S\tens x\omega)(\delta_A(a)_{13}(1_A\tens y\tens 1_S)q_{\beta_A\alpha,12}) \, ; \, a\in A,\, x\in\K,\, y\in S,\, \omega\in\B(\s H)_*]\\
&= [(\id_A\tens\id_S\tens\omega)(\delta_A(a)_{13}((y\tens 1_{\K})q_{\alpha}(1_S\tens x))_{23}) \, ; \, a\in A,\, x\in\K,\, y\in S,\, \omega\in\B(\s H)_*]\\
&= [(\id_A\tens\id_S\tens\omega)(\delta_A(a)_{13}((y\tens 1_{\K})W(1_S\tens x))_{23})\, ; \, a\in A,\, x\in\K,\, y\in S,\, \omega\in\B(\s H)_*]\\
&= [(\id_A\tens\id_S\tens\omega)((1_A\tens y\tens 1)\delta_A(a)_{13}W_{23})\, ; \, a\in A,\, x\in\K,\, y\in S,\, \omega\in\B(\s H)_*]\\
&= [(\id_A\tens\id_S\tens\omega)((1_A\tens (y\tens x)W)\delta_A^2(a))\, ; \, a\in A,\, x\in\K,\, y\in S,\, \omega\in\B(\s H)_*]\\
&=[(\id_A\tens\id_S\tens\omega)((1_A\tens y\tens x)\delta_A^2(a))\, ; \, a\in A,\, x\in\K,\, y\in S,\, \omega\in\B(\s H)_*]\\
&=[(1_A\tens y)\delta_A((\id_A\tens\omega x)\delta_A(a))\, ; \, a\in A,\, x\in\K,\, y\in S,\, \omega\in\B(\s H)_*]\\
&=[(1_A\tens S)\delta_A(A)].\qedhere
\end{align*}
\end{proof}

As a first application, we have the result below.

\begin{prop}
If the groupoid $\cal G$ is regular, then the trivial action of $\cal G$ on $N^{\rm o}$ (cf.\ \ref{ex2}) is strongly continuous and there exists a unique $\widehat{\cal G}$-equivariant *-isomorphism $\tau:N^{\rm o}\rtimes{\cal G}\rightarrow\widehat{S}$ such that $\tau(\pi(n^{\rm o})\widehat{\theta}(x))=\beta(n^{\rm o})x$ for all $n\in N$ and $x\in\widehat{S}$.
\end{prop}

\begin{proof}
In this proof, we use the notations of \ref{not15} with $A:=N^{\rm o}$. In this case, we have $q_I=\varepsilon_I^{\rm o}$ for all $I\in\s I$. According to Theorem \ref{theo6}, it suffices to show that the trivial action is weakly continuous. Since the C*-algebra $N$ is finite-dimensional, it amounts to proving that $\varepsilon_I^{\rm o}\in\langle(\id_{N^{\rm o}}\tens\omega)\delta_{N^{\rm o}}(n^{\rm o})\,;\,n\in N,\, \omega\in\B(\s H)_*\rangle$ for all $I\in\s I$. Let $I\in\s I$. For all $n'\in N$, there exists $\omega\in\B(\s H)_*$ such that $\omega(\alpha(n))=\epsilon(n'n)$ for all $n\in N$ (extension of normal linear forms). In particular, there exists $\omega\in\B(\s H)_*$ such that $\omega(\alpha(\varepsilon_J))=n_I\delta_J^I$ for all $J\in\s I$. By a straightforward computation, we have $\varepsilon_I^{\rm o}=(\id_{N^{\rm o}}\tens\omega)\delta_{N^{\rm o}}(1_{N^{\rm o}})$ and the weak continuity of the trivial action is then proved since $N^{\rm o}$ is unital. Since $N^{\rm o}$ is unital, we have $\widehat{\theta}(x)=\pi(1_{N^{\rm o}})\widehat{\theta}(x)\in N^{\rm o}\rtimes{\cal G}$ for all $x\in\widehat{S}$. Moreover, we have $\pi(n^{\rm o})\widehat{\theta}(x)=\widehat{\theta}(\beta(n^{\rm o})x)$ for all $n\in N$ and $x\in\widehat{S}$. Hence, the morphism $\widehat{\theta}$ induces a *-isomorphism $\widehat{S}\simeq N^{\rm o}\rtimes{\cal G}$. The equivariance is easily obtained from the definitions.
\end{proof}

\section{Notion of equivariant Hilbert C*-modules}
	\subsection{Actions of measured quantum groupoids on Hilbert C*-modules}\label{SectEqHilb}

In this section, we introduce a notion of ${\cal G}$-equivariant Hilbert C*-module for a measured quantum groupoid ${\cal G}$ on a finite basis in the spirit of \cite{BS1}. We fix a measured quantum groupoid $\cal G$ on a finite-dimensional basis $N=\bigoplus_{1\leqslant l\leqslant k}{\rm M}_{n_l}(\GC)$ endowed with the non-normalized Markov trace $\epsilon=\bigoplus_{1\leqslant l\leqslant k}n_l\cdot{\rm Tr}_l$. We use all the notations introduced in \S\S \ref{MQGfinitebasis}, \ref{WHC*A}. Let us fix a ${\cal G}$-C*-algebra $A$.

\paragraph{The three pictures.} Following \S 2 \cite{BS1}, an action of ${\cal G}$ on a Hilbert $A$-module $\s E$ will be defined by three equivalent data:
\begin{itemize}
\item a pair $(\beta_{\s E},\delta_{\s E})$ consisting of a *-homomorphism $\beta_{\s E}:N^{\rm o}\rightarrow\Lin(\s E)$ and a linear map $\delta_{\s E}:\s E\rightarrow\widetilde{\M}(\s E\tens S)$, cf.\ \ref{hilbmodequ};
\item a pair $(\beta_{\s E},\s V_{\s E})$ consisting of a *-homomorphism $\beta_{\s E}:N^{\rm o}\rightarrow\Lin(\s E)$ and an isometry $\s V\in\Lin(\s E\tens_{\delta_A}(A\tens S),\s E\tens S)$, cf.\ \ref{isometry};
\item an action $(\beta_J,\delta_J)$ of $\cal G$ on $J:=\K(\s E\oplus A)$, cf.\ \ref{compcoact};
\end{itemize}
satisfying some conditions.

\medskip

We have the following unitary equivalences of Hilbert modules: 
\begin{align}
A\tens_{\delta_A}(A\tens S) &\rightarrow q_{\beta_A \alpha}(A\tens S)\,;\, a\tens_{\delta_A}x \mapsto \delta_A(a)x;\label{ehmeq1}\\
(A\tens S)\tens_{\delta_A\tens\, \id_S}(A\tens S\tens S)&\rightarrow q_{\beta_A \alpha,12}(A\tens S\tens S) \, ;\, x\tens_{\delta_A\tens\, \id_S}y \mapsto (\delta_A\tens \id_S)(x)y;\label{ehmeq2}\\
(A\tens S)\tens_{\id_A\tens\,\delta}(A\tens S\tens S) &\rightarrow q_{\beta\alpha,23}(A\tens S\tens S)\,;\,x\tens_{\id_A\tens\,\delta}y\mapsto(\id_A\tens\delta)(x)y.\label{ehmeq3}
\end{align}

In the following, we fix a Hilbert $A$-module $\s E$. We will apply the usual identifications $\M(A\tens S)=\Lin(A\tens S)$, $\K(\s E)\tens S=\K(\s E\tens S)$ and $\M(\K(\s E)\tens S)=\Lin(\s E\tens S)$.

\begin{defin}\label{hilbmodequ}
An action of ${\cal G}$ on the Hilbert $A$-module $\s E$ is a pair $(\beta_{\s{E}},\delta_{\s{E}})$, where $\beta_{\s{E}}:N^{\rm o}\rightarrow\mathcal{L}(\s{E})$ is a non-degenerate *-homomorphism and $\delta_{\s{E}}:\s{E}\rightarrow\widetilde{\mathcal{M}}(\s{E}\tens S)$ is a linear map such that:
\begin{enumerate}
 \item for all $a\in A$ and $\xi,\,\eta\in \s{E}$, we have 
\[
\delta_\s{E}(\xi a)=\delta_\s{E}(\xi)\delta_A(a) \quad \text{and} \quad \langle\delta_\s{E}(\xi),\,\delta_\s{E}(\eta)\rangle=\delta_A(\langle\xi,\,\eta\rangle);
\]
 \item $[\delta_{\s{E}}(\s{E})(A\tens S)]=q_{\beta_{\s E}\alpha}(\s{E}\tens S)$;
 \item for all $\xi\in \s{E}$ and $n\in N$, we have $\delta_{\s{E}}(\beta_{\s{E}}(n^{\rm o})\xi)=(1_{\s E}\tens\beta(n^{\rm o}))\delta_{\s{E}}(\xi)$;
 \item the linear maps $\delta_{\s{E}}\tens \id_S$ and $\id_{\s{E}}\tens\delta$ extend to linear maps from $\mathcal{L}(A\tens S,\s{E}\tens S)$ to $\mathcal{L}(A\tens S\tens S,\s{E}\tens S\tens S)$ and we have 
\[
(\delta_\s{E}\tens \id_S)\delta_\s{E}(\xi)=(\id_\s{E}\tens\delta)\delta_{\s{E}}(\xi)\in\mathcal{L}(A\tens S\tens S,\s{E}\tens S\tens S),\quad \text{for all } \xi\in\s{E}.\qedhere
\]
\end{enumerate}
\end{defin}

\begin{rks}\label{rk2}
\begin{itemize}
	\item If the second formula of the condition 1 holds, then $\delta_{\s E}$ is isometric (cf.\ \cite{BS2}, \ref{rk4} 1).
	\item If the condition 1 holds, then the condition 2 is equivalent to:
	\[
	[\delta_{\s E}(\s E)(1_A\tens S)]=q_{\beta_{\s E}\alpha}(\s E\tens S).
	\]
	Indeed, if $(u_{\lambda})_{\lambda}$ is an approximate unit of $A$ we have
	\[
	\delta_{\s E}(\xi)=\lim_{\lambda}\,\delta_{\s E}(\xi u_{\lambda})=\lim_{\lambda}\,\delta_{\s E}(\xi)\delta_A(u_{\lambda})=\delta_{\s E}(\xi)q_{\beta_A\alpha},\quad \text{for all } \xi\in\s E.
	\]
	By strong continuity of the action $(\beta_A,\delta_A)$, the condition 1 of Definition \ref{hilbmodequ} and the equality $\s E A=\s E$, we then have $[\delta_{\s E}(\s E)(A\tens S)]=[\delta_{\s E}(\s E)(1_A\tens S)]$ and the equivalence follows.
	\item Note that we have $q_{\beta_{\s E}\alpha}\delta_{\s E}(\xi)=\delta_{\s E}(\xi)=\delta_{\s E}(\xi)q_{\beta_A\alpha}$ for all $\xi\in\s E$.
	\item We will prove (cf.\ \ref{rk9}) that if $\delta_{\s E}$ satisfies the conditions 1 and 2 of \ref{hilbmodequ}, then the extensions of $\delta_{\s E}\tens\id_S$ and $\id_{\s E}\tens\delta$ always exist and satisfy the formulas: 
	\begin{align*}
(\id_{\s E}\tens\delta)(T)(\id_A\tens\delta)(x)&=(\id_{\s E}\tens\delta)(Tx);\\
(\delta_{\s E}\tens\id_S)(T)(\delta_A\tens\id_S)(x)&=(\delta_{\s E}\tens\id_S)(Tx);
\end{align*}
for all $x\in A\tens S$ and $T\in\Lin(A\tens S,\s E\tens S)$.\qedhere
\end{itemize}
\end{rks}

\begin{nb}\label{not5}
For $\xi\in \s{E}$, let us denote by $T_{\xi}\in\mathcal{L}(A\tens S,\s{E}\tens_{\delta_A}(A\tens S))$ the operator defined by\index[symbol]{tb@$T_{\xi}$}
\[
T_{\xi}(x):=\xi\tens_{\delta_A}x,\quad\text{for all } x\in A\tens S.
\]
We have $T_{\xi}^*(\eta\tens_{\delta_A}y)=\delta_A(\langle\xi,\,\eta\rangle)y$ for all $\eta\in\s E$ and $y\in A\tens S$. In particular, we have $T_{\xi}^*T_{\eta}=\delta_A(\langle\xi,\,\eta\rangle)$ for all $\xi,\eta\in\s E$.
\end{nb}

\begin{defin}\label{isometry}
Let $\s{V}\in\mathcal{L}(\s{E}\tens_{\delta_A}(A\tens S),\s{E}\tens S)$ be an isometry and $\beta_{\s{E}}:N^{\rm o}\rightarrow\mathcal{L}(\s{E})$ a non-degenerate *-homomorphism such that:
\newcounter{saveenum}
\begin{enumerate}
 \item $\s{V}\s{V}^*=q_{\beta_{\s E}\alpha}$;
 \item $\s{V}(\beta_{\s{E}}(n^{\rm o})\tens_{\delta_A}1)=(1_{\s E}\tens\beta(n^{\rm o}))\s{V}$, for all $n\in N$. 
\setcounter{saveenum}{\value{enumi}}
\end{enumerate}
Then, $\s{V}$ is said to be admissible if we further have:
\begin{enumerate}
\setcounter{enumi}{\value{saveenum}}
 \item $\s{V}T_{\xi}\in\widetilde{\mathcal{M}}(\s{E}\tens S)$, for all $\xi\in \s{E}$;
 \item $(\s{V}\tens_{\mathbb{C}}\id_S)(\s{V}\tens_{\delta_A\tens\, \id_S}1)=\s{V}\tens_{\id_A\tens\,\delta}1\in\mathcal{L}(\s{E}\tens_{\delta_A^2}(A\tens S\tens S),\s{E}\tens S\tens S)$.\qedhere
\end{enumerate}
\end{defin}

The fourth statement in the previous definition makes sense since we have used the canonical identifications thereafter. By combining the associativity of the internal tensor product with the unitary equivalences (\ref{ehmeq2}) and (\ref{ehmeq3}), we obtain the following unitary equivalences of Hilbert $A\tens S$-modules:
\begin{align}
(\s{E}\tens_{\delta_A}(A\tens S))\tens_{\delta_A\tens\,\id_S}(A\tens S\tens S) &\rightarrow \s{E}\tens_{\delta_A^2}(A\tens S\tens S)\label{eq4}\\
(\xi\tens_{\delta_A}x)\tens_{\delta_A\tens\,\id_S}y &\mapsto \xi\tens_{\delta_A^2}(\delta_A\tens \id_S)(x)y;\notag\\[.5em]
(\s{E}\tens_{\delta_A}(A\tens S))\tens_{\id_A\tens\,\delta}(A\tens S\tens S) &\rightarrow \s{E}\tens_{\delta_A^2}(A\tens S\tens S)\label{eq5}\\
(\xi\tens_{\delta_A}x)\tens_{\id_A\tens\,\delta}y &\mapsto \xi\tens_{\delta_A^2}(\id_A\tens\delta)(x)y.\notag
\intertext{We also have the following:}
(\s{E}\tens S)\tens_{\delta_A\tens\,\id_S}(A\tens S\tens S) &\rightarrow (\s{E}\tens_{\delta_A}(A\tens S))\tens S \label{eq6}\\
(\xi\tens s)\tens_{\delta_A\tens\,\id_S}(x\tens t)&\mapsto(\xi\tens_{\delta_A}x)\tens st;\notag\\[.5em]
(\s{E}\tens S)\tens_{\id_A\tens\,\delta}(A\tens S\tens S) & \rightarrow q_{\beta\alpha,23}(\s{E}\tens S\tens S)\subset\s{E}\tens S\tens S \label{eq7}\\
\xi\tens_{\id_A\tens\,\delta}y &\mapsto(\id_{\s{E}}\tens\delta)(\xi)y.\notag
\end{align}
In particular, 
$\s{V}\tens_{\delta_A\tens\,\id_S}1\in\mathcal{L}(\s{E}\tens_{\delta_A^2}(A\tens S\tens S),(\s{E}\tens S)\tens_{\delta_A\tens\,\id_S}(A\tens S\tens S))$ (\ref{eq4}) and
$\s{V}\tens_{\mathbb{C}}\id_S\in\mathcal{L}((\s{E}\tens S)\tens_{\delta_A\tens\,\id_S}(A\tens S\tens S),\s{E}\tens S\tens S)$ (\ref{eq6}).
\vspace{10pt}

The next result provides an equivalence of the definitions \ref{hilbmodequ} and \ref{isometry}. 

\begin{prop}\label{prop27}
 \begin{enumerate}[label=\alph*)]
  \item Let $\delta_{\s{E}}:\s{E}\rightarrow\widetilde{\mathcal{M}}(\s{E}\tens S)$ be a linear map and $\beta_{\s{E}}:N^{\rm o}\rightarrow\mathcal{L}(\s{E})$ a non-degenerate *-homomorphism which satisfy the conditions 1, 2, and 3 of Definition \ref{hilbmodequ}. Then, there exists a unique isometry $\s{V}\in\mathcal{L}(\s{E}\tens_{\delta_A}(A\tens S),\s{E}\tens S)$ such that $\delta_\s{E}(\xi)=\s{V}T_{\xi}$ for all $\xi\in\s{E}$. 
  Moreover, the pair $(\beta_{\s{E}},\s V)$ satisfies the conditions 1, 2, and 3 of Definition \ref{isometry}.
  \item Conversely, let $\s{V}\in\mathcal{L}(\s{E}\tens_{\delta_A}(A\tens S),\s{E}\tens S)$ be an isometry and $\beta_{\s{E}}:N^{\rm o}\rightarrow\mathcal{L}(\s{E})$ a non-degenerate *-homomorphism, which satisfy the conditions 1, 2, and 3 of Definition \ref{isometry}. Let us consider the map $\delta_{\s{E}}:\s{E}\rightarrow\Lin(A\tens S,\s{E}\tens S)$ given by $\delta_{\s{E}}(\xi):=\s{V}T_{\xi}$ for all $\xi\in\s{E}$. Then, the pair $(\beta_{\s{E}},\delta_{\s{E}})$ satisfies the conditions 1, 2 and 3 of Definition \ref{hilbmodequ}.
  \item Let us assume that the above statements hold. Then, the pair $(\beta_{\s{E}},\delta_{\s{E}})$ is an action of $\cal G$ on $\s E$ if, and only if, $\s{V}$ is admissible.\qedhere
 \end{enumerate}
\end{prop}

In the proof, we will use the following notation.

\begin{nb}\label{not7}
Let ${\cal E}$ and ${\cal F}$ be Hilbert C*-modules. Let $q\in\Lin({\cal E})$ be a self-adjoint projection and $T\in\Lin(q{\cal E},{\cal F})$. Let $\widetilde{T}:{\cal E}\rightarrow{\cal F}$ be the map defined by $\widetilde{T}\xi:=Tq\xi$, for all $\xi\in{\cal E}$. Therefore, $\widetilde T\in\Lin({\cal E},{\cal F})$ and $\widetilde{T}^*=qT^*$. By abuse of notation, we will still denote by $T$ the adjointable operator $\widetilde T$.
\end{nb}

\begin{proof}[Proof of Proposition \ref{prop27}]
 $a)$ By definition of the internal tensor product and \ref{hilbmodequ} 1, there exists a unique isometric $(A\tens S)$-linear map $\s{V}:\s{E}\tens_{\delta_A}(A\tens S)\rightarrow \s{E}\tens S$ such that 
\[
 \s{V}(\xi\tens_{\delta_A}x)=\delta_{\s{E}}(\xi)x,\quad \text{for all } \xi\in \s{E} \text{ and } x\in A\tens S.
\]
In other words, we have $\s{V}T_{\xi}=\delta_{\s{E}}(\xi)$ for all $\xi\in \s{E}$. Now, it follows from  \ref{hilbmodequ} 2 that the ranges of $\s{V}$ and $q_{\beta_{\s E}\alpha}$ are equal. $\vphantom{v^{-1}}$Then, denote by $v$ the range restriction of $\s{V}$. Hence, the map $v^{-1}q_{\beta_{\s E}\alpha}$ is an adjoint for $\s{V}$. Indeed, for all $x\in \s{E}\tens S$ and $y\in \s{E}\tens_{\delta_A}(A\tens S)$ we have$\vphantom{v^{-1}}$
\begin{align*}
\langle v^{-1}q_{\beta_{\s E}\alpha}x,y\rangle &=\langle \s{V}v^{-1}(q_{\beta_{\s E}\alpha}x),\s{V}y\rangle & (\s V \text{ is isometric})\\
&=\langle q_{\beta_{\s E}\alpha}x,\s{V}y\rangle\\
&=\langle x,\s{V}y\rangle-\langle(1-q_{\beta_{\s E}\alpha})(x),\s{V}y\rangle\\
&=\langle x,\s{V}y\rangle. & (\s{V}y\in{\rm im}(q_{\beta_{\s E}\alpha}))
\end{align*}
Hence, $\s{V}\in\mathcal{L}(\s{E}\tens_{\delta_A}(A\tens S),\s{E}\tens S)$ and then $\s{V}^*\s{V}=1$ and $\s{V}\s{V}^*=\s{V}v^{-1}q_{\beta_{\s E}\alpha}=q_{\beta_{\s E}\alpha}$.

The conditions 1 and 3 of Definition \ref{isometry} are then fulfilled. Now, we have
\begin{align*}
\s{V}(\beta_{\s{E}}(n^{\rm o})\tens_{\delta_A}1)(\xi\tens_{\delta_A}x)&=\delta_{\s{E}}(\beta_{\s{E}}(n^{\rm o})\xi)x\\
&=(1_{\s E}\tens\beta(n^{\rm o}))\delta_{\s{E}}(\xi)x\\
&=(1_{\s E}\tens\beta(n^{\rm o}))\s{V}(\xi\tens_{\delta_A}x),
\end{align*}
for all $\xi\in\s{E}$, $x\in A\tens S$ and $n\in N$. Hence, the condition 2 of Definition \ref{isometry} holds. 

\medskip
 
$b)$ is straightforward.     

\medskip

$c)$ Let $T\in\mathcal{L}(A\tens S,\s{E}\tens S)$. By using \ref{not7} and the identifications (\ref{ehmeq3}, \ref{eq7}), we have $T\tens_{\id_A\tens\,\delta}1\in\mathcal{L}(A\tens S\tens S,\s{E}\tens S\tens S)$. Now, we can define the extension 
\[
\id_{\s{E}}\tens\delta:\mathcal{L}(A\tens S,\s{E}\tens S)\rightarrow\mathcal{L}(A\tens S\tens S,\s{E}\tens S\tens S)
\]
by setting 
\[
(\id_{\s{E}}\tens\delta)(T):=T\tens_{\id_A\tens\,\delta}1,\quad \text{for all } T\in\mathcal{L}(A\tens S,\s{E}\tens S).
\]
We also have $T\tens_{\delta_A\tens\,\id_S}1\in\mathcal{L}(A\tens S\tens S,(\s{E}\tens_{\delta_A}(A\tens S))\tens S)$ by using again \ref{not7} and the identifications (\ref{ehmeq2}, \ref{eq6}). Let us define the extension 
\[
\delta_{\s{E}}\tens \id_S:\mathcal{L}(A\tens S,\s{E}\tens S)\rightarrow\mathcal{L}(A\tens S\tens S,\s{E}\tens S\tens S)
\]
by setting 
\[
(\delta_{\s{E}}\tens \id_S)(T):=(\s{V}\tens_{\mathbb{C}}1_S)(T\tens_{\delta_A\tens\,\id_S}1),\quad \text{for all } T\in\mathcal{L}(A\tens S,\s{E}\tens S).
\]
Therefore, for all $\xi\in\s{E}$ we have:
\begin{align*}
(\delta_{\s{E}}\tens \id_S)\delta_{\s{E}}(\xi)&=(\s{V}\tens_{\mathbb{C}}1_S)(\s{V}\tens_{\delta_A\tens\,\id_S}1)(T_{\xi}\tens_{\delta_A\tens\,\id_S}1) \in \mathcal{L}(A\tens S\tens S,\s{E}\tens S\tens S);\\[0.3cm]
(\id_{\s{E}}\tens\delta)\delta_{\s{E}}(\xi)&=(\s{V}\tens_{\id_A\tens\,\delta}1)(T_{\xi}\tens_{\id_A\tens\,\delta}1)\in\mathcal{L}(A\tens S\tens S,\s{E}\tens S\tens S);
\end{align*}
where:
\begin{align*}
T_{\xi}\tens_{\delta_A\tens\,\id_S}1 &\in\mathcal{L}(A\tens S\tens S,\s{E}\tens_{\delta_A^2}(A\tens S\tens S));\\[0.3cm]
T_{\xi}\tens_{\id_A\tens\,\delta}1 &\in\mathcal{L}(A\tens S\tens S,\s{E}\tens_{\delta_A^2}(A\tens S\tens S));
\end{align*}
by using the identifications (\ref{ehmeq2}, \ref{eq4}) and (\ref{ehmeq3}, \ref{eq5}) respectively and \ref{not7}. In particular, if $\s{V}$ is admissible, then the condition 4 of Definition \ref{hilbmodequ} holds.

\medskip

Conversely, let us assume that the above condition is satisfied. In order to show that $\s V$ is admissible, we only have to prove that the restrictions of the operators $T_{\xi}\tens_{\delta_A\tens\,\id_S}1$ and $T_{\xi}\tens_{\id_A\tens\,\delta}1$ to the Hilbert submodule $q_{\beta_A\alpha,12}q_{\beta\alpha,23}(A\tens S\tens S)$ are surjective.\hfill\break
Let $a\in A$, $x\in A\tens S$ and $y\in A\tens S\tens S$. Let us set $z=(\delta_A\tens \id_S)(\delta_A(a)x)y$. It is clear that $z\in q_{\beta_A\alpha,12}q_{\beta\alpha,23}(A\tens S\tens S)$. By a straightforward computation, we have 
\begin{center}
$
(T_{\xi}\tens_{\delta_A\tens\,\id_S}1)(z)=\xi a\tens_{\delta_A^2}(\delta_A\tens \id_S)(x)y.
$ 
\end{center}
Hence, the restriction of $T_{\xi}\tens_{\delta_A\tens\,\id_S}1$ to $q_{\beta_A\alpha,12}q_{\beta\alpha,23}(A\tens S\tens S)$ is surjective in virtue of (\ref{eq4}) and the fact that $\s{E}A=\s{E}$. The same statement is also true for $T_{\xi}\tens_{\id_A\tens\,\delta}1$. 
\end{proof}

\begin{rks}\label{rk9} In the proof of Proposition \ref{prop27}, we have proved the statements below.
\begin{itemize}
\item  By applying \ref{not7} and the identifications (\ref{ehmeq3}, \ref{eq7}), we have obtained a linear map $\id_{\s E}\tens\delta:\Lin(A\tens S,\s E\tens S)\rightarrow \Lin(A\tens S\tens S,\s E\tens S\tens S)$ given by
\[
(\id_{\s E}\tens\delta)(T):=T\tens_{\id_A\tens\,\delta} 1,\quad \text{for all } T\in\Lin(A\tens S,\s E\tens S);
\]
\item If $\delta_{\s E}$ satisfies the conditions 1 and 2 of Definition \ref{hilbmodequ}, let $\s V$ be the isometry associated with $\delta_{\s E}$ (cf.\ \ref{prop27} a)). By applying \ref{not7} and the identifications (\ref{ehmeq2}, \ref{eq6}), the linear map $\delta_{\s E}\tens\id_S:\Lin(A\tens S,\s E\tens S)\rightarrow\Lin(A\tens S\tens S,\s E\tens S\tens S)$ is defined by
\[
(\delta_{\s E}\tens\id_S)(T):=(\s V \tens_{\GC} 1_S)(T\tens_{\delta_A\tens\,\id_S} 1),\quad \text{for all }  T\in\Lin(A\tens S,\s E\tens S).
\]
\end{itemize}
Note that the extensions $\id_{\s E}\tens\delta$ and $\delta_{\s E}\tens\id_S$ satisfy the following formulas:
\begin{equation}\label{eq26}
(\id_{\s E}\tens\delta)(T)(\id_A\tens\delta)(x)=(\id_{\s E}\tens\delta)(Tx); \,
(\delta_{\s E}\tens\id_S)(T)(\delta_A\tens\id_S)(x)=(\delta_{\s E}\tens\id_S)(Tx);
\end{equation}
for all $x\in A\tens S$ and $T\in\Lin(A\tens S,\s E\tens S)$.
\end{rks}

Let us denote by $J:=\K(\s E\oplus A)$ the linking C*-algebra associated with the Hilbert $A$-module $\s E$. In the following, we apply the usual identifications $\M(J)=\Lin(\s E\oplus A)$ and $\M(J\tens S)=\Lin((\s E\tens S)\oplus(A\tens S))$.

\begin{defin}\label{compcoact}
An action $(\beta_J,\delta_J)$ of ${\cal G}$ on $J$ is said to be compatible with the action $(\beta_A,\delta_A)$ if: 
\begin{enumerate}
 \item $\delta_J:J\rightarrow\M(J\tens S)$ is compatible with $\delta_A$, {\it i.e.} 
$ 
 \iota_{A\tens S}\circ\delta_A=\delta_J\circ\iota_A;
$
 \item $\beta_J:N^{\rm o}\rightarrow\M(J)$ is compatible with $\beta_A$, {\it i.e.} 
$ 
 \iota_A(\beta_A(n^{\rm o})a)=\beta_J(n^{\rm o})\iota_A(a)
$,
 for all $n\in N$ and $a\in A$.\qedhere
\end{enumerate}
\end{defin}

\begin{prop}\label{propfib}
Let $(\beta_J,\delta_J)$ be a compatible action of ${\cal G}$ on $J$. There exists a unique non-degenerate *-homomorphism $\beta_{\s E}:N^{\rm o}\rightarrow\Lin(\s{E})$ such that 
\[
\beta_J(n^{\rm o})=\begin{pmatrix}\beta_{\s E}(n^{\rm o})&0\\0&\beta_A(n^{\rm o})\end{pmatrix}\!,\quad\text{for all } n\in N.
\]
Moreover, we have
$
q_{\beta_J\alpha}=\begin{pmatrix}q_{\beta_{\s E}\alpha}&0\\0& q_{\beta_A\alpha}\end{pmatrix}\!.
$
\end{prop}

\begin{proof}
Note that since $\iota_A$, $\beta_A$ and $\beta_J$ are *-homomorphisms, the condition 2 of Definition \ref{compcoact} is equivalent to: 
\[
\iota_A(a\beta_A(n^{\rm o}))=\iota_A(a)\beta_J(n^{\rm o}),\quad\text{for all } a\in A,\,n\in N.
\]
Therefore, there exists a map $\beta_{\s E}:N^{\rm o}\rightarrow\Lin(\s E)$ necessarily unique such that
\[
\beta_J(n^{\rm o})=\begin{pmatrix}\beta_{\s E}(n^{\rm o})&0\\0&\beta_A(n^{\rm o})\end{pmatrix},
\]
for all $n\in N$. Then, it is clear that $\beta_{\s E}$ is a non-degenerate *-homomorphism and the last statement is then an immediate consequence.
\end{proof}

\begin{rks}\label{rk13}
Note that if $\beta_A$ is injective, then so is $\beta_J$. For all $n\in N$, $\xi\in\s E$ and $k\in\K(\s E)$, we have $\iota_{\K(\s E)}(\beta_{\s E}(n^{\rm o})k)=\beta_J(n^{\rm o})\iota_{\K(\s E)}(k)$ and $\iota_{\s E}(\beta_{\s E}(n^{\rm o})\xi)=\beta_J(n^{\rm o})\iota_{\s E}(\xi)$. In particular, we have $\beta_{\s E}(n^{\rm o})\theta_{\xi,\eta}=\theta_{\beta_{\s E}(n^{\rm o})\xi,\eta}$ for all $n\in N$ and $\xi,\eta\in\s E$ (cf.\ \ref{prop32} 2).
\end{rks}

\begin{prop}\label{prop1}
a) Let us assume that the C*-algebra $J$ is endowed with a compatible action $(\beta_J,\delta_J)$ of ${\cal G}$ such that $\delta_J(J)\subset\widetilde{\M}(J\tens S)$. Then, we have the following statements:
\begin{itemize}
 \item there exists a unique linear map $\delta_{\s E}:\s E\rightarrow\widetilde{\M}(\s E\tens S)$ such that 
$\iota_{\s E\tens S}\circ\delta_{\s E}=\delta_J\circ\iota_{\s E}$;
furthermore, the pair $(\beta_{\s E},\delta_{\s{E}})$ is an action of ${\cal G}$ on $\s E$, where $\beta_{\s E}:N^{\rm o}\rightarrow\Lin(\s E)$ is the *-homomorphism defined in \ref{propfib};
 \item there exists a unique faithful *-homomorphism $\delta_{\mathcal{K}(\s{E})}:\mathcal{K}(\s{E})\rightarrow\widetilde{\mathcal{M}}(\mathcal{K}(\s{E})\tens S)$ such that 
 $\iota_{\K(\s{E}\tens S)}\circ\delta_{\mathcal{K}(\s{E})}=\delta_J\circ\iota_{\mathcal{K}(\s{E})}$; moreover, the pair $(\beta_{\s E},\delta_{\mathcal{K}(\s{E})})$ is an action of ${\cal G}$ on $\mathcal{K}(\s{E})$.
\end{itemize}
b) Conversely, let $(\beta_{\s E},\delta_{\s{E}})$ be an action of ${\cal G}$ on the Hilbert $A$-module $\s E$. Then, there exists a faithful *-homomorphism $\delta_J:J\rightarrow\widetilde{\mathcal{M}}(J\tens S)$ such that 
$
\iota_{\s E\tens S}\circ\delta_\s{E}=\delta_J\circ\iota_\s{E}.
$
Moreover, we define a unique action $(\beta_J,\delta_J)$ of ${\cal G}$ on $J$ compatible with $(\beta_A,\delta_A)$ by setting
\[
\beta_J(n^{\rm o})=\begin{pmatrix}\beta_{\s{E}}(n^{\rm o})&0\\0&\beta_A(n^{\rm o})\end{pmatrix}\!, \quad \text{for all } n\in N. \qedhere
\]
\end{prop}

\begin{proof}
a) Let us assume that the C*-algebra $J$ is endowed with a compatible action $(\beta_J,\delta_J)$ of ${\cal G}$. Let $\beta_{\s E}:N^{\rm o}\rightarrow\Lin(\s E)$ be the *-homomorphism defined in Proposition \ref{propfib}.  
By strict continuity and \ref{compcoact} 1, we have $\delta_J(\iota_A(m))=\iota_{A\tens S}(\delta_A(m))$ for all $m\in\M(A)$. It then follows from \ref{propfib} that 
\[
\delta_J(\iota_{\K(\s E)}(1_{\s E}))=\delta_J(1_J)-\delta_J(\iota_A(1_A))=q_{\beta_J\alpha}-\iota_{A\tens S}(q_{\beta_A\alpha})=\iota_{\K(\s E\tens S)}(q_{\beta_{\s E}\alpha}).
\]
Let $\xi\in \s{E}$. We have $\iota_{\K(\s{E})}(1_{\s{E}})\iota_{\s{E}}(\xi)=\iota_{\s{E}}(\xi)$ and $\iota_{\s{E}}(\xi)\iota_{\K(\s{E})}(1_{\s{E}})=0$. Hence,
\[
\iota_{\K(\s{E}\tens S)}(q_{\beta_{\s E}\alpha})\delta_J(\iota_{\s{E}}(\xi))=\delta_J(\iota_{\s{E}}(\xi)) \quad \text{and} \quad
\delta_J(\iota_{\s{E}}(\xi))\iota_{\K(\s{E}\tens S)}(q_{\beta_{\s E}\alpha})=0.
\]
We have 
$
\iota_{A\tens S}(x)\delta_{J}(\iota_{\s{E}}(\xi))=\iota_{A\tens S}(x)\iota_{\mathcal{L}(\s{E}\tens S)}(q_{\beta_{\s E}\alpha})\delta_{J}(\iota_{\s{E}}(\xi))=0
$,
for all $x\in A\tens S$. Now, let $(u_{\lambda})_{\lambda}$ be an approximate unit of $A$. We have
\[
 \delta_J(\iota_{\s{E}}(\xi))=\lim_{\lambda}\ \delta_J(\iota_{\s{E}}(\xi u_{\lambda}))=\lim_{\lambda}\ \delta_J(\iota_{\s{E}}(\xi))\iota_{A\tens S}(\delta_A(u_{\lambda})).
\]
Hence, $\delta_J(\iota_{\s{E}}(\xi))\iota_{\s E\tens S}(\eta)=0$ for all $\eta\in \s{E}\tens S$. Hence, there exists a unique linear map 
$
\delta_{\s{E}}:\s{E}\rightarrow\mathcal{L}(A\tens S,\s{E}\tens S)
$
such that 
$
\iota_{\s E\tens S}\circ\delta_\s{E}=\delta_J\circ\iota_\s{E}
$
(cf.\ \ref{ehmlem1} 1). Moreover,$\vphantom{\widetilde{\M}}$ $\delta_{\s{E}}$ actually takes its values in the subspace $\widetilde{\mathcal{M}}(\s{E}\tens S)$ of $\Lin(A\tens S,\s{E}\tens S)$. Indeed, let us fix $\xi\in \s{E}$ and $s\in S$. By assumption, we have that$\vphantom{\widetilde{\M}}$
\[
\iota_{\s E\tens S}((1_{\s{E}}\tens s)\delta_{\s{E}}(\xi))=(1_J\tens s)\delta_J(\iota_{\s{E}}(\xi)) \; \text{ and } \;
\iota_{\s E\tens S}(\delta_{\s{E}}(\xi)(1_A\tens s))=\delta_J(\iota_{\s{E}}(\xi))(1_J\tens s)
\]
belong to $J\tens S=\K((\s{E}\tens S)\oplus(A\tens S))$. It then follows that $(1_{\s{E}}\tens s)\delta_{\s{E}}(\xi) \in \s{E}\tens S$ and $\delta_{\s{E}}(\xi)(1_A\tens s)\in \s{E}\tens S$. The first condition of \ref{hilbmodequ} is easily derived from the compatibility of $\delta_J$.
The vector subspace of $\delta_J(1_J)((\s{E}\oplus A)\tens S)$ spanned by 
\[
\{\delta_J(\theta_{\xi\,\oplus\, a,\,\eta\,\oplus\, b})(\zeta) \; ;\; \xi,\,\eta\in \s{E},\; a,\,b\in A,\; \zeta\in(\s{E}\oplus A)\tens S\}
\]
is dense. However, we have
\[
\delta_J(\theta_{\xi\,\oplus\, a,\,\eta\,\oplus\, b})(\zeta)=(\delta_{\s{E}}(\xi)\oplus\delta_A(a))(\delta_{\s{E}}(\eta)\oplus\delta_A(b))^*(\zeta),
\]
where $\delta_{\s{E}}(\xi)\oplus\delta_A(a),\,\delta_{\s{E}}(\eta)\oplus\delta_A(b)\in\mathcal{L}(A\tens S,\s{E}\tens S)\oplus\mathcal{L}(A\tens S)\subset\mathcal{L}(A\tens S,(\s{E}\oplus A)\tens S)$. In particular, the vector subspace of $\delta_J(1_J)((\s{E}\oplus A)\tens S)$ spanned by 
\[
\{\delta_{\s{E}}(\xi)x\oplus\delta_A(a)x \; ; \; \xi\in\s E,\, a\in A,\, x\in A\tens S\}
\]
is dense. Therefore, the relation $[\delta_{\s E}(\s E)(A\tens S)]=q_{\beta_{\s E}\alpha}(\s E\tens S)$ follows since we also have
$
\delta_J(1_J)((\s{E}\oplus A)\tens S)=q_{\beta_{\s E}\alpha}(\s{E}\tens S)\oplus q_{\beta_A\alpha}(A\tens S).
$\newline
Let us fix $\xi\in\s E$ and $n\in N$. We have 
\begin{multline*}
\iota_{\s E\tens S}(\delta_{\s E}(\beta_{\s E}(n^{\rm o})\xi))=\delta_J(\iota_{\s E}(\beta_{\s E}(n^{\rm o})\xi))=\delta_J(\beta_J(n^{\rm o})\iota_{\s E}(\xi))=(1_J\tens\beta(n^{\rm o}))\delta_J(\iota_{\s E}(\xi))\\
=\iota_{\s E\tens S}((1_{\s E}\tens\beta(n^{\rm o})\delta_{\s E}(\xi)).
\end{multline*} 
Hence, $\delta_{\s E}(\beta_{\s E}(n^{\rm o})\xi)=(1_{\s E}\tens\beta(n^{\rm o}))\delta_{\s E}(\xi)$.\newline
Let us consider the linear maps $\id_{\s E}\tens\delta:\Lin(A\tens S,\s E\tens S)\rightarrow\Lin(A\tens S\tens S,\s E\tens S\tens S)$ and $\delta_{\s E}\tens\id_S:\Lin(A\tens S,\s E\tens S)\rightarrow\Lin(A\tens S\tens S,\s E\tens S\tens S)$ (cf.\ \ref{rk9}). By using (\ref{eq26}) and the compatibility of $\delta_J$ with $\delta_A$ and $\delta_{\s E}$, it follows from a straightforward computation that
\[
\iota_{\s E\tens S\tens S}((\id_{\s E}\tens\delta)(T))=(\id_J\tens\delta)(\iota_{\s E\tens S}(T));\; 
\iota_{\s E\tens S\tens S}((\delta_{\s E}\tens\id_S)(T))=(\delta_J\tens\id_S)(\iota_{\s E\tens S}(T));
\]
for all $T\in\Lin(A\tens S,\s E\tens S)$. In particular, $\iota_{\s E\tens S\tens S}((\id_{\s E}\tens\delta)\delta_{\s E}(\xi))=(\id_J\tens\delta)\delta_J(\iota_{\s E}(\xi))$ and $\iota_{\s E\tens S\tens S}((\delta_{\s E}\tens\id_S)\delta_{\s E}(\xi))=(\delta_J\tens\id_S)\delta_J(\iota_{\s E}(\xi))$ for all $\xi\in\s E$. Hence, for all $\xi\in\s E$ we have $(\delta_{\s E}\tens\id_S)\delta_{\s E}(\xi)=(\id_{\s E}\tens\delta)\delta_{\s E}(\xi)$. Therefore, the pair $(\beta_{\s E},\delta_{\s E})$ is an action of $\cal G$ on $\s E$.

\medskip

We claim that there exists a unique *-homomor\-phism 
$
\delta_{\K(\s{E})}:\K(\s{E})\rightarrow\mathcal{M}(\mathcal{K}(\s{E})\tens S)
$
such that $\iota_{\K(\s E\tens S)}\circ\delta_{\K(\s{E})}=\delta_J\circ\iota_{\K(\s E)}$.  We recall that
$
\delta_J(\iota_{\K(\s E)}(1_{\s{E}}))=\iota_{\K(\s E\tens S)}(q_{\beta_{\s E}\alpha}). 
$
We also have 
$
\iota_{A\tens S}(x)\iota_{\K(\s E\tens S)}(q_{\beta_{\s E}\alpha})=0
$ 
and
$
\iota_{\K(\s E\tens S)}(q_{\beta_{\s E}\alpha})\iota_{A\tens S}(x)=0
$
for all $x\in A\tens S$. It follows that
$
\iota_{A\tens S}(x)\delta_J(\iota_{\K(\s E)}(k))=0
$
and
$
\delta_J(\iota_{\K(\s E)}(k))\iota_{A\tens S}(x)=0
$
for all $k\in\K(\s E)$ and $x\in A\tens S$. Hence, $\delta_J(\iota_{\K(\s E)}(k))\in\iota_{\K(\s E\tens S)}(\mathcal{L}(\s{E}\tens S))$ (cf.\ \ref{ehmlem1}) and the claim is proved since $\iota_{\K(\s E\tens S)}$ is faithful. Since $\iota_{\K(\s E\tens S)}$ is isometric and $\delta_J\circ\iota_{\K(\s E)}$ is strictly continuous, the *-homomorphism $\delta_{\K(\s{E})}$ is strictly continuous and extend uniquely to a strictly continuous *-homomorphism 
$
\delta_{\K(\s{E})}:\Lin(\s{E})\rightarrow\mathcal{M}(\mathcal{K}(\s{E})\tens S)
$
such that $\delta_{\K(\s E)}(1_{\s E})=q_{\beta_{\s E}\alpha}$. Moreover, for all $\xi,\eta\in \s{E}$ we have (cf.\ \ref{prop32} 2)
\begin{equation}\label{eq25}
\delta_{\mathcal{K}(\s{E})}(\theta_{\xi,\eta})=\delta_{\s{E}}(\xi)\circ\delta_{\s{E}}(\eta)^*=\theta_{\delta_{\s{E}}(\xi),\,\delta_{\s{E}}(\eta)}\in\mathcal{K}(\widetilde{\mathcal{M}}(\s{E}\tens S))\subset\widetilde{\mathcal{M}}(\mathcal{K}(\s{E})\tens S).
\end{equation}
Hence, $\delta_{\mathcal{K}(\s{E})}(\mathcal{K}(\s{E}))\subset \widetilde{\mathcal{M}}(\K(\s{E})\tens S)$. We have $\delta_{\K(\s E)}(\beta_{\s E}(n^{\rm o}))=(1_{\s E}\tens\beta(n^{\rm o}))q_{\beta_{\s E}\alpha}$, for all $n\in N$ (cf.\ \ref{eq25}, \ref{rk13}). By strict continuity, we have the formulas:
\begin{align*}
\iota_{\K(\s E\tens S\tens S)}(\id_{\K(\s E)}\tens\delta)(T)&=(\id_J\tens\delta)(\iota_{\K(\s E\tens S)}(T));\\
\iota_{\K(\s E\tens S\tens S)}(\delta_{\K(\s E)}\tens\id_S)(m))&=(\delta_J\tens\id_S)(\iota_{\K(\s E\tens S)}(T));
\end{align*}
for all $T\in\M(\K(\s E)\tens S)=\Lin(\s E\tens S)$. By applying the above formulas to $T:=\delta_{\K(\s E)}(k)$ for $k\in\K(\s E)$, we show that $(\delta_{\K(\s E)}\tens\id_S)\delta_{\K(\s E)}(k)=(\id_{\K(\s E)}\tens\delta)\delta_{\K(\s E)}(k)$.

\medskip

b) First, it is clear that $\beta_J$ is a non-degenerate *-homomorphism. It is also clear that $\beta_J$ is compatible with the fibration map $\beta_A$, {\it i.e.\ }$\beta_J(n^{\rm o})\iota_A(a)=\iota_A(\beta_A(n^{\rm o})a)$, for all $a\in A$ and $n\in N$. Let $\s{V}\in\mathcal{L}(\s{E}\tens_{\delta_A}(A\tens S),\s{E}\tens S)$ be the isometry associated with the action $\delta_{\s{E}}$. Let $i:q_{\beta_A\alpha}(A\tens S)\rightarrow A\tens S$ be the inclusion map. We verify that $i$ is an $(A\tens S)$-linear adjointable map and $i^*=q_{\beta_A\alpha}$. In particular, the map $i$ is an isometry since we have $i^*i(x)=q_{\beta_A\alpha}x=x$ for all $x\in q_{\beta_A\alpha}(A\tens S)$. Let
\[
\s W:=\s{V}\oplus i\in\mathcal{L}((\s{E}\tens_{\delta_A}(A\tens S))\oplus q_{\beta_A\alpha}(A\tens S),(\s E\tens S)\oplus (A\tens S)).
\] 
We have $\s W^*\s W=1$, then $\s W$ is an isometry. Henceforth, we will use the following identification (see (\ref{ehmeq1})):
\[
(\s{E}\tens_{\delta_A}\!(A\tens S))\oplus q_{\beta_A\alpha}(A\tens S)=(\s{E}\tens_{\delta_A}\!(A\tens S))\oplus(A\tens_{\delta_A}\!(A\tens S))=(\s{E}\oplus A)\tens_{\delta_A}\!(A\tens S).
\]
Hence, $\s W\in\mathcal{L}((\s{E}\oplus A)\tens_{\delta_A}(A\tens S),(\s{E}\oplus A)\tens S)$. Let us define 
\[
\delta_J(x):=\s W(x\tens_{\delta_A}1)\s W^*\in\M(J\tens S),\quad \text{for all } x\in J.
\]
In that way, we define a strictly continuous *-homomorphism $\delta_J:J\rightarrow\M(J\tens S)$ satisfying $\delta_J(1_J)=\s W\s W^*=q_{\beta_{\s E}\alpha}\oplus q_{\beta_A\alpha}=q_{\beta_J\alpha}$.
Let $a\in A$. Let us prove that $\iota_{A\tens S}(\delta_A(a))=\delta_J(\iota_A(a))$. Since $\iota_{A\tens S}(\delta_A(a))\s W\s W^*=\iota_{A\tens S}(\delta_A(a))$ and $\delta_J(\iota_A(a))\s W\s W^*=\delta_J(\iota_A(a))$, it amounts to proving that 
$
\iota_{A\tens S}(\delta_A(a))\s W=\delta_J(\iota_A(a))\s W,
$
for all $a\in A$.
Therefore, it is enough to prove that $\iota_{A\tens S}(\delta_A(a))\s W=\s W(\iota_A(a)\tens_{\delta_A}1)$ since $\s W^*\s W=1$. However, for all $\eta\in \s E$, $b\in A$ and $x\in A\tens S$ we have
\[
\s W((\eta\oplus b)\tens_{\delta_A}x)=\s{V}(\eta\tens_{\delta_A}x)\oplus\delta_A(b)x=\delta_{\s{E}}(\eta)x\oplus\delta_A(b)x.
\]
We finally obtain
\begin{align*}
\s W(\iota_A(a)\tens_{\delta_A}1)((\eta\oplus b)\tens_{\delta_A}x)&=\s W((0\oplus ab)\tens_{\delta_A}x)\\
&=(\s{V}\oplus i)(0\oplus\delta_A(ab)x)\\
&=0\oplus\delta_A(a)\delta_A(b)x\\
&=\iota_{A\tens S}(\delta_A(a))(\delta_{\s{E}}(\eta)x\oplus\delta_A(b)x)\\
&=\iota_{A\tens S}(\delta_A(a))\s W((\eta\oplus b)\tens_{\delta_A}x),
\end{align*}
for all $\eta\in \s{E}$, $b\in A$ and $x\in A\tens S$. By using similar arguments, we also prove that $\iota_{\s E\tens S}(\delta_{\s{E}}(\xi))=\delta_J(\iota_\s{E}(\xi))$ for all $\xi\in \s{E}$.
By strict continuity, we obtain the formulas: 
\[
(\delta_J\tens\id_S)\iota_{A\tens S}(m)=\iota_{A\tens S\tens S}(\delta_A\tens\id_S)(m);\; (\id_J\tens\delta)\iota_{A\tens S}(m)=\iota_{A\tens S\tens S}(\id_A\tens\delta)(m);
\] 
for all $m\in\M(A\tens S)$.
By applying the above formulas to $m:=\delta_A(a)$ for $a\in A$ and by using again the compatibility of $\delta_J$ with $\delta_A$, we obtain the formulas: 
\[
(\delta_J\tens\id_S)\delta_J(\iota_A(a))=\iota_{A\tens S\tens S}(\delta_A\tens\id_S)\delta_A(a); \; (\id_J\tens\delta)\delta_J(\iota_A(a))=\iota_{A\tens S\tens S}(\id_A\tens\delta)\delta_A(a).
\]
Hence, $(\delta_J\tens\id_S)\delta_J(\iota_A(a))=(\id_J\tens\delta)\delta_J(\iota_A(a))$ for all $a\in A$. In a similar way, we have $(\delta_J\tens\id_S)\delta_J(\iota_{\s E}(\xi))=(\id_J\tens\delta)\delta_J(\iota_{\s E}(\xi))$ for all $\xi\in\s E$. However, since $J$ is generated by $\iota_{\s{E}}(\s E)\cup\iota_A(A)$ as a C*-algebra, the coassociativity of $\delta_J$ is then proved.

\medskip

For all $\eta\in \s{E}$, $b\in A$, $x\in A\tens S$ and $n\in N$, we have
\begin{align*}
\delta_J(\beta_J(n^{\rm o}))\s W((\eta\oplus b)\tens_{\delta_A}x)
&=\s W(\beta_J(n^{\rm o})(\eta\oplus b)\tens_{\delta_A} x)\\
&=\s W((\beta_{\s{E}}(n^{\rm o})\eta \oplus \beta_A(n^{\rm o})b)\tens_{\delta_A} x)\\
&=\delta_{\s{E}}(\beta_{\s{E}}(n^{\rm o})\eta)x \oplus \delta_A(\beta_A(n^{\rm o})b)x\\
&=(1_J\tens\beta(n^{\rm o}))(\delta_{\s{E}}(\eta)x \oplus \delta_A(b)x)\\
&=(1_J\tens\beta(n^{\rm o}))\s W((\eta\oplus b)\tens_{\delta_A}x).
\end{align*}
Hence, 
$
\delta_J(\beta_J(n^{\rm o}))
=\delta_J(\beta_J(n^{\rm o}))\s W\s W^*=(1_J\tens\beta(n^{\rm o}))\s W\s W^*
=(1_J\tens\beta(n^{\rm o}))\delta_J(1_J),
$
for all $n\in N$. Therefore, $(\beta_J,\delta_J)$ is an action of $\cal G$ on $J$, compatible with $(\beta_A,\delta_A)$.
\end{proof}

\paragraph{Equivariant unitary equivalence. } In this paragraph, we define a notion of equivariance for unitary equivalences of Hilbert modules acted upon by $\cal G$. We refer the reader to \S \ref{UnitaryEq} for the definitions and notations used below.

\begin{defin}\label{def1}
Let $A$ and $B$ be two $\cal G$-C*-algebras and $\phi:A\rightarrow B$ a $\cal G$-equivariant *-isomorphism. Let $\s E$ and $\s F$ be two Hilbert modules over respectively $A$ and $B$ acted upon by $\cal G$. A $\phi$-compatible unitary operator $\Phi:\s E\rightarrow\s F$ is said to be $\cal G$-equivariant if we have
\[
\delta_{\s F}(\Phi\xi)=(\Phi\tens\id_S)\delta_{\s E}(\xi), \quad \text{for all } \xi\in\s E. \qedhere
\]
\end{defin}

We recall that the linear map $\Phi\tens\id_S:\Lin(A\tens S,\s E\tens S)\rightarrow\Lin(B\tens S,\s F\tens S)$ (\ref{not3}) is the extension of the $\phi\tens\id_S$-compatible unitary operator $\Phi\tens\id_S:\s E\tens S\rightarrow\s F\tens S$ (\ref{propdef6}).

\begin{prop}
With the notations and hypotheses of \ref{def1}, for all $n\in N$ we have $\beta_{\s F}(n^{\rm o})\circ\Phi=\Phi\circ\beta_{\s E}(n^{\rm o})$.
\end{prop}

\begin{proof}
It is clear that $(\Phi\tens\id_S)((1_{\s E}\tens s)T)=(1_{\s F}\tens s)(\Phi\tens\id_S)(T)$ for all $s\in S$ and $T\in\Lin(A\tens S,\s E\tens S)$. Let $n\in N$ and $\xi\in\s E$. We have
$
\delta_{\s F}(\Phi(\beta_{\s E}(n^{\rm o})\xi))=\delta_{\s F}(\beta_{\s F}(n^{\rm o})\Phi\xi)
$
by \ref{hilbmodequ} 3. Hence, $\Phi(\beta_{\s E}(n^{\rm o})\xi)=\beta_{\s F}(n^{\rm o})\Phi\xi$ by \ref{rk2} 1.
\end{proof}

\begin{defin}
Two Hilbert C*-modules $\s E$ and $\s F$ acted upon by $\cal G$ are said to be $\cal G$-equivariantly unitarily equivalent if there exists a $\cal G$-equivariant unitary operator from $\s E$ onto $\s F$.
\end{defin}

It is clear that the $\cal G$-equivariant unitary equivalence defines an equivalence relation on the class consisting of the Hilbert C*-modules acted upon by $\cal G$. In the following, we provide equivalent definitions of the $\cal G$-equivariant unitary equivalence in the two other pictures. 

\medskip

Let $A$ and $B$ be two $\cal G$-C*-algebras and $\phi:A\rightarrow B$ a $\cal G$-equivariant *-isomomorphism. Let $\s E$ and $\s F$ be two Hilbert C*-modules over $A$ and $B$ respectively and $\Phi:\s E\rightarrow\s F$ a $\phi$-compatible unitary operator. 

\begin{lem}\label{lem12}
The linear map
\[
\s E\tens_{\delta_A}(A\tens S)  \rightarrow \s F\tens_{\delta_B}(B\tens S)\; ;\; \xi\tens_{\delta_A}x  \mapsto \Phi\xi \tens_{\delta_B}(\phi\tens\id_S)(x)
\]
is a $\phi\tens\id_S$-compatible unitary operator.
\end{lem}

\begin{proof}
For all $\xi\in\s E$, $a\in A$ and $x\in A\tens S$, we have
\begin{align*}
\Phi(\xi a) \tens_{\delta_B} (\phi\tens\id_S)(x) &= \Phi(\xi)\phi(a) \tens_{\delta_B} (\phi\tens\id_S)(x)\\
&= \Phi\xi \tens_{\delta_B} \delta_B(\phi(a))(\phi\tens\id_S)(x) \\
&= \Phi\xi \tens_{\delta_B} (\phi\tens\id_S)(\delta_A(a)x).
\end{align*}
Therefore, we have a well-defined linear map 
\[
\Psi:\s E\odot_{A}(A\tens S) \rightarrow \s F\tens_{\delta_B}(B\tens S)\,;\, \xi\odot_A x \mapsto \Phi\xi \tens_{\delta_B}(\phi\tens\id_S)(x).
\]
For all $\xi,\eta\in\s E$, we have $\delta_B(\langle\Phi\xi,\, \Phi\eta\rangle)=\delta_B(\phi(\langle\xi,\,\eta\rangle))=(\phi\tens\id_S)\delta_A(\langle\xi,\,\eta\rangle)$. 
Therefore, for all $\xi,\eta\in\s E$ and $x,y\in A\tens S$, we have 
\[
\langle \Phi\xi \tens_{\delta_B}(\phi\tens\id_S)(x),\, \Phi\eta \tens_{\delta_B}(\phi\tens\id_S)(y)\rangle=(\phi\tens\id_S)(\langle \xi\tens_{\delta_A}x,\, \eta\tens_{\delta_A}y\rangle).
\] 
Hence, $\langle \Psi(t),\, \Psi(t')\rangle=(\phi\tens\id_S)(\langle t,\, t'\rangle)$ for all $t,t'\in \s E\odot_{A}(A\tens S)$. In particular, we have $\|\Psi(t)\|=\|t\|$ for all $t\in \s E\odot_{A}(A\tens S)$ ($\phi\tens\id_S$ is isometric). Therefore, $\Psi$ extends uniquely to a bounded operator from $\s E\tens_{\delta_A}(A\tens S)$ to $\s F\tens_{\delta_B}(B\tens S)$ still denoted by $\Psi$. We have $\langle \Psi(t),\, \Psi(t')\rangle=(\phi\tens\id_S)(\langle t,\, t'\rangle)$ for all $t,t'\in \s E\tens_{\delta_A}(A\tens S)$. Since $\Psi$ is isometric and has a dense range, we conclude that $\Psi$ is surjective. A staightforward computation shows that $\Psi(tx)=\Psi(t)(\phi\tens\id_S)(x)$ for all  $t\in \s E\tens_{\delta_A}(A\tens S)$ and $x\in A\tens S$.
\end{proof}

\begin{prop}
Let $(\beta_{\s E},\delta_{\s E})$ {\rm(}resp.\ $(\beta_{\s F},\delta_{\s F})${\rm)} be an action of $\cal G$ on $\s E$ {\rm(}resp.\ $\s F${\rm)}. Denote by $\s V_{\s E}\in\Lin(\s E\tens_{\delta_A}(A\tens S),\s E\tens S)$ {\rm(}resp.\ $\s V_{\s F}\in\Lin(\s F\tens_{\delta_B}(B\tens S),\s F\tens S)${\rm)} the isometry associated with $(\beta_{\s E},\delta_{\s E})$ {\rm(}resp.\ $(\beta_{\s F},\delta_{\s F})${\rm)}. Assume that $\Phi\circ\beta_{\s E}(n^{\rm o})=\beta_{\s F}(n^{\rm o})\circ\Phi$ for all $n\in N$. Then, $\Phi$ is $\cal G$-equivariant if, and only if, we have
\[
\s V_{\s F}^*(\Phi\tens\id_S)\s V_{\s E}(\xi\tens_{\delta_A}x)=\Phi\xi \tens_{\delta_B}(\phi\tens\id_S)(x),
\] 
for all $\xi\in\s E$ and $x\in A\tens S$.
\end{prop}

\begin{proof}
Let $\Psi:\s E\tens_{\delta_A}(A\tens S) \rightarrow  \s F\tens_{\delta_B}(B\tens S)$ be the $\phi\tens\id_S$-compatible unitary operator defined in the above lemma. For all $\xi\in\s E$ and $x\in A\tens S$, we have 
\[
\delta_{\s F}(\Phi\xi)(\phi\tens\id_S)(x)=\s V_{\s F}(\Phi\xi\tens_{\delta_B}x)=\s V_{\s F}\circ \Psi(\xi\tens_{\delta_A} 
x)
\]
and $(\Phi\tens\id_S)(\delta_{\s E}(\xi))(\phi\tens\id_S)(x)=(\Phi\tens\id_S)(\delta_{\s E}(\xi)x)=(\Phi\tens\id_S)\s V_{\s E}(\xi\tens_{\delta_A}x)$. Therefore, $\delta_{\s F}\circ \Phi=(\Phi\tens\id_S)\circ\delta_{\s E}$ if, and only if, $\s V_{\s F}\circ \Psi=(\Phi\tens\id_S)\s V_{\s E}$. In order for the formula $\s V_{\s F}\circ \Psi=(\Phi\tens\id_S)\s V_{\s E}$ to hold true, it is necessary and sufficient that $\Psi=\s V_{\s F}^*(\Phi\tens\id_S)\s V_{\s E}$. Indeed, it is necessary since $\s V_{\s F}^*\s V_{\s F}=1$. It is also sufficient since we have $\s V_{\s F}\s V_{\s F}^*=q_{\beta_{\s F}\alpha}$, $q_{\beta_{\s F}\alpha}(\Phi\tens\id_S)=(\Phi\tens\id_S)q_{\beta_{\s E}\alpha}$ (by assumption) and $q_{\beta_{\s E}\alpha}\s V_{\s E}=\s V_{\s E}$.
\end{proof}

\begin{rk}
Let $A$ be a $\cal G$-C*-algebra. Let $\s E$ and $\s F$ be two Hilbert $A$-modules acted upon by $\cal G$. Let $\Phi\in\Lin(\s E,\s F)$ be a unitary. The following statements are equivalent:
\begin{enumerate}[label=(\roman*)]
\item $\Phi$ is $\cal G$-equivariant;
\item $\Phi\circ\beta_{\s E}(n^{\rm o})=\beta_{\s F}(n^{\rm o})\circ\Phi$ for all $n\in N$ and $\s V_{\s F}^*(\Phi\tens 1_S)\s V_{\s E}=\Phi\tens_{\delta_A}1_{A\tens S}$;
\item $\Phi\circ\beta_{\s E}(n^{\rm o})=\beta_{\s F}(n^{\rm o})\circ\Phi$ for all $n\in N$ and $\s V_{\s F}(\Phi\tens_{\delta_A}1_{A\tens S})\s V_{\s E}^*=q_{\beta_{\s E}\alpha}(\Phi\tens 1_S)$.\qedhere
\end{enumerate}
\end{rk}

\begin{prop}\label{prop10}
Let $A$ and $B$ be two $\cal G$-C*-algebras and $\phi:A\rightarrow B$ a $\cal G$-equivariant *-isomorphism. Let $\s E$ and $\s F$ be two Hilbert modules over respectively $A$ and $B$ acted upon by $\cal G$. Let $\Phi:\s E\rightarrow\s F$ be a $\phi$-compatible unitary operator. Denote by $f:\K(\s E\oplus A)\rightarrow\K(\s F\oplus B)$ the unique *-homomorphism such that $f\circ\iota_{\s E}=\iota_{\s F}\circ T$ and $f\circ\iota_A=\iota_B\circ\phi$ (cf.\ \ref{prop7}). Then, $\Phi$ is $\cal G$-equivariant if, and only if, $f$ is $\cal G$-equivariant.
\end{prop}

\begin{proof}
Let $J:=\K(\s E\oplus A)$ and $L:=\K(\s F\oplus B)$. Assume that $\Phi$ is equivariant. The following formulas are immediate consequences of the definitions:
\begin{align*}
\iota_{B\tens S}\circ(\phi\tens\id_S)(m)&=(f\tens\id_S)\circ\iota_{A\tens S}(m),\quad m\in\M(A\tens S);\\
\iota_{\s F\tens S}\circ(\Phi\tens\id_S)(T)&=(f\tens\id_S)\circ\iota_{\s E\tens S}(T),\quad T\in\Lin(A\tens S,\s E\tens S).
\end{align*}
By combining the first (resp.\ second) formula with the $\cal G$-equivariance of $\phi$ (resp.\ $\Phi$) and the fact that $f\circ\iota_A=\iota_B\circ\phi$ (resp.\ $f\circ\iota_{\s E}=\iota_{\s F}\circ \Phi$), we obtain
\begin{align*}
\delta_{L}\circ f(\iota_A(a))&=(f\tens\id_S)\delta_{J}(\iota_A(a)),\quad \text{for all } a\in A\\[0.2cm]
(\text{resp.\ } \delta_{L}\circ f(\iota_{\s E}(\xi))&=(f\tens\id_S)\delta_{J}(\iota_{\s E}(\xi)),\quad \text{for all }\xi\in\s E).
\end{align*}
Therefore, we have $\delta_{L}(f(x))=(f\tens\id_S)\delta_{J}(x)$ for all $x\in J$. Moreover, by definition of the fibration map on a linking C*-algebra (cf.\ \ref{prop1}) and the $\cal G$-equivariance of $\Phi$, we have
\[
f(\beta_{J}(n^{\rm o}))=\begin{pmatrix} \beta_{\s E}(n^{\rm o}) & 0 \\ 0 & \beta_A(n^{\rm o})\end{pmatrix}=\begin{pmatrix}\Phi\circ\beta_{\s E}(n^{\rm o})\circ \Phi^{-1} & 0 \\ 0 & \phi(\beta_A(n^{\rm o}))\end{pmatrix}=\beta_{L}(n^{\rm o}),
\]
for all $n\in N$. The converse is proved in a similar way.
\end{proof}

\paragraph{Continuous actions.} In this paragraph, we introduce the notion of continuity for actions of the quantum groupoid $\cal G$ on Hilbert $A$-modules. If $\cal G$ is regular, we prove that any action of $\cal G$ is necessarily continuous.

\begin{defin}
An action $(\beta_{\s E},\delta_{\s E})$ of $\cal G$ on a Hilbert $A$-module $\s E$ is said to be continuous if we have $[(1_{\s E}\tens S)\delta_{\s E}(\s E)]=(\s E\tens S)q_{\beta_A\alpha}$. A $\cal G$-equivariant Hilbert $A$-module is a Hilbert $A$-module $\s E$ endowed with a continuous action of $\cal G$.
\end{defin}

\begin{prop}\label{prop31}
Let $\s E$ be a $\cal G$-equivariant Hilbert $A$-module. Let $B:=\K(\s E)$. We have the following statements:
\begin{enumerate}
\item the action $(\beta_B,\delta_B)$ of $\cal G$ on $B$ defined in \ref{prop1} is strongly continuous;
\item we define a continuous action of $\cal G$ on the Hilbert $B$-module $\s E^*$ by setting:
\begin{itemize}
\item $\beta_{\s E^*}(n^{\rm o})T:=\beta_{A}(n^{\rm o})\circ T$, for all $n\in N$ and $T\in\s E^*$,
\item $\delta_{\s E^*}(T)x:=\delta_{\s E}(T^*)^*\circ x$, for all $T\in\s E^*$ and $x\in B\tens S$;
\end{itemize}
where we have applied the usual identifications $B\tens S=\K(\s E\tens S)$ and $\s E=\K(A,\s E)$.\qedhere
\end{enumerate}
\end{prop}

\begin{proof}
1. We have $[\delta_B(B)(1_B\tens S)]=[\delta_B(\theta_{\xi,\eta})(1_B\tens y)\,;\,\xi,\, \eta\in\s E,\,y\in S]$. However, we have
$
\delta_{B}(\theta_{\xi,\eta})(1_B\tens y)=\delta_{\s E}(\xi)\delta_{\s E}(\eta)^*(1_B\tens y)=\delta_{\s E}(\xi)((1_B\tens y^*)\delta_{\s E}(\eta))^*
$
for all $y\in S$ and $\xi,\,\eta\in\s E$. It then follows from the continuity of the action $(\beta_{\s E},\delta_{\s E})$ and \ref{rk2} that 
\[
[\delta_B(B)(1_B\tens S)]=[\delta_{\s E}(\s E)q_{\beta_A\alpha}(\s E^*\tens S)]=[\delta_{\s E}(\s E)(\s E^*\tens S)].
\] 
Now, by combining the formulas $[\delta_{\s E}(\s E)(1_{\s E}\tens S)]=q_{\beta_{\s E}\alpha}(\s E\tens S)$ and $B=[\s E\s E^*]$ with the fact that any element of $S$ can be written as a product of two elements of $S$, we obtain 
$
[\delta_B(B)(1_B\tens S)]=[\delta_{\s E}(\s E)(1_{\s E}\tens S)(\s E^*\tens S)]=q_{\beta_{\s E}\alpha}(B\tens S).
$\newline
2. Straightforward.
\end{proof}

\begin{prop}\label{propcont}
Let $\s E$ be a Hilbert $A$-module endowed with an action $(\beta_{\s E},\delta_{\s E})$ of $\cal G$ on $\s E$. Let $J:=\K(\s E\oplus A)$ be the associated linking C*-algebra. Let $(\beta_J,\delta_J)$ be the action defined in \ref{prop1}. Then, the action $(\beta_{\s E},\delta_{\s E})$ is continuous if, and only if, the action $(\beta_J,\delta_J)$ is strongly continuous.
\end{prop}

\begin{proof}
Assume that the action $(\beta_{\s E},\delta_{\s E})$ is continuous. Let $B:=\K(\s E)$. Let us prove that $(\beta_J,\delta_J)$ is strongly continuous. Let $x\in J$ and $s\in S$. Let us write 
\[
x=\begin{pmatrix}b & \xi\\ \eta^* & a\end{pmatrix}\!,\quad \text{where }a\in A,\ b\in B\; \text{and} \; \xi,\,\eta\in\s E.
\]
Then, we have
\begin{multline*}
\delta_J(x)(1_J\tens s)=\iota_{B\tens S}(\delta_B(b)(1_B\tens s))+\iota_{\s E\tens S}(\delta_{\s E}(\xi)(1_A\tens s))+\iota_{\s E^*\tens S}(\delta_{\s E^*}(\eta^*)(1_B\tens s))\\
+\iota_{A\tens S}(\delta_A(a)(1_A\tens s)).
\end{multline*}
Then, the continuity of $(\beta_J,\delta_J)$ follows from the continuity of $(\beta_A,\delta_A)$, $(\beta_{\s E},\delta_{\s E})$ and the continuity of $(\beta_B,\delta_B)$ and $(\beta_{\s E^*},\delta_{\s E^*})$ (\ref{prop31}). Conversely, assume that $(\beta_J,\delta_J)$ is continuous. We have $\iota_{\s E\tens S}((\s E\tens S)q_{\beta_A\alpha})=(\iota_{\s E}(\s E)\tens S)q_{\beta_J\alpha}$. Let $\zeta\in\s E$ and $y\in S$. As in the above computation, we prove that $\iota_{\s E\tens S}((\zeta\tens y)q_{\beta_A\alpha})$ is the norm limit of finite sums of elements of the following forms:
$\iota_{B\tens S}((1_B\tens s)\delta_B(b))$, $\iota_{\s E\tens S}((1_{\s E}\tens s)\delta_{\s E}(\xi))$, $\iota_{\s E^*\tens S}((1_{\s E^*}\tens s)\delta_{\s E^*}(\eta^*))$ and $\iota_{A\tens S}((1_A\tens s)\delta_A(a))$. By multiplying by $\iota_{B\tens S}(1_{B\tens S})$ on the left and by $\iota_{A\tens S}(1_{A\tens S})$ on the right, we have that $\iota_{\s E\tens S}((\zeta\tens y)q_{\beta_A\alpha})$ is the norm limit of finite sums of elements of the form $\iota_{\s E\tens S}((1_{\s E}\tens s)\delta_{\s E}(\xi))$. The continuity of $(\beta_{\s E},\delta_{\s E})$ follows from the fact that $\iota_{\s E\tens S}$ is isometric.
\end{proof}

\begin{defin}\label{defLinkAlg}
A linking $\cal G$-C*-algebra is a quintuple $(J,\beta_J,\delta_J,e_1,e_2)$ consisting of a C*-algebra $J$ endowed with a continuous action $(\beta_J,\delta_J)$ and nonzero self-adjoint projections $e_1,e_2\in\M(J)$ satisfying the following conditions:
\begin{enumerate}
\item $e_1+e_2=1_J$;
\item $[Je_jJ]=J$, for all $j=1,2$;
\item $\delta_J(e_j)=q_{\beta_J\alpha}(e_j\tens 1_S)$, for all $j=1,2$.\qedhere
\end{enumerate} 
\end{defin}

\begin{rks}\label{rkLinkAlg} 
\begin{itemize}
\item Let $(A,\beta_A,\delta_A)$ be a $\cal G$-C*-algebra and $m\in\M(A)$ such that $\delta_A(m)=q_{\beta_A\alpha}(m\tens 1_S)$. Let $n\in N$, we have $[m,\,\beta_A(n^{\rm o})]=0$. Indeed, since $\alpha$ and $\beta$ commute pointwise we have $[q_{\beta_A\alpha}(1_A\tens\beta(n^{\rm o})),\,q_{\beta_A\alpha}(m\tens 1_S)]=0$. It then follows that $\delta_A([m,\,\beta_A(n^{\rm o})])=[\delta_A(m\beta_A(n^{\rm o})),\,\delta_A(\beta_A(n^{\rm o})m)]=0$. Hence, $[m,\,\beta_A(n^{\rm o})]=0$ by faithfulness of $\delta_A$. In particular, we have $[q_{\beta_A\alpha},\,m\tens 1_S]=0$.
\item Let $(J,\beta_J,\delta_J,e_1,e_2)$ be a linking $\cal G$-C*-algebra. By restriction of the action $(\beta_J,\delta_J)$, the corner $e_2Je_2$ (resp.\ $e_1Je_2$) turns into a $\cal G$-C*-algebra (resp.\ $\cal G$-equivariant Hilbert C*-module over $e_2Je_2$). Furthermore, we also have the identification of $\cal G$-C*-algebras $\K(e_1Je_2)=e_1Je_1$.
\item Conversely, if $(\s E,\beta_{\s E},\delta_{\s E})$ is a $\cal G$-equivariant Hilbert $A$-module, then the C*-algebra $J:=\K(\s E\oplus A)$ endowed with the continuous action $(\beta_J,\delta_J)$ (cf.\ \ref{prop1}, \ref{propcont}) and the projections $e_1:=\iota_{\s E}(1_{\s E})$ and $e_2:=\iota_{A}(1_A)$ is a linking $\cal G$-C*-algebra.\qedhere
\end{itemize}
\end{rks}

\begin{lem} \label{lem21}
Let $\s E$ be a Hilbert $A$-module endowed with an action $(\beta_{\s E},\delta_{\s E})$ of $\cal G$. We have
$
\s E=[(\id_{\s E}\tens\omega)\delta_{\s E}(\xi)\, ; \, \xi\in\s E,\, \omega\in\B(\s H)_*]
$ (cf.\ \ref{not4}).
\end{lem}

\begin{proof}
We have $\s E\supset[(\id_{\s E}\tens\omega)\delta_{\s E}(\xi)\, ; \, \xi\in\s E,\, \omega\in\B(\s H)_*]$ (cf.\ \ref{not4}). To obtain the converse inclusion, we combine the fact that there exists $\omega\in\B(\s H)_*$ such that $(\id_{\s E}\tens\omega)(q_{\beta_{\s E}\alpha})=1_{\s E}$ with the formula $[\delta_{\s E}(\s E)(1_B\tens S)]=q_{\beta_{\s E}\alpha}(\s E\tens S)$.
\end{proof}

Now, we can state the main results of this chapter.

\begin{thm}\label{theo5}
Let $A$ be a $\cal G$-C$^*$-algebra and $\s E$ a Hilbert $A$-module acted upon by $\cal G$. Let $J:=\K(\s E\oplus A)$ be the associated linking C*-algebra endowed with the action $(\beta_J,\delta_J)$ defined in \ref{prop1}. If $\cal G$ is semi-regular (resp.\ regular), then the action $(\beta_J,\delta_J)$ is weakly (resp.\ strongly) continuous.
\end{thm}

\begin{proof}
Assume that $\cal G$ is semi-regular. Denote by $T:=[(\id_J\tens\omega)\delta_J(x)\,;\, x\in J,\, \omega\in\B(\s H)_*]$ the C$^*$-algebra of continuous elements (cf.\ref{prop38}). By combining the compatibility of $\delta_J$ with $\delta_A$ (resp.\ $\delta_{\s E}$) with the fact that $(\beta_A,\delta_A)$ is (weakly) continuous (resp.\ Lemma \ref{lem21}), we obtain $\iota_A(A)\subset T$ (resp.\ $\iota_{\s E}(\s E)\subset T$). Hence, $J\subset T$. The converse inclusion also holds since $\delta_J(J)\subset\widetilde{\cal M}(J\tens S)$. Hence, $(\beta_J,\delta_J)$ is weakly continuous. It follows from \ref{theo6} that the action $(\beta_J,\delta_J)$ is strongly continuous if $\cal G$ is regular.
\end{proof}

\begin{cor}\label{corActReg}
Let $\s E$ be a Hilbert $A$-module. If the quantum groupoid $\cal G$ is regular, then any action of ${\cal G}$ on $\s E$ is continuous.
\end{cor}

\begin{proof}
This is an immediate consequence of \ref{propcont} and \ref{theo5}.
\end{proof}

	\subsection{Case of a colinking measured quantum groupoid}\label{ActColinkHilb}

Let us fix a colinking measured quantum groupoid ${\cal G}:={\cal G}_{\QG_1,\QG_2}$ between two monoidally equivalent locally compact quantum groups $\QG_1$ and $\QG_2$. Let $(A,\beta_A,\delta_A)$ be a $\cal G$-C*-algebra. We follow all the notations of \S \ref{sectionColinking} (resp. \ref{not12} and \ref{actprop}) concerning the objects associated with $\cal G$ (resp.\ $(A,\beta_A,\delta_A)$).

\medbreak

In the following, we provide a description of Hilbert modules acted upon by $\cal G$ in terms of Hilbert modules acted upon by $\QG_1$ and $\QG_2$. Let us fix a Hilbert $A$-module $\s E$ endowed with an action $(\beta_{\s E},\delta_{\s E})$ of $\cal G$.

\begin{nbs}\label{not8}
We introduce some useful notations to describe the action $(\beta_{\s E},\delta_{\s E})$.
\begin{itemize}
\item Let $q_{\s E,j}:=\beta_{\s E}(\varepsilon_j)$ for $j=1,2$. Note that $q_{\s E,1}$ and $q_{\s E,2}$ are orthogonal self-adjoint projections of $\Lin(\s E)$ and $q_{\s E,1}+q_{\s E,2}=1_{\s E}$.\index[symbol]{qc@$q_{\s E,j}$}
\item Let $J:=\K(\s E\oplus A)$ be the linking C*-algebra associated with $\s E$ endowed with the action $(\beta_J,\delta_J)$ of $\cal G$ (cf.\ \ref{prop1} b)). Since $\beta_J(\GC^2)\subset{\cal Z}({\cal M}(J))$ (cf.\ 3.2.3 \cite{BC}), we have $\beta_{\s E}(n)\xi=\xi\beta_A(n)$ in $\Lin(A,\s E)$ for all $n\in\GC^2$ and $\xi\in\s E$, {\it i.e.\ }$(\beta_{\s E}(n)\xi)a=\xi(\beta_A(n)a)$ for all $n\in\GC^2$, $\xi\in\s E$ and $a\in A$. Hence, 
\begin{equation}\label{eq1.11}
(q_{\s E,j}\xi)a=\xi(q_{A,j}a), \quad \text{for all } \xi\in\s E,\, a\in A,\, j=1,2.
\end{equation}
In particular, we have
\begin{equation*}
\langle q_{\s E,j}\xi,\, q_{\s E,j}\eta\rangle=q_{A,j}\langle\xi,\, \eta\rangle,\quad \text{for all } \xi,\eta\in\s E.
\end{equation*}
Indeed, fix $\xi,\eta\in\s E$ and write $\xi=\xi'a$ and $\eta=\eta' b$ with $\xi',\eta'\in\s E$ and $a,b\in A$. Since the projection $q_{A,j}$ is central in $A$, we have 
$
\langle q_{\s E,j}\xi,\, q_{\s E,j}\eta\rangle=\langle (q_{\s E,j}\xi')a,\, (q_{\s E,j}\eta')b \rangle =\langle \xi'(q_{A,j}a),\, \eta'(q_{A,j}b) \rangle =q_{A,j}a^*\langle\xi',\, \eta'\rangle b=q_{A,j}\langle\xi,\, \eta\rangle.
$
For $j=1,2$, we then define the following Hilbert $A_j$-module $\s E_j:=q_{\s E,j}\s E$. Note that $\s E=\s E_1\oplus\s E_2$. 
\item For $j,k=1,2$, let
$
\Pi_j^k:\s E_k\tens S_{kj}\rightarrow\s E\tens S.
$
\index[symbol]{pk@$\Pi_j^k$}be the inclusion map. It is clear that $\Pi_j^k$ is a $\pi_j^k$-compatible operator (cf.\ \ref{def2}). We can consider its canonical linear extension 
$
\Pi_j^k:\Lin(A_k\tens S_{kj},\s E_k\tens S_{kj})\rightarrow \Lin(A\tens S,\s E\tens S),
$
up to the canonical injective maps $\s E_k\tens S_{kj}\rightarrow\Lin(A_k\tens S_{kj},\s E_k\tens S_{kj})$ and $\s E\tens S\rightarrow\Lin(A\tens S,\s E\tens S)$, $\vphantom{\Pi_j^k}$defined by
$
\Pi_j^k(T)(x):=\Pi_j^k\circ T((q_{A,k}\tens p_{kj})x)
$ for all $T\in\Lin(A_k\tens S_{kj},\s E_k\tens S_{kj})$ and $x\in A\tens S$.\qedhere
\end{itemize}
\end{nbs}

\begin{lem}
With the above notations and hypotheses, we have a canonical unitary equivalence of Hilbert $A\tens S$-modules
$
\s E\tens_{\delta_A}(A\tens S)=\bigoplus_{j,k=1,2}\s E_j\tens_{\delta_{A_j}^k}\!\!(A_k\tens S_{kj}).
$
\end{lem}

\begin{proof}
This is a straightforward verification to see that we define a unitary adjointable operator by the following formula:
\[
\s E\tens_{\delta_A}(A\tens S) \rightarrow \bigoplus_{j,k=1,2}\s E_j\tens_{\delta_{A_j}^k}\!\!(A_k\tens S_{kj})\; ;\; \xi\tens_{\delta_A} x \mapsto \bigoplus_{j,k=1,2} q_{\s E,j}\xi \tens_{\delta_{A_j}^k} (q_{A,k}\tens p_{kj})x.\qedhere
\]
\end{proof}

\begin{propdef}\label{propdef3}
Let $\s V\in\Lin(\s E\tens_{\delta_A}(A\tens S),\s E\tens S)$ be the isometry associated with the action $(\beta_{\s E},\delta_{\s E})$ (cf.\ \ref{prop27} a). For all $j,k=1,2$, there exists a unique unitary\index[symbol]{ve@$\s V_j^k$}
\[
\s V_j^k\in\Lin(\s E_j\tens_{\delta_{A_j}^k}\!\!(A_k\tens S_{kj}),\s E_k\tens S_{kj})\\[-.5em]
\]
such that 
\[
\s V(\xi\tens_{\delta_A}x)=\sum_{j,k=1,2}\s V_j^k(q_{\s E,j}\xi\tens_{\delta_{A_j}^k}\!\!(q_{A,k}\tens p_{kj})x),\quad \text{for all } \xi\in\s E \text{ and } x\in A\tens S.\qedhere
\]
\end{propdef}

\begin{proof}
Let $j,k=1,2$. Fix $\xi\in\s E$, $x\in A\tens S$ and write $x=x'x''$ with $x',x''\in A\tens S$. We have
\begin{align*}
\s V(q_{\s E,j}\xi\tens_{\delta_A}\!(q_{A,k}\tens p_{kj})x) & =
(1_{\s E}\tens\beta(\varepsilon_j))\s V(\xi\tens_{\delta_A}\!(q_{A,k}\tens p_{kj})x) & \text{{\rm(}\ref{isometry} 2{\rm)}}\\
& =(1_{\s E}\tens\beta(\varepsilon_j))\s V(\xi\tens_{\delta_A}\!x'(q_{A,k}\tens p_{kj})x'') & (q_{A,k},\, p_{kj} \text{ are central).} \\
&= (1_{\s E}\tens\beta(\varepsilon_j))\s V(\xi\tens_{\delta_A}\!x')(q_{A,k}\tens p_{kj})x'' & (\s V \text{ is }A\tens S\text{-linear).}
\end{align*}
Now if $\eta\in\s E$, $y, s\in S$ and $a\in A$, we have 
\begin{align*}
(1_{\s E}\tens\beta(\varepsilon_j))(\eta\tens y)(q_{A,k}\tens p_{kj})(a\tens s)&=\eta(q_{A,k} a)\tens\beta(\varepsilon_j) y p_{kj} s \\
&= \eta(q_{A,k} a)\tens p_{kj} y s \\
&= q_{\s E,k}\eta a\tens p_{kj} y s \in \s E_k \tens S_{kj}
\end{align*}
by using (\ref{eq1.11}) and the fact that $p_{kj}$ is central in $S$. $\vphantom{\delta_{A_j}^k}$In particular, for all $\xi\in\s E$ and $x\in A\tens S$ we have $\s V(q_{\s E,j}\xi\tens_{\delta_A}(q_{A,k}\tens p_{kj})x)\in\s E_k\tens S_{kj}$. $\vphantom{\delta_{A_j}^k}$By combining the fact that $\s V$ is isometric with the fact that $x\delta_{A_j}^k(a)=x\delta_A(a)$ for all $a\in A_j$ and $x\in A_k\tens S_{kj}$, we obtain a well-defined isometric $A_k\tens S_{kj}$-linear map$\vphantom{\delta_{A_j}^k}$
\[
\s V_j^k:\s E_j\tens_{\delta_{A_j}^k}(A_k\tens S_{kj})\rightarrow \s E_k\tens S_{kj}\; ; \; \xi\tens_{\delta_{A_j}^k}x \mapsto \s V(\xi\tens_{\delta_A}x).
\]
It follows from ${\rm im}(\s V)=q_{\beta_{\s E}\alpha}(\s E\tens S)$ (\ref{isometry} 1) that $\s V_j^k$ is surjective. As a result, we have $\s V_j^k\in\Lin(\s E_j\tens_{\delta_{A_j}^k}(A_k\tens S_{kj}),\s E_k\tens S_{kj})$ and $\s V_j^k$ is unitary.
\end{proof}

For $j,k,l=1,2$ we have the following set of unitary equivalences of Hilbert modules:
\begin{align}
A_j\tens_{\delta_{A_j}^k}\!\!(A_k\tens S_{kj}) & \rightarrow A_k\tens S_{kj} \\[-.5em]
a\tens_{\delta_{A_j}^k}x & \mapsto \delta_{A_j}^k(a)x; \notag\\[.5em]
(A_k\tens S_{kj})\tens_{\delta_{A_k}^l\tens\,\id_{S_{kj}}}\!\!(A_l\tens S_{lk}\tens S_{kj}) &  \rightarrow A_l\tens S_{lk}\tens S_{kj} \label{eq24}\\[-.5em]
x\tens_{\delta_{A_k}^l\tens\,\id_{S_{kj}}}y & \mapsto (\delta_{A_k}^l\tens\id_{S_{kj}})(x)y; \notag\\[.5em]
(A_l\tens S_{lj})\tens_{\id_{A_l}\tens\,\delta_{lj}^k}\!\!(A_l\tens S_{lk}\tens S_{kj}) & \rightarrow A_l\tens S_{lk}\tens S_{kj} \label{eq9} \\[-.5em]
x\tens_{\id_{A_l}\tens\,\delta_{lj}^k}y & \mapsto (\id_{A_l}\tens\delta_{lj}^k)(x)y; \notag\\[.5em]
(\s E_j\tens_{\delta_{A_j}^k}\!\!(A_k\tens S_{kj}))\tens_{\delta_{A_k}^l\tens\,\id_{S_{kj}}}\!\!(A_l\tens S_{lk}\tens S_{kj}) & \rightarrow \s E_j\tens_{(\delta_{A_k}^l\tens\,\id_{S_{kj}})\delta_{A_j}^k}\!\!(A_l\tens S_{lk}\tens S_{kj}) \label{eq1.14}\\[-.5em]
(\xi\tens_{\delta_{A_j}^k}x)\tens_{\delta_{A_k}^l\tens\,\id_{S_{kj}}}y & \mapsto \xi\tens_{(\delta_{A_k}^l\tens\,\id_{S_{kj}})\delta_{A_j}^k}(\delta_{A_k}^l\tens\id_{S_{kj}})(x)y;\notag \\[.5em]
(\s E_j\tens_{\delta_{A_j}^l}\!\!(A_l\tens S_{lj}))\tens_{\id_{A_l}\tens\,\delta_{lj}^k}\!\!(A_l\tens S_{lk}\tens S_{kj}) & \rightarrow \s E_j\tens_{(\id_{A_l}\tens\,\delta_{lj}^k)\delta_{A_j}^l}\!\!(A_l\tens S_{lk}\tens S_{kj}) \label{eq1.15} \\[-.5em]
(\xi\tens_{\delta_{A_j}^l}x)\tens_{\id_{A_l}\tens\,\delta_{lj}^k}y & \mapsto \xi\tens_{(\id_{A_l}\tens\,\delta_{lj}^k)\delta_{A_j}^l}(\id_{A_l}\tens\delta_{lj}^k)(x)y; \notag \\[.5em]
(\s E_k\tens S_{kj})\tens_{\delta_{A_k}^l\tens\,\id_{S_{kj}}}\!\!(A_l\tens S_{lk}\tens S_{kj}) &\rightarrow (\s E_k\tens_{\delta_{A_k}^l}\!\!(A_l\tens S_{lk}))\tens S_{kj} \label{eq1.16}\\[-.5em]
(\xi\tens s)\tens_{\delta_{A_k}^l\tens\,\id_{S_{kj}}}(x\tens t) & \mapsto (\xi\tens_{\delta_{A_k}^l} x)\tens st; \notag \\[.5em]
(\s E_l\tens S_{lj})\tens_{\id_{A_l}\tens\,\delta_{lj}^k}\!\!(A_l\tens S_{lk}\tens S_{kj}) & \rightarrow \s E_l\tens S_{lk}\tens S_{kj} \label{eq10} \\[-.5em]
\xi\tens_{\id_{A_l}\tens\,\delta_{lj}^k}y & \mapsto (\id_{\s E_l}\tens\delta_{lj}^k)(\xi)y\notag.
\end{align}	

\begin{prop}\label{prop9}
For all $j,k,l=1,2$, we have
\[
(\s V_k^l\tens_{\GC}\id_{S_{kj}})(\s V_j^k\tens_{\delta_{A_k}^l\tens\id_{S_{kj}}}1)=\s V_j^l\tens_{\id_{A_l}\tens\delta_{lj}^k}1.\qedhere
\]
\end{prop}

For $j,k,l=1,2$, $\s V_k^l\tens_{\GC}\id_{S_{kj}}\in\Lin((\s E_k\tens S_{kj})\tens_{\delta_{A_k}^l\tens\,\id_{S_{kj}}}(A_l\tens S_{lk}\tens S_{kj}),\s E_l\tens S_{lk}\tens S_{kj})$ (\ref{eq1.16}), $\s V_j^k\tens_{\delta_{A_k}^l\tens\,\id_{S_{kj}}}1\in\Lin(\s E_j\tens_{(\delta_{A_k}^l\tens\,\id_{S_{kj}})\delta_{A_j}^k}(A_l\tens S_{lk}\tens S_{kj}),(\s E_k\tens S_{kj})\tens_{\delta_{A_k}^l\tens\,\id_{S_{kj}}}(A_l\tens S_{lk}\tens S_{kj}))$ (\ref{eq1.14}) and $\s V_j^l\tens_{\id_{A_l}\tens\,\delta_{lj}^k}1\in\Lin(\s E_j\tens_{(\id_{A_l}\tens\,\delta_{lj}^k)\delta_{A_j}^l}(A_l\tens S_{lk}\tens S_{kj}),\s E_l\tens S_{lk}\tens S_{kj})$ (\ref{eq10}).
Moreover, the composition $(\s V_k^l\tens_{\GC}\id_{S_{kj}})(\s V_j^k\tens_{\delta_{A_k}^l\tens\id_{S_{kj}}}1)$ does make sense since $(\delta_{A_k}^l\tens\id_{S_{kj}})\delta_{A_j}^k=(\id_{A_l}\tens\delta_{lj}^k)\delta_{A_j}^l$.

\begin{proof}
Straightforward consequence of $(\s V\tens_{\GC}\id_{S})(\s V\tens_{\delta_{A}\tens\id_{S}}1)=\s V\tens_{\id_{A}\tens\,\delta}1$.
\end{proof}

\begin{propdef}\label{propdef4}
For $j,k=1,2$, let 
$
\delta_{\s E_j}^k:\s E_j\rightarrow\Lin(A_k\tens S_{kj},\s E_k\tens S_{kj})
$
be the linear map defined by\index[symbol]{df@$\delta_{\s E_j}^k$}
\[
\delta_{\s E_j}^k(\xi)x:=\s V_j^k(\xi\tens_{\delta_{A_j}^k}x), \quad \text{for all } \xi\in\s E_j \text{ and } x\in A_k\tens S_{kj}.
\]
For all $j,\, k,\, l=1,2$, we have:
\begin{enumerate}[label=(\roman*)]
\item $\displaystyle{\delta_{\s E}(\xi)=\sum_{k,j=1,2} \Pi_j^k\circ\delta_{\s E_j}^k(q_{\s E,j}\xi)}$, for all $\xi\in\s E$;
\item $\delta_{\s E_j}^k(\s E_j)\subset\widetilde{\M}(\s E_k\tens S_{kj})$;$\vphantom{\displaystyle{\sum_{k,j=1,2}}}$ 
\item $\delta_{\s E_j}^k(\xi a)=\delta_{\s E_j}^k(\xi)\delta_{A_j}^k(a)$ and $\langle\delta_{\s E_j}^k(\xi),\, \delta_{\s E_j}^k(\eta)\rangle = \delta_{A_j}^k(\langle\xi,\, \eta\rangle)$, for all $\xi,\,\eta\in\s E_j$ and $a\in A_j$;$\vphantom{\displaystyle{\sum_{k,j=1,2}}}$
\item $[\delta_{\s E_j}^k(\s E_j)(1_{A_k}\tens S_{kj})]=\s E_k\tens S_{kj}$; in particular, we have
\[
\s E_k=[(\id_{\s E_k}\tens\omega)\delta_{\s E_j}^k(\xi) \; ; \; \omega\in\B(\s H_{kj})_*,\, \xi\in\s E_j] \quad (\text{cf.}\ \ref{not4}).
\]
\item $\delta_{\s E_k}^l\tens\id_{S_{kj}}$ {\rm(}resp.\ $\id_{\s E_l}\tens\delta_{lj}^k${\rm)} extends to a linear map from $\mathcal{L}(A_k\tens S_{kj},\s E_k\tens S_{kj})$ {\rm(}resp. $\Lin(A_l\tens S_{lj},\s E_l\tens S_{lj})${\rm)} to $\mathcal{L}(A_l\tens S_{lk}\tens S_{kj},\s E_l\tens S_{lk}\tens S_{kj})$ and for all $\xi\in\s E_j$ we have 
\[
(\delta_{\s E_k}^l\tens\id_{S_{kj}})\delta_{\s E_j}^k(\xi)=(\id_{\s E_l}\tens\delta_{lj}^k)\delta_{\s E_j}^l(\xi) \in \Lin(A_l\tens S_{lk}\tens S_{kj},\s E_l\tens S_{lk}\tens S_{kj});
\]
\item if $\s E$ is a $\cal G$-equivariant Hilbert $A$-module, then we have $[(1_{\s E_k}\tens S_{kj})\delta_{\s E_j}^k(\s E_j)]=\s E_k\tens S_{kj}$.
\end{enumerate}
If $\s E$ is a $\cal G$-equivariant Hilbert module, then $(\s E_j,\delta_{\s E_j}^j)$ is a $\QG_j$-equivariant Hilbert $A_j$-module and $\s V_j^j$ is the associated unitary.
\end{propdef}

\begin{proof}
It is clear that $\delta_{\s E_j}^k:\s E_j\rightarrow\Lin(A_k\tens S_{kj},\s E_k\tens S_{kj})$ is a well-defined linear map. Moreover, statement (i) follows straightforwardly from \ref{propdef3} and the fact that $\delta_{\s E}(\xi)x=\s V(\xi\tens_{\delta_A}x)$ for all $\xi\in\s E$ and $x\in A\tens S$.
Let $\xi\in\s E_j$ and $s\in S_{kj}$. We have
\[
\Pi_j^k(\delta_{\s E_j}^k(\xi)(1_{A_k}\tens s))=\delta_{\s E}(\xi)(1_A\tens s) \quad \text{and} \quad
\Pi_j^k((1_{\s E_k}\tens s)\delta_{\s E_j}^k(\xi))=(1_{\s E}\tens s)\delta_{\s E}(\xi)
\]
It then follows from $\delta_{\s E}(\s E)\subset\widetilde{\M}(\s E\tens S)$ that $\Pi_j^k(\delta_{\s E_j}^k(\xi)(1_{A_k}\tens s))$ and $\Pi_j^k((1_{\s E_k}\tens s)\delta_{\s E_j}^k(\xi))$ belong to $\s E\tens S$. Moreover, $(q_{\s E_,k}\tens p_{kj})\Pi_j^k(T)=\Pi_j^k(T)$ for all $T\in\Lin(A_k\tens S_{kj},\s E_k\tens S_{kj})$; hence, $\Pi_j^k(\delta_{\s E_j}^k(\xi)(1_{A_k}\tens s))\in\Pi_j^k(\s E_k\tens S_{kj})$ and $\Pi_j^k((1_{\s E_k}\tens s)\delta_{\s E_j}^k(\xi))\in\Pi_j^k(\s E_k\tens S_{kj})$. It then follows that $\delta_{\s E_j}^k(\xi)(1_{A_k}\tens s)$ and $(1_{\s E_k}\tens s)\delta_{\s E_j}^k(\xi)$ belong to $\s E_k\tens S_{kj}$ by injectivity of $\Pi_j^k:\Lin(A_k\tens S_{kj},\s E_k\tens S_{kj})\rightarrow\Lin(A\tens S,\s E\tens S)$. Hence, statement (ii) is proved.\newline
Let $\xi,\eta\in\s E_j$. We have
\begin{align*}
\pi_j^k(\langle\delta_{\s E_j}^k(\xi),\, \delta_{\s E_j}^k(\eta)\rangle)&=\langle\Pi_j^k(\delta_{\s E_j}^k(\xi)),\, \Pi_j^k(\delta_{\s E_j}^k(\eta))\rangle\\
&=\langle(q_{\s E,k}\tens p_{kj})\delta_{\s E}(\xi),\, (q_{\s E,k}\tens p_{kj})\delta_{\s E}(\eta)\rangle\\
&=(q_{A,k}\tens p_{kj})\delta_A(\langle\xi,\,\eta\rangle)\\
&=\pi_j^k(\delta_{A_j}^k(\langle\xi,\,\eta\rangle)).
\end{align*}
Hence $\langle\delta_{\s E_j}^k(\xi),\, \delta_{\s E_j}^k(\eta)\rangle=\delta_{A_j}^k(\langle\xi,\,\eta\rangle)$ by injectivity of $\pi_j^k$. The first formula of statement (iii) is derived immediately from the definition of $\delta_{\s E_j}^k$.\newline
The surjectivity of $\s V_j^k$ is just a restatement of $[\delta_{\s E_j}^k(\s E_j)(A_k\tens S_{kj})]=\s E_k\tens S_{kj}$. The identity $[\delta_{\s E_j}^k(\s E_j)(1_{A_k}\tens S_{kj})]=\s E_k\tens S_{kj}$ follows by combining the previous formula with the first relation of (iii) and the relation $[\delta_{A_j}^k(A_j)(1_{A_k}\tens S_{kj})]=A_k\tens S_{kj}$. Let us prove the formula $\s E_k=[(\id_{\s E_k}\tens\omega)\delta_{\s E_j}^k(\xi) \; ; \; \omega\in\B(\s H_{kj})_*,\, \xi\in\s E_j]$. By statement (ii) and \ref{not4}, we already have the relation $\s E_k\supset[(\id_{\s E_k}\tens\omega)\delta_{\s E_j}^k(\xi) \; ; \; \omega\in\B(\s H_{kj})_*,\, \xi\in\s E_j]$. Conversely, let us fix $\eta\in\s E_k$. Let $\omega\in\B(\s H_{kj})_*$ and $s\in S_{kj}$ such that $\omega(s)=1$. $\vphantom{\delta_{\s E_j}^k}$It then follows from the formula $[\delta_{\s E_j}^k(\s E_j)(1_{A_k}\tens S_{kj})]=\s E_k\tens S_{kj}$ that $\eta=(\id_{\s E_k}\tens\omega)(\eta\tens s)$ is the norm limit of finite sums of elements of the form $(\id_{\s E_k}\tens\omega)(\delta_{\s E_j}^k(\xi)(1_{A_k}\tens y))=(\id_{\s E_k}\tens y\omega)\delta_{\s E_j}^k(\xi)$, where $\xi\in\s E_j$ and $y\in S$. Therefore, statement (iv) is proved.\newline
By using the identifications (\ref{eq24}) and (\ref{eq1.16}) (resp.\ (\ref{eq9}) and (\ref{eq10})), the linear map 
\begin{align*}
\delta_{\s E_k}^l\tens\id_{S_{kj}} & :\mathcal{L}(A_k\tens S_{kj},\s E_k\tens S_{kj})\rightarrow\mathcal{L}(A_l\tens S_{lk}\tens S_{kj},\s E_l\tens S_{lk}\tens S_{kj})\\ 
({\rm resp}.\ \id_{\s E_l}\tens\delta_{lj}^k & :\Lin(A_l\tens S_{lj},\s E_l\tens S_{lj})\rightarrow\mathcal{L}(A_l\tens S_{lk}\tens S_{kj},\s E_l\tens S_{lk}\tens S_{kj}))
\end{align*} 
is defined for all $T\in\mathcal{L}(A_k\tens S_{kj},\s E_k\tens S_{kj})$ (resp.\ $T\in\Lin(A_l\tens S_{lj},\s E_l\tens S_{lj})$) by
\[
(\delta_{\s E_k}^l\tens\id_{S_{kj}})(T)  :=(\s V_k^j\tens 1_{S_{kj}})(T\tens_{\delta_{A_k}^l\tens\,\id_{S_{kj}}}1) \quad
({\rm resp}.\ (\id_{\s E_l}\tens\delta_{lj}^k)(T) :=T\tens_{\id_{A_l}\tens\,\delta_{lj}^k}1).
\]
The relation $(\delta_{\s E_k}^l\tens\id_{S_{kj}})\delta_{\s E_j}^k(\xi)=(\id_{\s E_l}\tens\delta_{lj}^k)\delta_{\s E_j}^l(\xi)$ for $\xi\in\s E_j$ is then derived from \ref{prop9} as in the proof of \ref{prop27}.
Assume that $(\beta_{\s E},\delta_{\s E})$ is continuous. Since $p_{kj}$ is central in $S$, we have
\begin{align*}
\Pi_j^k(\s E_k\tens S_{kj})&=(q_{\s E,j}\tens p_{kj})(\s E\tens S)q_{\beta_A\alpha}\\
&=(q_{\s E,j}\tens p_{kj})[(1_{\s E}\tens S)\delta_{\s E}(\s E)]\\
&=\Pi_j^k[(1_{\s E_k}\tens S_{kj})\delta_{\s E_j}^k(\s E_j)]
\end{align*}
and statement (vi) is proved.
\end{proof}

From this concrete description of $\cal G$-equivariant Hilbert C*-modules, we can also provide a corresponding description of the $\cal G$-equivariant unitary equivalences between them.

\begin{lem}\label{lem5}
Let $A$ and $B$ be $\cal G$-C*-algebras. Let $\s E$ and $\s F$ be Hilbert C*-modules over $A$ and $B$ respectively acted upon by $\cal G$.
\begin{enumerate}
\item Let $\Phi:\s E\rightarrow\s F$ be a $\cal G$-equivariant unitary equivalence over a $\cal G$-equivariant *-isomorphism $\phi:A\rightarrow B$. For $j=1,2$, there exists a unique map $\Phi_j:\s E_j\rightarrow\s F_j$ satisfying the formula $\Phi(\xi)=\Phi_1(q_{\s E,1}\xi)+\Phi_2(q_{\s E,2}\xi)$ for all $\xi\in\s E$. Moreover, we have:
\begin{enumerate}[label=(\roman*)]
\item for $j=1,\,2$, the map $\Phi_j$ is a unitary equivalence over the *-isomorphism $\phi_j:A_j\rightarrow B_j$ (cf.\ \ref{lem7} 1);
\item for all $j,k=1,2$, we have
\begin{equation}\label{eq14}
(\Phi_k\tens\id_{S_{kj}})\circ\delta_{\s E_j}^k=\delta_{\s F_j}^k\circ\Phi_j.
\end{equation}
\end{enumerate}
In particular, $\Phi_j$ is a $\QG_j$-equivariant $\phi_j$-compatible unitary operator.
\item Conversely, for $j=1,\, 2$ let $\Phi_j:\s E_j\rightarrow\s F_j$ be a $\QG_j$-equivariant unitary equivalence over a $\QG_j$-equivariant *-isomorphism $\phi_j:A_j\rightarrow B_j$ such that (\ref{eqmorpheq}) and (\ref{eq14}) hold for all $j,k=1,2$. Then, the map $\Phi:\s E\rightarrow\s F$, defined  by $\Phi(\xi):=\Phi_1(q_{\s E,1}\xi)+\Phi_2(q_{\s E,2}\xi)$ for all $\xi\in\s E$, is a $\cal G$-equivariant unitary equivalence over the $\cal G$-equivariant *-isomorphism $\phi:A\rightarrow B$ (cf.\ \ref{lem7} 2).\qedhere
\end{enumerate}
\end{lem}

\begin{proof}
1. Let $j=1,2$. Since $\Phi$ is $\cal G$-equivariant, we have $\Phi\circ q_{\s E,j}=q_{\s F,j}\circ\Phi$. It then follows that $\Phi(\s E_j)\subset\s F_j$. Let us denote $\Phi_j:=\restr{\Phi}{\s E_j}:\s E_j\rightarrow\s F_j$. For $\xi\in\s E$, we have $\xi=q_{\s E,1}\xi+q_{\s E,2}\xi$; hence, $\Phi(\xi)=\Phi_1(q_{\s E,1}\xi)+\Phi_2(q_{\s E,2}\xi)$. Moreover, such a decomposition of $\Phi$ is unique since $\s F_1$ and $\s F_2$ are orthogonal in $\s F$. Statement (i) is straightforward. Let $j,k=1,2$ and $x\in A_k\tens S_{kj}$. For all $T\in\Lin(A_k\tens S_{kj},\s E_k\tens S_{kj})$ we have 
\[
(\Phi\tens\id_S)(\Pi_j^k(T))(\phi\tens\id_S)(x)=(\Phi_k\tens\id_{S_{kj}})(T)(\phi_k\tens\id_{S_{kj}})(x).\]
In particular, $(\Phi\tens\id_S)(\Pi_j^k(\delta_{\s E_j}^k(\xi)))(\phi\tens\id_S)(x)=(\Phi_k\tens\id_{S_{kj}})(\delta_{\s E_j}^k(\xi))(\phi_k\tens\id_{S_{kj}})(x)$ for all $\xi\in\s E_j$; hence, $(\Phi\tens\id_S)(\delta_{\s E}(\xi))(\phi\tens\id_S)(x)=(\Phi_k\tens\id_{S_{kj}})(\delta_{\s E_j}^k(\xi))(\phi_k\tens\id_{S_{kj}})(x)$ (\ref{propdef4} (i)) for all $\xi\in\s E_j$. We also have $\delta_{\s F}(\Phi(\xi))(q_{A,k}\tens p_{kj})=\Pi_j^k(\delta_{\s F_j}^k(\Phi_j(\xi)))$ for all $\xi\in\s E_j$. Hence, $\delta_{\s F}(\Phi(\xi))(\phi\tens\id_S)(x)=\delta_{\s F_j}^k(\Phi_j(\xi))(\phi_k\tens\id_{S_{kj}})(x)$ for all $\xi\in\s E_j$ and statement (ii) is proved.\newline
2. Straightforward.
\end{proof}

\begin{ex}
Let $(\beta_N,\delta_N)$ be the trivial action (cf.\ \ref{ex2}). Let $i=1,2$. Consider the Hilbert $N$-module $\s E:=\s H_{i1}\oplus\s H_{i2}$. Let $\s V\in\Lin(\s E\tens_{\delta_N}(N\tens S),\s E\tens S)$ and $\beta_{\s E}:N\rightarrow\Lin(\s E)$ be the maps defined by the formulas:
\[
\s V(\xi\tens 1)=\sum_{k=1,2}V_{kj}^i(\xi\tens 1),\quad \xi\in\s H_{ij}\quad ;\quad \beta_{\s E}(\varepsilon_j)=p_{ij},\quad j=1,2.
\]
Then, the pair $(\s V,\beta_{\s E})$ is an action of $\cal G$ on $\s E$.
\end{ex}

	\subsection{Induction of equivariant Hilbert C*-modules}

Let $\QG_1$ and $\QG_2$ be two monoidally equivalent regular locally compact quantum groups.

\medskip
	
Fix a $\QG_1$-C*-algebra $(A_1,\delta_{A_1})$ and a $\QG_1$-equivariant Hilbert $A_1$-module $(\s E_1,\delta_{\s E_1})$. We denote by $J_1:=\K(\s E_1\oplus A_1)$ the associated linking C*-algebra endowed with the continuous action $\delta_{J_1}$ of $\QG_1$.

\begin{nbs} Let us fix some notations.
\begin{itemize}
\item Let $\id_{\s E_1}\tens\delta_{11}^2:\Lin(A_1\tens S_{11},\s E_1\tens S_{11})\rightarrow\Lin(A_1\tens S_{12}\tens S_{21},\s E_1\tens S_{12}\tens S_{21})$ be the unique linear extension of $\id_{\s E_1}\tens\delta_{11}^2:\s E_1\tens S_{11}\rightarrow\Lin(A_1\tens S_{12}\tens S_{21},\s E_1\tens S_{12}\tens S_{21})$ such that
$
(\id_{\s E_1}\tens\delta_{11}^2)(T)(\id_{A_1}\tens\delta_{11}^2)(x)=(\id_{\s E_1}\tens\delta_{11}^2)(Tx)
$
for all $x\in\M(A_1\tens S_{11})$ and $T\in\Lin(A_1\tens S_{11},\s E_1\tens S_{11})$.
\item Let $\delta_{\s E_1}^{(2)}:\s E_1\rightarrow\Lin(A_1\tens S_{12}\tens S_{21},\s E_1\tens S_{12}\tens S_{21})$ be the linear map defined by
$
\delta_{\s E_1}^{(2)}(\xi):=(\id_{\s E_1}\tens\delta_{11}^2)\delta_{\s E_1}(\xi)
$
for all $\xi\in\s E_1$.\index[symbol]{dg@$\delta_{\s E_1}^{(2)}$}
\item Consider the Banach subspace of $\Lin(A_1\tens S_{12},\s E_1\tens S_{12})$ defined by (cf.\ \ref{not4}):\index[symbol]{id@$\ind(\s E_1)$, induced Hilbert module} 
\[
\ind({\s E}_1):=\![(\id_{\s E_1\tens S_{12}}\tens\omega)\delta_{\s E_1}^{(2)}(\xi)\,;\,\xi\in\s E_1,\omega\in\B(\s H_{21})_*].\qedhere
\]
\end{itemize}
\end{nbs}

\begin{prop}\label{prop2} We have
$
[\ind({\s E}_1)(1_{A_1}\tens S_{12})]=\s E_1\tens S_{12}=[(1_{\s E_1}\tens S_{11})\ind({\s E}_1)].
$
In particular, $\ind({\s E}_1)\subset\widetilde{\M}(\s E_1\tens S_{12})$.
\end{prop} 

\begin{proof}
Let us prove the formula $[\ind({\s E}_1)(1_{A_1}\tens S_{12})]=\s E_1\tens S_{12}$. Fix $\xi\in\s E_1$, $s\in S_{12}$ and $\omega\in\B(\s H_{21})_*$. Write $\omega=s'\omega'$ with $s'\in S_{21}$ and $\omega'\in\B(\s H_{21})_*$. It follows from 
$S_{12}\tens S_{21}=[\delta_{11}^2(S_{11})(1_{S_{12}}\tens S_{21})]$ 
that 
\[
(\id_{\s E_1\tens S_{12}}\tens\omega)(\delta_{\s E_1}^{(2)}(\xi))(1_{A_1}\tens s)
=(\id_{\s E_1\tens S_{12}}\tens\omega')(\delta_{\s E_1}^{(2)}(\xi)(1_{\s E_1}\tens s\tens s'))
\] 
is the norm limit of finite sums of elements of the form
\[
\eta=(\id_{\s E_1}\tens\id_{S_{12}}\tens\omega')(\delta_{\s E_1}^{(2)}(\xi)(1_{\s E_1}\tens\delta_{11}^2(t')(1_{S_{12}}\tens t))), \quad \text{with } t'\in S_{11} \; \text{and} \; t\in S_{21}.
\]
It follows from $[\delta_{\s E_1}(\s E_1)(1_{A_1}\tens S_{11})]=\s E_1\tens S_{11}$ that
\[
\eta=(\id_{\s E_1\tens S_{12}}\tens\omega')((\id_{\s E_1}\tens\delta_{11}^2)(\delta_{\s E_1}(\xi)(1_{\s E_1}\tens t'))(1_{\s E_1\tens S_{12}}\tens t))
\] 
is the norm limit of finite sums of elements of the form
\[
\eta'=(\id_{\s E_1\tens S_{12}}\tens\omega')(\zeta\tens \delta_{11}^2(t'')(1_{S_{12}}\tens t)), \quad
\text{with } \zeta\in\s E_1 \; \text{and} \; t''\in S_{21}.
\]
By using 
$S_{12}=[(\id_{S_{12}}\tens\omega)(\delta_{11}^2(y))\,;\, \omega\in\B(\s H_{21})_*,\, y\in S_{11}]$,
we obtain 
\[
\eta'=\zeta\tens (\id_{S_{12}}\tens t\omega')(\delta_{11}^2(t''))\in \s E_1\tens S_{12}.
\] 
Hence, $(\id_{\s E_1\tens S_{12}}\tens\omega)(\delta_{\s E_1}^{(2)}(\xi))(1_{A_1}\tens s)\in\s E_1\tens S_{12}$ for all $\xi\in\s E_1$, $\omega\in\B(\s H_{21})_*$ and $s\in S_{12}$. Therefore, the inclusion \begin{center}
$[\ind({\s E}_1)(1_{A_1}\tens S_{12})] \subset \s E_1\tens S_{12}$
\end{center}
is proved. The converse inclusion is obtained by following backwards the above argument.\newline
By a similar argument, we prove by using the relation $(1_{\s E_1}\tens S_{12})\delta_{\s E_1}(\s E_1) \subset \s E_1\tens S_{12}$ that $(1_{\s E_1}\tens S_{12})\ind({\s E}_1)\subset \s E_1\tens S_{12}$. Hence, $[(1_{\s E_1}\tens S_{12})\ind({\s E}_1)]\subset \s E_1\tens S_{12}$. For the converse inclusion, it suffices to follow backwards the proof as above and to use the continuity of the action $\delta_{\s E_1}$.
\end{proof}

\begin{lem}\label{lem22} For all $a\in A_1$, $\xi\in\s E_1$, $k\in\K(\s E_1)$ and $\omega\in\B(\s H_{21})_*$, we have:
\begin{enumerate}
\item $\iota_{A_1\tens S_{11}}(\id_{A_1\tens S_{12}}\tens\omega)\delta_{A_1}^{(2)}(a)
=(\id_{J_1\tens S_{12}}\tens\omega)\delta_{J_1}^{(2)}(\iota_{A_1}(a))$;
\item $\iota_{\s E_1\tens S_{12}}(\id_{\s E_1\tens S_{12}}\tens\omega)\delta_{\s E_1}^{(2)}(\xi)
=(\id_{J_1\tens S_{12}}\tens\omega)\delta_{J_1}^{(2)}(\iota_{\s E_1}(\xi))$;
\item $\iota_{\K(\s E_1\tens S_{12})}(\id_{\K(\s E_1)\tens S_{12}}\tens\omega)\delta_{\K(\s E_1)}^{(2)}(k)=(\id_{J_1\tens S_{12}}\tens\omega)\delta_{J_1}^{(2)}(\iota_{\K(\s E_1)}(k))$.\qedhere
\end{enumerate}
\end{lem}

\begin{proof} These formulas are straightforward consequences of definitions and the compatibility of $\delta_{J_1}$ with $\delta_{A_1}$ and $\delta_{\s E_1}$ and $\delta_{\K(\s E_1)}$ (2.7 (b), 2.8 (a) \cite{BS1}, \ref{not4}).
\end{proof}

\begin{prop}\label{prop3}
Let $\ind(A_1)$ be the induced C*-algebra. Then $\ind(\s E_1)$ is a Hilbert $\ind(A_1)$-module for the right action by composition and the $\ind(A_1)$-valued inner product given by $\langle\xi,\, \eta\rangle:=\xi^*\circ\eta$ for $\xi,\eta\in\ind(\s E_1)$.
\end{prop}

\begin{proof}
Let $\omega,\omega'\in\B(\s H_{21})_*$, $a\in A_1$ and $\xi\in\s E_1$. Let $\eta:=(\id_{\s E_1\tens S_{12}}\tens\omega)\delta_{\s E_1}^{(2)}(\xi)$ and $x:=(\id_{A_1 \tens S_{12}}\tens\omega')\delta_{A_1}^{(2)}(a)$.
We have
\begin{align*}
\iota_{\s E_1\tens S_{12}}(\eta x) & = \iota_{\s E_1\tens S_{12}}(\eta)\iota_{A_1\tens S_{12}}(a) &\text{(\ref{prop32} 1)}\\
& = (\id_{J_1\tens S_{12}}\tens\omega)\delta_{J_1}^{(2)}(\iota_{\s E_1}(\xi))(\id_{J_1\tens S_{12}}\tens\omega')\delta_{J_1}^{(2)}(\iota_A(a)) &\text{(\ref{lem22})}\\
& = (\id_{J_1\tens S_{12}}\tens\omega\tens\omega')(\delta_{J_1}^{(2)}(\iota_{\s E_1}(\xi))_{123}\delta_{J_1}^{(2)}(\iota_A(a))_{124}).
\end{align*}
In virtue of 4.2 a) \cite{BC}, we have $\iota_{\s E_1\tens S_{12}}(\eta x)\in\ind(J_1)$. Therefore, $\iota_{\s E_1\tens S_{12}}(\eta x)$ is the norm limit of finite sums of elements of the form 
\[
y=(\id_{J_1\tens S_{12}}\tens\phi)\delta_{J_1}^{(2)}\begin{pmatrix} k & \zeta \\ \chi^* & b\end{pmatrix}\!,\; \text{with}\; k\in\K(\s E_1),\; \zeta\in\s E_1,\; \chi^*\in\s E_1^*,\; b\in A_1,\; \phi\in\B(\s H_{21})_*.
\]
We have (\ref{lem22})
\begin{multline*}
y  = \iota_{\K(\s E_1\tens S_{12})}(\id_{\K(\s E_1)\tens S_{12}}\tens\phi)\delta_{\K(\s E_1)}^{(2)}(k)+ \iota_{\s E_1\tens S_{12}}(\id_{\s E_1\tens S_{12}} \tens \phi)\delta_{\s E_1}^{(2)}(\zeta) \\
 +\iota_{\s E_1\tens S_{12}}(\id_{\s E_1\tens S_{12}} \tens \phi)\delta_{\s E_1}^{(2)}(\chi)^* + \iota_{A_1\tens S_{12}}(\id_{A_1\tens S_{12}}\tens\phi)\delta_{A_1}^{(2)}(b).
\end{multline*}
By multiplying on the left (resp.\ right) by $\iota_{\K(\s E_1\tens S_{12})}(1_{\s E_1\tens S_{12}})$ (resp.\ $\iota_{A_1\tens S_{12}}(1_{A_1\tens S_{12}})$), we obtain (\ref{ehmlem1}) that $\iota_{\s E_1\tens S_{12}}(\eta x)$ is the norm limit of finite sums of elements of the form 
\begin{center}
$\iota_{\s E_1\tens S_{12}}(\id_{\s E_1}\tens\id_{S_{12}}\tens\phi)\delta_{\s E_1}^{(2)}(\zeta)$,\quad with $\zeta\in\s E_1$ and $\phi\in\B(\s H_{21})_*$. 
\end{center}
Since $\iota_{\s E_1\tens S_{12}}$ is isometric, we  have proved that $\eta x\in\ind(\s E_1)$.

\medskip

Let us prove that $\zeta^*\circ\chi\in\ind(A_1)$ for all $\zeta,\chi\in\ind(\s E_1)\subset\Lin(A_1\tens S_{12},\s E_1\tens S_{12})$. Let us fix $\xi,\eta\in\s E_1$ and $\omega,\psi\in\B(\s H_{21})_*$. Let us denote
$
\zeta:=(\id_{\s E_1\tens S_{12}}\tens\omega)\delta_{\s E_1}^{(2)}(\xi)
$ 
and
$
\chi:=(\id_{\s E_1\tens S_{12}}\tens\psi)\delta_{\s E_1}^{(2)}(\eta).
$
We have
\begin{align*}
\iota_{A_1\tens S_{12}}(\zeta^*\circ\chi) & =\iota_{\s E_1\tens S_{12}}(\zeta)^*\iota_{\s E_1\tens S_{12}}(\chi) &\text{(\ref{prop32} 3)}\\
& = (\id_{J_1\tens S_{12}}\tens\omega)\delta_{J_1}^{(2)}(\iota_{\s E_1}(\xi)^*)(\id_{J_1\tens S_{12}}\tens\psi)\delta_{J_1}^{(2)}(\iota_{\s E_1}(\eta)) & \text{(\ref{lem22})}\\
& = (\id_{J_1\tens S_{12}}\tens\omega\tens\psi)(\delta_{J_1}^{(2)}(\iota_{\s E_1}(\xi)^*)_{123}\delta_{J_1}^{(2)}(\iota_{\s E_1}(\eta))_{124}).
\end{align*}
Hence, $\iota_{A_1\tens S_{12}}(\zeta^*\circ\chi)\in\ind(J_1)$ (4.2 a) \cite{BC}). As above, we prove that $\iota_{A_1\tens S_{12}}(\zeta^*\circ\chi)$ is the norm limit of finite sums of elements of the form $\iota_{A_1\tens S_{12}}(\id_{A_1\tens S_{12}}\tens\phi)\delta_{A_1}^{(2)}(a)$ with $a\in A_1$ and $\phi\in\B(\s H_{21})_*$. We have proved that $\zeta^*\circ\chi\in\ind(A_1)$ since $\iota_{A_1\tens S_{12}}$ is isometric.
\end{proof}

Let us denote $(A_2,\delta_{A_2}):=\ind(A_1,\delta_{A_1})$ and $(J_2,\delta_{J_2}):=\ind(J_1,\delta_{J_1})$ the induced $\QG_2$-C*-algebra of $(A_1,\delta_{A_1})$ and $(J_1,\delta_{J_1})$ respectively. We also denote $\s E_2:=\ind(\s E_1)$ the induced Hilbert $A_2$-module as defined above.$\vphantom{\ind}$

\medskip

In the technical lemma below, we make the identification $\M(A)=\Lin(A)$. We first recall a well-known corollary of Cohen-Hewitt factorization theorem.

\begin{lem}\label{lem3}
Let $A$ be a C*-algebra and $\cal E$ a Hilbert $A$-module. If $T:A\rightarrow{\cal E}$ is a map such that $T(ab)=T(a)b$ for all $a,b\in A$, then $T$ is linear and continuous.  
\end{lem}

\begin{lem}\label{lem1}
Let $A$ be a C*-algebra, $B\subset\M(A)$ a non-degenerate C*-subalgebra and ${\cal E}$ a Hilbert $A$-module. Let ${\cal F}\subset\Lin(A,{\cal E})$ be a Hilbert $B$-module (where $B$ is acting on the right by composition and the $B$-valued inner product is given by $\langle\eta_1,\eta_2\rangle:=\eta_1^*\circ\eta_2$, for all $\eta_1,\eta_2\in{\cal F}$) such that $[{\cal F} A]={\cal E}$. 
\begin{enumerate}[label=(\roman*)]
\item There exists a unique map 
$
i:\Lin(B,{\cal F})\rightarrow\Lin(A,{\cal E})
$
such that 
$
i(T)(ba)=(Tb)a
$
for all $T\in\Lin(B,{\cal F})$, $b\in B$ and $a\in A$.
Moreover, $i$ is an injective linear map whose image is
$
{\rm im}(i)=\{S\in\Lin(A,{\cal E})\,;\,SB\subset{\cal F},\,S^*{\cal F}\subset B\}.
$
\item There exists a unique map 
$
j:\Lin({\cal F}\oplus B)\rightarrow\Lin({\cal E}\oplus A)
$
such that
$
j(x)(\eta a)=(x\eta)a
$
for all $x\in\Lin({\cal F}\oplus B)$, $\eta\in{\cal F}\oplus B$ and $a\in A$.
Moreover, $j$ is a unital faithful *-homomorphism.\qedhere
\end{enumerate}
\end{lem}

\begin{proof}
(i) We have $A=BA$. Let $T\in\Lin(B,{\cal F})$. Let $(u_{\lambda})$ be an approximate unit of $B$, we have
$
(Tb)a=\lim_{\lambda}\,[T(u_{\lambda}b)]a=\lim_{\lambda}\,[T(u_{\lambda})b]a=\lim_{\lambda}\,T(u_{\lambda})ba,
$
for all $b\in B$ and $a\in A$.
In particular, we have $(Tb)a=(Tb')a'$ for all $b,b'\in B$ and $a,a'\in A$ such that $ba=b'a'$. Therefore, $i(T)$ is well defined. Moreover, we have $i(T)(aa')=(i(T)a)a'$ for all $a,a'\in A$. Indeed, let us fix $a,a'\in A$. Let us write $a=ba''$ with $b\in B$ and $a''\in A$. We have $i(T)(aa')=i(T)(b(a''a'))=(Tb)a''a'=i(T)(ba'')a'=(i(T)a)a'$. By Lemma \ref{lem3}, it then follows that $i(T)$ is a bounded linear map. By a straightforward computation, we have
$
\langle i(T)(ba'),\, \eta a\rangle=\langle ba',\, T^*(\eta)a\rangle,
$
for all $b\in B$, $a,\,a'\in A$ and $\eta\in{\cal F}$. Hence, $\langle i(T)x,\, \eta a\rangle = \langle x,\, T^*(\eta)a\rangle$ for all $x,a\in A$ and $\eta\in{\cal F}$. Let $S\in\Lin({\cal F},B)$. We have
\[
\langle\, x,\, \sum_{l=1}^n S(\eta_l)a_l \,\rangle = \langle\,  i(S^*)x,\, \sum_{l=1}^n\eta_l a_l\, \rangle, \quad \text{for all } \, a_1,\cdots a_n\in A,\, \eta_1,\cdots, \eta_n\in{\cal F} \; \text{ and } \; x\in A.
\]
As a consequence, the following map
\[
i'(S):\langle{\cal F} A\rangle \rightarrow A \; ; \; \sum_{l=1}^n\eta_l a_l \mapsto \sum_{l=1}^n S(\eta_l)a_l
\]
is well-defined and we have $\langle x,\, i'(S)(\xi)\rangle=\langle i(S^*)x,\, \xi\rangle$ for all $\xi\in\langle{\cal F} A\rangle$ and $x\in A$. It follows from the boundedness of the linear operator $i(S^*)$ and the Cauchy-Schwarz inequality that
$
\|i'(S)\xi\|^2=\|\langle i'(S)\xi,\, i'(S)\xi\rangle\|=\|\langle\xi,\, i(S^*)(i'(S)\xi)\rangle\|\leqslant  \|\xi\|\|i(S^*)(i'(S)\xi)\| \leqslant \|\xi\|\|i(S^*)\|\|i'(S)\xi\|
$
for all $\xi\in{\cal E}$. Hence, $\|i'(S)\xi\|\leqslant\|i(S^*)\| \|\xi\|$ for all $\xi\in{\cal E}$, which proves the continuity of $i'(S)$ since $i'(S)$ is linear by definition. In particular, $i'(S)$ extends uniquely to a bounded linear map $i'(S):{\cal E}\rightarrow A$. By continuity of the inner product, we have proved that $i'(S)\in\Lin({\cal E},A)$ and $i'(S)^*=i(S^*)$. As a result, we have well-defined maps $i:\Lin(A,{\cal E})\rightarrow\Lin(B,{\cal F})$ and $i':\Lin({\cal F},B)\rightarrow\Lin({\cal E},A)$ such that $i(T)^*=i'(T^*)$ for all $T\in\Lin(A,{\cal E})$. It is clear that $i$ is linear and injective.

\medbreak

It remains to prove that ${\rm im}(i)=\{S\in\Lin(A,{\cal E})\,;\,SB\subset{\cal F},\,S^*{\cal F}\subset B\}$. Let $T\in\Lin(B,{\cal F})$ and $b\in B\subset\Lin(A)$. For all $a\in A$, we have $[i(T)\circ b]a=i(T)(ba)=(Tb)a$. Hence $i(T)\circ b=T(b)\in{\cal F}$. Fix $\eta\in{\cal F}$. Write $\eta=\zeta b$ with $\zeta\in{\cal F}$ and $b\in B$. For all $a\in A$, we have 
$
[i(T)^*\circ \eta]a=[i'(T^*)\eta]a=i'(T^*)(\eta a)=i'(T^*)(\zeta(ba))=(T^*\zeta)ba=T^*(\zeta b)a=T^*(\eta)a.
$
Hence, $i(T)^*\circ\eta=T^*(\eta)\in B$. Conversely, let us fix $S\in\Lin(A,{\cal E})$ such that $SB\subset{\cal F}$ and $S^*{\cal F}\subset B$. Let $T:A\rightarrow{\cal E}$ and $T':{\cal F}\rightarrow B$ be the maps defined by: 
\[
T(b):=S\circ b,\quad b\in B ; \quad T'(\eta):=S^*\circ\eta,\quad \eta\in{\cal F}.
\]
For all $b\in B$ and $\eta\in{\cal F}$, we have
$
\langle T(b),\, \eta\rangle = (S\circ b)^*\circ\eta=b^* (S^*\circ\eta)=\langle b, T'(\eta)\rangle.
$
Hence $T\in\Lin(A,{\cal E})$ and $T^*=T'$. Moreover, we have  $i(T)(ba)=T(b)a=S(ba)$ for all $a\in A$ and $b\in B$. Thus, we have $S=i(T)$.\newline
(ii) Since $[{\cal F} A]={\cal E}$, we have $[({\cal F} \oplus B) A]={\cal E} \oplus A$, which proves the uniqueness of $j$. Let $i_{12}:=i$ and $i_{21}:=i'$. By a similar argument as in statement (i), we prove that there exists a unique map $i_{11}:\Lin({\cal F})\rightarrow\Lin({\cal E})$ such that $i_{11}(T)(\eta a)=(T\eta) a$ for all $\eta\in{\cal F}$ and $a\in A$. The non-degenerate inclusion of C*-algebras $B\subset\M(A)$ extends to a unital *-homomorphism $i_{22}:\M(B)\rightarrow\M(A)$. Then we consider the map $j:\Lin({\cal F}\oplus B)\rightarrow\Lin({\cal E}\oplus A)$ defined by $j(x):=(i_{kl}(x_{kl}))_{k,l=1,2}$ for all $x=(x_{kl})_{k,l=1,2}\in\Lin({\cal F}\oplus B)$.
It is clear that $j(x)(\eta a)=(x\eta)a$ for all $x\in\Lin({\cal F}\oplus B)$, $\eta\in{\cal F}\oplus B$ and $a\in A$. The fact that $j$ is a unital faithful *-homomorphism is then straightforward.
\end{proof}

\begin{rks}\label{rk10}
With the notations and hypotheses of the previous proposition, we have:
\begin{enumerate}[label=(\roman*)]
\item for all $T\in\Lin(B,{\cal F})$, $j(\iota_{{\cal F}}(T))=\iota_{{\cal E}}(i(T))$;
\item for all $m\in\M(B)$, $j(\iota_{B}(m))=\iota_{A}(m)$, where we identify $\M(B)\subset\M(A)$.\qedhere
\end{enumerate}
\end{rks}

\begin{lem}\label{lem2}
Let $j=1,2$. We have a canonical embedding
\[
\Lin(A_2\tens S_{2j},\s E_2\tens S_{2j})\rightarrow\Lin(A_1\tens S_{12}\tens S_{2j},\s E_1\tens S_{12}\tens S_{2j})\,;\,T\mapsto\widetilde{T},
\]
where for $T\in\Lin(A_2\tens S_{2j},\s E_2\tens S_{2j})$ the operator $\widetilde{T}\in\Lin(A_1\tens S_{12}\tens S_{2j},\s E_1\tens S_{12}\tens S_{2j})$ is defined by
$
\widetilde{T}(xa)=T(x)a
$
for all $x\in A_2\tens S_{2j}$ and $a\in A_1\tens S_{12}\tens S_{2j}$.
Moreover, the image of $\Lin(A_2\tens S_{2j},\s E_2\tens S_{2j})\rightarrow\Lin(A_1\tens S_{12}\tens S_{2j},\s E_1\tens S_{12}\tens S_{2j})$ is$\vphantom{\widetilde{T}}$
\begin{align*}
\{ X \in \Lin(A_1\tens S_{12}\tens S_{2j},\s E_1\tens S_{12}\tens S_{2j}) \, ; \, X(A_2 \tens S_{2j}) \subset \s E_2 \tens & S_{2j} \;\text{ and } \\
&X^*(\s E_2 \tens S_{2j}) \subset A_2 \tens S_{2j} \}.\qedhere
\end{align*}
\end{lem}

\begin{proof}
This follows from \ref{lem1} with $A:=A_1\tens S_{12}\tens S_{2j}$, $B:=A_2\tens S_{2j}$, ${\cal E}:=\s E_1\tens S_{12}\tens S_{2j}$ and ${\cal F}:=\s E_2\tens S_{2j}\subset\Lin(A_1\tens S_{12}\tens S_{2j},\s E_1\tens S_{12}\tens S_{2j})$. The assumptions of \ref{lem1} are satisfied in this case in virtue of \ref{propind4} 1 and \ref{prop2}.
\end{proof}

\begin{nb}
Let 
\[
\id_{\s E_1}\tens\delta_{12}^2:\Lin(A_1\tens S_{12},\s E_1\tens S_{12})\rightarrow\Lin(A_1\tens S_{12}\tens S_{22},\s E_1\tens S_{12}\tens S_{22})
\] 
be the unique linear extension of $\id_{\s E_1}\tens\delta_{12}^2:\s E_1\tens S_{12}\rightarrow\Lin(A_1\tens S_{12}\tens S_{22},\s E_1\tens S_{12}\tens S_{22})$ such that
$
(\id_{\s E_1}\tens\delta_{12}^2)(T)(\id_{A_1}\tens\delta_{12}^2)(x)=(\id_{\s E_1}\tens\delta_{12}^2)(Tx)
$
for all $x\in\M(A_1\tens S_{12})$ and $T\in\Lin(A_1\tens S_{12}, \s E_1\tens S_{12})$.
\end{nb}

\begin{propdef}\label{prop12}
There exists a unique linear map 
\begin{center}
$
\delta_{\s E_2}:\s E_2\rightarrow\Lin(A_2\tens S_{22},\s E_2\tens S_{22})
$
\end{center} 
satisfying the relation $[\delta_{\s E_2}(\xi)a]b=(\id_{\s E_1}\tens\delta_{12}^2)(\xi)(ab)$ for all $\xi\in\s E_2$, $a\in A_2\tens S_{22}$ and $b\in A_1\tens S_{12}\tens S_{22}$.
\end{propdef}

\begin{proof}
Let us prove the inclusion
$
(\id_{\s E_1}\tens\delta_{12}^2)(\s E_2)(A_2\tens S_{22})\subset\s E_2\tens S_{22}.
$
It follows from \ref{lem22} 2 that $\iota_{\s E_1\tens S_{12}}(\s E_2)\subset J_2$. Fix $\xi\in\s E_2$ and $x\in A_2\tens S_{22}$. We have
\begin{align*}
\iota_{\s E_1\tens S_{12} \tens S_{22}}((\id_{\s E_1}\tens\delta_{12}^2)(\xi)x)&=(\id_{J_1}\tens\delta_{12}^2)(\iota_{\s E_1 \tens S_{12}}(\xi))\iota_{A_1\tens S_{12}\tens S_{22}}(x)\\
&=\delta_{J_2}(\iota_{\s E_1 \tens S_{12}}(\xi))\iota_{A_2\tens S_{22}}(x)\in J_2 \tens S_{22}.
\end{align*}
As in the proof of \ref{prop3}, $\iota_{\s E_1\tens S_{12} \tens S_{22}}((\id_{\s E_1}\tens\delta_{12}^2)(\xi)x)$ is the norm limit of finite sums of elements of the form 
$
\iota_{\s E_1\tens S_{12} \tens S_{22}}((\id_{\s E_1\tens S_{12}} \tens \omega)\delta_{\s E_1}^{(2)}(\eta)\tens s)
$
with $\eta\in\s E_1$, $\omega\in\B(\s H_{21})_*$ and $s\in S_{22}$.
Hence, $(\id_{\s E_1}\tens\delta_{12}^2)(\xi)x \in \s E_2 \tens S_{22}$ since $\iota_{\s E_1\tens S_{12} \tens S_{22}}$ is isometric. $\vphantom{\delta_{\s E_1}^{(2)}}$Therefore, we have $(\id_{\s E_1}\tens\delta_{12}^2)(\s E_2)(A_2\tens S_{22})\subset\s E_2\tens S_{22}$. $\vphantom{\delta_{\s E_1}^{(2)}}$The inclusion $(\id_{\s E_1}\tens\delta_{12}^2)(\s E_2)^*(\s E_2\tens S_{22})\subset A_2\tens S_{22}$ is obtained by a similar argument. $\vphantom{\delta_{\s E_1}^{(2)}}$Then, the existence and uniqueness of the operator $\delta_{\s E_2}(\xi)\in\Lin(A_2\tens S_{22},\s E_2\tens S_{22})$ follows as an application of \ref{lem2} with $j=2$.
$\vphantom{\delta_{\s E_1}^{(2)}}$It is clear that the map $\delta_{\s E_2}:\s E_2\rightarrow\Lin(A_2\tens S_{22},\s E_2\tens S_{22})$ is linear.$\vphantom{\delta_{\s E_1}^{(2)}}$
\end{proof}

In the following, we prove that $\delta_{\s E_2}$ is a continuous action of $\QG_2$ on $\s E_2$. We also show that the induction procedure for equivariant Hilbert modules is equivalent to that of \S 4.3 \cite{BC}.

\begin{nbs}
Let $e_{1,1}:=\iota_{\K(\s E_1)}(1_{\s E_1})\in\M(J_1)$ and $e_{2,1}:=\iota_{A_1}(1_{A_1})\in\M(J_1)$, where we identify $\M(J_1)=\Lin(\s E_1\oplus A_1)$. Let $(J_2,\delta_{J_2},e_{1,2},e_{2,2})$ be the induced linking $\QG_2$-C*-algebra, with $e_{l,2}:=e_{l,1}\tens 1_{S_{12}}\in\M(J_2)$ for $l=1,2$ (cf.\ 4.14 \cite{BC}). Consider $e_{2,2}J_2e_{2,2}$ and $e_{1,2}J_2e_{2,2}$ endowed with their structure of $\QG_2$-C*-algebra and $\QG_2$-equivariant Hilbert $e_{2,2}J_2e_{2,2}$-module \cite{BS1}. Recall that the morphism $\ind\iota_{A_1}:A_2\rightarrow J_2\,;\, x\mapsto(\iota_{A_1}\tens\id_{S_{12}})(x)$ induces a $\QG_2$-equivariant *-isomormorphism $A_2\rightarrow e_{2,2}J_2e_{2,2}$ (cf.\ 4.17, 4.18 \cite{BC}).
\end{nbs}

\begin{prop}\label{prop11}
We use the above notations.
\begin{enumerate}[label=(\roman*)]
\item The map $\delta_{\s E_2}:\s E_2\rightarrow\Lin(A_2\tens S_{22},\s E_2\tens S_{22})$ is a continuous action of $\QG_2$ on $\s E_2$.
\item There exists a unique bounded linear map $\ind\iota_{\s E_1}:\s E_2\rightarrow J_2$ such that
\[
\ind\iota_{\s E_1}((\id_{\s E_1\tens S_{12}}\tens\omega)\delta_{\s E_1}^{(2)}(\xi))
=(\id_{J_1\tens S_{12}}\tens\omega)\delta_{J_1}^{(2)}(\iota_{\s E_1}(\xi)),
\]
for all $\xi\in\s E_1$ and $\omega\in\B(\s H_{21})_*$. Moreover, we have $\ind\iota_{\s E_1}(\s E_2)=e_{1,2}J_2e_{2,2}$ and $\ind\iota_{\s E_1}$ induces a $\QG_2$-equivariant unitary equivalence
$
\s E_2 \rightarrow e_{1,2}J_2e_{2,2} \; ; \; \xi \mapsto \ind\iota_{\s E_1}(\xi)
$
over the $\QG_2$-equivariant *-isomorphism 
$
A_2 \rightarrow e_{2,2}J_2 e_{2,2} \; ; \; a \mapsto \ind\iota_{A_1}(a).
$
\item There exists a unique *-homomorphism $\tau:\K(\s E_2\oplus A_2)\rightarrow J_2$ such that
$
\tau\circ\iota_{\s E_2}=\ind\iota_{\s E_1}
$
and
$
\tau\circ\iota_{A_2}=\ind\iota_{A_1}.
$ 
Moreover, $\tau$ is an isomorphism of linking $\QG_2$-C*-algebras.
\item If $T\in\ind(\K(\s E_1))\subset\Lin(\s E_1\tens S_{12})$ and $\eta\in\s E_2\subset\Lin(A_1\tens S_{12},\s E_1\tens S_{12})$, then we have $T\circ\eta\in\s E_2$. Moreover, for all $T\in\ind(\K(\s E_1))$, we have $[\eta\mapsto T\circ\eta]\in\K(\s E_2)$. More precisely, the map
$
\ind(\K(\s E_1))\rightarrow\K(\s E_2)\,;\, T \mapsto [\eta\mapsto T\circ\eta]
$
is a $\QG_2$-equivariant *-isomorphism.$\vphantom{\ind}$\qedhere
\end{enumerate}
\end{prop}

\begin{proof}
Let us denote $B:=e_{2,2}J_2e_{2,2}$ and $\s F:=e_{1,2}J_2e_{2,2}$ for short. \newline
(i)-(ii) $\vphantom{\delta_{\s E_1}^{(2)}}$We have $\iota_{\s E_1\tens S_{12}}(\s E_2)\subset J_2$ (cf.\ \ref{lem22}). Let 
$
\ind\iota_{\s E_1}:=\restr{\iota_{\s E_1\tens S_{12}}}{\s E_2}:\s E_2\rightarrow J_2.
$
It also follows from the formulas $\delta_{J_1}^{(2)}(e_{l,1})=e_{l,2}\tens 1_{S_{21}}$ for $l=1,2$ (4.14 \cite{BC}) that 
\begin{align*}
e_{1,2}J_2e_{2,2}&=[(\id_{J_1\tens S_{12}}\tens\omega)\delta_{J_1}^{(2)}(e_{1,1}xe_{2,1})\,;\, x\in J_1,\, \omega\in\B(\s H_{21})_*] \\
&=[(\id_{J_1\tens S_{12}}\tens\omega)\delta_{J_1}^{(2)}(\iota_{\s E_1}(\xi)) \, ; \, \xi\in\s E_1,\, \omega\in\B(\s H_{21})_*]\\
&=\ind\iota_{\s E_1}(\s E_2).
\end{align*}
Let $\xi,\eta\in\s E_2$. Since $\langle\xi,\,\eta\rangle\in A_2$, we have (cf.\ \ref{prop32} 3)
\[
\langle\ind\iota_{\s E_1}(\xi),\, \ind\iota_{\s E_1}(\eta)\rangle=\iota_{\s E_1\tens S_{12}}(\xi)^*\iota_{\s E_1\tens S_{12}}(\eta)=\iota_{A_1\tens S_{12}}(\langle\xi,\eta\rangle)=\ind\iota_{A_1}(\langle\xi,\,\eta\rangle).
\]  
We also have $\ind\iota_{\s E_1}(\xi)=\ind\iota_{\s E_1}(\xi)\ind\iota_{A_1}(a)$ for all $a\in A_2$ and $\xi\in\s E_2$ (cf.\ \ref{prop32} 1). The map $\Phi:\s E_2\rightarrow \s F\,;\, \xi\mapsto\ind\iota_{\s E_1}(\xi)$ is a unitary equivalence of Hilbert modules over the *-isomorphism $\phi:A_2 \rightarrow B \; ; \; a \mapsto \ind\iota_{A_1}(a)$.

\medskip

Let us prove that $(\Phi\tens\id_{S_{22}})\circ\delta_{\s E_2}=\delta_{\s F}\circ\Phi$. It is immediately verified that for $\xi\in\s E_1\tens S_{12}$, the formula $(\iota_{\s E_1\tens S_{12}}\tens\id_{S_{22}})(\id_{\s E_1}\tens\delta_{12}^2)(\xi)=(\id_{J_1}\tens\delta_{12}^2)(\iota_{\s E_1\tens S_{12}}(\xi))$ holds true. Let us fix $\xi\in\Lin(A_1\tens S_{12},\s E_1\tens S_{12})$. For all $a\in A_1\tens S_{12}$ and $x\in A_1\tens S_{12}\tens S_{22}$, we have
\begin{multline*}
\iota_{\s E_1\tens S_{12}\tens S_{22}}((\id_{\s E_1}\tens\delta_{12}^2)  ( \xi )) \iota_{A_1\tens S_{12}\tens S_{22}} ((\id_{A_1}\tens\delta_{12}^2)(a)x)\\=(\id_{J_1}\tens\delta_{12}^2)(\iota_{\s E_1\tens S_{12}}(\xi a))\iota_{A_1\tens S_{12}\tens S_{22}}(x)
\end{multline*}
and 
$
(\id_{J_1}\tens\delta_{12}^2)(\iota_{\s E_1\tens S_{12}}(\xi a))=(\id_{J_1}\tens\delta_{12}^2)(\iota_{\s E_1\tens S_{12}}(\xi))\iota_{A_1\tens S_{12}\tens S_{22}}((\id_{A_1}\tens\delta_{12}^2)(a)).
$
Hence,
\begin{multline*}
\iota_{\s E_1\tens S_{12}\tens S_{22}}((\id_{\s E_1}\tens\delta_{12}^2)(\xi))\iota_{A_1\tens S_{12}\tens S_{22}} ((\id_{A_1}\tens\delta_{12}^2)(a)x)\\
=(\id_{J_1}\tens\delta_{12}^2)(\iota_{\s E_1\tens S_{12}}(\xi))\iota_{A_1\tens S_{12}\tens S_{22}}((\id_{A_1}\tens\delta_{12}^2)(a)x).
\end{multline*}
Thus,
$
\iota_{\s E_1\tens S_{12}\tens S_{22}}((\id_{\s E_1}\tens\delta_{12}^2)(\xi))\iota_{A_1\tens S_{12}\tens S_{22}}(x)
=(\id_{J_1}\tens\delta_{12}^2)(\iota_{\s E_1\tens S_{12}}(\xi))\iota_{A_1\tens S_{12}\tens S_{22}}(x)
$
for all $\xi\in\Lin(A_1\tens S_{12},\s E_1\tens S_{12})$ and $a\in A_1\tens S_{12}\tens S_{22}$ in virtue of the non-degeneracy of $\id_{A_1}\tens\delta_{12}^2$. Let us fix $\xi\in\s E_2$. For all $x\in A_2\tens S_{22}$ and $y\in A_1\tens S_{12}\tens S_{21}$ we have
\begin{align*}
[(\Phi\tens\id_{S_{22}})\delta_{\s E_2}(\xi)(\phi\tens\id_{S_{22}})(x)] \iota_{A_1\tens S_{12}\tens S_{22}}(y)
&=\iota_{\s E_1\tens S_{12}\tens S_{22}}((\delta_{\s E_2}(\xi)x)y)\\
&=\iota_{\s E_1\tens S_{12}\tens S_{22}}((\id_{\s E_1}\tens\delta_{12}^2)(\xi)(xy))\\
&=(\id_{J_1}\tens\delta_{12}^2)(\iota_{\s E_1\tens S_{12}}(\xi))\iota_{A_1\tens S_{12}\tens S_{22}}(xy)\\
&=[\delta_{\s F}(\Phi(\xi))(\phi\tens\id_{S_{22}})(x)]\iota_{A_1\tens S_{12}\tens S_{22}}(y),
\end{align*}
which also holds for all $y\in\M(A_1\tens S_{12}\tens S_{22})$ by strict continuity. In particular, by applying this formula for $y\in A_2\tens S_{22}$, we have then proved that
\[
(\Phi\tens\id_{S_{22}})(\delta_{\s E_2}(\xi))(\phi\tens\id_{S_{22}})(x)=\delta_{\s F}(\Phi(\xi))(\phi\tens\id_{S_{22}})(x)
\]
for all $x\in A_2\tens S_{22}$. Hence, $(\Phi\tens\id_{S_{22}})(\delta_{\s E_2}(\xi))=\delta_{\s F}(\Phi(\xi))$ for all $\xi\in\s E_2$. This proves that $\delta_{\s E_2}$ is a continuous action of $\QG_2$ on $\s E_2$ and $\Phi$ is $\QG_2$-equivariant.\newline
(iii) There exists a unique unital faithful *-homomorphism 
\begin{center}
$
j:\Lin(\s E_2\oplus A_2)\rightarrow\Lin((\s E_1\tens S_{12})\oplus(A_1\tens S_{12}))
$
\end{center}
such that $j(x)(\eta a)=(x\eta)a$ for all $x\in\Lin(\s E_2\oplus A_2)$, $\eta\in \s E_2\oplus A_2$ and $a\in A_1\tens S_{12}$ (\ref{lem2} (ii), with $A:=A_1\tens S_{12}$, $B:=A_2$, ${\cal E}:=\s E_1\tens S_{12}$ and ${\cal F}:=\s E_2$). Now, it should be noted that we have the following canonical identifications
\begin{center}
$
J_2\subset\M(J_1\tens S_{12})=\Lin((\s E_1\oplus A_1)\tens S_{12})=\Lin((\s E_1\tens S_{12})\oplus(A_1\tens S_{12})).
$
\end{center}
We have $j(\iota_{\s E_2}(\xi))=\ind\iota_{\s E_1}(\xi)$ for all $\xi\in\s E_2$ and $j(\iota_{A_2}(b))=\ind\iota_{A_1}(b)$ for all $b\in A_2$ (cf.\ \ref{rk10}). In particular, we have $j(\K(\s E_2\oplus A_2))\subset J_2$. Let $\tau:=\restr{j}{\K(\s E_2\oplus A_2)}:\K(\s E_2\oplus A_2)\rightarrow J_2$. Since $J_2$ is generated by $e_{1,2}J_2e_{2,2}$ and $e_{2,2}J_2e_{2,2}$ as a C*-algebra, $\tau$ has dense range (cf.\ (ii)); moreover, $\tau$ is also isometric (faithful), therefore $\tau$ is surjective. Thus, we have proved that $\tau$ is a *-isomorphism. The $\QG_2$-equivariance of $\tau$ is derived from straightforward computations.\newline
(iv) Consider the $\QG_2$-equivariant *-isomorphism 
\begin{center}
$\varphi:\ind(\K(\s E_1))\rightarrow e_{1,2}J_2e_{1,2}\,;\, k\mapsto\ind\iota_{\K(\s E_1)}(k)$
\end{center} 
(cf.\ 4.18 \cite{BC}, note that $\K(\s F)=e_{1,2}J_2e_{1,2}$). $\vphantom{\ind}$By statement (ii), $\tau$ induces by restriction a $\QG_2$-equivariant *-isomorphism$\vphantom{\ind}$ $\tau:f_{1,2}\K(\s E_2\oplus A_2)f_{1,2}\rightarrow e_{1,2}J_2e_{1,2}$, where $f_{1,2}:=\iota_{\s E_2}(1_{\s E_2})$ and $f_{2,2}:=\iota_{A_2}(1_{A_2})$. We have an isomorphism$\vphantom{\ind}$ 
$
\psi:\K(\s E_2)\rightarrow f_{1,2}\K(\s E_2\oplus A_2)f_{1,2}\; ;\; k\mapsto\iota_{\K(\s E_2)}(k)
$
of $\QG_2$-C*-algebras. Hence, $\chi:=\psi^{-1}\circ\tau^{-1}\circ\varphi:\ind(\K(\s E_1))\rightarrow\K(\s E_2)$ is an isomorphism of $\QG_2$-C*-algebras. It is clear that $\chi(T)\xi=T\circ\xi$ for all $T\in\ind(\K(\s E_1))\subset\Lin(\s E_1\tens S_{12})$ and $\xi\in\s E_2\subset\Lin(A_1\tens S_{12},\s E_1\tens S_{12})$.$\vphantom{\ind}$
\end{proof}

\begin{propdef}\label{prop14}
Let us fix some notations. Consider: 
\begin{itemize}
\item two $\QG_1$-C*-algebras $A_1$ and $B_1$; 
\item two $\QG_1$-equivariant Hilbert modules $\s E_1$ and $\s F_1$ over $A_1$ and $B_1$ respectively;
\item a $\QG_1$-equivariant unitary equivalence $\Phi_1:\s E_1\rightarrow\s F_1$ over a $\QG_1$-equivariant *-isomorphism $\phi_1:A_1\rightarrow B_1$.
\end{itemize} 
Denote by:
\begin{itemize}
\item $A_2:=\ind(A_1)$ and $B_2:=\ind(B_1)$ the induced $\QG_2$-C*-algebras;
\item $\ind(\phi_1):A_2\rightarrow B_2$ the induced $\QG_2$-equivariant *-isomorphism;
\item $\s E_2:=\ind(\s E_1)$ and $\s F_2:=\ind(\s F_1)$ the induced $\QG_2$-equivariant Hilbert modules over $A_2$ and $B_2$ respectively;
\item $
\Phi_1\tens\id_{S_{12}}:\Lin(A_1\tens S_{12},\s E_1\tens S_{12})\rightarrow\Lin(B_1\tens S_{12},\s F_1\tens S_{12})
$
the unique linear map such that $(\Phi_1\tens\id_{S_{12}})(T)(\phi_1\tens\id_{S_{12}})(x)=(\Phi_1\tens\id_{S_{12}})(Tx)$ for all $\Lin(A_1\tens S_{12},\s E_1\tens S_{12})$ and $x\in A_1\tens S_{12}$ (cf.\ \ref{not3}).
\end{itemize}
Then, $(\Phi_1\tens\id_{S_{12}})(\s E_2)\subset\s F_2$ and the map 
$
\ind(\Phi_1):=\restr{(\Phi_1\tens\id_{S_{12}})}{\s E_2}:\s E_2\rightarrow\s F_2
$ 
is a $\QG_2$-equivariant unitary equivalence over $\ind(\phi_1):A_2\rightarrow B_2$. Moreover, for all $\xi\in\s E_1$ and $\omega\in\B(\s H_{21})_*$ we have
$
\ind(\Phi_1)((\id_{\s E_1 \tens S_{12}}\tens\omega)\delta_{\s E_1}^{(2)}(\xi))=(\id_{\s F_1 \tens S_{12}}\tens\omega)\delta_{\s F_1}^{(2)}(\Phi_1\xi).
$
\end{propdef}

\begin{proof}
Denote by $J_1:=\K(\s E_1\oplus A_1)$ and $K_1:=\K(\s F_1\oplus B_1)$ the linking $\QG_1$-C*-algebras, whose linking structures are respectively defined by:
$e_{1,1}:=\iota_{\s E_1}(1_{\s E_1})$, $e_{2,1}:=\iota_{A_1}(1_{A_1})$;
$f_{1,1}:=\iota_{\s F_1}(1_{\s F_1})$, $f_{2,1}:=\iota_{B_1}(1_{B_1})$.
We also denote by $(J_2,\delta_{J_2},e_{1,2},e_{2,2})$ and $(K_2,\delta_{K_2},f_{1,2},f_{2,2})$ the induced linking $\QG_2$-C*-algebras, where $e_{l,2}:=e_{l,1}\tens 1_{S_{12}}$ and $f_{l,2}:=f_{l,1}\tens 1_{S_{12}}$ for $l=1,2$ (cf.\ 4.14 \cite{BC}).
There exists a unique *-isomorphism $\tau_1:J_1\rightarrow K_1$ such that $\tau_1\circ\iota_{\s E_1}=\iota_{\s F_1}\circ\Phi_1$ and $\tau_1\circ\iota_{A_1}=\iota_{B_1}\circ\phi_1$ (cf.\ \ref{prop7} and \ref{prop10}). We then denote by 
\begin{center}
$
\tau_2:=\ind \tau_1:J_2\rightarrow K_2
$
\end{center} 
the induced morphism. Since $\tau_2$ is an isomorphism of linking $\QG_2$-C*-algebras, it induces a $\QG_2$-equivariant unitary equivalence $\Psi:e_{1,2}J_2 e_{2,2}\rightarrow f_{1,2} K_2 f_{2,2}$ over the isomorphism of $\QG_2$-C*-algebras $\psi:e_{2,2}J_2e_{2,2}\rightarrow f_{2,2}K_2 f_{2,2}$. Since $\tau_1\circ\iota_{A_1}=\iota_{B_1}\circ\phi_1$, we have
\begin{center} 
$\tau_2\circ\ind\iota_{A_1}=\ind\iota_{B_1}\circ\phi_2$.
\end{center} 
Therefore, by composition of $\QG_2$-equivariant unitary equivalences (cf.\ \ref{rk4} 2) and by applying \ref{prop11}, we obtain a $\QG_2$-equivariant $\phi_2$-compatible unitary operator $\Phi_2:\s E_2\rightarrow\s F_2$. By a straightforward computation, we show that $\Phi_2=\restr{(\Phi_1\tens\id_{S_{12}})}{\s E_2}$.
\end{proof}

By exchanging the roles of $\QG_1$ and $\QG_2$, we define as above an induction procedure for $\QG_2$-equivariant Hilbert modules.

\medskip 

In the following, we investigate the composition of $\ind$ and $\iind$. Let $A_1$ be a $\QG_1$-C*-algebra and $\s E_1$ a $\QG_1$-equivariant Hilbert $A_1$-module. Denote by:
\begin{itemize}
\item $A_2:=\ind(A_1)$ and $\s E_2=\ind(\s E_1)\subset\Lin(A_1\tens S_{12},\s E_1\tens S_{12})$ the induced $\QG_2$-C*-algebra and the induced $\QG_2$-equivariant Hilbert $A_2$-module;
\item $C=\iind(A_2)$ and $\s F:=\iind(\s E_2)\subset\Lin(A_2\tens S_{21},\s E_2\tens S_{21})$ the induced $\QG_1$-C*-algebra  and the induced $\QG_1$-equivariant Hilbert $C$-module.
\end{itemize}

\begin{prop}\label{prop13}
With the above notations and hypotheses, we have the following statements:
\begin{enumerate}
\item there exists a unique map $\Pi_1:\s E_1 \rightarrow\s F$ such that
\[
(\Pi_1(\xi)x)a=\delta_{\s E_1}^{(2)}(\xi)(xa), \quad \text{for all } \; \xi\in\s E_1,\, x\in A_2\tens S_{21} \; \text{ and } \; a\in A_1\tens S_{12}\tens S_{21};
\]
moreover, $\Pi_1$ is a $\QG_1$-equivariant unitary equivalence over the $\QG_1$-equivariant *-isomorphism $\pi_1:A_1\rightarrow C\,;\,a\mapsto\delta_{A_1}^{(2)}(a)$;\index[symbol]{pl@$\Pi_j$}
\item  
$
\delta_{\s E_1}^2:\s E_1 \rightarrow \widetilde{\M}(\s E_2\tens S_{21})\,;\, \xi \mapsto \Pi_1(\xi)
$
is a well-defined linear map such that:
\begin{enumerate}[label=(\roman*)]
\item $\delta_{\s E_1}^2(\xi a)=\delta_{\s E_1}^2(\xi)\delta_{A_1}^2(a)$ and $\langle\delta_{\s E_1}^2(\xi),\, \delta_{\s E_1}^2(\eta)\rangle = \delta_{A_1}^{2}(\langle\xi,\, \eta\rangle)$ for all $\xi,\eta\in\s E_1$ and $a\in A_1$,
\item $[\delta_{\s E_1}^2(\s E_1)(1_{A_2}\tens S_{21})]=\s E_2\tens S_{21}=[(1_{\s E_2}\tens S_{21})\delta_{\s E_1}^2(\s E_1)]$.\qedhere
\end{enumerate}
\end{enumerate}
\end{prop}

\begin{proof}
1.\ The existence and uniqueness of $\Pi_1$ is an immediate application of \ref{lem2} with $j=1$ and the proof is very similar to that of \ref{prop12}. The fact that $\Pi_1$ is a $\QG_1$-equivariant unitary equivalence over $\pi_1$ is a straightforward consequence of \ref{prop11} (ii), (iii) and \ref{propind1} 2.\newline
2.\ Statement (ii) and the fact that $\delta_{\s E_1}^2$ takes its values in $\widetilde{\M}(\s E_2\tens S_{21})$ are proved by combining the formulas $[\s F(1_{A_2}\tens S_{21})]=\s E_2\tens S_{21}=[(1_{\s E_2}\tens S_{21})\s F]$ (cf.\ \ref{prop2}) with the fact that $\Pi_1$ is bijective. Statement (i) follows from the compatibility of $\Pi_1$ with $\pi_1$.$\vphantom{\widetilde{\M}}$
\end{proof}

We have proved the following result:

\begin{thm}
Let $\QG_1$ and $\QG_2$ be two monoidally equivalent regular locally compact quantum groups. The map 
\begin{align*}
\ind &:(\s E_1,\delta_{\s E_1}) \mapsto (\s E_2:=\ind(\s E_1),\, \delta_{\s E_2}:\xi\in\s E_2\mapsto[x\in A_2\tens S_{22}\mapsto(\id_{\s E_1}\tens\delta_{12}^2)(\xi)x]),\\
\intertext{where $\s E_1$ is a Hilbert module over the $\QG_1$-C*-algebra $A_1$ and $A_2=\ind(A_1)$ denotes the induced $\QG_2$-C*-algebra, is a one-to-one correspondence up to unitary equivalence. The inverse map, up to unitary equivalence, is}
\iind &: (\s F_2,\delta_{\s F_2}) \mapsto (\s F_1:=\iind(\s F_2),\, \delta_{\s F_1}:\xi\in\s F_1\mapsto[x\in B_1\tens S_{11}\mapsto(\id_{\s F_2}\tens\delta_{21}^1)(\xi)x]),
\end{align*} 
where $\s F_2$ is a Hilbert module over the $\QG_2$-C*-algebra $B_2$ and $B_1=\ind(B_2)$ denotes the induced $\QG_1$-C*-algebra.
\end{thm}

\begin{proof}
This is a consequence of Propositions \ref{prop13}, \ref{prop14} and the corresponding results obtained by exchanging the roles of $\QG_1$ and $\QG_2$.
\end{proof}

Let $B_1$ be a $\QG_1$-C*-algebra. Let us denote by $B_2:=\ind(B_1)$ the induced $\QG_2$-C*-algebra. Let
$
\delta_{B_j}^k:B_j\rightarrow\M(B_k\tens S_{kj})
$
for $j,k=1,2$ be the *-homomorphisms defined in \ref{not11}.

\begin{nbs}
Let $\s E_1$ be a $\QG_1$-equivariant Hilbert $B_1$-module. Let us denote by $\s F_2=\ind(\s F_1)$ the induced $\QG_2$-equivariant Hilbert $B_2$-module. We have four linear maps
\[
\delta_{\s F_j}^k:\s F_j\rightarrow\Lin(B_k\tens S_{kj},\s F_k\tens S_{kj}), \quad \text{for } j,k=1,2,
\]
defined as follows:
\begin{itemize}
\item $\delta_{\s F_1}^1:=\delta_{\s F_1}$ and $\delta_{\s F_2}^2:=\delta_{\s F_2}$;
\item $\delta_{\s F_1}^2:\s F_1\rightarrow\Lin(B_2\tens S_{21},\s F_2\tens S_{21})$ is the unique linear map such that 
\[
(\delta_{\s F_1}^2(\xi)x)b=\delta_{\s F_1}^{(2)}(\xi)(xb)
\] 
for all $\xi\in\s F_1$, $x\in B_2\tens S_{21}$ and $b\in B_1\tens S_{12}\tens S_{22}$, where $\delta_{\s F_1}^{(2)}(\xi):=(\id_{\s E_1}\tens\delta_{11}^2)\delta_{\s F_1}(\xi)$ (cf.\ \ref{prop13});
\item $\delta_{\s F_2}^1:\s F_2\rightarrow\Lin(B_1\tens S_{12},\s F_1\tens S_{12})$ is the unique linear map such that for all $\xi\in\s F_2$, $x\in\iind(B_2)\tens S_{12}$ and $y\in B_2\tens S_{21}\tens S_{12}$, we have
\[
[(\Pi_1\tens\id_{S_{12}})(\delta_{\s F_2}^1(\xi))x]y=\delta_{\s F_2}^{(1)}(\xi)(xy),
\] 
where $\delta_{\s F_2}^{(1)}(\xi):=(\id_{\s F_1}\tens\delta_{22}^1)\delta_{\s F_2}(\xi)$ and $\Pi_1:\s F_1\rightarrow\iind(\s F_2)$ (cf.\ \ref{prop13} 1).\qedhere
\end{itemize}
\end{nbs}

\begin{lem}\label{lem4}
For all $j,k,l=1,2$, we have the following statements:
\begin{enumerate}
\item $\delta_{\s F_j}^k(\s F_j)\subset\widetilde{\M}(\s F_k\tens S_{kj})$;
\item $\delta_{\s F_j}^k(\xi b)=\delta_{\s F_j}^k(\xi)\delta_{B_j}^k(b)$ and $\langle\delta_{\s F_j}^k(\xi),\, \delta_{\s F_j}^k(\eta)\rangle = \delta_{B_j}^k(\langle\xi,\, \eta\rangle)$ for all $\xi,\eta\in\s F_j$ and $b\in B_j$;
\item $[\delta_{\s F_j}^k(\s F_j)(1_{B_k}\tens S_{kj})]=\s F_k\tens S_{kj}=[(1_{\s F_k}\tens S_{kj})\delta_{\s F_j}^k(\s F_j)]$;
\item $\delta_{\s F_k}^l\tens\id_{S_{kj}}$ {\rm(}resp.\ $\id_{\s F_l}\tens\delta_{lj}^k)$ extends uniquely to a linear map from $\Lin(B_k\tens S_{kj},\s E_k\tens S_{kj})$ to $\Lin(B_l\tens S_{lk}\tens S_{kj},\s E_l\tens S_{lk}\tens S_{kj})$ such that
\begin{align*}
(\delta_{\s F_k}^l\tens\id_{S_{kj}})(T)(\delta_{B_k}^l\tens\id_{S_{kj}})(x)&=(\delta_{\s F_k}^l\tens\id_{S_{kj}})(Tx)\\
\text{{\rm(}resp.\ }(\id_{\s F_l}\tens\delta_{lj}^k)(T)(\id_{B_l}\tens\delta_{lj}^k)(x)&=\id_{\s F_l}\tens\delta_{lj}^k)(Tx)\text{{\rm)}}
\end{align*}
for all $T\in\Lin(B_k\tens S_{kj},\s E_k\tens S_{kj})$ and $x\in B_k\tens S_{kj}$;
\item $(\delta_{\s F_k}^l\tens\id_{S_{kj}})\delta_{\s F_j}^k=(\id_{\s F_l}\tens\delta_{lj}^k)\delta_{\s F_j}^l$.\qedhere
\end{enumerate}
\end{lem}

\begin{proof}
Let $C:=\iind(B_2)$, $D:=\ind(C)$, $\s K=\iind(\s F_2)$ and $\s L=\ind(\s K)$. There exists a unique $\QG_2$-equivariant unitary equivalence $\Pi_2:\s F_2\rightarrow\s L$ (\ref{prop13} 1, after exchanging the roles of $\QG_1$ and $\QG_2$) over the $\QG_2$-equivariant *-isomorphism $\pi_2:B_2\rightarrow D$.\newline
1.\ This statement will follow straightforwardly from the third one.\newline
2.\ This statement has already been proved for $(j,k)=(1,1)$ (by definition), $(j,k)=(2,2)$ (cf.\ \ref{prop11} (i)) and for $(j,k)=(1,2)$ (cf.\ \ref{prop13}). Moreover, the case $(j,k)=(2,1)$ follows from the formulas $\delta_{\s F_2}^1=(\Pi_1^{-1}\tens\id_{S_{12}})\Pi_2$ and $\delta_{B_2}^1=(\pi_1^{-1}\tens\id_{S_{12}})\pi_2$.\newline
3.\ This statement is true by assumption for $(j,k)=(1,1)$, for $(j,k)=(2,2)$ (cf.\ \ref{prop11} (i)) and for $(j,k)=(1,2)$ (cf.\ \ref{prop13} 2 (ii)). By \ref{prop2} and \ref{prop13}, we have $[\s L(C\tens S_{12})]=\s K\tens S_{12}$, $\s L=\Pi_2(\s F_2)$ and $\s K=\Pi_1(\s F_1)$.
Therefore, we have
\begin{align*}
[\delta_{\s F_2}^1(\s F_2)(B_1\tens S_{12})]&=[(\Pi_1^{-1}\tens\id_{S_{12}})(\s L)(\pi_1^{-1}\tens\id_{S_{12}})(D\tens S_{12})]\\
&=[(\Pi_1^{-1}\tens\id_{S_{12}})(\s L(D\tens S_{12}))]\\
&=\s F_1\tens S_{12}.
\end{align*}
It then follows from the second statement and the fact that $[\delta_{B_2}^1(B_2)(1_{B_1}\tens S_{12})]=B_1\tens S_{12}$ that $[\delta_{\s F_2}^1(\s F_2)(1_{B_1}\tens S_{12})]=\s F_1\tens S_{12}$, which is statement 3 for $(j,k)=(2,1)$.\newline
4.\ Let $j,k,l=1,2$. The uniqueness of the extensions is obvious by the non-degeneracy of $\delta_{lj}^k$ and $\delta_{B_k}^l$. The linear map $\id_{\s F_l}\tens\delta_{lj}^k:\Lin(B_k\tens S_{kj},\s E_k\tens S_{kj})\rightarrow\Lin(B_l\tens S_{lk}\tens S_{kj},\s E_l\tens S_{lk}\tens S_{kj})$ is defined by
\begin{center}
$(\id_{\s F_l}\tens\delta_{lj}^k)(T):=T\tens_{\id_{B_l}\tens\delta_{lj}^k}1$, \quad for all $T\in\Lin(B_k\tens S_{kj},\s E_k\tens S_{kj})$,
\end{center}
where we use the identifications (\ref{eq9}) and (\ref{eq10}). As in 2.4 (a) \cite{BS1}, there exists a unique unitary $\s V_k^l\in\Lin(\s F_k\tens_{\delta_{B_k}^l}(B_l\tens S_{lk}),\s F_l\tens S_{lk})$ such that\index[symbol]{ve@$\s V_j^k$} 
\begin{center}
$\s V_k^l(\xi\tens_{\delta_{B_k}^l}x)=\delta_{\s F_k}^l(\xi)x$, \quad for all $\xi\in\s F_k$ and  $x\in B_l\tens S_{lk}$.
\end{center}
$\vphantom{tens_{\delta_{B_k}^l\tens\,\id_{S_{kj}}}}$The linear extension $\delta_{\s F_k}^l\tens\id_{S_{kj}}:\Lin(B_k\tens S_{kj},\s E_k\tens S_{kj})\rightarrow\Lin(B_l\tens S_{lk}\tens S_{kj},\s E_l\tens S_{lk}\tens S_{kj})$ is defined by $(\delta_{\s F_k}^l\tens\id_{S_{kj}})(T):=(\s V_k^l\tens_{\GC}1)(T\tens_{\delta_{B_k}^l\tens\,\id_{S_{kj}}}1)$ for all $T\in\Lin(B_k\tens S_{kj},\s E_k\tens S_{kj})$, up to the identifications (\ref{eq24}) and (\ref{eq1.16}).$\vphantom{\delta_{\s F_k}^l}$\newline
5.\ The formula $(\delta_{\s F_k}^l\tens\id_{S_{kj}})\delta_{\s F_j}^k=(\id_{\s F_l}\tens\delta_{lj}^k)\delta_{\s F_j}^l$ is derived from \ref{prop11} after long but straightforward computations.
\end{proof}

Let us consider the C*-algebra $B:=B_1\oplus B_2$ endowed with the continuous action $(\beta_B,\delta_B)$ (cf.\ \ref{prop37}).

\begin{prop}\label{prop16}
Let $\s F_1$ be a $\QG_1$-equivariant Hilbert $B_1$-module. Let $\s F_2:=\ind(\s F_1)$ be the induced $\QG_2$-equivariant Hilbert $B_2$-module. Consider the Hilbert $B$-module $\s F:=\s F_1\oplus\s F_2$. Denote by
$
\Pi_j^k:\Lin(B_k\tens S_{kj},\s F_k\tens S_{kj})\rightarrow\Lin(B\tens S,\s F\tens S) 
$ 
\index[symbol]{pk@$\Pi_j^k$}the linear extension of the canonical injection $\s F_k\tens S_{kj}\rightarrow\s F\tens S$.
Let us consider the linear maps $\delta_{\s F}:\s F\rightarrow\Lin(B\tens S,\s F\tens S)$ and $\beta_{\s F}:\GC^2\rightarrow\Lin(\s F)$ defined by:
\[
\delta_{\s F}(\xi):=\sum_{k,j=1,2}\Pi_j^k\circ\delta_{\s F_j}^k(\xi_j), \quad \xi=(\xi_1,\xi_2)\in\s F ;\quad 
\beta_{\s F}(\lambda,\mu):=\begin{pmatrix} \lambda & 0 \\ 0 & \mu\end{pmatrix}\!, \quad (\lambda,\mu)\in\GC^2.
\]
Then, the triple $(\s F,\beta_{\s F},\delta_{\s F})$ is a $\cal G$-equivariant Hilbert $B$-module.
\end{prop}

\begin{proof}
Let us consider $J_1:=\K(\s F_1\oplus B_1)$ (resp.\ $\K(\s F_2\oplus B_2)$) the linking $\QG_1$-C*-algebra (resp.\ linking $\QG_2$-C*-algebra) associated with $\s F_1$ (resp.\ $\s F_2$). Let $J_2:=\ind(J_1)$ be the induced $\QG_2$-C*-algebra. Let us consider $J:=J_1\oplus J_2$ endowed with the continuous action $(\beta_J,\delta_J)$ of $\cal G$ (see above). We denote $L:=\K(\s F\oplus B)$ the linking C*-algebra associated with $\s F$ and we identify $L=J_1\oplus\K(\s F_2\oplus B_2)$. We have an isomorphism of linking C*-algebras $f:=\id_{J_1}\oplus\tau:L\rightarrow J$ (\ref{prop11} (ii)). Let $(\beta_L,\delta_L)$ be the continuous action of $\cal G$ on $L$ obtained by transport of structure, {\it i.e.}:
\[
\delta_L(x):=(f^{-1}\tens\id_S)\delta_J(f(x)), \quad x\in L ; \quad \beta_L(n):=f^{-1}(\beta_J(n)),\quad n\in\GC^2.
\]
By straightforward computations, we show that $(\beta_L,\delta_L)$ is compatible with $(\beta_B,\delta_B)$ (cf.\ \ref{compcoact}) and we prove that
$
\delta_L(\iota_{\s F}(\xi))=\iota_{\s F\tens S}(\delta_{\s F}(\xi)),
$
for all $\xi\in\s F$. Therefore, the result follows from \ref{prop1} a) and \ref{propcont}. 
\end{proof}

\begin{prop}\label{prop17}
Let $(\s E,\beta_{\s E},\delta_{\s E})$ be a $\cal G$-equivariant Hilbert $A$-module. In the following, we use the notations of \ref{propdef4}. Let $j,k=1,2$ with $j\neq k$. Let 
\[
\widetilde{A}_j:={\rm Ind}_{\QG_k}^{\QG_j}(A_k,\delta_{A_k}^k) \quad \text{and} \quad \widetilde{\s E}_j:={\rm Ind}_{\QG_k}^{\QG_j}(\s E_k,\delta_{\s E_k}^k).
\] 
If $\xi\in\s E_j$, then we have $\delta_{\s E_j}^k(\xi)\in\widetilde{\s E}_j\subset\widetilde{\M}(\s E_k\tens S_{kj})$ and the map $\widetilde{\Pi}_j:\s E_j \rightarrow\widetilde{\s E}_j \, ; \, \xi  \mapsto \delta_{\s E_j}^k(\xi)$ is a $\QG_j$-equivariant unitary equivalence over $\widetilde{\pi}_j:A_j\rightarrow\widetilde{A}_j$ (cf.\ \ref{prop4}).\index[symbol]{pm@$\widetilde{\Pi}_j$}
\end{prop}

\begin{proof}
We have $\s E_j=[(\id_{\s E_j}\tens\omega)\delta_{\s E_k}^j(\xi) \; ; \; \omega\in\B(\s H_{jk})_*,\, \xi\in\s E_k]$ (cf.\ \ref{propdef4} (iv)) and for all $\xi\in\s E_j$ and $\omega\in\B(\s H_{jk})_*$ we have
\[
\delta_{\s E_j}^k(\id_{\s E_j}\tens\omega)\delta_{\s E_k}^j(\xi)=(\id_{\s E_k\tens S_{kj}}\tens\omega)(\delta_{\s E_j}^k\tens\id_{S_{jk}})\delta_{\s E_k}^j(\xi)=(\id_{\s E_k\tens S_{kj}}\tens\omega)\delta_{\s E_j}^{(k)}(\xi)
\]
(cf.\ \ref{propdef4} (v)), where $\delta_{\s E_j}^{(k)}(\xi):=(\id_{\s E_j}\tens\delta_{jj}^k)\delta_{\s E_j}^j(\xi)$. As a consequence, statement 1 is proved as well as the surjectivity of $\widetilde{\Pi}_j$. The fact that $\widetilde{\Pi}_j$ is a $\QG_j$-equivariant $\widetilde{\pi}_j$-compatible unitary operator is just a restatement of \ref{propdef4} (iii) and $(\id_{\s E_k}\tens\delta_{kj}^j)\delta_{\s E_j}^k=(\delta_{\s E_j}^k\tens\id_{S_{jj}})\delta_{\s E_j}^j$ (\ref{propdef4} (v)).
\end{proof}

\begin{thm}
Let ${\cal G}_{\QG_1,\QG_2}$ be a colinking measured quantum groupoid between two regular monoidally equivalent locally compact quantum groups $\QG_1$ and $\QG_2$. 
Let $j=1,2$. The map $(\s E,\beta_{\s E},\delta_{\s E})\mapsto(\s E_j,\delta_{\s E_j}^j)$ is a one-to-one correspondence up to unitary equivalence (cf.\ \ref{propdef4} and \ref{lem5} 1). The inverse map, up to unitary equivalence, is $(\s F_j,\delta_{\s F_j})\mapsto(\s F,\beta_{\s F},\delta_{\s F})$ (cf.\ \ref{prop16}, \ref{prop14} and \ref{lem5} 2).
\end{thm}

\begin{proof}
Let $A$ be a $\cal G$-C*-algebra and $\s E$ a $\cal G$-equivariant Hilbert $A$-module. Let us use all the notations introduced in \S\ref{ActColinkHilb}. Let us denote:
\begin{align*}
(B_1,\delta_{B_1})&:=(A_1,\delta_{A_1}^1), \; (B_2,\delta_{B_2}):=\ind(B_1,\delta_{B_1});\\
(\s F_1,\delta_{\s F_1})&:=(\s E_1,\delta_{\s E_1}^1), \;\, (\s F_2,\delta_{\s F_2}):=\ind(\s F_1,\delta_{\s F_1}).
\end{align*}
Let us endow the C*-algebra $B:=B_1\oplus B_2$ with the continuous action $(\beta_B,\delta_B)$ of $\cal G$ and $\s F:=\s F_1\oplus\s F_2$ with the structure of $\cal G$-equivariant Hilbert $B$-module $(\beta_{\s F},\delta_{\s F})$ (cf.\ \ref{prop37}, \ref{prop16}). Let $\psi_A:A\rightarrow B$ the canonical $\cal G$-equivariant *-isomorphism defined for all $a\in A$ by $\psi_A(a):=(q_{A,1}a,\widetilde{\pi}_2(q_{A,2}a))$ (cf.\ 4.10 \cite{BC}). Then, we consider the map $\Psi:\s E\rightarrow\s F$ given by
\[
\Psi(\xi):=(q_{\s E,1}\xi,\widetilde{\Pi}_2(q_{\s E,2}\xi)),\; \text{ for all }\; \xi\in\s E.
\]
It is clear from \ref{prop17} that $\Psi$ is a $\psi_A$-compatible unitary operator. Let us consider the $\cal G$-C*-algebras $K:=\K(\s E\oplus A)$ and $L:=\K(\s F\oplus B)$. Let $f:K\rightarrow L$ be the associated isomorphism of linking C*-algebras (cf.\ \ref{prop7}). In virtue of  \ref{prop10}, it only remains to prove that $f$ is $\cal G$-equivariant. We also consider the $\QG_1$-C*-algebra $J_1:=\K(\s F_1\oplus B_1)$ and the induced $\QG_2$-C*-algebra $J_2:=\ind(J_1)$. We recall that we have a canonical isomorphism $\tau:\K(\s F_2\oplus B_2)\rightarrow J_2$ (cf.\ \ref{prop11} (ii)). Let us endow the C*-algebra $J:=J_1\oplus J_2$ with the continuous action $(\beta_J,\delta_J)$ of $\cal G$. Therefore, it amounts to proving that the *-isomorphism $(\id_{J_1}\oplus\tau)f:K\rightarrow J$ is $\cal G$-equivariant (we identify $L=J_1\oplus\K(\s F_2\oplus B_2)$). We apply the notations of \S\ref{ActColinkHilb} to the $\cal G$-C*-algebra $K$ and identify $K_j:=q_{K,j}K=\K(\s E_j\oplus A_j)$ for $j=1,2$. Let us consider as above (by exchanging the roles of $A$ and $K$) the $\cal G$-equivariant *-isomorphism $\psi_K:K\rightarrow J$. By evaluating on elements of the form $\iota_{\s E}(\xi)$ for $\xi\in\s E$ and $\iota_A(a)$ for $a\in A$, it is staightforward to see that $(\id_{J_1}\oplus\tau)f=\psi_K$. 
\end{proof}

\section{Takesaki-Takai duality and equivariant Morita equivalence}

\setcounter{thm}{0}

\numberwithin{thm}{section}
\numberwithin{prop}{section}
\numberwithin{lem}{section}
\numberwithin{cor}{section}
\numberwithin{propdef}{section}
\numberwithin{nb}{section}
\numberwithin{nbs}{section}
\numberwithin{rk}{section}
\numberwithin{rks}{section}
\numberwithin{defin}{section}
\numberwithin{ex}{section}
\numberwithin{exs}{section}
\numberwithin{noh}{section}

In this section, we fix a measured quantum groupoid ${\cal G}=(N,M,\alpha,\beta,\Delta,T,T',\epsilon)$ on the finite-dimensional basis $N=\bigoplus_{1\leqslant l\leqslant k}\,{\rm M}_{n_l}(\GC)$ and we use all the notations introduced in \S\S\ \ref{MQGfinitebasis} and \ref{WHC*A}. We will use the notations and results of \S\S \ref{subsectionBiduality}, \ref{ActRegMQG} and \ref{SectEqHilb}.

\paragraph{Equivariant Hilbert bimodules and Morita equivalence.} In this paragraph, we introduce the notion of equivariant representation of a $\cal G$-C*-algebra on a Hilbert module acted upon by $\cal G$. We then introduce the notion of equivariant Morita equivalence. 

\begin{nb}\label{rk8}
Let $A$ and $B$ be C*-algebras. Let $\s E$ be a Hilbert $B$-module. If $\gamma:A\rightarrow\Lin(\s E)$ is a *-homomorphism then, up to the identification $\M(\K(\s E)\tens S)=\Lin(\s E\tens S)$, we can extend $\gamma\tens\id_S$ to a *-homomorphism $\gamma\tens\id_S:\widetilde\M(A\tens S)\rightarrow\Lin(\s E\tens S)$ (cf.\ \S \ref{sectionNotations}).
\end{nb}

As in 2.9 \cite{BS1}, we have:

\begin{defin}\label{defbimod}
Let $A$ and $B$ be two ${\cal G}$-C*-algebras, $\s E$ a Hilbert $B$-module, $(\beta_{\s E},\delta_{\s E})$ an action of $\cal G$ on $\s E$ and $\gamma:A\rightarrow\Lin(\s E)$ a *-representation. We say that $\gamma$ is ${\cal G}$-equivariant if we have:
\begin{enumerate}
\item $\delta_{\s E}(\gamma(a)\xi)=(\gamma\tens\id_S)(\delta_A(a))\circ\delta_{\s E}(\xi)$, for all $a\in A$ and $\xi\in\s E$;
\item $\beta_{\s E}(n^{\rm o})\circ\gamma(a)=\gamma(\beta_A(n^{\rm o})a)$, for all $n\in N$ and $a\in A$.\qedhere
\end{enumerate}
\end{defin}

\begin{rks}\label{rk11}
\begin{enumerate}
\item Provided that the second condition in the above definition is verified, the first condition is equivalent to: 
\begin{equation}\label{eq21}
\s V(\gamma(a)\tens_{\delta_B}1)\s V^*=(\gamma\tens\id_S)\delta_A(a), \; \text{ for all } \; a\in A,
\end{equation}
where $\s V\in\Lin(\s E\tens_{\delta_B}(B\tens S),\s E\tens S)$ denotes the isometry defined in \ref{prop27} a). Indeed, we can interpret it as follows: $\s V(\gamma(a)\tens_{\delta_B}1)=(\gamma\tens\id_S)(\delta_A(a))\s V$, for all $a\in A$. Moreover, for all $a\in A$ we have 
\begin{align*}
(\gamma\tens\id_S)(\delta_A(a))\s V\s V^*&=(\gamma\tens\id_S)(\delta_A(a))q_{\beta_{\s E}\alpha}\\
&=(\gamma\tens\id_S)(\delta_A(a)q_{\beta_A\alpha})\\
&=(\gamma\tens\id_S)\delta_A(a).
\end{align*}
Hence, $(\s V(\gamma(a)\tens_{\delta_B}1)=(\gamma\tens\id_S)(\delta_A(a))\s V \Leftrightarrow \s V(\gamma(a)\tens_{\delta_B}1)\s V^*=(\gamma\tens\id_S)\delta_A(a))$, for all $a\in A$.
\item We recall that the action $\delta_{\K(\s E)}$ of $\cal G$ on $\K(\s E)$ is defined by $\delta_{\K(\s E)}(k):=\s V(k\tens_{\delta_B}1)\s V^*$ for all $k\in\K(\s E)$. Hence, (\ref{eq21}) can be restated as follows: $\delta_{\K(\s E)}(\gamma(a))=(\gamma\tens\id_S)\delta_A(a)$ for all $a\in A$. In particular, if $\gamma$ is non-degenerate, then Definition \ref{defbimod} simply means that the *-homomorphism $\gamma:A\rightarrow\M(\K(\s E))$ is $\cal G$-equivariant (cf.\ \ref{defEquiHom}).
\item If $\gamma:A\rightarrow\Lin(\s E)$ is a non-degenerate *-representation such that
\[
\delta_{\s E}(\gamma(a)\xi)=(\gamma\tens\id_S)(\delta_A(a))\circ\delta_{\s E}(\xi),\; \text{ for all } \; a\in A \; \text{ and } \; \xi\in\s E,
\]
then we have $\beta_{\s E}(n^{\rm o})\circ\gamma(a)=\gamma(\beta_A(n^{\rm o})a)$ for all $n\in N$ and $a\in A$. Indeed, this will be inferred from \ref{rkEqMorph} and the previous remark.\qedhere
\end{enumerate}
\end{rks}

\begin{defin}(cf.\ \S 6 \cite{Rie74})
Let $A$ and $B$ be two C*-algebras. An imprimitivity $A$-$B$-bimodule is an $A$-$B$-bimodule $\s E$, which is a full left Hilbert $A$-module for an $A$-valued inner product ${}_A\langle\cdot,\,\cdot\rangle$ and a full right Hilbert $B$-module for a $B$-valued inner product $\langle\cdot,\,\cdot\rangle_B$ such that ${}_A\langle\xi,\,\eta\rangle\zeta=\xi\langle\eta,\,\zeta\rangle_B$ for all $\xi,\,\eta,\,\zeta\in\s E$.
\end{defin}

\begin{rks}
Let $A$ and $B$ be two C*-algebras and $\s E$ an imprimitivity $A$-$B$-bimodule. We recall that the norms defined by the inner products ${}_A\langle\cdot,\,\cdot\rangle$ on ${}_A\s E$ and $\langle\cdot,\,\cdot\rangle_B$ on $\s E_B$ coincide. We also recall that the left (resp.\ right) action of $A$ (resp.\ $B$) on $\s E$ defines a non-degenerate *-homomorphism $\gamma:A\rightarrow\Lin(\s E_B)$ (resp.\ $\rho:B\rightarrow\Lin({}_A\s E)$).
\end{rks}

\begin{defin}
Let $A$ and $B$ be two $\cal G$-C*-algebras. A $\cal G$-equivariant imprimitivity $A$-$B$-bimodule is an imprimitivity $A$-$B$-bimodule $\s E$ endowed with a continuous action of $\cal G$ on $\s E_B$ such that the left action $\gamma:A\rightarrow\Lin(\s E_B)$ is $\cal G$-equivariant.
\end{defin}

\begin{exs}\label{ex1} Let $A$ and $B$ be two $\cal G$-C*-algebras. 
\begin{enumerate}[label=(\roman*)]
\item $B$ is a $\cal G$-equivariant imprimitivity $B$-$B$-bimodule for the inner products given by ${}_B\langle x,\, y\rangle:=xy^*$ and $\langle x,\, y\rangle_B:=x^*y$ for all $x,y\in B$.
\item Let $\s E$ be a $\cal G$-equivariant Hilbert $B$-module. If $\s E$ is full, then $\s E$ is a $\cal G$-equivariant imprimitivity $\K(\s E)$-$B$-bimodule for the natural left action and the inner product given by ${}_{\K(\s E)}\langle\xi,\,\eta\rangle:=\theta_{\xi,\eta}$ for all $\xi,\,\eta\in\s E$. Conversely, if $\s E$ is a $\cal G$-equivariant imprimitivity $A$-$B$-bimodule, then the the left action $\gamma:A\rightarrow\Lin(\s E_B)$ induces an isomorphism of $\cal G$-C*-algebras $A\simeq\K(\s E_B)$.
\item Let $(J,\beta_J,\delta_J,e_1,e_2)$ be a linking $\cal G$-C*-algebra (cf.\ \ref{defLinkAlg}). Let $A:=e_1Je_1$ and $B:=e_2Je_2$ be the corner C*-algebras endowed with the continuous actions of $\cal G$ induced by $(\beta_J,\delta_J)$. Let us endow $\s E:=e_1Je_2$ with its structure of $\cal G$-equivariant Hilbert $B$-module (cf.\ \ref{rkLinkAlg}). Then, $\s E$ is a $\cal G$-equivariant imprimitivity $A$-$B$-module whose actions and inner products are defined as in (i). 
\item If $\s E$ is a $\cal G$-equivariant imprimitivity $A$-$B$-bimodule, then $\s E^*$ turns into a $\cal G$-equivariant imprimitivity $B$-$A$-bimodule for the actions and inner products given by the following formulas: $b\xi^*a:=(a^*\xi b^*)^*$, for $\xi^*\in\s E^*$, $a\in A$ and $b\in B$; ${}_B\langle\xi^*,\,\eta^*\rangle:=\langle\xi,\,\eta\rangle_B$ and $\langle\xi^*,\,\eta^*\rangle_A:={}_A\langle\xi,\,\eta\rangle$, for $\xi^*,\,\eta^*\in\s E^*$.\qedhere
\end{enumerate}
\end{exs}

\begin{prop}
Let $A$ and $B$ be $\cal G$-C*-algebras. The following statements are equivalent:
\begin{enumerate}[label=(\roman*)]
\item there exists a $\cal G$-equivariant imprimitivity $A$-$B$-bimodule;
\item there exists a full $\cal G$-equivariant Hilbert $B$-module $\s E$ such that we have an isomorphism $A\simeq\K(\s E)$ of $\cal G$-C*-algebras;
\item there exists a linking $\cal G$-C*-algebra $(J,\beta_J,\delta_J,e_1,e_2)$ such that we have $\cal G$-equivariant *-isomorphisms $A\simeq e_1Je_1$ and $B\simeq e_2Je_2$.\qedhere
\end{enumerate}
\end{prop}

\begin{proof}
This is a straightforward consequence of \ref{ex1} (ii), (iii), \ref{prop1} b) and \ref{rkLinkAlg}. 
\end{proof}

Now, we investigate the tensor product construction (cf.\ 2.10 \cite{BS1} for the quantum group case).

\begin{prop}\label{prop18} Let $C$ {\rm(}resp.\ $B${\rm)} be a $\cal G$-C*-algebra. Let $\s E_1$ {\rm(}resp.\ $\s E_2${\rm)} be a Hilbert module over $C$ {\rm(}resp.\ $B${\rm)} endowed with an action $(\beta_{\s E_1},\delta_{\s E_1})$ {\rm(}resp.\ $(\beta_{\s E_2},\delta_{\s E_2})${\rm)} of $\cal G$. Let $\gamma_2:C\rightarrow\Lin(\s E_2)$ be a $\cal G$-equivariant *-representation. Consider the Hilbert $B$-module $\s E:=\s E_1\tens_{\gamma_2}\s E_2$. Denote by
\[
\Delta(\xi_1,\xi_2):=(\delta_{\s E_1}(\xi_1)\tens_{\widetilde{\gamma}_2\tens\id_S}1)\circ\delta_{\s E_2}(\xi_2), \; \text{ for } \; \xi_1\in\s E_1 \; \text{ and } \; \xi_2\in\s E_2.
\]
We have $\Delta(\xi_1,\xi_2)\in\widetilde{\M}(\s E\tens S)$ for all $\xi_1\in\s E_1$ and $\xi_2\in\s E_2$. Let $\beta_{\s E}:N^{\rm o}\rightarrow\Lin(\s E)$ be the *-homomorphism defined by
\[
\beta_{\s E}(n^{\rm o}):=\beta_{\s E_1}(n^{\rm o})\tens_{\gamma_2}1, \; \text{ for all } \; n\in N.
\]
There exists a unique map $\delta_{\s E}:\s E\rightarrow\widetilde{\M}(\s E\tens S)$ defined by the formula
$
\delta_{\s E}(\xi_1\tens_{\gamma_2}\xi_2):=\Delta(\xi_1,\xi_2)
$
for $\xi_1\in\s E_1$ and $\xi_2\in\s E_2$ such that the pair $(\beta_{\s E},\delta_{\s E})$ is an action of $\cal G$ on $\s E$. 
\end{prop}

The operator $\delta_{\s E_1}(\xi_1)$ is considered here as an element of $\Lin(\widetilde{C}\tens S,\s E_1\tens S)\supset\widetilde{\M}(\s E_1\tens S)$. In particular, we have $\delta_{\s E_1}(\xi_1)\tens_{\widetilde{\gamma}_2\tens\id_S}1\in\Lin(\s E_2\tens S,\s E\tens S)$ since we use the identifications:
\begin{align}
(\widetilde{C}\tens S)\tens_{\widetilde{\gamma}_2\tens\id_S}(\s E_2\tens S)&=\s E_2\tens S,\; x\tens_{\widetilde{\gamma}_2\tens\id_S}\eta\mapsto(\widetilde{\gamma}_2\tens\id_S)(x)\eta;\\[0.5em]
(\s E_1\tens S)\tens_{\widetilde{\gamma}_2\tens\id_S}(\s E_2\tens S)&=\s E\tens S,\; (\xi_1\tens s)\tens_{\widetilde{\gamma}_2\tens\id_S}(\xi_2\tens t)\mapsto (\xi_1\tens_{\gamma_2}\xi_2)\tens st.\label{eqId}
\end{align}

\begin{proof}
The proof is basically the same as that of 2.10 \cite{BS1}. For example, we refer the reader to it for the proof of the fact that $\Delta(\xi_1,\xi_2)\in\widetilde{\M}(\s E\tens S)$ for all $\xi_1\in\s E_1$ and $\xi_2\in\s E$. Let $\s V_1$ and $\s V_2$ be the isometries associated with $\delta_{\s E_1}$ and $\delta_{\s E_2}$. Since $\s V_2$ intertwines the left actions $c\mapsto\gamma_2(c)\tens_{\delta_B}1$ and $(\gamma_2\tens\id_S)\delta_C$ of $C$, there exists a unique isometry $\widetilde{\s V}_2\in\Lin(\s E\tens_{\delta_B}(B\tens S),\s E_1\tens_{(\gamma_2\tens\id_S)\delta_C}(\s E_2\tens S))$ such that
\[
\widetilde{\s V}_2((\xi_1\tens_{\gamma_2}\xi_2)\tens_{\delta_B}x)=\xi_1\tens_{(\gamma_2\tens\id_S)\delta_C}\s V_2(\xi_2\tens_{\delta_B}x), \text{ for all } \; \xi_1\in\s E_1,\, \xi_2\in\s E_2  \text{ and }  x\in B\tens S.
\]
Let us prove that $\widetilde{\s V}_2$ is a unitary. It amounts to proving that $\widetilde{\s V}_2$ is surjective. Since ${\rm im}(\s V_2)={\rm im}(q_{\beta_{\s E_2}\alpha})$, we have 
$
{\rm im}(\widetilde{\s V}_2)=[\xi\tens_{(\gamma_2\tens\id_S)\delta_C}q_{\beta_{\s E_2}\alpha}\eta\,;\,\xi\in\s E_1,\,\eta\in\s E_2\tens S].
$
Let $\xi\in\s E_1$ and $\eta\in\s E_2\tens S$. Write $\xi=\xi'c$ with $\xi'\in\s E_1$ and $c\in C$. Since $\s V_2^{\vphantom{*}}\s V_2^*=q_{\s E_2\alpha}$, we have 
$
(\gamma_2\tens\id_S)\delta_C(c)q_{\beta_{\s E_2}\alpha}=(\gamma_2\tens\id_S)\delta_C(c)
$
(cf.\ \ref{rk11}). Hence, 
\begin{align*}
\xi\tens_{(\gamma_2\tens\id_S)\delta_C}q_{\beta_{\s E_2}\alpha}\eta&=\xi'\tens_{(\gamma_2\tens\id_S)\delta_C}(\gamma_2\tens\id_S)\delta_C(c)q_{\beta_{\s E_2}\alpha}\eta\\
&=\xi'\tens_{(\gamma_2\tens\id_S)\delta_C}(\gamma_2\tens\id_S)\delta_C(c)\eta\\
&=\xi\tens_{(\gamma_2\tens\id_S)\delta_C}\eta.
\end{align*}
Therefore we have shown that ${\rm im}(\widetilde{\s V}_2)=\s E_1\tens_{(\gamma_2\tens\id_S)\delta_C}(\s E_2\tens S)$, which proves that $\widetilde{\s V}_2$ is unitary. Let us identify
\begin{align*}
(\s E_1\tens_{\delta_C}(C\tens S))\tens_{\gamma_2\tens\id_S}(\s E_2\tens S)&\rightarrow\s E_1\tens_{(\gamma_2\tens\id_S)\delta_C}(\s E_2\tens S)\\ 
(\xi_1\tens_{\delta_C}x)\tens_{\gamma_2\tens\id_S}\eta&\mapsto \xi_1\tens_{(\gamma_2\tens\id_S)\delta_C}(\gamma_2\tens\id_S)(x)\eta\\[-.5em]
\end{align*}
and $(\s E_1\tens S)\tens_{\gamma_2\tens\id_S}(\s E_2\tens S)=\s E\tens S$ (cf.\ (\ref{eqId})).
Let 
\[
\s V:=(\s V_1\tens_{\gamma_2\tens\id_S}1)\widetilde{\s V}_2\in\Lin(\s E\tens_{\delta_B}(B\tens S),\s E\tens S).
\]
It follows from the formulas $\widetilde{\s V}_2^*\widetilde{\s V}_2^{\vphantom{*}}=1$, $\widetilde{\s V}_2^{\vphantom{*}}\widetilde{\s V}_2^*=1$, $\s V_1^*\s V_1^{\vphantom{*}}=1$ and $\s V_1^{\vphantom{*}}\s V_1^*=q_{\beta_{\s E_1}\alpha}$ that $\s V^*\s V=1$ and $\s V\s V^*=q_{\beta_{\s E_1}\alpha}\tens_{\gamma_2\tens\id_S}1=q_{\beta_{\s E}\alpha}$ (by definition of $\beta_{\s E}$).\newline
Let $n\in N$. On one hand, we have 
\begin{center}
$
\widetilde{\s V}_2(\beta_{\s E}(n^{\rm o})\tens_{\delta_B}1)=(\beta_{\s E_1}(n^{\rm o})\tens_{(\gamma_2\tens\id_S)\delta_C}1)\widetilde{\s V}_2
$ 
\end{center}
(by definition of $\beta_{\s E}$ and $\widetilde{\s V}_2$). On the other, we have 
\begin{center}
$
(\s V_1\tens_{\gamma_2\tens\id_S}1)(\beta_{\s E_1}(n^{\rm o})\tens_{(\gamma_2\tens\id_S)\delta_C}1)=((1_{\s E_1}\tens\beta(n^{\rm o}))\tens_{\gamma_2\tens\id_S}1)(\s V_1\tens_{\gamma_2\tens\id_S}1).
$ 
\end{center}
Hence, we have proved that $\s V(\beta_{\s E}(n^{\rm o})\tens_{\delta_B}1)=(1\tens\beta(n^{\rm o}))\s V$ for all $n\in N$. Exactly as in the proof of 2.10 \cite{BS1}, $\vphantom{\widetilde{\M}}$we have $\s VT_{\xi_1\tens_{\gamma_2}\xi_2}=\Delta(\xi_1,\xi_2)$ for all $\xi_1\in\s E_1$ and $\xi_2\in\s E_2$ (cf.\ \ref{not5} for the definition of $T_{\xi}$). In particular, $\s VT_{\xi}\in\widetilde{\M}(\s E\tens S)$ for all $\xi\in\s E$. It then follows from \ref{prop27} b) that the pair $(\beta_{\s E},\delta_{\s E})$, where $\delta_{\s E}:\s E\rightarrow\widetilde{\M}(\s E\tens S)$ is defined for all $\xi\in\s E$ by $\delta_{\s E}(\xi):=\s VT_{\xi}$, satisfies the conditions 1, 2, and 3 of Definition \ref{isometry}. $\vphantom{\widetilde{\M}}$The coassociativity condition of $\delta_{\s E}$ is derived from those of $\delta_{\s E_1}$ and $\delta_{\s E_2}$ exactly as in the proof of 2.10 \cite{BS1}.$\vphantom{\widetilde{\M}}$
\end{proof}

\begin{prop}\label{propleftaction}
We use all the notations and hypotheses of \ref{prop18}. If $A$ is a $\cal G$-C*-algebra and $\gamma_1:A\rightarrow\Lin(\s E_1)$ is a $\cal G$-equivariant *-representation, then $\gamma:A\rightarrow\Lin(\s E_1\tens_{\gamma_2}\s E_2)$ the *-representation defined by $\gamma(a):=\gamma_1(a)\tens_{\gamma_2}1$ for all $a\in A$ is $\cal G$-equivariant.
\end{prop}

\begin{proof}
Through the identification (\ref{eqId}), for all $x\in A\tens S$ the operator $(\gamma_1\tens\id_S)(x)\tens_{\widetilde{\gamma}_2\tens\id_S}1$ is identified to $(\gamma\tens\id_S)(x)$. This identification also holds for $x\in\widetilde{\M}(A\tens S)$ (by using the fact that any element of $S$ can be written as a product of two elements of $S$). In particular, for all $a\in A$ the operator $(\gamma_1\tens\id_S)\delta_A(a)\tens_{\widetilde{\gamma}_2\tens\id_S}1$ is identified to $(\gamma\tens\id_S)\delta_A(a)$. Hence, $\delta_{\s E}(\gamma(a)\xi)=(\gamma\tens\id_S)\delta_A(a)\circ\delta_{\s E}(\xi)$ for all $\xi\in\s E$ and $a\in A$ by definition of $\delta_{\s E}$. The relation $\beta_{\s E}(n^{\rm o})\circ\gamma(a)=\gamma(\beta_A(n^{\rm o})a)$ for $n\in N$ and $a\in A$ is straightforward.
\end{proof}

From now on, we assume the quantum groupoid $\cal G$ to be regular. We recall that any action of the quantum groupoid $\cal G$ on a Hilbert module is necessarily continuous (cf.\ \ref{corActReg}).

\begin{propdef}\label{propImpProd}
Let $A$, $C$ and $B$ be $\cal G$-C*-algebras. Let $\s E_1$ {\rm(}resp.\ $\s E_2${\rm)} be a $\cal G$-equivariant imprimitivity $A$-$C$-bimodule {\rm(}resp.\ $C$-$B$-bimodule{\rm)}. Denote by $\s E_1\tens_C\s E_2$ the internal tensor product $\s E_1\tens_{\gamma_2}\s E_2$, where $\gamma_2:C\rightarrow\Lin(\s E_2)$ is the $\cal G$-equivariant *-representation defined by the left action of $C$ on $\s E_2$. The Hilbert $B$-module $\s E_1\tens_C\s E_2$ endowed with the action of $\cal G$ defined in \ref{prop18} is a $\cal G$-equivariant imprimitivity $A$-$B$-bimodule for the left action of $A$ and the $A$-valued inner product defined by the formulas:
\begin{itemize}
\item $a(\xi_1\tens_C\xi_2):=a\xi_1\tens_C\xi_2$,\, for all $a\in A$, $\xi_1\in\s E_1$ and $\xi_2\in\s E_2$;
\item ${}_A\langle\xi_1\tens_C\xi_2,\, \eta_1\tens_C\eta_2\rangle:={}_A\langle\xi_1,\, \eta_1\,{}_C\langle\xi_2,\,\eta_2\rangle\rangle$,\, for all $\xi_1,\,\eta_1\in\s E_1$ and $\xi_2,\,\eta_2\in\s E_2$.\qedhere
\end{itemize}
\end{propdef}

\begin{proof}
It is known that $\s E_1\tens_C\s E_2$ is an imprimitivity $A$-$B$-bimodule. The rest of the proof is contained in \ref{prop18} and \ref{propleftaction}.
\end{proof}

\begin{prop}
Let $A$ and $B$ be $\cal G$-C*-algebras. Let $\s E$ be a $\cal G$-equivariant imprimitivity $A$-$B$-bimodule. Then, the map 
$
\s E^*\tens_A\s E\rightarrow B\; ; \; \xi^*\tens_A\eta \mapsto \langle\xi,\,\eta\rangle_B
$
defines an isomorphism of $\cal G$-equivariant imprimitivity $B$-$B$-bimodules.
\end{prop}

\begin{proof}
It is known that the map $\Phi:\s E^*\tens_A\s E\rightarrow B\; ; \; \xi^*\tens_A\eta \mapsto \langle\xi,\,\eta\rangle_B$ is an isomorphism of imprimitivity $B$-$B$-bimodules. The $\cal G$-equivariance of $\Phi$ is a restatement of the formula $\delta_B(\langle\xi,\,\eta\rangle_B)=\delta_{\s E}(\xi)^*\circ\delta_{\s E}(\eta)$ for $\xi,\,\eta\in\s E$.
\end{proof}

\begin{defin}
Let $A$ and $B$ be $\cal G$-C*-algebras. We say that $A$ and $B$ are $\cal G$-equivariantly Morita equivalent if there exists a $\cal G$-equivariant imprimitivity $A$-$B$-bimodule. The $\cal G$-equivariant Morita equivalence is a reflexive (\ref{ex1} (i)), symmetric (\ref{ex1} (iv)) and transitive (\ref{propImpProd}) relation on the class of $\cal G$-C*-algebras.\qedhere
\end{defin}

\paragraph{Biduality and equivariant Morita equivalence.} In this paragraph, the measured quantum groupoid $\cal G$ is assumed to be regular. Let us fix a $\cal G$-C*-algebras $A$. We show that there is a canonical $\cal G$-equivariant Morita equivalence between $A$ and the double crossed product $(A\rtimes{\cal G})\rtimes\widehat{\cal G}$.

\begin{nbs}
Denote by $\K:=\K(\s H)$ for short. Consider the Hilbert $A$-modules ${\cal E}_0:=A\tens\s H$ and $\er:=q_{\beta_A\widehat{\alpha}}(A\tens\s H)$. Let ${\cal V}\in\mathcal{L}(\s H\tens S)$ be the unique partial isometry such that $(\id_{\K}\tens L)({\cal V})=V$. 
\end{nbs}

\begin{prop}
There exists a unique bounded linear map $\delta_{\eo}:\eo\rightarrow\Lin(A\tens S,\eo\tens S)$ such that
$
\delta_{\eo}(a\tens\zeta)={\cal V}_{23}\delta_A(a)_{13}(1_A\tens\zeta\tens 1_S)
$,
for all $a\in A$ and $\zeta\in\s H$.
\end{prop}

\begin{proof}
If $B$ is a C*-algebra and $\s K$ a Hilbert space, we identify $\M(B)\tens\s K$ with a closed vector subspace of $\Lin(B,B\tens\s K)$. We have $(\delta_A\tens \id_{\s H})(\xi)\in\Lin(A\tens S,A\tens S\tens\s H)$ and $(\delta_A\tens \id_{\s H})(\xi)^*=(\delta_A\tens \id_{\s H^*})(\xi^*)$ for $\xi\in{\cal E}_0$. Let $\sigma\in\Lin(S\tens\s H,\s H\tens S)$ be the flip map. Denote by $\delta_{\eo}:\eo\rightarrow\Lin(A\tens S,\eo\tens S)$ the map defined by
$
\delta_{\eo}(\xi):={\cal V}_{23}\sigma_{23}(\delta_A\tens \id_{\s H})(\xi)
$
for $\xi\in{\cal E}_0$. It is clear that $\delta_{\eo}:\eo\rightarrow\Lin(A\tens S,\eo\tens S)$ is linear map satsifying the formula $\delta_{\eo}(a\tens\xi)={\cal V}_{23}\delta_A(a)_{13}(1_A\tens\xi\tens 1_S)$ for all $a\in A$ and $\zeta\in\s H$. 
\end{proof}

\begin{prop}\label{exehmprop1}
We have the following statements:
\begin{enumerate}
\item $\delta_{\eo}(\xi)^*\delta_{\eo}(\eta)=\delta_A(\langle q_{\beta_A\widehat{\alpha}}\xi,\, q_{\beta_A\widehat{\alpha}}\eta\rangle)$, for all $\xi,\eta\in\eo$; 
\item $\delta_{\eo}(\xi a)=\delta_{\eo}(\xi)\delta_A(a)$, for all $\xi\in\eo$ and $a\in A$;
\item $\delta_{\eo}(q_{\beta_A\widehat{\alpha}}\xi)=\delta_{\eo}(\xi)$, for all $\xi\in\eo$;
\item $\delta_{\eo}(\eo)(A\tens S)\subset\er\tens S$.\qedhere
\end{enumerate}
\end{prop}

\begin{proof}
1. Let $\xi,\eta\in\eo$, we have
$
\delta_{\eo}(\xi)^*\delta_{\eo}(\eta)=(\delta_A\tens \id)(\xi^*)\sigma_{23}^*{\cal V}_{23}^*{\cal V}_{23}^{\vphantom{*}}\sigma_{23}^{\vphantom{*}}(\delta_A\tens \id)(\eta).
$
We have $\sigma^*{\cal V}^*{\cal V}\sigma=q_{\beta\widehat{\alpha}}$. Let $n,n'\in N$. For all $a\in A$ and $\zeta\in\s H$, we have
\begin{align*}
(1_A\tens\beta(n^{\rm o})\tens \widehat{\alpha}(n'))(\delta_A\tens \id_{\s H})(a\tens\zeta)&
=(1_A\tens\beta(n^{\rm o}))\delta_A(a)\tens \widehat{\alpha}(n')\zeta\\
&=\delta_A(\beta_A(n^{\rm o})a)\tens \widehat{\alpha}(n')\zeta\\
&=(\delta_A\tens \id_{\s H})((\beta_A(n^{\rm o})\tens \widehat{\alpha}(n'))(a\tens\zeta)).
\end{align*}
Hence, 
$
(1_A\tens\beta(n^{\rm o})\tens \widehat{\alpha}(n'))(\delta_A\tens \id_{\s H})(\eta)=(\delta_A\tens \id_{\s H})((\beta_A(n^{\rm o})\tens \widehat{\alpha}(n'))\eta).
$
It then follows that 
$
\sigma_{23}^*{\cal V}_{23}^*{\cal V}_{23}\sigma_{23}(\delta_A\tens \id_{\s H})(\eta)=(\delta_A\tens \id_{\s H})(q_{\beta_A\widehat{\alpha}}\eta).
$
We finally have
\begin{align*}
\delta_{\eo}(\xi)^*\delta_{\eo}(\eta)&=(\delta_A\tens \id_{\s H^*})(\xi^*)(\delta_A\tens \id_{\s H})(q_{\beta_A\widehat{\alpha}}\eta)\\
&=\delta_A(\langle\xi,\, q_{\beta_A\widehat{\alpha}}\eta\rangle)\\
&=\delta_A(\langle q_{\beta_A\widehat{\alpha}}\xi,\, q_{\beta_A\widehat{\alpha}}\eta\rangle),
\end{align*}
where the last equality follows from the fact that $q_{\beta_A\widehat{\alpha}}\in\Lin(\eo)$ is a self-adjoint projection.

\medskip

2. Let us fix $a,b\in A$ and $\zeta\in\s H$. We have $\delta_A(a)_{13}(1_A\tens\zeta\tens 1_S)=(1_A\tens\zeta\tens 1_S)\delta_A(a)$ in $\Lin(A\tens S,{\cal E}_0\tens S)$. Hence,
$
\delta_{\eo}((b\tens\zeta)a)=\delta_{\eo}(b\tens\zeta)\delta_A(a).
$

\medskip

3. Let $a\in A$ and $\zeta\in\s H$. For all $n,n'\in N$, we have
\begin{align*}
\delta_{\eo}(\beta_A(n^{\rm o})a\tens \widehat\alpha(n')\zeta)
&={\cal V}_{23}(1_A\tens\widehat\alpha(n')\tens 1_S)\delta_A(\beta_A(n^{\rm o})a)_{13}(1_A\tens\zeta\tens 1_S)\\
&={\cal V}_{23}(1_A\tens\widehat{\alpha}(n')\tens\beta(n^{\rm o}))\delta_A(a)_{13}(1_A\tens\zeta\tens 1_S).
\end{align*}
Hence, $\delta_{\eo}(q_{\beta_A\widehat{\alpha}}(a\tens\zeta))={\cal V}_{23}q_{\widehat{\alpha}\beta,23}\delta_A(a)_{13}(1_A\tens\zeta\tens 1_S)=\delta_{\eo}(a\tens\zeta)$.

\medskip

4. It suffices to show that $q_{\beta_A\widehat{\alpha},12}\delta_{\eo}(\xi)=\delta_{\eo}(\xi)$ for all $\xi\in\eo$. We recall (cf.\ \ref{prop34}) that $(\widehat\alpha(n)\tens 1_S){\cal V}={\cal V}(1_{\K}\tens\alpha(n))$ for all $n\in N$. Hence, $q_{\beta_A\widehat{\alpha},12}{\cal V}_{23}={\cal V}_{23}q_{\beta_A\alpha,13}$. It then follows from $q_{\beta_A\alpha}=\delta_A(1_A)$ that $q_{\beta_A\widehat{\alpha},12}{\cal V}_{23}\delta_A(a)_{13}={\cal V}_{23}\delta_A(a)_{13}$ for all $a\in A$. Hence, $q_{\beta_A\widehat{\alpha},12}\delta_{\eo}(a\tens\zeta)=\delta_{\eo}(a\tens\zeta)$ for all $a\in A$ and $\zeta\in\s H$.
\end{proof}

\begin{nbs}\label{not1}
According to the previous proposition, $\delta_{\eo}$ restricts to a linear map
\[\delta_{\er}:\er\rightarrow\Lin(A\tens S,\er\tens S),\] 
which satisfies the following statements:
\begin{itemize}
\item $\delta_{\er}(\xi)^*\delta_{\er}(\eta)=\delta_A(\langle\xi,\,\eta\rangle)$, for all $\xi,\,\eta\in\er$;
\item $\delta_{\er}(\xi a)=\delta_{\er}(\xi)\delta_A(a)$, for all $\xi\in\er$ and $a\in A$.\qedhere
\end{itemize}
Since $[\widehat{\alpha}(n'),\,\beta(n^{\rm o})]=0$ for all $n,n'\in N$, we have $[1_A\tens\beta(n^{\rm o}),\,q_{\beta_A\widehat{\alpha}}]=0$ for all $n\in N$. We then have a non-degenerate *-homomorphism
\[
\beta_{\er}:N^{\rm o}\rightarrow\Lin(\er)\; ; \; n \mapsto \restr{(1_A\tens\beta(n^{\rm o}))}{\er}\!.
\]
Since $\beta$ and $\widehat{\alpha}$ commute pointwise and ${\cal V}{\cal V}^*=q_{\beta\alpha}$, we have $[{\cal V}_{23}^{\phantom{*}}{\cal V}_{23}^*,\, q_{\beta_A\widehat{\alpha},12}]=0$. Hence, $q_{\beta_{\er}\alpha}=\restr{{\cal V}_{23}^{\phantom{*}}{\cal V}_{23}^*}{\er\tens S}\,\in\Lin(\er\tens S)$.
\end{nbs}

\begin{prop}
We have the following statements:
\begin{enumerate}
\item $\delr(\er)\subset\widetilde{\M}(\er\tens S)$;
\item $[\delr(\er)(A\tens S)]=q_{{\beta_{\er}}\alpha}(\er\tens S)$; 
$\phantom{\widetilde{\M}}$
\item $\delr(\betr(n^{\rm o})\xi)=(1_{\er}\tens\beta(n^{\rm o}))\delr(\xi)$, for all $\xi\in\er$ and $n\in N$.\qedhere
\end{enumerate}
\end{prop}

\begin{proof}
1. Let us prove that $\delr(\xi)(1_A\tens s)\in\er\tens S$ for all $\xi\in\er$ and $s\in S$. It amounts to proving that $\delo(\xi)(1_A\tens s)\in\eo\tens S$ for all $\xi\in\eo$ and $s\in S$ (cf.\ \ref{exehmprop1} 3, 4). Let $a\in A$ and $\zeta\in\s H$. It follows from the relation $\delta_A(A)(1_A\tens S)\subset A\tens S$ that $\delo(a\tens\zeta)(1_A\tens s)=(1_A\tens{\cal V}(\zeta\tens 1_S))\delta_A(a)(1_A\tens s)$ is the norm limit of finite sums of the form
$
\sum_{i}a_i\tens{\cal V}(\zeta\tens s_i),
$
where $a_i\in A$ and $s_i\in S$. Hence, $\delo(a\tens\zeta)(1_A\tens s)\in\eo\tens S$.\hfill\break
Now, let us prove that $(1_{\er}\tens y)\delr(\xi)\in\er\tens S$ for all $\xi\in\er$ and $y\in S$. This also amounts to proving that $(1_{\eo}\tens y)\delta_{\eo}(\xi)\in\eo\tens S$ for all $\xi\in\eo$ and $y\in S$. Let $a\in A$, $\zeta\in\s H$ and $y\in S$, we have
$
(1_{\eo}\tens y)\delta_{\eo}(a\tens\zeta)=(1_A\tens(1_{\s H}\tens y){\cal V}(\zeta\tens 1_S))\delta_A(a).
$
Write $\zeta=\rho(x)\eta$ with $x\in\widehat{S}$ and $\eta\in\s H$. We have
$
(1_{\K}\tens y){\cal V}(\zeta\tens 1_S)=(\rho\tens \id_S)((1_{\widehat{S}}\tens y)V(x\tens 1_S))(\eta\tens 1_S).
$
Since $\cal G$ is regular, we have $(1_{\widehat{S}}\tens y)V(x\tens 1_S)\in\widehat{S}\tens S$ (cf.\ \ref{corReg} 2). Hence, $(1_{\eo}\tens y)\delta_{\eo}(a\tens\zeta)$ is the norm limit of finite sums of elements of the form $(1_A\tens\rho(x')\eta\tens y')\delta_A(a)$ with $x'\in\widehat{S}$ and $y'\in S$. Hence, $(1_{\eo}\tens y)\delta_{\eo}(a\tens\zeta)\in\eo\tens S$ since $(1_A\tens S)\delta_A(A)\subset A\tens S$.

\medskip

2. Since ${\cal V}{\cal V}^*{\cal V}={\cal V}$, we have ${\cal V}_{23}^{\phantom{*}}{\cal V}_{23}^*\delta_{\eo}(\xi)=\delta_{\eo}(\xi)$ for all $\xi\in\eo$. It then follows that $q_{{\beta_{\er}}\alpha}\delr(\xi)=\delr(\xi)$, for all $\xi\in\er$. By the first statement, we then obtain
\[
\delr(\er)(A\tens S)\subset q_{{\beta_{\er}}\alpha}(\er\tens S).
\]
Conversely, let $a\in A$, $\zeta\in\s H$ and $s\in S$. Since ${\cal V}_{23}q_{\beta_A\alpha,13}=q_{\beta_A\widehat{\alpha},12}{\cal V}_{23}$, we have
\[
q_{{\beta_{\er}}\alpha}(q_{\beta_A\widehat{\alpha}}(a\tens\zeta)\tens s)={\cal V}_{23}{\cal V}_{23}^*q_{\beta_A\widehat{\alpha},12}(a\tens\zeta\tens s)={\cal V}_{23}q_{\beta_A\alpha,13}(a\tens{\cal V}^*(\zeta\tens s)).
\]
Hence, $q_{{\beta_{\er}}\alpha}(q_{\beta_A\widehat{\alpha}}(a\tens\zeta)\tens s)$ is the norm limit of finite sums of elements of the form:
\[
{\cal V}_{23}q_{\beta_A\alpha,13}(a\tens\zeta'\tens s')={\cal V}_{23}(q_{\beta_A\alpha}(a\tens s'))_{23}(1_A\tens\zeta'\tens 1_S),\quad \text{where } \zeta'\in\s H,\ s'\in S.
\]
By continuity of the action $(\delta_A,\beta_A)$, ${\cal V}_{23}q_{\beta_A\alpha,13}(a\tens\zeta'\tens s')$ is the norm limit of finite sums of the form
$
\sum_i{\cal V}_{23}\delta_A(a_i)_{13}(1_A\tens\zeta'\tens s_i)=\sum_i\delr(q_{\beta_A\widehat{\alpha}}(a_i\tens\zeta'))(1_A\tens s_i),
$
where $a_i\in A$ and $s_i\in S$.
As a result, we have
\[
q_{{\beta_{\er}}\alpha}(q_{\beta_A\widehat{\alpha}}(a\tens\zeta)\tens s)\in[\delr(\er)(A\tens S)]
\] 
for all $a\in A$, $\zeta\in\s H$ and $s\in S$. Hence, $q_{\beta_{\er}\alpha}(\er\tens S)\subset[\delr(\er)(A\tens S)]$.

\medskip

3. Let $\xi=q_{\beta_A\widehat{\alpha}}(a\tens\zeta)$, with $a\in A$ and $\zeta\in\s H$. We have 
\[
\betr(n^{\rm o})\xi=(1_A\tens \beta(n^{\rm o}))q_{\beta_A\widehat{\alpha}}(a\tens\zeta)=q_{\beta_A\widehat{\alpha}}(a\tens \beta(n^{\rm o})\zeta).
\]
Moreover, we have ${\cal V}(\beta(n^{\rm o})\tens 1_S)=(1_{\K}\tens\beta(n^{\rm o})){\cal V}$ for all $n\in N$ (cf.\ \ref{prop34}). It then follows that
\[
\delr(\betr(n^{\rm o})\xi)\!=\!\delo(a\tens \beta(n^{\rm o})\zeta)\!=\!(1_{\eo}\!\tens\beta(n^{\rm o}))\delo(a\tens\zeta)\!=\!(1_{\er}\!\tens\beta(n^{\rm o}))\delr(\xi).\qedhere
\]
\end{proof}

Consequently, $\delr\tens \id_S$ and $\id_{\er}\tens\delta$ extend to linear maps from $\Lin(A\tens S,\er\tens S)$ to $\Lin(A\tens S\tens S,\er\tens S\tens S)$ (cf.\ \ref{rk9}) and we have: 
\begin{align*}
(\delr\tens \id_S)(T)(\delta_A\tens \id_S)(x)&=(\delr\tens \id_S)(Tx);\\[0.3cm]
(\id_{\er}\tens\delta)(T)(\id_A\tens\delta)(x)&=(\id_{\er}\tens\delta)(Tx);
\end{align*}
for all $x\in A\tens S$ and $T\in\Lin(A\tens S,\er\tens S)$.

\begin{prop}
For all $\xi\in\er$,
$
(\delr\tens \id_S)\delr(\xi)=(\id_{\er}\tens\delta)\delr(\xi).
$
\end{prop}

\begin{proof}
Let $a\in A$, $\zeta\in\s H$ and $x\in A\tens S$. Let $\xi:=q_{\beta_A\widehat{\alpha}}(a\tens\zeta)$. We have
\[
(\delr\tens \id_S)\delr(\xi)(\delta_A \tens \id_S)(x)=(\delo\tens \id_S)({\cal V}_{23}(\delta_A(a)x)_{13}(1_A\tens\zeta\tens 1_S)).
\]
For all $b\in A$, $\zeta'\in\s H$ and $s'\in S$, we have
\[
(\delo\tens \id_S)(b\tens\zeta'\tens s')={\cal V}_{23}\delta_A(b)_{13}(1_A\tens\zeta'\tens 1_S\tens s').
\]
Hence, $(\delo\tens \id_S)(b\tens X)={\cal V}_{23}\delta_A(b)_{13}X_{24}\in\Lin(A\tens S\tens S,A\tens\s H\tens S\tens S)$ for all $b\in A$ and $X\in\s H\tens S$. In particular, we have
\begin{align*}
(\delo\tens \id_S)({\cal V}_{23}(b\tens\zeta\tens s))&=(\delo\tens \id_S)(b\tens{\cal V}(\zeta\tens s))\\
&={\cal V}_{23}\delta_A(b)_{13}{\cal V}_{24}(1_A\tens\zeta\tens 1_S\tens s)\\
&={\cal V}_{23}{\cal V}_{24}\delta_A(b)_{13}(1_A\tens\zeta\tens 1_S\tens s).
\end{align*}
However, we have $(\id_{\K}\tens\delta)({\cal V})={\cal V}_{12}{\cal V}_{13}$. Hence, ${\cal V}_{23}{\cal V}_{24}=(\id_{A\tens\K}\tens\delta)({\cal V}_{23})$. Moreover, we have
$
\delta_A(b)_{13}(1_A\tens\zeta\tens 1_S\tens s)=(\delta_{A,13}\tens \id_S)(b\tens s)(1_A\tens\zeta\tens 1_S\tens 1_S),
$
for all $b\in A$ and $s\in S$, where $\delta_{A,13}:A\rightarrow\M(A\tens\K\tens S)$ is the strictly continuous *-homomorphism defined by $\delta_{A,13}(a):=\delta_A(a)_{13}$ for all $a\in A$. As a result, for all $Y\in A\tens S$ we have
\[
(\delo\tens \id_S)({\cal V}_{23}Y_{13}(1_A\tens\zeta\tens 1_S))=(\id_{A\tens\K}\tens\delta)({\cal V}_{23})(\delta_{A,13}\tens \id_S)(Y)(1_A\tens\zeta\tens 1_S\tens 1_S).
\]
In particular, we have 
\begin{multline*}
(\delo\tens \id_S)(\delo(a\tens\zeta)x)=(\id_{A\tens\K}\tens\delta)({\cal V}_{23})(\delta_{A,13}\tens \id_S)(\delta_A(a)x)(1_A\tens\zeta\tens 1_S\tens 1_S)\\
=(\id_{A\tens\K}\tens\delta)({\cal V}_{23})(\delta_{A,13}\tens \id_S)\delta_A(a)(1_A\tens\zeta\tens 1_S\tens 1_S)(\delta_A\tens \id_S)(x).
\end{multline*}
Moreover, we have $(\delta_{A,13}\tens \id_S)\delta_A=(\id_{A\tens\K}\tens\delta)\delta_{A,13}$. Hence,
\[
(\delr\tens \id_S)\delr(\xi)x=(\id_{A\tens\K}\tens\delta)({\cal V}_{23}\delta_A(a)_{13})(1_A\tens\zeta\tens 1_S\tens 1_S)x,
\]
for all $x\in q_{\beta_A\alpha,12}(A\tens S\tens S)$. In particular, if $x\in q_{\beta_A\alpha,12}q_{\beta\alpha,23}(A\tens S\tens S)$ we have
\begin{align*}
(\delr\tens \id_S)\delr(\xi)x &=(\id_{A\tens\K}\tens\delta)({\cal V}_{23}\delta_A(a)_{13})q_{\beta\alpha,34}(1_A\tens\zeta\tens 1_S\tens 1_S)x\\
&=(\id_{A\tens\K}\tens\delta)({\cal V}_{23}\delta_A(a)_{13})(\id_{\eo}\tens\delta)(1_A\tens\zeta\tens 1_S)x\\
&=(\id_{\eo}\tens\delta)\delo(a\tens\zeta)x\\
&=(\id_{\er}\tens\delta)\delr(\xi)x.
\end{align*}
Hence, $(\delr\tens \id_S)\delr(\xi)=(\id_{\er}\tens\delta)\delr(\xi)$.
\end{proof}

Now, we can assemble the previous results (see also \ref{corActReg}) in the statement below.

\begin{prop}
The triple $(\er,\betr,\delr)$ is a ${\cal G}$-equivariant Hilbert $A$-module.
\end{prop}

Let $D$ be the bidual $\cal G$-C*-algebra of $A$. We have a canonical $\cal G$-equivariant *-isomorphism $\phi:(A\rtimes{\cal G})\rtimes{\widehat{\cal G}}\rightarrow D$ of ${\cal G}$-C*-algebras (cf.\ \ref{IsoTT}). Let $j_D:\M(D)\rightarrow\Lin(\eo)$ be the unique faithful continuous *-homomorphism for the strict/*-strong topologies such that $j_D(1_D)=q_{\beta_A\widehat{\alpha}}$.

\begin{prop}\label{prop33}
The *-representation $(A\rtimes{\cal G})\rtimes{\widehat{\cal G}}\rightarrow\Lin(\er)\,;\, x\mapsto \restr{\phi(x)}{\er}$ is $\cal G$-equivariant.
\end{prop}

\begin{proof}
We have to prove that $\delo(d\xi)=(j_D\tens \id_S)(\delta_D(d))\circ\delo(\xi)$ for all $d\in D$ and $\xi\in{\cal E}_0$ (cf.\ \ref{rk11} 3 and \ref{exehmprop1} 3).
Let us prove it in three steps:\hfill\break
$\bullet$ Let $b\in A$, $x\in\widehat{S}$ and $\zeta\in\s H$. We have
\[
\delo(b\tens\lambda(x)\zeta)=(1_A\tens{\cal V}(\lambda(x)\tens 1_S))\delta_A(b)_{13}(1_A\tens\zeta\tens 1_S).
\]
However, $[{\cal V},\,\lambda(x)\tens 1_S]=0$ (as $\lambda(\widehat{S})\subset\widehat{M}$ and $V\in\widehat{M}'\tens M$). Hence,
\[
\delo(b\tens\lambda(x)\zeta)=(1_A\tens\lambda(x)\tens 1_S)\delo(b\tens\zeta)
\]
and then $\delo((1_A\tens\lambda(x))\xi)=(1_A\tens\lambda(x)\tens 1_S)\delo(\xi)$, for all $x\in\widehat{S}$ and $\xi\in\eo$.\hfill\break
$\bullet$ Let $y\in S$. Since $L(y)\in M\subset\widehat{\alpha}(N)'$, we have 
\[
{\cal V}(L(y)\tens 1_S)\!=\!{\cal V}q_{\widehat{\alpha}\beta}(L(y)\tens 1_S)\!=\!{\cal V}(L(y)\tens 1_S)q_{\widehat{\alpha}\beta}\!=\!{\cal V}(L(y)\tens 1_S){\cal V}^*{\cal V}\!=\!(L\tens \id_S)\delta(y){\cal V}.
\]
For all $b\in A$ and $\zeta\in\s H$, we have
\begin{align*}
\delo((1_A\tens L(y))(b\tens\zeta))&=(1_A\tens{\cal V}(L(y)\tens 1_S))\delta_A(b)_{13}(1_A\tens\zeta\tens 1_S)\\
&=(1_A\tens(L\tens \id_S)\delta(y))\delo(b\tens\zeta).
\end{align*}
Hence, $\delo((1_A\tens L(y))\xi)=(1_A\tens(L\tens \id_S)\delta(y))\delo(\xi)$ for all $y\in S$ and $\xi\in\eo$.

\medskip

In virtue of the first two steps, for all $\xi\in\er$ we have 
\[
\delr((1_A\tens\lambda(x)L(y))\xi)=(1_A\tens\lambda(x)\tens 1_S)(1_A\tens(L\tens \id_S)\delta(y))\delr(\xi).
\]
$\bullet$ Let $s\in S$. We have (cf.\ \ref{inifinproj})
\[
(R(s)\tens 1)V=(U\tens 1)\Sigma(1\tens L(s))\Sigma(U^*\tens 1)V\\
=(U\tens 1)\Sigma(1\tens L(s))W\Sigma(U^*\tens 1).
\]
Besides, $(1\tens L(s))W=(1\tens L(s))WW^*W=WW^*(1\tens L(s))W=W\delta(s)$ since we have $WW^*=q_{\alpha\widehat{\beta}}$ and $L(s)\in M\subset\widehat{\beta}(N^{\rm o})'$. Therefore, since $(U\tens 1)\Sigma W=V(U\tens 1)\Sigma$ we have
$
(R(s)\tens 1)V=V\Sigma(1\tens U)\delta(s)(1\tens U^*)\Sigma.
$
Hence,
$
(R(s)\tens 1_S){\cal V}={\cal V}\sigma(\id_S\tens R)(\delta(s))\sigma^*
$
for all $s\in S$. We then have $((\id_A\tens R)(x)\tens 1_S){\cal V}_{23}={\cal V}_{23}\sigma_{23}(\id_{A\tens S}\tens R)((\id_A\tens\delta)(x))\sigma_{23}^*$ for all $x\in A\tens S$. But, since $R$ and $\delta$ are strictly continuous this equality also holds for all $x\in\M(A\tens S)$. In particular, we have
$
\pi_R(a)_{12}{\cal V}_{23}={\cal V}_{23}\sigma_{23}(\id_{A\tens S}\tens R)(\delta_A^2(a))\sigma_{23}^*
$
for all $a\in A$. By coassociativity of $\delta_A$, we have
\[
\pi_R(a)_{12}{\cal V}_{23}={\cal V}_{23}\sigma_{23}(\delta_A\tens \id_{\K})(\pi_R(a))\sigma_{23}^*,\quad \text{for all }a\in A.
\]
It then follows that
\begin{align*}
\pi_R(a)_{12}\delo(\xi)&={\cal V}_{23}\sigma_{23}(\delta_A\tens \id_{\K})(\pi_R(a))(\delta_A\tens \id_{\s H})(\xi)\\
&={\cal V}_{23}\sigma_{23}(\delta_A\tens \id_{\s H})(\pi_R(a)\xi)\\
&=\delo(\pi_R(a)\xi),
\end{align*}
for all $a\in A$ and $\xi\in\eo$. In particular, $\pi_R(a)_{12}\delo(\xi_0)=\delo(\pi_R(a)\xi_0)$ for all $a\in A$ and $\xi\in\eo$.

\medskip

We have proved that for all $ a\in A$, $x\in\widehat{S}$, $y\in S$ and $\xi\in\eo$, we have
\[
\delta_{\eo}(\pi_R(a)(1_A\tens\lambda(x)L(y))\xi)=\pi_R(a)_{12}(1_A\tens\lambda(x)\tens 1_S)(1_A\tens(L\tens\id_S)\delta(y))\delta_{\eo}(\xi).
\]
However, for all $a\in A$, $x\in\widehat{S}$ and $y\in S$ we have (cf.\ 3.37 d) \cite{BC})
\[
(j_D\tens\id_S)\delta_D(\pi_R(a)(1_A\tens\lambda(x)L(y)))=\pi_R(a)_{12}(1_A\tens\lambda(x)\tens 1_S)(1_A\tens(L\tens \id_S)\delta(y)).
\]
If $d=\pi_R(a)(1_A\tens\lambda(x)L(y))\in D$, where $a\in A$, $x\in\widehat{S}$ and $y\in S$, we have proved that $\delta_{\eo}(d\xi)=(j_D\tens\id_S)(\delta_D(d))\circ\delta_{\eo}(\xi)$ for all $\xi\in\eo$. Thus, the statement is proved since $D=[\pi_R(a)(1_A\tens\lambda(x)L(y))\,;\,a\in A,\,x\in\widehat{S},\,y\in S]$.
\end{proof}

\begin{thm}\label{prop21}
The $\cal G$-C*-algebras $(A\rtimes{\cal G})\rtimes\widehat{\cal G}$ and $A$ are Morita equivalent via the $\cal G$-equivariant imprimitivity $(A\rtimes{\cal G})\rtimes\widehat{\cal G}$-$A$-bimodule $\er$.
\end{thm}

\begin{proof}
Let us prove that the Hilbert $A$-module ${\cal E}_{A,R}$ is full. Fix $x\in A$ and write $x=a^*b$ for $a,b\in A$. There exists $\omega\in\B(\s H)_*$ such that $(\id_A\tens\omega)(q_{\beta_A\widehat{\alpha}})=1_A$. Hence, $x$ is the norm limit of finite sums of elements of the form $a^*(\id_A\tens\omega_{\xi,\eta})(q_{\beta_A\widehat{\alpha}})b$, where $\xi,\eta\in\s H$. However, for all $\xi,\eta\in\s H$ we have 
\[
a^*(\id_A\tens\omega_{\xi,\eta})(q_{\beta_A\widehat{\alpha}})b=(a\tens\xi)^*q_{\beta_A\widehat{\alpha}}(b\tens\eta)=\langle q_{\beta_A\widehat{\alpha}}(a\tens\xi),\, q_{\beta_A\widehat{\alpha}}(b\tens\eta)\rangle.
\] 
Hence, $A=[\langle\xi,\, \eta\rangle\,;\, \xi,\eta\in{\cal E}_{A,R}]$. Now, we recall that $D=q_{\beta_A\widehat{\alpha}}(A\tens\K)q_{\beta_A\widehat{\alpha}}$ (cf.\ \ref{BidulityTheo}). It is easily seen that the left action of $(A\rtimes{\cal G})\rtimes\widehat{\cal G}$ (cf.\ \ref{prop33}) induces a $\cal G$-equivariant *-isomorphism $(A\rtimes{\cal G})\rtimes\widehat{\cal G}\simeq\K({\cal E}_{A,R})$.
\end{proof} 

\section{Appendix}\label{appendix}

	\subsection[Normal linear forms, weights and operator-valued weights]{Normal linear forms, weights and operator-valued weights on von Neumann algebras \cite{Co2}}\label{integration}

\setcounter{thm}{0}

\numberwithin{thm}{subsection}
\numberwithin{prop}{subsection}
\numberwithin{lem}{subsection}
\numberwithin{cor}{subsection}
\numberwithin{propdef}{subsection}
\numberwithin{nb}{subsection}
\numberwithin{nbs}{subsection}
\numberwithin{rk}{subsection}
\numberwithin{rks}{subsection}
\numberwithin{defin}{subsection}
\numberwithin{ex}{subsection}
\numberwithin{exs}{subsection}
\numberwithin{noh}{subsection}

Let $M$ be a von Neumann algebra. Denote by $M_*$ (resp.\ $M_*^+$) the Banach space (resp.\ positive cone) of the normal linear forms (resp.\ positive normal linear forms) on $M$. Let $\omega\in M_*$ and $a,b\in M$. Denote by $a\omega\in M_*$ and $\omega b\in M_*$ the normal linear functionals on $M$ given for all $x\in M$ by: 
\[
(a\omega)(x):=\omega(xa);\quad (\omega b)(x):=\omega(bx).
\]
We have $a'(a\omega)=(a'a)\omega$ and $(\omega b)b'=\omega(bb')$, for all $a,a,b,b'\in M$. We also denote 
\[
a\omega b:=a(\omega b)=(a \omega) b ; \quad \omega_a:=a^*\omega a.
\]
If $\omega\in M_*^+$, then $\omega_a\in M_*^+$. Note that $(\omega_a)_b=\omega_{ab}$ for all $a,b\in M$. If $\omega\in M_*$ we define $\overline{\omega}\in M_*$\index[symbol]{oa@$\overline{\omega}$} by setting
\[
\overline{\omega}(x):=\overline{\omega(x^*)},\quad \text{for all } x\in M.
\]
Let $\s H$ be a Hilbert space and let us fix $\xi,\eta\in\s H$. Denote by $\omega_{\xi,\eta}\in\B(\s H)_*$\index[symbol]{ob@$\omega_{\xi,\eta}$} the normal linear form defined by
\[
\omega_{\xi,\eta}(x):=\langle\xi,\, x\eta\rangle,\quad \text{for all } x\in\B(\s H).
\]
Note that we have $\overline{\omega}_{\xi,\eta}=\omega_{\eta,\xi}$, $a\omega_{\xi,\eta}=\omega_{\xi,a\eta}$ and $\omega_{\xi,\eta}a=\omega_{a^*\xi,\eta}$ for all $a\in\B(\s H)$.

\begin{noh}{\it Tensor product of normal linear forms.} Let $M$ and $N$ be von Neumann algebras, $\phi\in M_*$ and $\psi\in N_*$. There exists a unique $\phi\tens\psi\in(M\tens N)_*$ such that $(\phi\tens\psi)(x\tens y)=\phi(x)\psi(y)$ for all $x\in M$ and $y\in N$. Moreover, $\|\phi\tens\psi\|\leqslant\|\phi\|\cdot\|\psi\|$. Actually, it is known that we have an (completely) isometric identification $M_*\widehat{\tens}_{\pi}N_*=(M\tens N)_*$, where $\widehat{\tens}_{\pi}$ denotes the projective tensor product of Banach spaces. In particular, any $\omega\in(M\tens N)_*$ is the norm limit of finite sums of the form $\sum_i\phi_i\tens\psi_i$, where $\phi_i\in M_*$ and $\psi_i\in N_*$.
\end{noh}

\begin{noh}{\it Slicing with normal linear forms.} We will also need to slice maps with normal linear forms. Let $M_1$ and $M_2$ be von Neumann algebras, $\omega_1\in (M_1)_*$ and $\omega_2\in(M_2)_*$. Therefore, the maps $\omega_1\odot\id:M_1\odot M_2\rightarrow M_1$ and $\id\odot\omega_2:M_1\odot M_2\rightarrow M_2$ extend uniquely to norm continuous normal linear maps $\omega_1\tens\id:M_1\tens M_2\rightarrow M_2$ and $\id\tens\omega_2:M_1\tens M_2\rightarrow M_1$. 
Let $\s H$ and $\s K$ be Hilbert spaces, for $\xi\in\s H$ and $\eta\in\s K$ we define 
$\theta_{\xi}\in\B(\s K,\s H\tens\s K)$ and $\theta'_{\eta}\in\B(\s H,\s H\tens\s K)$
by setting: 
\begin{center}
$\theta_{\xi}(\zeta):=\xi\tens\zeta$, \quad for all $\zeta\in\s K$;\quad $\theta'_{\eta}(\zeta):=\zeta\tens\eta$, \quad for all $\zeta\in\s H$.
\end{center} 
If $T\in\B(\s H\tens\s K)$, $\phi\in\B(\s K)_*$ and $\omega\in\B(\s H)_*$, then the operators $(\id\tens\phi)(T)\in\B(\s H)$ and $(\omega\tens\id)(T)\in\B(\s K)$ are determined by the formulas:
\begin{align*}
\langle\xi_1,\,(\id\tens\phi)(T)\xi_2\rangle &=\phi(\theta_{\xi_1}^*T\theta_{\xi_2}),\quad \xi_1,\,\xi_2\in\s H;\\[0.3cm]
\langle\eta_1,\,(\omega\tens\id)(T)\eta_2\rangle &=\omega(\theta_{\eta_1}^{\prime *}T\theta'_{\eta_2}),\quad \eta_1,\,\eta_2\in\s K.
\end{align*}
In particular, we have:
\[
(\id\tens\omega_{\eta_1,\eta_2})(T)=\theta_{\eta_1}^{\prime *}T\theta'_{\eta_2},\quad \eta_1,\,\eta_2\in\s K;\quad (\omega_{\xi_1,\xi_2}\tens\id)(T)=\theta_{\xi_1}^*T\theta_{\xi_2},\quad \xi_1,\,\xi_2\in\s H.
\]
Let us recall some formulas that will be used several times. For all $\phi\in\B(\s K)_*$, $\omega\in\B(\s H)_*$ and $T\in\B(\s H\tens\s K)$, we have:
\[
x(\id\tens\phi)(T)y=(\id\tens\phi)((x\tens 1)T(y\tens 1)),\ (y\omega x\tens\id)(T)=(\omega\tens\id)((x\tens 1)T(y\tens 1))
\]
for all $x,y\in\B(\s H)$;
\[
a(\omega\tens\id)(T)b=(\omega\tens\id)((1\tens a)T(1\tens b)), \ (\id\tens b\phi a)(T)=(\id\tens\phi)((1\tens a)T(1\tens b))
\]
for all $a,b\in\B(\s K)$. We also have 
\[
(\id\tens\phi)(T)^*=(\id\tens\overline{\phi})(T^*),\quad (\omega\tens\id)(T)^*=(\overline{\omega}\tens\id)(T^*),
\]
\[
(\phi\tens\id)(\Sigma_{\s H\tens\s K} T\Sigma_{\s K\tens\s H})=(\id\tens\phi)(T),\quad (\id\tens\omega)(\Sigma_{\s H\tens\s K} T\Sigma_{\s K\tens\s H})=(\omega\tens\id)(T),
\]
for all $T\in\B(\s H\tens\s K)$, $\phi\in\B(\s K)_*$ and $\omega\in\B(\s H)_*$.
\end{noh}

\begin{defin}\label{weight}
A weight $\varphi$ on $M$ is a map $\varphi:M_+\rightarrow[0,\infty]$ such that: 
\begin{itemize}
	\item for all $x,y\in M_+$, $\varphi(x+y)=\varphi(x)+\varphi(y)$;
	\item for all $x\in M_+$ and $\lambda\in\GR_+$, $\varphi(\lambda x)=\lambda\varphi(x)$.
\end{itemize}
We denote by $\f N_{\varphi}:=\{x\in M\,;\,\varphi(x^*x)<\infty\}$
\index[symbol]{na@$\f N_{\varphi}$, $\f M_{\varphi}^+$} 
the left ideal of square $\varphi$-integrable elements of $M$, $\f M_{\varphi}^+:=\{x\in M_+\,;\,\varphi(x)<\infty\}$ the cone of positive $\varphi$-integrable elements of $M$ and $\f M_{\varphi}:=\langle\f M_{\varphi}^+\rangle$ the space of $\varphi$-integrable elements of $M$.
\end{defin}	

\begin{defin}
Let $\varphi$ be a weight on $M$. The opposite weight of $\varphi$ is the weight $\varphi^{\rm o}$\index[symbol]{pn@$\varphi^{\rm o}$, opposite weight} on $M^{\rm o}$ given by $\varphi ^{\rm o}(x^{\rm o}):=\varphi(x)$ for all $x\in M_+$. Then, we have $\f N_{\varphi^{\rm o}}=(\f N_{\varphi}^*)^{\rm o}$, $\f M_{\varphi^{\rm o}}^+=(\f M_{\varphi}^+)^{\rm o}$ and $\f M_{\varphi^{\rm o}}=(\f M_{\varphi})^{\rm o}$.
\end{defin}

\begin{defin}\label{nsf}
A weight $\varphi$ on $M$ is called:
\begin{itemize}
	\item semi-finite, if $\f N_{\varphi}$ is $\sigma$-weakly dense in $M$;
	\item faithful, if for $x\in M_+$ the condition $\varphi(x)=0$ implies $x=0$;
	\item normal, if $\varphi(\sup_{i\in\cal I} x_i)=\sup_{i\in\cal I}\varphi(x_i)$ for all increasing bounded net $(x_i)_{i\in\cal I}$ of $M_+$.\qedhere
\end{itemize}
\end{defin}

From now on, we will mainly use normal semi-finite faithful (n.s.f.)\ weights. Fix a \nsf weight $\varphi$ on $M$.

\begin{defin}\label{GNS}
We define an inner product on $\f N_{\varphi}$ by setting
\[
\langle x,\,y\rangle_{\varphi}:=\varphi(x^*y),\quad \text{for all } x,\,y\in\f N_{\varphi}.
\]
We denote by $(\s H_{\varphi},\Lambda_{\varphi})$ the Hilbert space completion of $\f N_{\varphi}$ with respect to this inner product, where $\Lambda_{\varphi}:\f N_{\varphi}\rightarrow\s H_{\varphi}$ is the canonical map. There exists a unique unital normal *-representation $\pi_{\varphi}:M\rightarrow\B(\s H_{\varphi})$ such that
\[
\pi_{\varphi}(x)\Lambda_{\varphi}(y)=\Lambda_{\varphi}(xy),\quad \text{for all } x\in M \text{ and } y\in\f N_{\varphi}.
\]
The triple $(\s H_{\varphi},\pi_{\varphi},\Lambda_{\varphi})$ is called the \GNS construction for $(M,\varphi)$.
\end{defin}

\begin{rks}
The linear map $\Lambda_{\varphi}$ is called the \GNS map. We have that $\Lambda_{\varphi}(\f N_{\varphi})$ is dense in $\s H_{\varphi}$ and $\langle\Lambda_{\varphi}(x),\Lambda_{\varphi}(y)\rangle_{\varphi}=\varphi(x^*y)$ for all $x,y\in\f N_{\varphi}$. In particular, $\Lambda_{\varphi}$ is injective. Moreover, we also call $\pi_{\varphi}$ the \GNS representation.
\end{rks}

We recall below the main objects of the Tomita-Takesaki modular theory.

\begin{propdef}
Let $M$ be a von Neumann algebra and $\varphi$ a \nsf weight on $M$. The anti-linear map
$
\Lambda_{\varphi}(\f N_{\varphi}^*\cap\f N_{\varphi}) \rightarrow \Lambda_{\varphi}(\f N_{\varphi}^*\cap\f N_{\varphi}) \; ; \; \Lambda_{\varphi}(x) \mapsto \Lambda_{\varphi}(x^*)
$
is closable and its closure is a possibly unbounded anti-linear map ${\cal T}_{\varphi}:D({{\cal T}_{\varphi}})\subset\s H_{\varphi}\rightarrow\s H_{\varphi}$ such that $D({{\cal T}_{\varphi}})={\rm im}\,{{\cal T}_{\varphi}}$ and ${\cal T}_{\varphi}\circ{\cal T}_{\varphi}(x)=x$ for all $x\in D({{\cal T}_{\varphi}})$.\newline
Let
$
{\cal T}_{\varphi}=J_{\varphi}\nabla_{\varphi}^{1/2}
$
be the polar decomposition of ${\cal T}_{\varphi}$. The anti-unitary $J_{\varphi}:\s H_{\varphi}\rightarrow\s H_{\varphi}$ is called the modular conjugation for $\varphi$ and the injective positive self-adjoint operator $\nabla_{\varphi}$ is called the modular operator for $\varphi$.
\end{propdef}

\begin{propdef}
There exists a unique one-parameter group $(\sigma_t^{\varphi})_{t\in\GR}$ of automorphisms on $M$, called the modular automorphism group of $\varphi$, such that
\[
\pi_{\varphi}(\sigma_t^{\varphi}(x))=\nabla_{\varphi}^{{\rm i}t}\pi_{\varphi}(x)\nabla_{\varphi}^{-{\rm i}t},\quad \text{for all } t\in\GR \text{ and } x\in M.
\]
Then, for all $t\in\GR$ and $x\in M$ we have $\sigma_t^{\varphi}(x)\in\f N_{\varphi}$ and 
$\Lambda_{\varphi}(\sigma_t^{\varphi}(x))=\nabla_{\varphi}^{{\rm i}t}\Lambda_{\varphi}(x)$.
\end{propdef}

\begin{propdef}
The map $C_M:M\rightarrow M'\,;\, x \mapsto J_{\varphi}\pi_{\varphi}(x)^*J_{\varphi}$ is a normal unital *-antihomomorphism.\index[symbol]{cb@$C_M$}
\end{propdef}

\begin{defin}
Let $N$ be a von Neumann algebra. The extended positive cone of $N$ is the set $N_+^{\ext}$ consisting of the maps $m:N_*^+\rightarrow[0,\infty]$, which satisfy the following conditions:\index[symbol]{nb@$N_+^{{\rm ext}}$, extended positive cone}
\begin{itemize}
	\item for all $\omega_1,\omega_2\in N_*^+$, $m(\omega_1+\omega_2)=m(\omega_1)+m(\omega_2)$;
	\item for all $\omega\in N_*^+$ and $\lambda\in\GR_+$, $m(\lambda\omega)=\lambda m(\omega)$;
	\item $m$ is lower semicontinuous with respect to the norm topology on $N_*$.\qedhere
\end{itemize}
\end{defin}

\begin{nbs}
Let $N$ be a von Neumann algebra.
\begin{enumerate}
	\item From now on, we will identify $N_+$ with its part inside $N_+^{\rm ext}$. Accordingly, if $m\in N_+^{\ext}$ and $\omega\in N_*^+$ we will denote by $\omega(m)$ the evaluation of $m$ at $\omega$.
	\item Let $a\in N$ and $m\in N_+^{\ext}$, we define $a^*m a\in N_+^{\ext}$ by setting
	$\omega(a^*ma):=a\omega a^*(m)$ for all $\omega\in N_*^+$.
If $m,n\in N_+^{\ext}$ and $\lambda\in\GR_+$, we also define $m+n\in N_+^{\ext}$ and $\lambda m\in N_+^{\ext}$ by setting
	$\omega(m+n):=\omega(m)+\omega(n)$ and $\omega(\lambda m):=\lambda\omega(m)$ for all $\omega\in N_*^+$.\qedhere
\end{enumerate}
\end{nbs}

\begin{defin}
Let $N\subset M$ be a unital normal inclusion of von Neumann algebras. An operator-valued weight from $M$ to $N$ is a map $T:M_+\rightarrow N_+^{\rm ext}$ such that:
\begin{itemize}
	\item for all $x,y\in M_+$, $T(x+y)=T(x)+T(y)$;
	\item for all $x\in M_+$, $\forall\lambda\in\GR_+$, $T(\lambda x)=\lambda T(x)$;
	\item for all $x\in M_+$ and $a\in N$, $T(a^*xa)=a^*T(x)a$.
\end{itemize}
Let $\f N_T:=\{x\in M\,;\,T(x^*x)\in N_+\}$, $\f M_T^+:=\{x\in M_+\,;\,T(x)\in N_+\}$ and ${\f M}_T:=\langle{\f M}_T^+\rangle$.\index[symbol]{nc@${\f N}_T$, ${\f M}_T^+$, ${\f M}_T$}
\end{defin}

\begin{defin}
Let $N\subset M$ be a unital normal inclusion of von Neumann algebras. An operator-valued weight $T$ from $M$ to $N$ is said to be:
\begin{itemize}
	\item semi-finite, if $\f N_T$ is $\sigma$-weakly dense in $M$;
	\item faithful, if for $x\in M_+$ the condition $T(x)=0$ implies $x=0$;
	\item normal, if for every increasing bounded net $(x_i)_{i\in\cal I}$ of elements of $M_+$ and $\omega\in N_*^+$, we have $\omega(T(\sup_{i\in\cal I} x_i))=\lim_{i\in\cal I}\omega(T(x_i))$.\qedhere
\end{itemize}
\end{defin}

Note that if $T:M_+\rightarrow N_+^{\ext}$ is an operator-valued weight, it extends uniquely to a semi-linear map $\overline{T}:M_+^{\ext}\rightarrow N_+^{\ext}$. This will allow us to compose \nsf operator-valued weights. Indeed, let $P\subset N\subset M$ be unital normal inclusions of von Neumann algebras. Let $S$ (resp.\ $T$) be an operator-valued weight from $N$ (resp.\ $M$) to $P$ (resp.\ $N$). We define an operator-valued weight from $M$ to $P$ by setting $(S\circ T)(x):=\overline{S}(T(x))$ for all $x\in N_+$.
	
	\subsection[Relative tensor product and fiber product]{Relative tensor product of Hilbert spaces and fiber product of von Neumann algebras}\label{tensorproduct}
	
In this paragraph, we will recall the definitions, notations and important results concerning the  relative tensor product and the fiber product which are the main technical tools of the theory of measured quantum groupoids. For more information, we refer the reader to \cite{Co}.\newline

In the whole section, $N$ is a von Neumann algebra endowed with a \nsf weight $\varphi$. Let $\pi$ (resp.\ $\gamma$) be a normal unital *-representation of $N$ (resp.\ $N^{\rm o}$) on a Hilbert space $\cal H$ (resp.\ $\cal K$).

\paragraph{Relative tensor product.} The Hilbert space $\cal H$ (resp.\ $\cal K$) may be considered as a left (resp.\ right) $N$-module. Moreover, $\s H_{\varphi}$ is an $N$-bimodule whose actions are given by
\begin{center}
$x\xi:=\pi_{\varphi}(x)\xi$ \quad and \quad $\xi y:=J_{\varphi}\pi_{\varphi}(y^*)J_{\varphi}\xi$,\quad for all $\xi\in\s H_{\varphi}$ and $x,\,y\in N$.
\end{center}

\begin{defin}\label{defLeftBoundedVector}
We define the set of right (resp.\ left) bounded vectors with respect to $\varphi$ and $\pi$ (resp.\ $\gamma$) to be:
\begin{align*}
_{\varphi}(\pi,{\cal H})&:=\{\xi\in{\cal H}\,;\,\exists\,C\in\GR_+,\,\forall\,x\in\f N_{\varphi},\,\|\pi(x)\xi\|\leqslant C\|\Lambda_{\varphi}(x)\|\},\\[0.3cm]
\text{(resp.\ }({\cal K},\gamma)_{\varphi}&:=\{\xi\in{\cal K}\,;\,\exists\,C\in\GR_+,\,\forall\,x\in\f N_{\varphi}^*,\,\|\gamma(x^{\rm o})\xi\|\leqslant C\|\Lambda_{\varphi^{\rm o}}(x^{\rm o})\|\}\text{{\rm)}}.
\end{align*}

If $\xi\in{}_{\varphi}(\pi,{\cal H})$, we denote by $R^{\pi,\varphi}_{\xi}\in\B(\s H_{\varphi},{\cal H})$\index[symbol]{r@$R_{\xi}^{\pi}$, $L_{\eta}^{\gamma}$} (or simply $R^{\pi}_{\xi}$ if $\varphi$ is understood) the unique bounded operator such that 
\[
R^{\pi,\varphi}_{\xi}\Lambda_{\varphi}(x)=\pi(x)\xi, \quad \text{for all }x\in\f N_{\varphi}.
\]
Similarly, if $\xi\in({\cal K},\gamma)_{\varphi}$ we denote $L^{\gamma,\varphi}_{\xi}\in\B(\s H_{\varphi},{\cal K})$ (or simply $L_{\xi}^{\gamma}$ if $\varphi$ is understood) the unique bounded operator such that
\[
L^{\gamma,\varphi}_{\xi}J_{\varphi}\Lambda_{\varphi}(x^*)=\gamma(x^{\rm o})\xi, \quad \text{for all }x\in\f N_{\varphi}^*,
\]
where we have used the identification $\s H_{\varphi^{\rm o}}\rightarrow\s H_{\varphi}\,;\,\Lambda_{\varphi^{\rm o}}(x^{\rm o})\mapsto J_{\varphi}\Lambda_{\varphi}(x^*)$.
\end{defin}

Note that $\xi\in\cal K$ is left bounded with respect to $\varphi$ and $\gamma$ if, and only if, it is right bounded with respect to the \nsf weight $\varphi^{\rm c}:=\varphi\circ C_N^{-1}$\index[symbol]{po@$\varphi^{\rm c}$, commutant weight} on $N'$ and the normal unital *-representation $\gamma^{\rm c}:=\gamma\circ C_N^{-1}$ of $N'$. It is important to note that $(\K,\gamma)_{\varphi}$ (resp.\ $_{\varphi}(\pi,\cal H)$) is dense in $\cal K$ (resp.\ $\cal H$) (cf.\ Lemma 2 of \cite{Co}).

\medskip

If $\xi\in{}_{\varphi}(\pi,{\cal H})$ (resp.\ $\xi\in({\cal K},\gamma)_{\varphi}$), we have that $R^{\pi,\varphi}_{\xi}$ (resp.\ $L^{\gamma,\varphi}_{\xi}$) is left (resp.\ right) $N$-linear. Therefore, for all $\xi,\eta\in{}_{\varphi}(\pi,{\cal H})$ (resp.\ $({\cal K},\gamma)_{\varphi}$) we have
\[
(R^{\pi,\varphi}_{\xi})^*R^{\pi,\varphi}_{\eta}\in\pi_{\varphi}(N)'=C_N(N) \;\; \text{and} \;\;
R^{\pi,\varphi}_{\xi}(R^{\pi,\varphi}_{\eta})^*\in\pi(N)'
\]
\[
\text{(resp.\ }(L^{\gamma,\varphi}_{\xi})^*L^{\gamma,\varphi}_{\eta}\in\pi_{\varphi}(N)  \;\; \text{and} \;\; 
L^{\gamma,\varphi}_{\xi}(L^{\gamma,\varphi}_{\eta})^*\in\gamma(N^{\rm o})'\text{{\rm)}}.
\]

\begin{nbs}
(cf.\ 2.1 \cite{E05}) Let\index[symbol]{kf@$\K_{\pi}$, $\K_{\gamma}$}
\[
\K_{\pi,\varphi}:=[R^{\pi,\varphi}_{\xi} (R^{\pi,\varphi}_{\eta})^*\,;\,\xi,\eta\in{}_{\varphi}(\pi,{\cal H})]\quad
\text{(resp.\ }\K_{\gamma,\varphi}:=[L^{\gamma,\varphi}_{\xi}(L^{\gamma,\varphi}_{\eta})^*\,;\, \xi,\eta\in({\cal H},\gamma)_{\varphi}]\text{{\rm)}}.
\]
Note that $\K_{\pi,\varphi}$ {\rm(}resp.\ $\K_{\gamma,\varphi}${\rm )} is a weakly dense ideal of $\pi(N)'$ {\rm(}resp.\ $\gamma(N^{\rm o})'${\rm)} {\rm(}cf. Proposition 3 of \cite{Co}{\rm)}. If $\varphi$ is understood, we denote $\K_{\pi}$ {\rm(}resp.\ $\K_{\gamma})$ instead of $\K_{\pi,\varphi}$ {\rm(}resp.\ $\K_{\gamma,\varphi})$.
\end{nbs}

\begin{nbs}
Let $\xi,\eta\in{}_{\varphi}(\pi,{\cal H})$ {\rm(}resp.\ $({\cal K},\gamma)_{\varphi})$, we denote
\[
\langle\xi,\,\eta\rangle_{N^{\rm o}}:=C_N^{-1}((R^{\pi,\varphi}_{\xi})^*R^{\pi,\varphi}_{\eta})^{\rm o}\in N^{\rm o}\quad 
\text{(resp.\ }\langle\xi,\,\eta\rangle_N:=\pi_{\varphi}^{-1}((L^{\gamma,\varphi}_{\xi})^*L^{\gamma,\varphi}_{\eta})\in N\text{{\rm)}}.\qedhere
\]
\end{nbs}

\begin{prop}
For all $\xi , \eta \in {}_\varphi(\pi,{\cal H})$ {\rm(}resp.\ $\xi , \eta \in ({\cal K}, \gamma)_\varphi)$ and $y\in N$ analytic for $(\sigma_t^\varphi)_{t\in\GR}$,  we have:
\begin{enumerate}
	\item $\langle \xi , \eta \rangle_{N^{\rm o}}^* = \langle \eta, \xi \rangle_{N^{\rm o}}$ {\rm(}resp.\ $\langle \xi , \eta \rangle_{ N}^* = \langle \eta, \xi \rangle_{N}${\rm)};
	\item $\langle \xi , \eta y^{\rm o}\rangle_{N^{\rm o}} =  \langle \xi , \eta \rangle_{N^{\rm o}} \sigma_{{\rm i}/2}^\varphi(y)^{\rm o}$ {\rm(}resp.\ $\langle \xi , \eta y \rangle_N =  \langle \xi , \eta \rangle_N\sigma_{-{{\rm i}/2}}^\varphi(y)${\rm)}.\qedhere
\end{enumerate}
\end{prop}

\begin{lem}
For all $\xi_1 ,\, \xi_2 \in{}_\varphi(\pi,{\cal H})$ and $\eta_1 ,\, \eta_2 \in ({\cal K}, \gamma)_\varphi$, we have 
\[
\langle \eta_1 ,\, \gamma(\langle \xi_1 ,\,  \xi_2 \rangle_{N^{\rm o}})\eta_2 \rangle_{\cal K} = 
\langle \xi_1 , \pi(\langle \eta_1 ,\,  \eta_2 \rangle_N)\xi_2 \rangle_{\cal H}.\qedhere
\]
\end{lem}

\begin{defin}
The relative tensor product 
\[
\cst{\cal K}{\gamma}{\varphi}{\pi}{\cal H}\quad \text{(or simply denoted by }\reltens{\cal K}{\gamma}{\pi}{\cal H}\text{{\rm)}}
\]
is the Hausdorff completion of the pre-Hilbert space $({\cal K}, \gamma)_\varphi \odot {}_\varphi(\pi,{\cal H})$, whose inner product is given by
\[
\langle \eta_1 \otimes \xi_1 ,\, \eta_2 \otimes \xi_2 \rangle := \langle \eta_1 ,\, \gamma(\langle \xi_1 ,\,  \xi_2 \rangle_{N^{\rm o}})\eta_2 \rangle_{\cal K} = 
\langle \xi_1 ,\, \pi(\langle \eta_1 ,\,  \eta_2 \rangle_N)\xi_2 \rangle_{\cal H},
\]
for all $\eta_1,\,\eta_2\in({\cal K}, \gamma)_\varphi$ and $\xi_1,\xi_2\in{}_\varphi(\pi,{\cal H})$. If $\eta\in({\cal K}, \gamma)_\varphi$ and $\xi\in{}_\varphi(\pi,{\cal H})$, we will denote by
\[
\cst{\eta}{\gamma}{\varphi}{\pi}{\xi} \quad \text{(or simply }\reltens{\eta}{\gamma}{\pi}{\xi}\text{{\rm)}}
\]
the image of $\eta\tens\xi$ by the canonical map $({\cal K}, \gamma)_\varphi \odot {}_\varphi(\pi,{\cal H})\rightarrow\reltens{\cal K}{\gamma}{\pi}{\cal H}$ (isometric dense range).
\end{defin}

\begin{rks}
\begin{enumerate}
	\item By applying this construction to $(N^{\rm o},\varphi^{\rm o})$ instead of $(N,\varphi)$ we obtain the relative tensor product $\cst{\cal H}{\pi}{\varphi^{\rm o}}{\gamma}{\cal K}$.
	\item The relative tensor product $\reltens{\cal K}{\gamma}{\pi}{\cal H}$ is also the Hausdorff completion of the pre-Hilbert space $({\cal K},\gamma)_{\varphi}\odot{\cal H}$ (resp.\ ${\cal K}\odot{}_{\varphi}(\pi,{\cal H})$), whose inner product is given by:
	\begin{align*}	
	\langle \eta_1 \otimes \xi_1 ,\, \eta_2 \otimes \xi_2 \rangle & :=
	\langle \xi_1 ,\, \pi(\langle \eta_1 ,\,  \eta_2 \rangle_N)\xi_2 \rangle_{\cal H}\\[0.3cm]
	(\text{resp. }\langle \eta_1 \otimes \xi_1 ,\, \eta_2 \otimes \xi_2 \rangle & := \langle \eta_1 ,\, \gamma(\langle \xi_1 ,\,  \xi_2 \rangle_{N^{\rm o}})\eta_2 \rangle_{\cal K}).
	\end{align*}
	\item Moreover, for all $\eta\in{\cal K} ,\, \xi\in {}_\varphi(\pi,{\cal H})$ and $y\in N$ analytic for $(\sigma_t^\varphi)_{t\in\GR}$ we have
\[
\reltens{\gamma(y^{\rm o})\eta}{\gamma}{\pi}{\xi} =  \reltens{\eta}{\gamma}{\pi}{\pi(\sigma_{-{{\rm i}/2}}^\varphi(y))\xi}.
\qedhere\]
\end{enumerate}
\end{rks}

\begin{noh}
The relative flip map is the isomorphism $\sigma_{\varphi}^{\gamma\pi}$ from $\cst{\cal K}{\gamma}{\varphi}{\pi}{\cal H}$ onto $\cst{\cal H}{\pi}{\varphi^{\rm o}}{\gamma}{\cal K}$ given by:
\[
\sigma_{\varphi}^{\gamma\pi}(\cst{\eta}{\gamma}{\varphi}{\pi}{\xi}):=\cst{\xi}{\pi}{\varphi^{\rm o}}{\gamma}{\eta},\quad \text{for all } \xi\in({\cal K},\gamma)_{\varphi} \text{ and } \eta\in{}_{\varphi}(\pi,{\cal H})\quad \text{(or simply }\sigma_{\gamma\pi}\text{{\rm)}}.
\]
Note that $\sigma_{\varphi}^{\gamma\pi}$ is unitary and $(\sigma_{\varphi}^{\gamma\pi})^*=\sigma_{\varphi^{\rm o}}^{\pi\gamma}$. Then, we can define a relative flip *-homomorphism 
\[
\varsigma_{\varphi}^{\gamma\pi}:\B(\cst{\cal K}{\gamma}{\varphi}{\pi}{\cal H})\rightarrow\B(\cst{\cal H}{\pi}{\varphi^{\rm o}}{\gamma}{\cal K})\quad \text{(or simply denoted by } \varsigma_{\gamma\pi}\text{{\rm)}}
\]
by setting
$
\varsigma_{\varphi}^{\gamma\pi}(X):=\sigma_{\varphi}^{\gamma\pi}X(\sigma_{\varphi}^{\gamma\pi})^*
$
for all $X\in\B(\cst{\cal K}{\gamma}{\varphi}{\pi}{\cal H})$.\index[symbol]{sg@$\sigma_{\gamma\pi}$/$\varsigma_{\gamma\pi}$, relative flip map/*-homomorphism}
\end{noh}

\paragraph{Fiber product of von Neumann algebras.} We continue to use the notations of the previous paragraph.

\begin{propdef}\label{propdef2}
Let ${\cal K}_i$ and ${\cal H}_i$ be Hilbert spaces, and $\gamma_i:N^{\rm o}\rightarrow\B({\cal K}_i)$ and $\pi_i:N\rightarrow\B({\cal H}_i)$ be unital normal *-homomorphisms for $i=1,2$. Let $T\in\B({\cal K}_1,{\cal K}_2)$ and $S\in\B({\cal H}_1,{\cal H}_2)$ such that 
$T\circ\gamma_1(n^{\rm o})=\gamma_2(n^{\rm o})\circ T$ and $S\circ\pi_1(n)=\pi_2(n)\circ S$ for all $n\in N$.
Then, the linear map
\begin{center}
$
({\cal K}_1,\gamma_1)_{\varphi}\odot{}_{\varphi}(\pi_1,{\cal H}_1) \rightarrow \reltens{{\cal K}_2}{\gamma_2}{\pi_2}{{\cal H}_2} \; ; \; \xi\odot\eta \mapsto \reltens{T\xi}{\gamma_2}{\pi_2}{S\eta}
$
\end{center}
extends uniquely to a bounded operator 
$
_{\gamma_2}\reltens{T}{\gamma_1}{\pi_2}{S}_{\pi_1}\in\B(\reltens{{\cal K}_1}{\gamma_1}{\pi_1}{{\cal H}_1},\reltens{{\cal K}_2}{\gamma_2}{\pi_2}{{\cal H}_2})
$
(or simply denoted by $\reltens{T}{\gamma_1}{\pi_2}{S}$), whose adjoint operator is $_{\gamma_1}\reltens{T^*}{\gamma_2}{\pi_1}{S^*}_{\pi_2}$ {\rm(}or simply $\reltens{T^*}{\gamma_2}{\pi_1}{S^*})$. In particular, if $x\in\gamma(N^{\rm o})'$ and $y\in\pi(N)'$, then the linear map
\[
({\cal K},\gamma)_{\varphi}\odot{}_{\varphi}(\pi,{\cal H}) \rightarrow \reltens{\cal K}{\gamma}{\pi}{\cal H} \; ; \; \xi\odot\eta \mapsto \reltens{x\xi}{\gamma}{\pi}{y\eta}
\]
extends uniquely to a bounded operator on $\reltens{\cal K}{\gamma}{\pi}{\cal H}$ denoted by $\reltens{x}{\gamma}{\pi}{y}\in\B(\reltens{\cal K}{\gamma}{\pi}{\cal H})$.
\end{propdef}

\begin{rk}
With the notations of \ref{propdef2}, let $T:{\cal K}_1\rightarrow{\cal H}_2$ and $S:{\cal H}_1\rightarrow{\cal K}_2$ be bounded antilinear maps such that $T\circ\gamma_1(n^{\rm o})^*=\pi_2(n)\circ T$ and $S\circ\pi_1(n)=\gamma_2(n^{\rm o})^*\circ S$ for all $n\in N$.
In a similar way, we define 
$
_{\pi_2}\reltens{T}{\gamma_1}{\gamma_2}{S}_{\pi_1}\in\B(\reltens{{\cal K}_1}{\gamma_1}{\pi_1}{{\cal H}_1},\reltens{{\cal H}_2}{\pi_2}{\gamma_2}{{\cal K}_2})
$ 
(or simply $\reltens{T}{\gamma_1}{\gamma_2}{S}$). Note that these notations are different from those used in \cite{E08,Le}.
\end{rk}

Let $M\subset\B({\cal K})$ and $P\subset\B({\cal H})$ be two von Neumann algebras. Let us assume that $\pi(N)\subset P$ and $\gamma(N^{\rm o})\subset M$.

\begin{defin}
The fiber product 
$
\fprod{M}{\gamma}{\pi}{P}
$
of $M$ and $P$ over $N$ is the commutant of 
$
\{\reltens{x}{\gamma}{\pi}{y}\,;\, x\in M',\,y\in P'\}\subset \B(\reltens{\cal K}{\gamma}{\pi}{\cal H}).
$
Then, $\fprod{M}{\gamma}{\pi}{P}$ is a von Neumann algebra. 
\end{defin}

Note that we have $\varsigma_{\gamma\pi}(\fprod{M}{\gamma}{\pi}{P})=\fprod{P}{\pi}{\gamma}{M}$.
We still denote by $\varsigma_{\gamma\pi}:\fprod{M}{\gamma}{\pi}{P}\rightarrow\fprod{P}{\pi}{\gamma}{M}$ the restriction of $\varsigma_{\gamma\pi}$ to $\fprod{M}{\gamma}{\pi}{P}$.

\begin{noh}\textit{Slicing with normal linear forms.} Now, let us recall how to slice with normal linear forms. For $\xi\in({\cal K},\gamma)_{\varphi}$ and $\eta\in{}_{\varphi}(\pi,{\cal H})$, we consider the following bounded linear maps:
\[
\lambda_{\xi}^{\gamma\pi}:{\cal H}\rightarrow\reltens{\cal K}{\gamma}{\pi}{\cal H},\; 
\zeta \mapsto \reltens{\xi}{\gamma}{\pi}{\zeta}  ; \quad
\rho_{\eta}^{\gamma\pi}:{\cal K}\rightarrow\reltens{\cal K}{\gamma}{\pi}{\cal H},\; 
\zeta \mapsto \reltens{\zeta}{\gamma}{\pi}{\eta}.
\]
Let $T\in\B(\reltens{{\cal K}}{\gamma}{\pi}{{\cal H}})$ and $\omega\in\B({\cal H})_*$ (resp.\ $\omega\in\B({\cal K})_*$). By using the fact that $({\cal K},\gamma)_{\varphi}$ (resp.\ $_{\varphi}(\pi,{\cal H})$) is dense in ${\cal H}$ (resp.\ ${\cal K}$), there exists a unique $(\fprod{\id}{\gamma}{\pi}{\omega})(T)\in\B({\cal K})$ (resp.\ $(\fprod{\omega}{\gamma}{\pi}{\id})(T)\in\B({\cal H})$) such that
\begin{align*}
\langle\xi_1,\,(\fprod{\id}{\gamma}{\pi}{\omega})(T)\xi_2\rangle &=\omega((\lambda_{\xi_1}^{\gamma\pi})^*T\lambda_{\xi_2}^{\gamma\pi}),\quad \text{for all } \xi_1,\,\xi_2\in({\cal K},\gamma)_{\varphi}\\[.5em]
\text{(resp. }\langle\eta_1,\,(\fprod{\omega}{\gamma}{\pi}{\id})(T)\eta_2\rangle &=\omega((\rho_{\eta_1}^{\gamma\pi})^*T\rho_{\eta_2}^{\gamma\pi}),\quad \text{for all } \eta_1,\,\eta_2\in{}_{\varphi}(\pi,{\cal H})\text{{\rm)}}.
\end{align*}
In particular, we have:
\begin{align*}
(\fprod{\id}{\gamma}{\pi}{\omega_{\eta_1,\eta_2}})(T)&=(\rho_{\eta_1}^{\gamma\pi})^*T\rho_{\eta_2}^{\gamma\pi}\in\B({\cal K}),\quad \text{for all } \eta_1,\,\eta_2\in{}_{\varphi}(\pi,{\cal H});\\[.5em]
(\fprod{\omega_{\xi_1,\xi_2}}{\gamma}{\pi}{\id})(T)&=(\lambda_{\xi_1}^{\gamma\pi})^*T\lambda_{\xi_2}^{\gamma\pi}\in\B({\cal H}),\quad \text{for all } \xi_1,\,\xi_2\in({\cal K},\gamma)_{\varphi}.
\end{align*}
If $x\in\fprod{M}{\gamma}{\pi}{P}$, then for all $\omega\in\B(\cal H)_*$ (resp.\ $\omega\in\B(\cal K)_*$) we have
$
(\fprod{\id}{\gamma}{\pi}{\omega})(x)\in M
$
(resp.\ $(\fprod{\omega}{\gamma}{\pi}{\id})(x)\in P$).
We refrain from writing the details but we can easily define the slice maps if $T$ takes its values in a different relative tensor product. Note that we can extend the notion of slice maps for normal linear forms to normal semi-finite weights.
\end{noh}

\paragraph{Fiber product over a finite-dimensional von Neumann algebra.} Now, let us assume that 
\[
N:=\bigoplus_{1\leqslant l\leqslant k}{\rm M}_{n_l}(\GC) \quad \text{and} \quad \varphi:=\bigoplus_{1\leqslant l\leqslant k}{\rm Tr}_l(F_l-),
\]
where $F_l$ is a positive invertible matrix of ${\rm M}_{n_l}(\GC)$ and ${\rm Tr}_l$ is the non-normalized trace on ${\rm M}_{n_l}(\GC)$. Denote by $(F_{l,i})_{1\leqslant i\leqslant n_l}$ the eigenvalues of $F_l$. 

\begin{propdef}($\S 7$ \cite{DC2})
The bounded linear map\index[symbol]{vh@$v_{\gamma\pi}$, canonical coisometry}
\[
v_{\varphi}^{\gamma\pi}:{\cal K}\tens{\cal H}\rightarrow\cst{\cal K}{\gamma}{\varphi}{\pi}{\cal H} \; ;\; \xi\tens\eta \mapsto \cst{\xi}{\gamma}{\varphi}{\pi}{\eta} \quad \text{(or simply denoted by }v_{\gamma\pi}\text{{\rm)}}
\]
is a coisometry if, and only if, we have $\sum_{1\leqslant i\leqslant n_l}F_{l,i}^{-1}=1$ for all $1\leqslant l\leqslant k$. 
\end{propdef}

In the following, we assume the above condition to be satisfied. 

\begin{propdef}\label{propcoiso}($\S 7$ \cite{DC2})
Let us denote\index[symbol]{qd@$q_{\gamma\pi}$, $q_{\pi\gamma}$} 
\[
q_{\varphi}^{\gamma\pi}:=(v_{\varphi}^{\gamma\pi})^*v_{\varphi}^{\gamma\pi} \quad (\text{or simply }q_{\gamma\pi}).
\]
Then, $q_{\varphi}^{\gamma\pi}$ is a self-adjoint projection of $\B({\cal K}\tens{\cal H})$ such that
\[
q_{\varphi}^{\gamma\pi}=\sum_{1\leqslant l\leqslant k}\sum_{1\leqslant i,j\leqslant n_l}F_{l,i}^{-1/2}F_{l,j}^{-1/2}\gamma(e_{ij}^{(l)\,{\rm o}})\tens\pi(e_{ji}^{(l)}),
\] 
where, for all $1\leqslant l\leqslant k$, $(e_{ij}^{(l)})_{1\leqslant i,j\leqslant n_l}$ is a system of matrix units (s.m.u.) diagonalizing $F_l$, {\it i.e.} $F_l=\sum_{1\leqslant i\leqslant n_l}F_{l,i}e_{ii}^{(l)}$. Moreover,
$
\fprod{M}{\gamma}{\pi}{P} \rightarrow q_{\varphi}^{\gamma\pi}(M\tens P)q_{\varphi}^{\gamma\pi} \; ; \; x \mapsto (v_{\varphi}^{\gamma\pi})^*xv_{\varphi}^{\gamma\pi}
$
is a unital normal *-isomorphism.$\vphantom{e_{ij}^{(l)}}$
\end{propdef}

Since $N$ is finite-dimensional, the inner product given by $\langle x,\,y\rangle:=\varphi(x^*y)$ for all $x,y\in N$ defines a structure of finite-dimensional Hilbert space on $N$. We have a (bounded) linear map $\mu_{\varphi}:N\tens N\rightarrow N$ defined for all $x,y\in N$ by $\mu_{\varphi}(x\tens y)=xy$, where $N\tens N$ is endowed with its canonical structure of finite-dimensional Hilbert space. 

\begin{propdef}\label{propdef5}
For $i=1,2$, let $\pi_i:N\rightarrow\B({\cal H}_i)$ be a unital normal *-representation of $N$ on a Hilbert space ${\cal H}_i$. Let us denote\index[symbol]{qe@$q_{\pi_1\pi_2}$, $q_{\pi_1}$}
\[
q_{\varphi}^{\pi_1\pi_2}:=(\pi_1\tens\pi_2)(\mu_{\varphi}^*(1_N))\in\B({\cal H}_1\tens{\cal H}_2) \quad \text{{\rm(}or simply } q_{\pi_1\pi_2}\text{{\rm)}}.
\]
We denote $q_{\varphi}^{\pi_1}:=q_{\varphi}^{\pi_1\pi_1}$ {\rm(}or simply $q_{\pi_1}${\rm)} for short. Then, we have
\[
q_{\varphi}^{\pi_1\pi_2}=\sum_{1\leqslant l\leqslant k}\sum_{1\leqslant i,j\leqslant n_l}F_{l,j}^{-1}\pi_1(e_{ij}^{(l)})\tens\pi_2(e_{ji}^{(l)}),
\]
where, for all $1\leqslant l\leqslant k$, $(e_{ij}^{(l)})_{1\leqslant i,j\leqslant n_l}$ is a s.m.u.\ diagonalizing $F_l$.
\end{propdef}

\begin{proof}
For $1\leqslant l\leqslant k$, fix a s.u.m.\ $(e_{ij}^{(l)})_{1\leqslant i,j\leqslant n_l}$ of ${\rm M}_{n_l}(\GC)$ diagonalizing $F_l$. It suffices to prove that
\[
\mu_{\varphi}^*(1_N)=\sum_{1\leqslant l\leqslant k}\sum_{1\leqslant i,j\leqslant n_l}F_{l,j}^{-1}e_{ij}^{(l)}\tens e_{ji}^{(l)}.
\]
Since $1_N=\sum_{1\leqslant l\leqslant k}\sum_{1\leqslant i\leqslant n_l}e_{ii}^{(l)}$, it is enough to prove that
\[
\mu_{\varphi}^*(e_{rs}^{(l)})=\sum_{1\leqslant j\leqslant n_l}F_{l,j}^{-1} e_{rj}^{(l)}\tens e_{js}^{(l)}, \quad \text{for all } 1\leqslant r,s\leqslant n_l.
\]
Let $(f_{ij}^{(l)})$ be the family of $N$ given by 
$
f_{ij}^{(l)}:=F_{l,j}^{-1/2}e_{ij}^{(l)}
$
for all $1\leqslant l\leqslant k$ and $1\leqslant i,j\leqslant n_l$. It is clear that $(f_{ij}^{(l)})$ is an orthonormal basis of $N$. We have
\[
\varphi(e_{sq}^{(l)})={\rm Tr}_l(F_l \, e_{sq}^{(l)})=\sum_{i\,=\,1}^{n_l}F_{l,i} \, {\rm Tr}_l(e_{ii}^{(l)}e_{sq}^{(l)})=F_{l,s} \, {\rm Tr}_l(e_{sq}^{(l)})= F_{l,s} \, \delta_q^s.
\]
We have
\begin{align*}
\mu_{\varphi}^*(e_{rs}^{(l)}) &= \sum_{l'\!,\,l''=\,1}^k \,\sum_{i,\,j^{\phantom{a}}\!=\,1}^{n_{l'}} \,\sum_{p,\,q^{\phantom{a}}\!=\,1}^{n_{l''}} \langle\mu_{\varphi}^*(e_{rs}^{(l)}),\, f_{ij}^{(l')}\tens f_{pq}^{(l'')}\rangle \, f_{ij}^{(l')}\tens f_{pq}^{(l'')}\\
&=\sum_{l'\!,\,l''=\,1}^k \,\sum_{i,\,j^{\phantom{a}}\!=\,1}^{n_{l'}} \,\sum_{p,\,q^{\phantom{a}}\!=\,1}^{n_{l''}} \delta_{l''}^{l'} \, \delta_p^j \, F_{l'\!,\,j}^{-1} \, F_{l'' \! ,\,q}^{-1} \, \langle e_{rs}^{(l)},\, e_{iq}^{(l')}\rangle \, e_{ij}^{(l')}\tens e_{pq}^{(l'')}\\
&=\sum_{l'=\,1}^k \,\sum_{i,j,\, q^{\phantom{a}}\!=\,1}^{n_{l'}} \delta_{l'}^l \, \delta_i^r \, \, F_{l'\!,\,j}^{-1} \, F_{l' \! ,\,q}^{-1} \, \varphi(e_{sq}^{(l)}) \, e_{ij}^{(l')}\tens e_{jq}^{(l')}\\
&=\sum_{j,\, q\,=\,1}^{n_{l}} F_{l,\,j}^{-1} \, F_{l ,\,q}^{-1} \, F_{l ,\,s}\, \delta_q^s \, e_{rj}^{(l)}\tens e_{jq}^{(l)}\\
&= \sum_{j\,=\,1}^{n_{l}} F_{l,\,j}^{-1} \,  e_{rj}^{(l)}\tens e_{js}^{(l)}. \qedhere
\end{align*}
\end{proof}

\begin{rks}
\begin{enumerate}
\item For $i=1,2$, let $\gamma_i:N^{\rm o}\rightarrow\B({\cal K}_i)$ be a unital normal *-representa\-tion of $N^{\rm o}$ on a Hilbert space ${\cal K}_i$. In a similar way, we define $q_{\varphi^{\rm o}}^{\gamma_1\gamma_2}\in\B(\K_1\tens\K_2)$ (or simply $q_{\gamma_1\gamma_2}$) such that
\[
q_{\varphi^{\rm o}}^{\gamma_1\gamma_2}=\sum_{1\leqslant l\leqslant k}\sum_{1\leqslant i,j\leqslant n_l}F_{l,j}^{-1}\gamma_1(e_{ij}^{(l){\rm o}})\tens\gamma_2(e_{ji}^{(l){\rm o}}),
\]
where, for all $1\leqslant l\leqslant k$, $(e_{ij}^{(l)})_{1\leqslant i,j\leqslant n_l}$ is a s.m.u.\ diagonalizing $F_l$.
\item It should be noted that $q_{\varphi}^{\pi_1\pi_2}$ and $q_{\varphi^{\rm o}}^{\gamma_1\gamma_2}$ are self-adjoint but not idempotent in general. If $N$ is commutative ({\it i.e.\ }$N=\GC^k$), then $q_{\varphi}^{\pi_1\pi_2}$ and $q_{\varphi^{\rm o}}^{\gamma_1\gamma_2}$ are projections.\qedhere
\end{enumerate}
\end{rks}

\paragraph{Case of the non-normalized Markov trace.} In this paragraph, we take for $\varphi$ the non-normalized Markov trace on $\displaystyle{N=\oplus_{1\leqslant l\leqslant k}\,{\rm M}_{n_l}(\GC)}$, {\it i.e.\ }$\epsilon=\oplus_{1\leqslant l\leqslant k}\, n_l\cdot{\rm Tr}_l$.
From now on, the operators $q_{\epsilon}^{\gamma\pi}$, $q_{\epsilon^{\rm o}}^{\pi\gamma}$ $q_{\epsilon}^{\pi_1\pi_2}$ and $q_{\epsilon^{\rm o}}^{\gamma_1\gamma_2}$ will be simply denoted by $q_{\gamma\pi}$, $q_{\pi\gamma}$, $q_{\pi_1\pi_2}$ and $q_{\gamma_1\gamma_2}$. As a corollary of \ref{propcoiso}, we have:

\begin{prop}\label{Projection}
For all s.u.m.\ $(e_{ij}^{(l)})_{1\leqslant l\leqslant k,\, 1\leqslant i,j\leqslant n_l}$ of $N$, we have
\[
q_{\gamma\pi}=\sum_{1\leqslant l\leqslant k}n_l^{-1}\!\sum_{1\leqslant i,j\leqslant n_l}\!\gamma(e_{ij}^{(l){\rm o}})\tens\pi(e_{ji}^{(l)}) \;\; \text{and} \;\;
q_{\pi\gamma}=\sum_{1\leqslant l\leqslant k}n_l^{-1}\!\sum_{1\leqslant i,j\leqslant n_l}\!\pi(e_{ij}^{(l)})\tens\gamma(e_{ji}^{(l){\rm o}}).\qedhere
\]
\end{prop}

As a corollary of \ref{propdef5}, we have:

\begin{prop}\label{ProjectionBis}
For all s.u.m.\ $(e_{ij}^{(l)})_{1\leqslant l\leqslant k,\, 1\leqslant i,j\leqslant n_l}$ of $N$, we have
\begin{align*}
q_{\pi_1\pi_2}&=\sum_{1\leqslant l\leqslant k}n_l^{-1}\sum_{1\leqslant i,j\leqslant n_l}\pi_1(e_{ij}^{(l)})\tens\pi_2(e_{ji}^{(l)})\;\; \text{and}\\
q_{\gamma_1\gamma_2}&=\sum_{1\leqslant l\leqslant k}n_l^{-1}\sum_{1\leqslant i,j\leqslant n_l}\gamma_1(e_{ij}^{(l){\rm o}})\tens\gamma_2(e_{ji}^{(l){\rm o}}).\qedhere
\end{align*}
\end{prop}

The following result is a slight generalization of \ref{Projection} to the setting of C*-algebras.

\begin{propdef}\label{ProjectionCAlg}(2.6 \cite{BC})
Let $A$, $B$ be two C*-algebras. We consider two non-degener\-ate *-homomorphisms $\gamma_A:N^{\rm o}\rightarrow\M(A)$ and $\pi_B:N\rightarrow\M(B)$. There exists a unique self-adjoint projection $q_{\gamma_A\pi_B}\in\M(A\tens B)$ {\rm(}resp.\ $q_{\pi_B\gamma_A}\in\M(B\tens A)${\rm)} such that
\begin{align*}
q_{\gamma_A\pi_B}&=\sum_{1\leqslant l\leqslant k}n_l^{-1}\sum_{1\leqslant i,j\leqslant n_l}\gamma_A(e_{ij}^{(l){\rm o}})\tens\pi_B(e_{ji}^{(l)})\\
\text{{\rm(}resp.\ }q_{\pi_B\gamma_A}&=\sum_{1\leqslant l\leqslant k}n_l^{-1}\sum_{1\leqslant i,j\leqslant n_l}\pi_B(e_{ij}^{(l)})\tens\gamma_A(e_{ji}^{(l){\rm o}})\text{{\rm)},}
\end{align*}
for all s.u.m.\ $(e_{ij}^{(l)})_{1\leqslant l\leqslant k,\, 1\leqslant i,j\leqslant n_l}$ of $N$.
\end{propdef}

\begin{proof}
The uniqueness of such a self-adjoint projection is straightforward. In virtue of the Gelfand-Naimark theorem, we can consider faithful non-degenerate *-homomorphisms $\theta_A:A\rightarrow\B(\K)$ and $\theta_B:B\rightarrow\B({\cal H})$. Let us denote $\gamma:=\theta_A\circ\gamma_A$ and $\pi:=\theta_B\circ\pi_B$. Then, $\gamma:N^{\rm o}\rightarrow\B(\K)$ and $\pi:N\rightarrow\B({\cal H})$ are normal unital *-representations. Let us fix an arbitrary s.u.m.\ $(e_{ij}^{(l)})_{1\leqslant i,j\leqslant n_l}$ for ${\rm M}_{n_l}(\GC)$ for each $1\leqslant l\leqslant k$. We define a self-adjoint projection $q_{\gamma_A\pi_B}\in\M(A\tens B)$ by setting:
\[
q_{\gamma_A\pi_B}:=\sum_{1\leqslant l\leqslant k}n_l^{-1}\sum_{1\leqslant i,j\leqslant n_l}\gamma_A(e_{ij}^{(l){\rm o}})\tens\pi_B(e_{ji}^{(l)}).
\]
By \ref{Projection}, we have $q_{\gamma\pi}=(\theta_A\tens\theta_B)(q_{\gamma_A\pi_B})$. By using again \ref{Projection} and the fact that $\theta_A\tens\theta_B$ is faithful, we obtain that $q_{\gamma_A\pi_B}$ is independent of  the chosen systems of matrix units. Moreover, the definition of $q_{\gamma_A\pi_B}$ shows that $q_{\gamma_A\pi_B}$ is also independent of the chosen faithful non-degenerate *-homomorphisms $\theta_A$ and $\theta_B$.
\end{proof}

In a similar way, we have the following generalization of \ref{ProjectionBis} to the setting of C*-algebras.

\begin{prop}
For $i=1,2$, let $B_i$ {\rm(}resp.\ $A_i${\rm)} be a C*-algebra and $\pi_i:N\rightarrow\M(B_i)$ {\rm(}resp.\ $\gamma_i:N^{\rm o}\rightarrow\M(A_i)${\rm)} a non-degenerate *-homomorphism. Then, there exists a unique $q_{\pi_1\pi_2}\in\M(B_1\tens B_2)$ {\rm(}resp.\ $q_{\gamma_1\gamma_2}\in\M(A_1\tens A_2)${\rm)} such that 
\begin{align*}
q_{\pi_1\pi_2}&=\sum_{1\leqslant l\leqslant k}n_l^{-1}\sum_{1\leqslant i,j\leqslant n_l}\pi_1(e_{ij}^{(l)})\tens\pi_2(e_{ji}^{(l)})\\
\text{{\rm(}resp.\ }q_{\gamma_1\gamma_2}&=\sum_{1\leqslant l\leqslant k}n_l^{-1}\sum_{1\leqslant i,j\leqslant n_l}\gamma_1(e_{ij}^{(l){\rm o}})\tens\gamma_2(e_{ji}^{(l){\rm o}})\text{{\rm)}},
\end{align*}
for all s.u.m.\ $(e_{ij}^{(l)})_{1\leqslant l\leqslant k,\, 1\leqslant i,j\leqslant n_l}$ of $N$.
\end{prop}

In the following, we adopt a multi-index notation to simplify formulas and computations.

\begin{nbs}\label{not13}
\begin{enumerate}
\item Consider the index sets $\s I:=\{(l,i,j)\,;\, 1\leqslant l\leqslant k,\, 1\leqslant i,j\leqslant n_l\}$ and $\s I_0:=\s I\sqcup\{\varnothing\}$.\index[symbol]{id@$\s I$, multi-index set} 
\item For $I=(l,i,j)\in\s I$, we denote $\overline{I}:=(l,j,i)\in\s I$. Denote also $\overline{\varnothing}:=\varnothing$. The map $\s I_0\rightarrow\s I_0\,;\, I\mapsto\overline{I}$ is involutive.
\item A pair of indices $(I,J)\in\s I\times\s I$ is said to be composable if we have $I=(l,i,m)$ and $J=(l,m,j)$ for some indices $1\leqslant l\leqslant k$ and $1\leqslant i,m,j\leqslant n_l$. In this case, we denote $I J:=(l,i,j)\in\s I$. We also denote $IJ=\varnothing$ if $I$ and $J$ are not composable, $I=\varnothing$ or $J=\varnothing$. This defines a map $\s I_0\times\s I_0\rightarrow\s I_0\,;\, (I,J)\mapsto IJ$. It is clear that $\overline{IJ}=\overline{J}\,\overline{I}$ for all $I,J\in\s I_0$.\qedhere
\end{enumerate}
\end{nbs}

Let us fix a s.u.m.\ $(e_{ij}^{(l)})_{1\leqslant l\leqslant k,\, 1\leqslant i,j\leqslant n_l}$ of $N$.

\begin{nbs}\label{not14}
\begin{enumerate}
\item Denote by $\varepsilon_I:=e_{ij}^{(l)}$\index[symbol]{ek@$\varepsilon_I$} for $I=(l,i,j)\in\s I$ and $\varepsilon_{\varnothing}:=0$. Denote by $e_I:=\pi(\varepsilon_I)$ and $f_I:=\gamma(\varepsilon_I^{\rm o})$ for $I\in\s I_0$. Denote by $n_I:=n_l$ for $I=(l,i,j)\in\s I$ and $n_{\varnothing}:=1$. Notice that we have $n_{\overline{I}}=n_I$ for all $I\in\s I_0$.\index[symbol]{el@$e_I$, $f_I$}
\item Since $(\varepsilon_I)_{I\in\s I}$ is a basis of $N$, for $x\in N$ we denote $x=\sum_{I\in\s I}x_I\cdot \varepsilon_I$, with $x_I\in\GC$ for $I\in\s I$. Note that $x^*=\sum_{I\in\s I}\overline{x_{\overline{I}}}\cdot \varepsilon_I$.\qedhere
\end{enumerate}
\end{nbs}

\begin{rks}\label{rk14}
\begin{enumerate}
\item For all $I,J\in\s I_0$, we have $\varepsilon_I^*=\varepsilon_{\overline{I}}$ and $\varepsilon_I\varepsilon_J=\varepsilon_{IJ}$. For all $I,J\in\s I_0$, we have:
\[
e_I^*=e_{\overline{I}},\quad e_Ie_J=e_{IJ} ; \quad f_I^*=f_{\overline{I}},\quad f_If_J=f_{JI}.
\]
\item We have $q_{\gamma\pi}=\sum_{I\in\s I} n_I^{-1} f_I \tens e_{\overline{I}}$, $q_{\pi\gamma}=\sum_{I\in\s I} n_I^{-1} e_I \tens f_{\overline{I}}$, $q_{\pi}=\sum_{I\in\s I} e_I\tens e_{\overline{I}}$ and $q_{\gamma}=\sum_{I\in\s I} f_I\tens f_{\overline{I}}$.\qedhere
\end{enumerate}
\end{rks}

	\subsection{Unitary equivalence of Hilbert C*-modules}\label{UnitaryEq}

In the following, we recall the notion of morphism between Hilbert modules over possibly different C*-algebras. 

\begin{defin}\label{def2} 
Let $A$ and $B$ be two C*-algebras and $\phi:A\rightarrow B$ a *-homomorphism. Let $\s E$ and $\s F$ be two Hilbert C*-modules over $A$ and $B$ respectively. A $\phi$-compatible operator from $\s E$ to $\s F$ is a linear map $\Phi:\s E\rightarrow\s F$ such that:
\begin{enumerate}[label=(\roman*)]
\item for all $\xi\in\s E$ and $a\in A$, $\Phi(\xi a)=\Phi(\xi)\phi(a)$;
\item for all $\xi,\eta\in\s E$, $\langle \Phi\xi,\, \Phi\eta\rangle=\phi(\langle\xi,\,\eta\rangle)$.
\end{enumerate}
Furthermore, if $\phi$ is a *-isomorphism and $\Phi$ is surjective, we say that $\Phi$ is $\phi$-compatible unitary operator (or a unitary equivalence over $\phi$) from $\s E$ onto $\s F$.
\end{defin}

\begin{rks}\label{rk4}
\begin{enumerate}
\item It follows from (ii) that $\Phi:\s E\rightarrow\s F$ is bounded and even isometric if $\phi$ is faithful. Indeed, we have $\|\langle\Phi\xi,\, \Phi\eta\rangle\|=\|\phi(\langle\xi,\,\eta\rangle)\|=\|\langle\xi,\,\eta\rangle\|$ for all $\xi,\eta\in\s E$. Then,  for all $\xi\in\s E$ we have
	 $
	 \|\Phi\xi\|^2=\|\langle\Phi\xi,\, \Phi\xi\rangle\|=\|\langle\xi,\, \xi\rangle\|=\|\xi\|^2
	 $.
In particular, if $\phi$ is a *-isomorphism and $\Phi$ is a $\phi$-compatible unitary operator, then $\Phi$ is bijective and the inverse map $\Phi^{-1}:\s F\rightarrow\s E$ is a $\phi^{-1}$-compatible unitary operator.
\item It is clear that $\id_{\s E}$ is a $\id_A$-compatible unitary operator. Let $A$, $B$ and $C$ be C*-algebras and $\s E$, $\s F$ and $\s G$ be Hilbert modules over $A$, $B$ and $C$ respectively. Let $\phi:A\rightarrow B$ and $\psi:B\rightarrow C$ be *-homomorphisms (resp.\ *-isomorphisms). If $\Phi:\s E\rightarrow\s F$ is a $\phi$-compatible operator (resp.\ unitary operator) and $\Psi:\s F\rightarrow\s G$ a $\psi$-compatible operator (resp.\ unitary operator), then $\Psi\circ \Phi:\s E\rightarrow\s G$ is a $\psi\circ\phi$-compatible operator (resp.\ unitary operator).
\item Let $\Phi:\s E\rightarrow\s F$ be a unitary equivalence over a given *-isomorphism $\phi$. If $T\in\Lin(\s E)$, then the map $\Phi\circ T\circ\Phi^{-1}:\s F\rightarrow\s F$ is an adjointable operator whose adjoint operator is $\Phi^{-1}\circ T^* \circ \Phi$. We define a *-isomorphism $\Lin(\s E)\rightarrow\Lin(\s F)\,;\, T\mapsto\Phi\circ T\circ\Phi^{-1}$. Note that $\Phi\circ\theta_{\xi,\eta}\circ\Phi^{-1}=\theta_{\Phi\xi,\Phi\eta}$ for all $\xi,\eta\in\s E$. In particular,  for all $k\in\K(\s E)$ we have $\Phi\circ k\circ\Phi^{-1}\in\K(\s F)$. More precisely, the map $\K(\s E)\rightarrow\K(\s F)\,;\,k\mapsto\Phi\circ k\circ\Phi^{-1}$ is a *-isomorphism.\qedhere
\end{enumerate} 
\end{rks}

The notion of unitary equivalence defines an equivalence relation on the class consisting of all Hilbert C*-modules (cf.\ \ref{rk4} 1, 2). Actually, this notion of morphism between Hilbert modules over possibly different C*-algebra can be understood in terms of unitary adjointable operator between two Hilbert modules over the same C*-algebra.

\begin{prop}\label{prop8}
Let $A$ and $B$ be two C*-algebras and $\phi:A\rightarrow B$ a *-isomorphism. Let $\s E$ and $\s F$ be two Hilbert C*-modules over $A$ and $B$ respectively. 
\begin{enumerate}
\item If $\Phi:\s E\rightarrow\s F$ is a surjective $\phi$-compatible unitary operator, then there exists a unique unitary adjointable operator $U\in\Lin(\s E\tens_{\phi}B,\s F)$ such that
$U(\xi\tens_{\phi}b)=\Phi(\xi)b$, for all $\xi\in\s E$ and $b\in B$.
\item Conversely, if $U\in\Lin(\s E\tens_{\phi}B,\s F)$ is a unitary, then there exists a unique $\phi$-compatible unitary operator $\Phi:\s E\rightarrow\s F$ such that $\Phi(\xi)b=U(\xi\tens_{\phi}b)$ for all $\xi\in\s E$ and $b\in B$.\qedhere
\end{enumerate}
\end{prop}

As an application of the above proposition, we can the state the following result.

\begin{propdef}\label{propdef6}
Let $A_1$, $B_1$, $A_2$ and $B_2$ be C*-algebras, $\phi_1:A_1\rightarrow B_1$ and $\phi_2:A_2\rightarrow B_2$ *-isomorphisms. Let $\s E_1$, $\s F_1$, $\s E_2$ and $\s F_2$ be Hilbert C*-modules over $A_1$, $B_1$, $A_2$ and $B_2$ respectively. Let $\Phi_1:\s E_1 \rightarrow \s F_1$ and $\Phi_2:\s E_2 \rightarrow \s F_2$ be unitary equivalences over $\phi_1$ and $\phi_2$ respectively. Then, the linear map
$
\s E_1 \odot \s E_2 \rightarrow \s F_1\tens \s F_2\; ;\; \xi_1 \tens \xi_2 \mapsto \Phi_1(\xi_1) \tens \Phi_2(\xi_2)
$
extends to a bounded linear map $\Phi_1\tens\Phi_2:\s E_1 \tens \s E_2\rightarrow\s F_1 \tens \s F_2$. Moreover, $\Phi_1\tens \Phi_2$ is a $\phi_1\tens\phi_2$-compatible unitary operator.
\end{propdef}

The notion of unitary equivalence can also be understood in terms of isomorphism between the associated linking C*-algebras.

\begin{prop}\label{prop7}
Let $A$ and $B$ be two C*-algebras and $\phi:A\rightarrow B$ a *-isomorphism. Let $\s E$ and $\s F$ be two Hilbert C*-modules over $A$ and $B$ respectively. 
\begin{enumerate}
\item If $\Phi:\s E\rightarrow\s F$ is a $\phi$-compatible unitary operator, then there exists a unique *-homomorphism 
$
f:\K(\s E\oplus A)\rightarrow\K(\s F\oplus B)
$
such that 
$f\circ\iota_{\s E}=\iota_{\s F}\circ \Phi$ and $f\circ\iota_A=\iota_B\circ\phi$.
Moreover, $f$ is a *-isomorphism.
\item Conversely, let $f:\K(\s E\oplus A)\rightarrow\K(\s F\oplus B)$ be a *-isomorphism such that $f\circ\iota_A=\iota_B\circ\phi$. Then, there exists a unique map $\Phi:\s E\rightarrow\s F$ such that $f\circ\iota_{\s E}=\iota_{\s F}\circ \Phi$. Moreover, $\Phi$ is a $\phi$-compatible unitary operator.\qedhere
\end{enumerate}
\end{prop}

\begin{proof}
1. The *-homomorphism $f:\K(\s E\oplus A)\rightarrow\K(\s F\oplus B)$ is defined by (cf.\ \ref{rk4} 3):
\begin{center}
$
f\begin{pmatrix} k & \xi \\ \eta^* & a\end{pmatrix}:=\begin{pmatrix} \Phi\circ k\circ \Phi^{-1} & \Phi\xi \\ (\Phi\eta)^* & \phi(a)\end{pmatrix}\!,\quad \text{for all } k\in\K(\s E),\; \xi,\,\eta\in\s E \text{ and } a\in A.
$
\end{center}
2. This is a straightforward consequence of \ref{ehmlem1} 1.
\end{proof}

\begin{nb}\label{not3}
Let $A$, $B$ be C*-algebras and $\s E$ and $\s F$ be two Hilbert C*-modules over $A$ and $B$ respectively. Let $\phi:A\rightarrow B$ be a *-isomorphism and $\Phi:\s E\rightarrow\s F$ a $\phi$-compatible unitary operator. If $T\in\Lin(A,\s E)$, we define the map 
$
\widetilde{\Phi}(T):= \Phi\circ T\circ\phi^{-1}:B\rightarrow\s F.
$
By a straightforward computation, we show that $\widetilde{\Phi}(T)\in\Lin(B,\s F)$ whose adjoint operator is $\widetilde{\Phi}(T)^*=\phi\circ T^*\circ \Phi^{-1}$. We have a bounded linear map 
$
\widetilde{\Phi}:\Lin(A,\s E)\rightarrow\Lin(B,\s F),
$
which is an extension of $\Phi$ up to the canonical injections $\s E\rightarrow\Lin(A,\s E)$ and $\s F\rightarrow\Lin(B,\s F)$.
\end{nb}


\bigbreak

\bigbreak

\noindent{\sc Vrije Universiteit Brussel}, Vakgroep Wiskunde, Pleinlaan 2, B-1050 Brussels (Belgium).\hfill\break 
{\it E-mail addresses}: \texttt{jonathan.crespo@wanadoo.fr}, \texttt{jonathan.crespo@vub.ac.be}\hfill\break
Supported by the FWO grant G.0251.15N

\printindex[symbol]
\addcontentsline{toc}{section}{Index of notations and symbols}

\end{document}